\documentclass[reqno,10pt]{amsart}

\usepackage{amsmath,amssymb,amsthm,mathrsfs,enumitem}
\usepackage{upgreek}
\usepackage{slashed}
\usepackage{xifthen}
\usepackage{graphicx}
\usepackage{longtable}

\usepackage{xcolor}
\definecolor{winered}{rgb}{0.7,0,0}
\definecolor{lessblue}{rgb}{0,0,0.7}

\usepackage[pdftex,colorlinks=true,linkcolor=winered,citecolor=lessblue,breaklinks=true,bookmarksopen=true]{hyperref}

\usepackage{titletoc}

\makeatletter

\def\@tocline#1#2#3#4#5#6#7{
\begingroup
  \par
    \parindent\z@ \leftskip#3 \relax \advance\leftskip\@tempdima\relax
                  \rightskip\@pnumwidth plus 4em \parfillskip-\@pnumwidth
    \ifcase #1 
       \vskip 0.6em \hskip 0em 
       \or
       \or \hskip 0em 
       \or \hskip 1em 
    \fi%
    #6
    \nobreak\relax{\leavevmode\leaders\hbox{\,.}\hfill}
    \hbox to\@pnumwidth {\@tocpagenum{#7}}
  \par
\endgroup
}

 \def\l@section{\@tocline{0}{0pt}{0pc}{}{}}

\renewcommand{\tocsection}[3]{%
  \indentlabel{\@ifnotempty{#2}{ 
    \ignorespaces\bfseries{#2. #3}}}
  \indentlabel{\@ifempty{#2}{\ignorespaces\bfseries{#3}}{}} 
    \vspace{1.5pt}}

\renewcommand{\tocsubsection}[3]{%
  \indentlabel{\@ifnotempty{#2}{
    \ignorespaces#2. #3}}
  \indentlabel{\@ifempty{#2}{\ignorespaces #3}{}}
    \vspace{1.5pt}}

\renewcommand{\tocsubsubsection}[3]{%
  \indentlabel{\@ifnotempty{#2}{
    \ignorespaces#2. #3}}
  \indentlabel{\@ifempty{#2}{\ignorespaces #3}{}}
    \vspace{1.5pt}}

\makeatother

\makeatletter
\def\@nomenstarted{0}
\newlength{\@nomenoldtabcolsep}

\newcommand{\nomenstart}
  {%
    \def\@nomenstarted{1}%
    \setlength{\@nomenoldtabcolsep}{\tabcolsep}%
    \setlength{\tabcolsep}{3.5pt}%
    \begin{longtable}{p{0.12\textwidth} p{0.84\textwidth}}
  }

\newcommand{\nomenitem}[2]{%
    \ifcase\@nomenstarted%
      \or 
      \or \\ 
    \fi%
    #1\,{\leavevmode\leaders\hbox{\,.}\hfill} & #2%
    \def\@nomenstarted{2}%
  }%
\newcommand{\nomenend}
  {\\%
      \end{longtable}%
      \setlength{\tabcolsep}{\@nomenoldtabcolsep}%
      \def\@nomenstarted{0}%
  }
\makeatother

\allowdisplaybreaks

\numberwithin{equation}{section}
\newtheorem{thm}{Theorem}[section]
\newtheorem{prop}[thm]{Proposition}
\newtheorem{lemma}[thm]{Lemma}
\newtheorem{cor}[thm]{Corollary}
\newtheorem*{thm*}{Theorem}
\newtheorem*{prop*}{Proposition}
\newtheorem*{cor*}{Corollary}
\newtheorem*{conj*}{Conjecture}

\theoremstyle{definition}
\newtheorem{definition}[thm]{Definition}

\theoremstyle{remark}
\newtheorem{rmk}[thm]{Remark}

\usepackage{chngcntr}
\counterwithin{figure}{section}

\newcommand{\mc}{\mathcal}
\newcommand{\cA}{\mc A}
\newcommand{\cB}{\mc B}
\newcommand{\cC}{\mc C}
\newcommand{\cD}{\mc D}

\newcommand{\cH}{\mc H}
\newcommand{\cI}{\mc I}

\newcommand{\cL}{\mc L}
\newcommand{\cM}{\mc M}

\newcommand{\cO}{\mc O}
\newcommand{\cP}{\mc P}

\newcommand{\cR}{\mc R}

\newcommand{\cU}{\mc U}
\newcommand{\cV}{\mc V}

\newcommand{\cX}{\mc X}

\newcommand{\cZ}{\mc Z}
\newcommand{\cCH}{\cC\cH}

\newcommand{\ms}{\mathscr}
\newcommand{\sB}{\ms B}
\newcommand{\sC}{\ms C}
\newcommand{\sD}{\ms D}

\newcommand{\sE}{\ms E}

\newcommand{\sI}{\ms I}
\newcommand{\sL}{\ms L}

\newcommand{\sR}{\ms R}

\newcommand{\C}{\mathbb{C}}
\newcommand{\N}{\mathbb{N}}
\newcommand{\R}{\mathbb{R}}

\newcommand{\Sph}{\mathbb{S}}

\newcommand{\vect}{\mathbb{V}}

\newcommand{\ulR}{\ul{\R}}

\newcommand{\bfzero}{\mathbf{0}}
\newcommand{\bfa}{\mathbf{a}}
\newcommand{\bfc}{\mathbf{c}}

\newcommand{\bfB}{\mathbf{B}}
\newcommand{\bfE}{\mathbf{E}}
\newcommand{\bfX}{\mathbf{X}}

\newcommand{\frakg}{\mathfrak{g}}

\newcommand{\sld}{\slashed{d}{}}
\newcommand{\slg}{\slashed{g}{}}
\newcommand{\sln}{\slashed{n}{}}
\newcommand{\slG}{\slashed{G}{}}
\newcommand{\slH}{\slashed{H}{}}
\newcommand{\slGamma}{\slashed{\Gamma}{}}
\newcommand{\sldelta}{\slashed{\delta}{}}
\newcommand{\slDelta}{\slashed{\Delta}{}}
\newcommand{\slnabla}{\slashed{\nabla}{}}
\newcommand{\slpi}{\slashed{\pi}{}}
\newcommand{\slRic}{\slashed{\Ric}{}}
\newcommand{\slstar}{\slashed{\star}}
\newcommand{\sltr}{\operatorname{\slashed{\tr}}}

\newcommand{\hg}{{\wh{g}}{}}
\newcommand{\hG}{{\wh{G}}{}}

\newcommand{\hL}{{\wh{L}}{}}
\newcommand{\hR}{\wh{R}{}}
\newcommand{\hX}{{\wh{X}}}
\newcommand{\hBox}{\wh{\Box}{}}
\newcommand{\hGamma}{\wh{\Gamma}{}}
\newcommand{\hdelta}{\wh{\delta}{}}
\newcommand{\tcL}{\wt{\cL}{}}
\newcommand{\hnabla}{\wh{\nabla}{}}
\newcommand{\hpi}{\wh{\pi}{}}
\newcommand{\hRic}{\wh{\mathrm{Ric}}{}}
\newcommand{\hstar}{\wh{\star}}
\newcommand{\htr}{\operatorname{\wh{\tr}}}
\newcommand{\hvol}{\wh{\mathrm{vol}}{}}

\newcommand{\asph}{{\mathrm{AS}}}
\newcommand{\TAS}{T_\asph}
\newcommand{\TS}{T_{\mathrm{S}}}

\newcommand{\ran}{\operatorname{ran}}

\newcommand{\End}{\operatorname{End}}

\renewcommand{\Re}{\operatorname{Re}}
\renewcommand{\Im}{\operatorname{Im}}
\newcommand{\Id}{\operatorname{Id}}

\newcommand{\mathspan}{\operatorname{span}}
\newcommand{\supp}{\operatorname{supp}}
\newcommand{\sgn}{\operatorname{sgn}}

\newcommand{\tr}{\operatorname{tr}}

\newcommand{\diag}{\operatorname{diag}}
\newcommand{\Res}{\operatorname{Res}}

\newcommand{\IVP}{{\mathrm{IVP}}}
\newcommand{\Ups}{\Upsilon}

\newcommand{\eps}{\epsilon}

\newcommand{\hra}{\hookrightarrow}
\newcommand{\la}{\langle}

\newcommand{\ol}{\overline}
\newcommand{\pa}{\partial}
\newcommand{\ra}{\rangle}

\newcommand{\tn}{\textnormal}

\newcommand{\ul}{\underline}
\newcommand{\upd}{\updelta}

\newcommand{\wh}{\widehat}
\newcommand{\wt}{\widetilde}
\newcommand{\xra}{\xrightarrow}

\newcommand{\gdot}{\dot{g}}
\newcommand{\hdot}{\dot{h}}
\newcommand{\kdot}{\dot{k}}
\newcommand{\qdot}{\dot{q}}

\newcommand{\Adot}{\dot{A}}
\newcommand{\Fdot}{\dot{F}}
\newcommand{\Qdot}{\dot{Q}}
\newcommand{\Xdot}{\dot{X}}
\newcommand{\Ydot}{\dot{Y}}
\newcommand{\Zdot}{\dot{Z}}

\newcommand{\mudot}{\dot{\mu}}
\newcommand{\phidot}{\dot{\phi}}

\newcommand{\bfBdot}{\dot{\bfB}}
\newcommand{\bfEdot}{\dot{\bfE}}

\newcommand{\bop}{{\mathrm{b}}}
\newcommand{\bl}{{\mathrm{b}}}

\newcommand{\cp}{{\mathrm{c}}}

\newcommand{\semi}{\hbar}

\newcommand{\Diff}{\mathrm{Diff}}

\newcommand{\Vf}{\mathcal V}

\newcommand{\Vb}{\Vf_\bop}
\newcommand{\Diffb}{\Diff_\bop}

\newcommand{\Diffbh}{\Diff_{\bop,\semi}}

\newcommand{\Omegab}{{}^{\bop}\Omega}

\newcommand{\TC}{{}^{\C}T}
\newcommand{\Tb}{{}^{\bop}T}

\newcommand{\rcTbdual}[1][]{\ensuremath{\overline{{}^{\bop}T^*\ifthenelse{\isempty{#1}}{}{_#1}}}}

\newcommand{\Sb}{{}^{\bop}S}
\newcommand{\Nb}{{}^{\bop}N}

\newcommand{\bdiff}{{}^{\bop}d}

\newcommand{\half}{\frac{1}{2}}

\newcommand{\sigmabh}{\sigma_{\bop,\semi}}

\newcommand{\ham}{H}

\newcommand{\sub}{{\mathrm{sub}}}

\newcommand{\bhm}[1][]{\ensuremath{M_{\bullet\ifthenelse{\isempty{#1}}{}{,#1}}}}

\newcommand{\CI}{\cC^\infty}
\newcommand{\CIdot}{\dot\cC^\infty}

\newcommand{\CIc}{\cC^\infty_\cp}

\newcommand{\Hb}{H_{\bop}}

\newcommand{\Hbext}{\bar H_{\bop}}
\newcommand{\Hbsupp}{\dot H_{\bop}}
\newcommand{\Hext}{\bar H}

\newcommand{\Hbh}{H_{\bop,h}}

\newcommand{\hd}{\wh{d}{}}
\newcommand{\hdel}{\wh{\delta}{}}
\newcommand{\td}{\wt{d}{}}
\newcommand{\tdel}{\wt{\delta}{}}

\newcommand{\Riem}{\mathrm{Riem}}
\newcommand{\Ric}{\mathrm{Ric}}
\newcommand{\Ein}{\mathrm{Ein}}

\newcommand{\openbigpmatrix}[1]
  {%
    \def\@bigpmatrixsize{#1}%
    \addtolength{\arraycolsep}{-#1}%
    \begin{pmatrix}%
  }
\newcommand{\closebigpmatrix}
  {%
    \end{pmatrix}%
    \addtolength{\arraycolsep}{\@bigpmatrixsize}%
  }

\newcommand{\setarraystretch}{\def\arraystretch{1.5}}

\newcommand{\itref}[1]{(\ref{#1})}

\newcommand{\KI}[2]{(#1$\cdot$#2)}

\DeclareGraphicsExtensions{.mps}
\newcommand{\inclfig}[1]{\includegraphics{#1}}

\begin{document}

\title[Non-linear stability of Kerr--Newman--de~Sitter]{Non-linear stability of the Kerr--Newman-- de~Sitter family of charged black holes}

\author{Peter Hintz}
\address{Department of Mathematics, University of California, Berkeley, CA 94720-3840, USA}
\email{phintz@berkeley.edu}

\date{December 13, 2016. Final revision: January 11, 2018.}

\subjclass[2010]{Primary 83C57, Secondary 83C22, 35B40, 83C35}

\begin{abstract}
  We prove the global non-linear stability, without symmetry assumptions, of slowly rotating charged black holes in de~Sitter spacetimes in the context of the initial value problem for the Einstein--Maxwell equations: if one perturbs the initial data of a slowly rotating Kerr--Newman--de~Sitter (KNdS) black hole, then in a neighborhood of the exterior region of the black hole, the metric and the electromagnetic field decay exponentially fast to their values for a possibly different member of the KNdS family. This is a continuation of recent work of the author with Vasy on the stability of the Kerr--de~Sitter family for the Einstein vacuum equations. Our non-linear iteration scheme automatically finds the final black hole parameters as well as the gauge in which the global solution exists; we work in a generalized wave coordinate/Lorenz gauge, with gauge source functions lying in a suitable finite-dimensional space.

  We include a self-contained proof of the linear mode stability of Reissner--Nordstr\"om--de~Sitter black holes, building on work by Kodama--Ishibashi. In the course of our non-linear stability argument, we also obtain the first proof of the linear (mode) stability of slowly rotating KNdS black holes using robust perturbative techniques.
\end{abstract}

\maketitle

\tableofcontents

\section{Introduction}
\label{SecIntro}

The Einstein--Maxwell system describes the interaction of gravity and electromagnetism in the context of Einstein's Theory of General Relativity. On a $4$-manifold $M^\circ$, this system consists of coupled equations for the Lorentzian metric $g$ with signature $(+,-,-,-)$ and the electromagnetic field $F$, a 2-form, on $M^\circ$:
\begin{equation}
\label{EqIntroEM}
  \Ric(g) + \Lambda g = 2 T(g,F), \quad d F=0,\quad \delta_g F = 0.
\end{equation}
Here, $\Lambda$ is the cosmological constant, which we take to be positive. (This is consistent with the currently accepted $\Lambda$CDM model of cosmology \cite{RiessEtAlLambda,PerlmutterEtAlLambda}.) Moreover, $T$ is the energy-momentum tensor associated with the electromagnetic field $F$, and $\delta_g$ is the codifferential. Thus, the first equation in \eqref{EqIntroEM}, Einstein's field equation, describes the dynamics of the spacetime metric $g$ in the presence of an electromagnetic field $F$, and the second equation, Maxwell's equation, describes the propagation of electromagnetic waves on this spacetime. In the absence of an electromagnetic field, the system \eqref{EqIntroEM} reduces to the Einstein vacuum equation $\Ric(g)+\Lambda g=0$.

Equation~\eqref{EqIntroEM} is very close in character to a coupled system of wave equations; it is however not quite such a system due to its gauge invariance: pulling a solution $(g,F)$ back by any diffeomorphism of $M^\circ$ produces another solution of \eqref{EqIntroEM}. The correct formulation of the non-characteristic initial value problem is therefore somewhat subtle; this was first accomplished by Choquet-Bruhat \cite{ChoquetBruhatLocalEinstein} for the vacuum Einstein equations, see also \cite{ChoquetBruhatGerochMGHD}. For the Einstein--Maxwell system, an initial data set
\[
  (\Sigma_0,h,k,\bfE,\bfB)
\]
consists of a smooth $3$-manifold $\Sigma_0$, a Riemannian metric $h$, a symmetric 2-tensor $k$, and 1-forms $\bfE$ and $\bfB$ on $\Sigma_0$, subject to the constraint equations; these are the Gauss--Codazzi equations for a metric subject to \eqref{EqIntroEM}, together with equations for $\bfE$ and $\bfB$ asserting that the electromagnetic field is source-free. Given such data, one can find a unique (up to gauge equivalence) solution $(M^\circ,g,F)$ of \eqref{EqIntroEM}, with an embedding of $\Sigma_0$ into $M^\circ$ as a spacelike hypersurface, attaining these data at $\Sigma_0$ in the sense that $h$ and $k$ are the metric and second fundamental form on $\Sigma_0$ induced by $g$, and $\bfE$ and $\bfB$ are the electric and magnetic field, encoded by the electromagnetic 2-form $F$, as measured by an observer crossing $\Sigma_0$ in the normal direction at unit speed. See \cite[\S6.10]{ChoquetBruhatGR} and \S\ref{SecBasic} below for a detailed discussion.

\subsection{Statement of the main result}
\label{SubsecIntroNL}

We analyze the global behavior of solutions of the initial value problem for \eqref{EqIntroEMA} with initial data which are close to those of a slowly rotating Kerr--Newman--de~Sitter (KNdS) black hole. The KNdS family of solutions of \eqref{EqIntroEMA} describes stationary, charged, and rotating black holes inside of a universe undergoing accelerated expansion (consistent with $\Lambda>0$); it was discovered by Carter \cite{CarterHamiltonJacobiEinstein}, following the discovery of the Kerr \cite{KerrKerr} and Kerr--Newman (KN) solutions \cite{NewmanCouchChinnaparedExtonPrakashTorrenceKN} describing neutral or charged rotating black holes in asymptotically flat spacetimes (for which $\Lambda=0$). Thus, fixing $\Lambda>0$, a KNdS solution $(g_b,F_b)$, defined on a manifold
\begin{equation}
\label{EqIntroNLManifold}
  M^\circ=\R_{t_*}\times(0,\infty)_r\times\Sph^2,
\end{equation}
is characterized by its parameters
\[
  b=(\bhm,\bfa,Q_e,Q_m) \in \R^6,
\]
with $\bhm>0$ denoting the mass of the black hole, $\bfa\in\R^3$ its angular momentum (direction and magnitude), and $Q_e,Q_m\in\R$ its electric and magnetic charge. KNdS solutions are stationary, i.e.\ translations in the time variable $t_*$ are isometries, and axisymmetric around the axis of rotation $\bfa/|\bfa|$ if $\bfa\neq 0$. (If $Q_e=Q_m=0$, then the electromagnetic field vanishes, and the metric $g_b$ reduces to a Kerr--de~Sitter (KdS) metric with parameters $(\bhm,\bfa)$; if furthermore $\bfa=\bfzero$, the metric $g_b$ is a spherically symmetric Schwarzschild--de~Sitter (SdS) metric.) A non-degenerate KNdS black hole has an event horizon $\cH^+$ at a radius $r=r_{b,-}>0$ and a cosmological horizon $\ol\cH{}^+$ at $r=r_{b,+}>r_{b,-}$. In the vicinity of a KNdS black hole, i.e.\ near its event horizon, the effect of the cosmological constant is small, and the black hole behaves very much like a KN black hole with the same mass, charge and angular momentum, while far away, near and beyond the cosmological horizon, the KNdS geometry is close to the geometry of (a static patch of) de Sitter space, with the same cosmological constant. See Figure~\ref{FigIntroNLPenrose}. The Reissner--Nordstr\"om--de~Sitter (RNdS) family is the subfamily of non-rotating (spherically symmetric) charged black holes, which thus have parameters $b_0=(\bhm,\bfzero,Q_e,Q_m)$.

\begin{figure}[!ht]
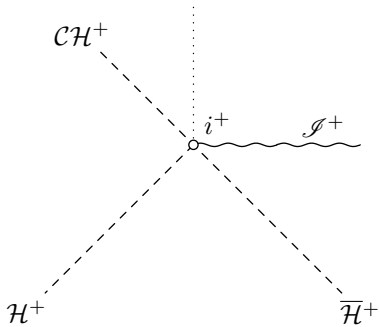

\centering
\inclfig{IntroNLPenrose}
\caption{Part of the Penrose diagram of a maximally extended Kerr--Newman--de~Sitter metric. Near and inside of the event horizon $\cH^+$, the black hole resembles a Kerr--Newman black hole (with Cauchy horizon denoted $\cCH^+$), while near and beyond the cosmological horizon $\ol\cH{}^+$, it resembles a de~Sitter spacetime with conformal boundary $\sI^+$.}
\label{FigIntroNLPenrose}
\end{figure}

\subsubsection{Black hole stability}
\label{SubsubsecIntroNLStab}

We fix the domain
\[
  \Omega^\circ=\{t_*\geq 0, r_< \leq r \leq r_>\}\subset M^\circ,
\]
with $r_< <r_{b_0,-}$ and $r_> >r_{b_0,+}$, which thus contains a neighborhood of the exterior region $r_{b_0,-}<r<r_{b_0,+}$ of a fixed non-degenerate RNdS spacetime. The RNdS solution $(g_{b_0},F_{b_0})$ induces initial data $(h_{b_0},k_{b_0},\bfE_{b_0},\bfB_{b_0})$ on
\[
  \Sigma_0=\{t_*=0\}\subset\Omega^\circ.
\]
See Figure~\ref{FigIntroNLStabPenrose}.

\begin{figure}[!ht]
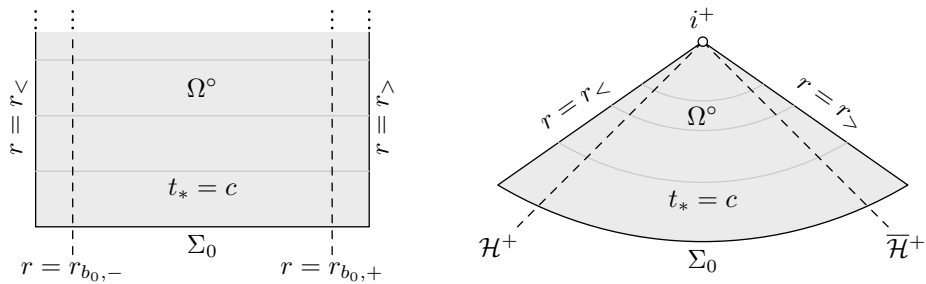

\centering
\inclfig{IntroNLStabPenrose}
\caption{\textit{Left:} The domain $\Omega^\circ$ on which we will solve the initial value problem for the Einstein--Maxwell system, with initial surface $\Sigma_0$. The location of the horizons of the RNdS solution with parameters $b_0$ is indicated by dashes. \textit{Right:} The same domain, drawn as a subset of the Penrose compactification. The artificial boundaries at $r=r_<$ and $r=r_>$ are spacelike.}
\label{FigIntroNLStabPenrose}
\end{figure}

Our main result is the \emph{full global non-linear stability of slowly rotating Kerr--Newman--de~Sitter black holes} in the neighborhood $\Omega^\circ$ of the black hole exterior, \emph{without any symmetry assumptions} on the data:

\begin{thm}
\label{ThmIntroNL}
  Suppose $(h,k,\bfE,\bfB)$ are smooth initial data on $\Sigma_0$ which satisfy the constraint equations and which are close (in a high regularity norm) to the data $(h_{b_0},k_{b_0},\bfE_{b_0},\bfB_{b_0})$ of a non-degenerate Reissner--Nordstr\"om--de~Sitter solution. Then there exist a smooth solution $(g,F)$ of the Einstein--Maxwell system \eqref{EqIntroEM} on $\Omega^\circ$, attaining the given initial data, and Kerr--Newman--de~Sitter black hole parameters $b\in\R^6$ close to $b_0$ such that
  \begin{equation}
  \label{EqIntroNLSol}
    g = g_b + \wt g,\quad F = F_b + \wt F,
  \end{equation}
  where
  \[
    \wt g,\ \wt F = \cO(e^{-\alpha t_*}),
  \]
  with $\alpha>0$ only depending on the parameters $b_0$ of the black hole we are perturbing. That is, the metric $g$ and the electromagnetic field $F$ decay exponentially fast to the stationary Kerr--Newman--de~Sitter metric $g_b$ and electromagnetic field $F_b$.
\end{thm}

Thus, we only consider the future development of initial data, but Theorem~\ref{ThmIntroNL} applies \emph{mutatis mutandis} to the past development of perturbations of RNdS data provided the past domain of dependence of $\Sigma_0$ within RNdS contains a neighborhood of past timelike infinity $i^-$; an example of such an initial surface is labelled $\Sigma_i$ in Figure~\ref{FigNLIniSurface}. Phrased more geometrically, Theorem~\ref{ThmIntroNL} states that the maximal globally hyperbolic development of the given initial data contains a region of the form indicated on the right in Figure~\ref{FigIntroNLStabPenrose} (we do not construct the horizons $\cH^+$ and $\ol\cH{}^+$ here, however).

See Theorem~\ref{ThmIntroNLFull} for the detailed statement. The iteration scheme which we use to construct the solution $(g,F)$ automatically finds the final black hole parameters $b$. Theorem~\ref{ThmIntroNL} extends the earlier result on the non-linear stability of slowly rotating Kerr--de~Sitter black holes \cite{HintzVasyKdSStability}, obtained in joint work with Vasy, to the coupled Einstein--Maxwell system; if $\bfE=0$ and $\bfB=0$, then Theorem~\ref{ThmIntroNL} reduces to \cite[Theorem~1.1]{HintzVasyKdSStability}. Prior to the present work, even the \emph{mode stability} for coupled gravitational and electromagnetic \emph{linearized} perturbations of rotating KNdS (or KN) black holes was not known; we discuss this further in \S\ref{SubsecIntroLin}.

There is no experimental evidence for the existence of magnetic charges in our universe. In the context of Theorem~\ref{ThmIntroNL}, the magnetic charge of the solution is $Q_m=\frac{1}{4\pi}\int_S F$, where $S$ is any 2-sphere homologous to the sphere $S_0$ defined by $t_*=0$, $r=r_0\in(r_<,r_>)$. In particular, $Q_m$ can be computed on the level of initial data, namely $Q_m=\frac{1}{4\pi}\int_{S_0}\star_h\bfB$; if the initial data have vanishing magnetic charge, then the final black hole will have vanishing magnetic charge as well. Now $F\mapsto\frac{1}{4\pi}\int_S F$ induces an isomorphism of the cohomology group $H^2(\Sigma_0,\R)\cong\R$; thus, the vanishing of $Q_m$ is a cohomological condition which is equivalent to the condition that $F=d A$ for an electromagnetic 4-potential $A$ on $\Omega^\circ$. Thus, restricting to the setting in which there are no magnetic charges, the system \eqref{EqIntroEM} simplifies to
\begin{equation}
\label{EqIntroEMA}
  \Ric(g)+\Lambda g = 2 T(g,d A),\quad \delta_g d A = 0,
\end{equation}
which we continue to call the Einstein--Maxwell system. The formulation of its initial value problem uses the same data as the formulation for \eqref{EqIntroEM}, with the additional condition that there are no magnetic charges, i.e.\ that $\star_h\bfB$ is trivial in the cohomology group $H^2(\Sigma_0,\R)$. One then solves for the 4-potential $A$, with the electromagnetic field being $F=d A$. The introduction of $A$ gives rise to an additional gauge invariance: if $(g,A)$ solves the system \eqref{EqIntroEMA}, then adding any exact 1-form $d a$, $a\in\CI(\Omega^\circ)$, to $A$ gives another solution of \eqref{EqIntroEMA}. This will prove to be very useful, as it allows for more flexibility in the formulation of a gauge-fixed version of \eqref{EqIntroEMA}; see \S\ref{SubsecIntroStr}. Our proof of Theorem~\ref{ThmIntroNL} automatically finds a generalized wave map/Lorenz gauge within a suitable finite-dimensional family of gauges in which one can find the global solution $(g,A)$ of the initial value problem. 

In the context of Theorem~\ref{ThmIntroNL}, the fact that $H^2(\Sigma_0,\R)\cong\R$ allows one to use the electric-magnetic duality of the Einstein--Maxwell system to reduce the study of the initial value problem for the general equation \eqref{EqIntroEM} with data close to (magnetically charged) RNdS data to the study of the 4-potential formulation \eqref{EqIntroEMA} with data close to RNdS data without magnetic charge; see \S\ref{SubsecBasicDer}. \emph{For the remainder of this introduction, we will assume that the magnetic charge vanishes, and consequently only discuss the formulation \eqref{EqIntroEMA}.} The KNdS family without magnetic charge then consists of pairs $(g_b,A_b)$, with $F_b=d A_b$.

Like the magnetic charge, the electric charge $Q_e$ of the solution \eqref{EqIntroNLSol}, thus of the final state $(g_b,A_b)$, can also be computed on the level of initial data by means of $Q_e=\frac{1}{4\pi}\int_S\star_g F=\frac{1}{4\pi}\int_S\star_h\bfE$. Therefore, if the initial data are free of charges (which requires the data to be close to Schwarzschild--de~Sitter data), then the final black hole will be a slowly rotating Kerr--de~Sitter black hole, and the electromagnetic field decays exponentially. (However, our proof of Theorem~\ref{ThmIntroNL} does not simplify for uncharged data with $\bfE$ or $\bfB$ non-vanishing, see Remark~\ref{RmkNLChargeConverges}.)

From a physical perspective, the most interesting setting for Theorem~\ref{ThmIntroNL} arises when the initial data have very small (or vanishing) charge, so $|Q_e/\bhm|\ll 1$. (A black hole with higher charge-to-mass ratio would selectively attract particles of the opposite charge, see \cite[\S{12.3}]{WaldGR}.) We stress that we make \emph{no restrictions on the charge} here, apart from the assumption of non-degeneracy, as our methods extend easily to the case of large charges. The restriction to small angular momenta is more serious; see the discussion in \cite[Remark~1.5]{HintzVasyKdSStability}, which also applies in the context of the present paper. One benefit of not assuming $Q_e$ to be small is that it allows for the study of near-extremal black holes while staying close to spherical symmetry, which opens the door to a quantitative study of Penrose's Strong Cosmic Censorship conjecture in this context; see Remark~\ref{RmkIntroScalarField} for further details. Note that while the size of the perturbation of RNdS initial data allowed by Theorem~\ref{ThmIntroNL} is not specified explicitly, it must shrink to $0$ as the RNdS parameters $b_0$ approach extremality: our methods fail for extremal black holes, where in fact qualitatively different behavior is expected, see the work by Aretakis \cite{AretakisExtremalRN1,AretakisExtremalRN2}.

Theorem~\ref{ThmIntroNL} is the first definitive result on the stability of black holes under non-linearly coupled gravitational and electromagnetic perturbations; the only other non-linear stability result on black hole spacetimes known to the author is the aforementioned work \cite{HintzVasyKdSStability}. We point out however that Dafermos, Holzegel, and Rodnianski \cite{DafermosHolzegelRodnianskiSchwarzschildStability} recently proved the \emph{linear} stability of the Schwarzschild solution under linearized gravitational perturbations. We also mention the recent experimental evidence \cite{LIGOBlackHoleMerger}.

The non-linear stability of de~Sitter space in the presence of Maxwell or Yang--Mills fields was proved by Friedrich \cite{FriedrichEinsteinMaxwellYangMills}, following his earlier work on the Einstein vacuum equations \cite{FriedrichStability}, see also \cite{AndersonStabilityEvenDS}; Ringstr\"om \cite{RingstromEinsteinScalarStability} discusses a closely related problem for (non-linear) scalar fields. Based on the monumental work by Christodoulou and Klainerman \cite{ChristodoulouKlainermanStability} on the stability of Minkowski space as a solution of the Einstein vacuum equation $\Ric(g)=0$, with refinements by Bieri \cite{BieriZipserStability}, Zipser \cite{ZipserThesis} proved the stability of Minkowski space as a solution of the system \eqref{EqIntroEM} with $\Lambda=0$: Minkowski space itself solves this system with $F\equiv 0$. Lindblad and Rodnianski \cite{LindbladRodnianskiGlobalExistence,LindbladRodnianskiGlobalStability} and Speck \cite{SpeckEinsteinMaxwell} gave different proofs of (variants of) the latter two results, based on wave coordinate gauges. The non-linear stability of Kerr--Newman black holes was investigated numerically in \cite{ZilhaoCardosoHerdeiroLehnerSperhakeKNStab}, including in near-extremal regimes (large angular momenta and/or large charges); no signs of developing instabilities were found. Since such numerical simulations are very expensive, the parameter space explored is rather small. The mode stability of Kerr--Newman solutions for a wider range of parameters was checked numerically in \cite{DiasGodazgarSantosKerrNewman}. We refer the reader to \S\ref{SubsecIntroWork} for further pointers to the literature. 

\begin{rmk}
\label{RmkIntroScalarField}
  Using the techniques developed in this paper, one can couple a massless neutral scalar field into the Einstein--Maxwell system and prove the global non-linear stability of the slowly rotating KNdS family with constant scalar field. (We remark that for the Einstein--Maxwell--scalar field system, there are non-trivial dynamics even in spherical symmetry.) In order to study Penrose's Strong Cosmic Censorship conjecture in either of these contexts, it is necessary to provide a (near) sharp bound for the decay rate $\alpha$, see \cite{DafermosEinsteinMaxwellScalarStability,CostaGiraoNatarioSilvaCauchy1,CostaGiraoNatarioSilvaCauchy2,CostaGiraoNatarioSilvaCauchy3,HintzVasyCauchyHorizon,DafermosLukKerrCauchyHorI}. Our arguments can in fact easily be seen to provide a value for $\alpha$ in terms of the spectral gaps of certain wave equations (which in turn can be computed numerically); however, we do not obtain any \emph{explicit} bounds here.
  
  In fact, one could allow the scalar field to have non-zero mass and show the non-linear stability of the KNdS family with vanishing scalar field (as solutions to the scalar Klein--Gordon equation decay exponentially fast to $0$). Shlapentokh-Rothman's work on `black hole bombs' \cite{ShlapentokhRothmanBlackHoleBombs} for linear scalar Klein--Gordon equations on Kerr spacetimes suggests that non-linear stability may fail for certain values of the scalar field mass and the black hole angular momentum.
\end{rmk}

The proof of Theorem~\ref{ThmIntroNL} builds on the framework developed in \cite{HintzVasyKdSStability} for the study of the Einstein vacuum equations; we will recall this framework in \S\ref{SubsecIntroStr} below. The study of the coupled Einstein--Maxwell system presents a number of additional difficulties which we resolve in this paper:

\begin{enumerate}
\item In addition to the diffeomorphism invariance, which for the Einstein vacuum equations is eliminated by imposing a (generalized) wave coordinate gauge (see \cite{FriedrichHyperbolicityEinstein} for a very general version of this), one needs to fix a gauge for the electromagnetic 4-potential $A$ in order to transform the system \eqref{EqIntroEMA} into a quasilinear system of wave equations, called the gauge-fixed Einstein--Maxwell system.
\item In order to be able to relate non-decaying modes of the linearized gauge-fixed Einstein--Maxwell system to physical degrees of freedom, we need to establish \emph{constraint damping} for this system. Constraint damping, which under the name `stable constraint propagation' plays a central role in \cite{HintzVasyKdSStability}, first appeared in the numerics literature \cite{GundlachCalabreseHinderMartinConstraintDamping}, and was used in particular for the simulation of binary black hole mergers \cite{PretoriusBinaryBlackHole}.
\item The linearization of the system \eqref{EqIntroEMA}, as well as of the gauge-fixed system linearized around a RNdS solution, yields two decoupled equations on an SdS background if the black hole has vanishing charge $Q_e$, one being the linearized Einstein equation, the other being Maxwell's equation. For $Q_e\neq 0$ however, the equations are coupled already at the linearized level, which complicates the symbolic analysis at the trapped set (normally hyperbolic trapping) and at the horizons (radial point estimates).
\item We need to show the (generalized) mode stability for the linearization of \eqref{EqIntroEMA} around a non-degenerate RNdS solution: all generalized mode solutions are sums of linearized KNdS solutions and pure gauge solutions. In this paper, we give a full, self-contained proof of generalized mode stability, completing and extending work by Kodama and Ishibashi \cite{KodamaIshibashiMaster,KodamaIshibashiCharged}, which in turn relies on earlier work by Kodama, Ishibashi and Seto \cite{KodamaIshibashiSetoBranes} as well as Kodama and Sasaki \cite{KodamaSasakiPerturbation}.
\end{enumerate}

We postpone the detailed discussion of these issues, and the way by which we overcome them, to \S\ref{SubsecIntroStr}.

\subsubsection{The role of the cosmological constant}
\label{SubsubsecIntroNLLambda}

Working with $\Lambda>0$ simplifies several aspects of the analysis underlying the proof of Theorem~\ref{ThmIntroNL}; this is not due to the positivity of $\Lambda$ per se, but rather due to the fact that the de~Sitter black hole spacetimes have a convenient (asymptotically hyperbolic) structure far away from the black hole, namely they have a cosmological horizon $\ol\cH{}^+$. Here, we use the framework developed in a seminal paper by Vasy \cite{VasyMicroKerrdS} and extended by the author (partly in joint work with Vasy) \cite{HintzVasySemilinear,HintzQuasilinearDS,HintzVasyQuasilinearKdS}, in combination with analysis at the trapped set, starting with the important work of Wunsch--Zworski \cite{WunschZworskiNormHypResolvent}, and developed further by Dyatlov \cite{DyatlovSpectralGaps} and the author \cite{HintzVasyNormHyp,HintzPsdoInner}; see also \cite{NonnenmacherZworskiQuantumDecay} and \cite{HirschShubPughInvariantManifolds}. This followed earlier work by S\'a Barreto and Zworski \cite{SaBarretoZworskiResonances}, Bony and H\"afner \cite{BonyHaefnerDecay}, Melrose, S\'a Barreto, and Vasy \cite{MelroseSaBarretoVasyResolvent}, as well as Dyatlov \cite{DyatlovQNM,DyatlovQNMExtended,DyatlovAsymptoticDistribution}. (For a broader perspective on scattering theory and the role of resonances, see \cite{ZworskiResonanceReview}.) Using the structure of (slowly rotating) KNdS spacetimes, one can show that solutions of linear tensor-valued wave equations (with appropriate behavior at the trapped set) have an asymptotic expansion as $t_*\to\infty$ into finitely many (generalized) modes with $t_*$-dependence $e^{-i t_*\sigma}$ (times a power $t_*^j$, $j\in\N_0$), $\Im\sigma>-\alpha$, plus an \emph{exponentially decaying} $\cO(e^{-\alpha t_*})$ remainder term. In the setting of the linear stability statement, Theorem~\ref{ThmIntroLin} below, we will argue that the main term of such an asymptotic expansion is a stationary mode, i.e.\ with $\sigma=0$, corresponding to a linearized KNdS solution, which also allows us to find the final black hole parameters $b$ in the non-linear setting. In Theorem~\ref{ThmIntroNL} then, $\alpha>0$ is chosen so small that all other physical degrees of freedom decay faster than $e^{-\alpha t_*}$. We remark that it is an open problem to prove more precise asymptotics for the solution \eqref{EqIntroNLSol} in the spirit of the phenomenon of \emph{ringdown} for linear waves as in \cite{BonyHaefnerDecay,DyatlovAsymptoticDistribution}.

The exponential decay is also very helpful in the non-linear analysis, as it obviates the need to exploit any special structure of the non-linear terms of \eqref{EqIntroEMA}; this is in stark contrast to the case of semilinear or quasilinear wave equations on $(3+1)$-dimensional Minkowski space, where the null condition, introduced by Klainerman \cite{KlainermanNullCondition}, or weaker versions thereof play a crucial role in ensuring global existence.

The key difference between Kerr--Newman--de~Sitter ($\Lambda>0$) and Kerr--Newman ($\Lambda=0$) spacetimes from the point of view of stationary scattering theory (which is a key ingredient in the framework developed in \cite{HintzVasyKdSStability}) is the behavior at low frequencies: for KNdS spacetimes, this is closely related to scattering on asymptotically hyperbolic manifolds, while for Kerr--Newman spacetimes, the far end is asymptotically Euclidean, which leads to much more delicate low frequency behavior. We refer the reader to the introduction of \cite{VasyMicroKerrdS} for details and references.

The limit $\Lambda\to 0$ is thus rather singular, and Theorem~\ref{ThmIntroNL} by itself provides no information on the stability of Kerr--Newman black holes: as $\Lambda\to 0$, the exponential decay rate $\alpha$ is expected to tend to $0$ as well, and we do not have any explicit control on the size of the allowed departure of the initial data from RNdS data. The conjectured decay rate for gravitational perturbations of the Kerr spacetime is in fact polynomial with a fixed rate, see for instance the lecture notes by Dafermos and Rodnianski \cite{DafermosRodnianskiLectureNotes} as well as the works by Tataru \cite{TataruDecayAsympFlat} and Dafermos--Rodnianski--Shlapentokh-Rothman \cite{DafermosRodnianskiShlapentokhRothmanDecay}.

\subsection{Results for the linearized Einstein--Maxwell system}
\label{SubsecIntroLin}

In the course of the proof of Theorem~\ref{ThmIntroNL}, we establish the linear stability of slowly rotating KNdS black holes, which is the analogue in this setting of the linear stability result (for the linearized Einstein vacuum equations) for the Schwarzschild spacetime by Dafermos, Holzegel and Rodnianski \cite{DafermosHolzegelRodnianskiSchwarzschildStability}.

\begin{thm}
\label{ThmIntroLin}
  Fix a Kerr--Newman--de~Sitter solution $(\Omega^\circ,g_b,F_b)$, with $b$ close to the parameters of a non-degenerate Reissner--Nordstr\"om--de~Sitter spacetime. Suppose $(\hdot,\kdot,\bfEdot,\bfBdot)$ are high regularity initial data on $\Sigma$, that is, $\hdot$ and $\kdot$ are symmetric 2-tensors and $\bfEdot$ and $\bfBdot$ are 1-forms on $\Sigma_0$, which satisfy the linearization of the constraint equations around the data $(h_b,k_b,\bfE_b,\bfB_b)$. Then there exist a solution $(\gdot,\Fdot)$ of the linearization of the Einstein--Maxwell system \eqref{EqIntroEM} around $(g_b,F_b)$, attaining these data, and linearized black hole parameters $b'\in\R^6$ such that
  \begin{equation}
  \label{EqIntroLin}
    (\gdot,\Fdot) - \frac{d}{ds}(g_{b+s b'},F_{b+s b'})|_{s=0} = \cO(e^{-\alpha t_*}),
  \end{equation}
  where $\cL$ denotes the Lie derivative. That is, linearized coupled gravitational and electromagnetic perturbations of a slowly rotating KNdS black hole decay exponentially fast to a linearized KNdS solution.
\end{thm}

We will prove this theorem by fixing a wave map/Lorenz gauge, and showing that the solution of the gauge-fixed linearized Einstein--Maxwell system satisfies \eqref{EqIntroLin} after subtraction of a pure gauge solution $\cL_V(g_b,F_b)$, with $\cL_V$ the Lie derivative along a suitable vector field $V$.

\begin{thm}
\label{ThmIntroLinMode}
  Fix a KNdS solution $(\Omega^\circ,g_b,F_b)$ as above. Suppose $\Im\sigma\geq 0$, and
  \[
    (\gdot,\Fdot) = e^{-i\sigma t_*}(\gdot_0,\Fdot_0),
  \]
  with $(\gdot_0,\Fdot_0)\in\CI(\Omega^\circ;S^2 T^*\Omega^\circ\oplus T^*\Omega^\circ)$ independent of $t_*$, is a mode solution of the Einstein--Maxwell system \eqref{EqIntroEM} linearized around $(g_b,F_b)$. Then there exist parameters $b'\in\R^6$ and a vector field $V$ on $\Omega^\circ$ such that
  \[
    (\gdot,\Fdot) = \frac{d}{ds}(g_{b+s b'},F_{b+s b'})|_{s=0} + \cL_V(g_b,F_b).
  \]
  (If $\sigma\neq 0$, then $b'=0$.) The vector field $V$ can be taken to be of the form $\sum_{j=0}^k e^{-i\sigma t_*}t_*^j V_j$, $k\leq 1$, with $V_j$ vector fields on $\Omega^\circ$ which are independent of $t_*$.
\end{thm}

See Theorems~\ref{ThmLinKNdS} and \ref{ThmLinMode} for the precise statements. As alluded to in \S\ref{SubsubsecIntroNLStab}, the mode stability of KNdS black holes was not rigorously known before, likewise even for the algebraically simpler Kerr--Newman solution, see \cite[Chapter~11, \S{111}]{ChandrasekharBlackHoles}. In this paper, we show that the mode stability --- and in fact the full non-linear stability --- for small $a=|\bfa|$ can be deduced from the stability for $a=0$ by using a robust perturbative framework.

For asymptotically flat black hole spacetimes, the study of mode stability, and black hole perturbation theory in general, was initiated by Regge and Wheeler \cite{ReggeWheelerSchwarzschild} for the Schwarzschild spacetime, with further contributions by Zerilli \cite{ZerilliPotential} and Vishveshwara \cite{VishveshwaraSchwarzschild}. Moncrief \cite{MoncriefRNOdd,MoncriefRNEven,MoncriefRNGaugeInv} subsequently established the mode stability of Reissner--Nordstr\"om black holes. Whiting \cite{WhitingKerrModeStability} proved mode stability for the Teukolsky equation on Kerr spacetimes, with refinements and extensions due to Shlapentokh-Rothman \cite{ShlapentokhRothmanModeStability}, Andersson--Ma--Paganini--Whiting \cite{AnderssonMaPaganiniWhitingModeStab}, as well as Civin in the Kerr--Newman case \cite{CivinKerrNewman}.

\subsection{Further related work}
\label{SubsecIntroWork}

Besides the references given in \S\ref{SubsubsecIntroNLLambda} for wave equations on black hole spacetimes with $\Lambda>0$, we mention \cite{DafermosRodnianskiSdS,IantchenkoRNdSDirac,IantchenkoKNdSDirac}, as well as Schlue's work on the stability properties of the cosmological region of Kerr--de~Sitter black holes \cite{SchlueCosmological,SchlueWeylDecay}. Related to the scattering theoretic underpinnings of the present work, we mention Warnick's physical space approach to the study of resonances \cite{WarnickQNMs}.

There is a large amount of literature on asymptotically flat black hole spacetimes; we only give a brief account here. The seminal papers by Wald \cite{WaldSchwarzschild} and Kay--Wald \cite{KayWaldSchwarzschild} proved the boundedness of linear scalar waves on the Schwarzschild spacetime; Dafermos and Rodnianski \cite{DafermosRodnianskiRedShift} gave a more robust proof, exploiting the red-shift effect, and established polynomial decay. Polynomial decay for linear scalar waves on Kerr was proved by Andersson--Blue \cite{AnderssonBlueHiddenKerr} and in the aforementioned \cite{DafermosRodnianskiShlapentokhRothmanDecay}; see also \cite{FinsterKamranSmollerYauKerr}. In his thesis, Civin obtained analogous results for waves on Kerr--Newman spacetimes \cite{CivinStability}. The precise decay rates (Price's law \cite{PriceLawI}) were obtained by Tataru \cite{TataruDecayAsympFlat}, see also \cite{MetcalfeTataruTohaneanuPriceNonstationary}, as well as \cite{DafermosRodnianskiPrice} in a spherically symmetric but non-linear context. Strichartz estimates were proved by Tataru and Tohaneanu \cite{TataruTohaneanuKerrLocalEnergy}, also in collaboration with Marzuola and Metcalfe \cite{MarzuolaMetcalfeTataruTohaneanuStrichartz}. For non-linear results, we refer to Luk's work on Kerr \cite{LukKerrNonlinear}, Ionescu--Klainerman \cite{IonescuKlainermanWavemap}, and Stogin \cite{StoginKerrWaveMap} for a wave map equation, and Dafermos--Holzegel--Rodnianski \cite{DafermosHolzegelRodnianskiKerrBw} for a scattering construction of dynamical Kerr black holes. Many of these works rely on the influential `vector field method' developed by Klainerman \cite{KlainermanUniformDecay}; see also \cite{MoschidisRp} for a discussion of recent developments.

In the context of non-scalar fields, we mention the work by Blue \cite{BlueMaxwellSchwarzschild} and Sterbenz--Tataru \cite{SterbenzTataruMaxwellSchwarzschild} on Maxwell's equation on Schwarzschild spacetimes, as well as the paper by Andersson and Blue on Kerr \cite{AnderssonBlueMaxwellKerr}. For gravitational perturbations of Schwarzschild, we recall the recent stability result by Dafermos--Holzegel--Rodnianski \cite{DafermosHolzegelRodnianskiSchwarzschildStability}; see also \cite{FinsterSmollerKerrStability}. Dirac waves on Kerr and Kerr--Newman were studied by Finster--Kamran--Smoller--Yau \cite{FinsterKamranSmollerYauDiracKerrNewman} as well as H\"afner--Nicolas \cite{HaefnerNicolasDiracKerr}.

We also mention the works \cite{BaskinParamKGondS,BaskinStrichartzDeSitterLong} by Baskin, and \cite{BaskinVasyWunschRadMink,BaskinVasyWunschRadMink2} by Baskin--Vasy--Wunsch; while they do not concern black hole spacetimes, some of the microlocal analysis used in these papers is closely related to aspects of the analysis underlying the proof of our main theorem, as will be discussed in detail in upcoming work by H\"afner and Vasy \cite{HaefnerVasyKerr}.

The study of mode stability of \emph{higher-dimensional black holes} has received considerable attention in recent years. For $D$-dimensional spacetimes with $\Lambda>0$, the mode stability of SdS was proved for $D=4$ in \cite{KodamaIshibashiMaster} and for $D=5,6$ in \cite{KodamaIshibashiCharged}; numerical results extend this to $D\leq 11$ \cite{KonoplyaZhidenkoSdSHighDim}. For RNdS spacetimes, stability is known by analytic means for $D=4,5$, and numerically for $D=6$, but there is evidence for mode \emph{instability} for near-extremal RNdS black holes with $D\geq 7$ \cite{KonoplyaZhidenkoInstability}. This surprising finding is in striking contrast to the asymptotically flat case, $\Lambda=0$, where the mode stability of Schwarzschild is known for all $D$ \cite{IshibashiKodamaHigherDim}, and that of Reissner--Nordstr\"om for $D=4,5$, with numerical results for $D\leq 11$.

\subsection{Strategy of the proof}
\label{SubsecIntroStr}

We present the general framework in a schematic form, thereby also summarizing and streamlining the ideas developed in \cite{HintzVasyKdSStability}. We consider (non-linear) gauge-invariant evolution equations of hyperbolic character in the following setup:

\begin{enumerate}[label=(\alph*)]
\item $M$ is a product manifold, $M=\R_{t_*}\times X$, and $E,F\to M$ are stationary vector bundles (i.e.\ equipped with an action of the group of translations in $t_*$).
\item $P$ is a second order differential operator acting on (suitable) smooth sections of $E$.
\item \label{ItIntroStrGaugeInv} \emph{Gauge invariance:} A gauge group $G$ acts on sections of $E$, $g\cdot\phi:=\Gamma(g)\phi$, so that for any $g\in G$ and any solution of $P(\phi)=0$, we also have $P(\Gamma(g)\phi)=0$. We assume that $T_{\Id}G=\CI(M;\wt G)$ is the space of sections of some vector bundle $\wt G\to M$.
\item \label{ItIntroStrGaugeI} \emph{Gauge fixing I:} There exists a (non-linear) operator $\Ups\in\Diff^1(M;E,F)$ and a linear operator $\cD\in\Diff^1(M;F,E)$ such that the gauge-fixed equation
  \begin{equation}
  \label{EqIntroStrGaugeFix}
    P(\phi)-\cD(\Ups(\phi))=0
  \end{equation}
  is a quasilinear system of wave equations with $\Sigma_0:=\{t_*=0\}\subset M$ as a Cauchy hypersurface.
\item \label{ItIntroStrGaugeII} \emph{Gauge fixing II:} For any $\phi$ with $P(\phi)=0$, there exists a gauge transformation $g$ (locally near $\Sigma_0$) such that $\Ups(\Gamma(g)\phi)=0$; one can find $g$ by solving a (non-linear) wave equation. On the linearized level, the equation
\[
  \Box_\phi^\Ups\frakg := D_\phi\Ups\circ(D_{\Id}\Gamma(\frakg)(\phi))=0
\]
is principally a wave equation for $\frakg\in T_{\Id}G$, with $\Sigma_0$ as a Cauchy hypersurface.
\item \label{ItIntroStrBianchi} \emph{Generalized Bianchi identity:} There exists a linear differential operator $\cB(\phi)\in\Diff^1(M;E,F)$, with coefficients depending possibly non-linearly on $\phi$, such that $\cB(\phi)P(\phi)=0$ for all $\phi$, and such that $\cB(\phi)\circ\cD\in\Diff^2(M;F)$ is principally a wave operator with $\Sigma_0$ as a Cauchy hypersurface.
\end{enumerate}

Due to the gauge invariance, an initial value problem for $P$ can in general only be solved for initial data satisfying certain constraint equations. The solution of such an initial value problem for $P$ can be found by the following well-known procedure: one arranges the gauge condition $\Ups(\phi)=0$ at $\Sigma_0$, and then solves the equation \eqref{EqIntroStrGaugeFix}, which implies that the 1-jet of $\Ups(\phi)$ vanishes at $\Sigma_0$ (see \S\ref{SubsecBasicNL} for a detailed discussion in the Einstein--Maxwell case). Since $\Ups(\phi)$ solves the wave equation
\begin{equation}
\label{EqIntroStrGaugeProp}
  \cB(\phi)\cD(\Ups(\phi))=0,
\end{equation}
one then infers that $\Ups(\phi)\equiv 0$, hence one has $P(\phi)=0$, as desired.

Suppose now that we are given a finite-dimensional family $\phi_b$ of stationary solutions, with $b\in\R^N$ close to some $b_0\in\R^N$, so $P(\phi_b)=0$. We choose the gauge $\Ups$ so that $\Ups(\phi_{b_0})=0$. We wish to study the non-linear stability of the family $\phi_b$. The key ingredients of our framework are:

\begin{enumerate}
\item \label{ItIntroStrMS} \emph{Linear mode stability at $b_0$:} A smooth (generalized) mode solution $\phidot$ of $D_{\phi_{b_0}}P(\phidot)=0$, so $\phidot(t_*,x)=\sum_{j=0}^k e^{-i t_*\sigma}t_*^j\phidot_j(x)$, with $\Im\sigma\geq 0$, $\sigma\neq 0$, is pure gauge: that is, $\phidot=D_{\Id}\Gamma(\frakg)(\phi_{b_0})$ for some $\frakg\in T_{\Id}G$. If $\sigma=0$, then $\phidot$ is equal to a pure gauge solution plus $\phi_{b_0}'(b'):=\frac{d}{ds}\phi_{b_0+s b'}|_{s=0}$ for some $b'\in\R^N$.
\item \label{ItIntroStrCD} \emph{Constraint damping at $b_0$:} All solutions $y\in\CI(M;F)$ of $(\cB(\phi_{b_0})\circ\cD)(y)=0$ are exponentially decaying in $t_*$ at a rate $\alpha>0$.
\item \label{ItIntroStrESG} \emph{High energy estimates at $b_0$:} Solutions of the linearized equation
  \begin{equation}
  \label{EqIntroStrESG}
    L_{b_0}\phidot:=D_{\phi_{b_0}}(P-\cD\circ\Ups)(\phidot)=D_{\phi_{b_0}}P(\phidot) - \cD(D_{\phi_{b_0}}\Ups(\phidot)) = 0
  \end{equation}
  (and of stationary perturbations thereof) with smooth initial data have an asymptotic expansion into a finite number of generalized modes (resonant states) with frequencies $\sigma_j\in\C$, $j=1,\ldots,M$, $\sigma_1=0$, $\Im\sigma_j\geq 0$, plus an $\cO(e^{-\alpha t_*})$ remainder.
\end{enumerate}

We stress that these assumptions only concern the linearization at the \emph{fixed solution} $\phi_{b_0}$; this will be the RNdS solution later on, so in particular we will not need to prove mode stability of slowly rotating KNdS solutions directly in order to show non-linear stability!

\begin{rmk}
  The reason for the terminology in point~\itref{ItIntroStrESG} is that obtaining such an asymptotic expansion typically relies on understanding the behavior of the Fourier transform of $L_{b_0}$, which is $\wh{L_{b_0}}(\sigma)=e^{i t_*\sigma}L_{b_0}e^{-i t_*\sigma}$ acting on $t_*$-independent sections of $E$, as $|\Re\sigma|\to\infty$ (with $\Im\sigma\geq -\alpha$) \cite{VasyMicroKerrdS}. This in turn relies on a precise understanding, which is robust under perturbations, of semiclassical microlocal estimates for $\wh{L_{b_0}}(\sigma)$ at trapping, radial sets, etc. The behavior for small frequencies is also important for obtaining exponential decay of the remainder; this was already discussed in \S\ref{SubsubsecIntroNLLambda}. See Figure~\ref{FigIntroStrRes}.
\end{rmk}

\begin{figure}[!ht]
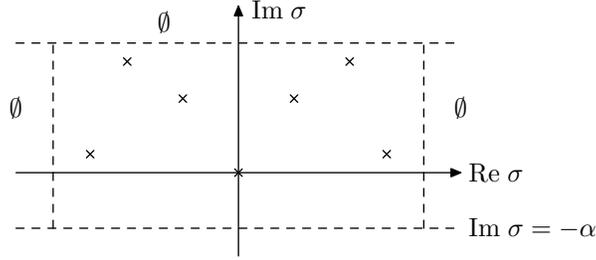

\centering
\inclfig{IntroStrRes}
\caption{Resonances of $L_{b_0}$ in the context of the Einstein--Maxwell system. Semiclassical microlocal analysis yields high energy bounds in $\Im\sigma\geq-\alpha$ and gives the absence of resonances for large $|\Re\sigma|$. Energy estimates imply the absence of resonances with $\Im\sigma\gg 0$. Only finitely many non-decaying resonances, drawn as crosses, remain. This qualitative picture persists for stationary perturbations of $L_{b_0}$.}
\label{FigIntroStrRes}
\end{figure}

\begin{rmk}
\label{RmkIntroStrNoCD}
  Points~\itref{ItIntroStrMS} and \itref{ItIntroStrESG} immediately imply the linear stability of $\phi_{b_0}$: a solution of $L_{b_0}\phidot=0$ with initial data satisfying the linearized constraints and put into the linearized gauge $D_{\phi_{b_0}}\Ups(\phidot)=0$ (see \ref{ItIntroStrGaugeII}) satisfies $\phidot=\phi'_{b_0}(b')+D_{\Id}\Gamma(\frakg)(\phi_{b_0})$ for some $b'$ and $\frakg$. Indeed, the linearization of the gauge propagation argument given above implies $D_{\phi_{b_0}}P(\phidot)=0$, thus this holds for each term in the asymptotic expansion of $\phidot$, to which in turn point~\itref{ItIntroStrMS} applies. This argument is not suited for non-linear analysis, or even for the linear stability analysis of $\phi_b$ for $b$ close to but different from $b_0$. See also \cite[Theorem~10.2]{HintzVasyKdSStability} and its subsequent discussion.
\end{rmk}

The statement in point~\itref{ItIntroStrCD} could also be called `damping of the gauge violation:' For any global solution $\phi$ of \eqref{EqIntroStrGaugeFix}, the failure $\Ups(\phi)$ of the gauge condition is exponentially damped. Arranging this means choosing the operator $\cD$ carefully; notice that in the above setup, only its principal part matters, so the task is to choose the zeroth order terms appropriately.

Every term, say $a_j(t_*,x):=e^{-i t_*\sigma_j}\phidot_j(x)$ (omitting potential powers of $t_*$ for brevity), in the asymptotic expansion of the solution of equation~\eqref{EqIntroStrESG} solves the equation $L_{b_0}(a_j)=0$ separately, hence the (linearized) generalized Bianchi identity implies $\cB(\phi_{b_0})\cD(D_{\phi_{b_0}}\Ups(a_j))=0$. But then
\begin{equation}
\label{EqIntroStrGrowingSatisfiesGauge}
  D_{\phi_{b_0}}\Ups(a_j)=0
\end{equation}
due to constraint damping (note that $a_j$ is a non-decaying mode). Thus, $a_j$ also satisfies $D_{\phi_{b_0}}P(a_j)=0$ and therefore has a very particular structure in view of the mode stability statement \itref{ItIntroStrMS}. Put differently, constraint damping ensures that what at first looks like an analytic obstacle to the solvability of the quasilinear gauge-fixed equation \eqref{EqIntroStrGaugeFix} (growing modes, such as $a_j$, of the `ungeometric' linearized gauge-fixed equation) is in fact geometric (pure gauge, or a linearized solution $\phi'_{b_0}(b')$).

One expects that for fixed non-linear wave equations, non-decaying solutions, such as $a_j$, of the linearization preclude the existence of global non-linear solutions; here however, the equation is not fixed due to the flexibility in the choice of gauge. We exploit this as follows: fixing a basis $a_j=D_{\Id}\Gamma(\frakg_j)(\phi_{b_0})$ of the finite-dimensional space of non-decaying pure gauge solutions of \eqref{EqIntroStrESG}, we can define the space
\[
  \Theta := \mathspan\{ \theta_j \} \subset \CIc(M;E),\quad \theta_j:=D_{\phi_{b_0}}\Ups(D_{\Id}\Gamma(\chi\frakg_j)(\phi_{b_0})),
\]
of gauge modifications, where $\chi=\chi(t_*)$ is a cutoff, $0$ near $\Sigma_0$ and $1$ for large $t_*$. (The inclusion into $\CIc$ follows from \eqref{EqIntroStrGrowingSatisfiesGauge}.) Note here that
\[
  L_{b_0}(D_{\Id}\Gamma(\chi\frakg_j)) = -\cD\theta_j
\]
in view of the gauge invariance \ref{ItIntroStrGaugeInv}. The point then is that one can solve
\begin{equation}
\label{EqIntroStrLb0RangeMod}
  L_{b_0}\phidot=-\cD\theta
\end{equation}
(with prescribed initial data) for $\phidot$, which is the sum of the contribution from the resonance at $\sigma_1=0$ plus an exponentially decaying tail, provided we choose $\theta\in\Theta$ appropriately (and the choice of $\theta$ can be read off from the asymptotic behavior of the solution of $L_{b_0}\phidot_0=0$ with the same initial data). Note that equation~\eqref{EqIntroStrLb0RangeMod} can be rewritten as
\[
  D_{\phi_{b_0}}P(\phidot) - \cD(D_{\Phi_{b_0}}\Ups(\phidot)-\theta) = 0,
\]
making the change of the gauge evident on the linearized level.

Suppose the family $\phi_b$ is trivial, i.e.\ the only stationary solution of interest is $\phi_{b_0}$.\footnote{This is the setting for the non-linear stability of the static model of de~Sitter space, both for the Einstein vacuum equations, see \cite[Appendix~C]{HintzVasyKdSStability}, and for the Einstein--Maxwell system.} (Mode stability then becomes the statement that all non-decaying generalized mode solutions are pure gauge, including at 0 frequency.) Then we can prove the non-linear stability of $\phi_{b_0}$ by solving the quasilinear wave equation
\[
  \Phi(\wt\phi,\theta) := P(\phi_{b_0}+\wt\phi) - \cD(\Ups(\phi_{b_0}+\wt\phi)-\theta) = 0
\]
for $\wt\phi=\cO(e^{-\alpha t_*})$ and $\theta\in\Theta$ by means of a simple Newton-type iteration scheme, starting with the guess $(\wt\phi_0,\theta_0)=(0,0)$: we solve the linear wave equation
\begin{equation}
\label{EqIntroStrNoFamily}
  D_{(\wt\phi_k,\theta_k)}\Phi(\wt\phi',\theta') = -\Phi(\wt\phi_k,\theta_k)
\end{equation}
\emph{globally} at each step, where we choose $\theta'\in\Theta$ so that this has a solution $\wt\phi'=\cO(e^{-\alpha t_*})$; and we then update $\wt\phi_{k+1}=\wt\phi_k+\wt\phi'$, $\theta_{k+1}=\theta_k+\theta'$. Note here that the space $\cD(\Theta)$ is a finite-dimensional complement of the range of $D_0\Phi(\cdot,0)$ acting on sections of size $\cO(e^{-\alpha t_*})$; if one has a robust Fredholm theory for linearizations of $\Phi(\cdot,0)$ around small $\wt\phi$, the \emph{same} space $\cD(\Theta)$ will be a complement to the range of $D_{\wt\phi}\Phi(\cdot,0)$ as well, ensuring the solvability of \eqref{EqIntroStrNoFamily}. We point out that this iteration scheme \emph{automatically} finds the gauge modification $\theta=\lim_{k\to\infty}\theta_k\in\Theta$ in which the global solution $\phi=\phi_{b_0}+\wt\phi$, $\wt\phi=\lim_{k\to\infty}\wt\phi_k$, of $P(\phi)=0$ exists. Furthermore, it is \emph{unnecessary to know} the dimension of the space of non-decaying mode solutions of $L_{b_0}$ in \eqref{EqIntroStrESG} --- \emph{only the finite-dimensionality matters}. (Without the finite-dimensionality, one would encounter additional analytic problems when dealing with the gauge modification $\theta$.)

If the family $\phi_b$ is non-trivial, we in addition need to deal with the parameter $b\in\R^N$. For the linearized equation \eqref{EqIntroStrESG}, we note that by \ref{ItIntroStrGaugeII}, there exists a linearized gauge $\frakg_b(b')$ such that
\[
  \phi^{\prime\Ups}_b(b'):=\phi'_b(b')+D_{\Id}\Gamma(\frakg_b(b'))
\]
solves $L_b(\phi^{\prime\Ups}_b(b'))=0$; thus, after adding a pure gauge solution, one can exhibit the linearization of the family $\phi_b$ as a stationary solution of $L_b(\phidot)=0$.

\begin{rmk}
\label{RmkIntroStrIncompatible}
  Since one finds $\frakg_{b_0}(b')$ by solving the wave equation
\[
  \Box_{\phi_{b_0}}^\Ups(\frakg_{b_0}(b')) = -D_{\phi_{b_0}}\Ups(\phi'_{b_0}(b')),
\]
the raison d'\^etre for $\frakg_{b_0}(b')$ is the incompatibility of $\phi_b$, with $b$ close to $b_0$, with the gauge condition $\Ups=0$.
\end{rmk}

The linear stability of $\phi_b$ for $b$ near $b_0$ can now be proved as follows: consider the space
\[
  \cZ_b := \cD(\Theta) + \{ L_b(\chi\phi^{\prime\Ups}_b(b')) \colon b'\in\R^N \} \subset \CIc(M;E),
\]
with the second summand taking care of the zero resonant states coming from the linearization of the family $\phi_b$. (Recall that the space $\cD(\Theta)$ is \emph{the same} for all $b$ close to $b_0$.) Then by construction, we can solve the initial value problem for
\[
  L_{b_0}\wt\phi' = -z \in \cZ_{b_0}
\]
with $\wt\phi'=\cO(e^{-\alpha t_*})$ exponentially decaying, provided we choose $z$ suitably; the choice of $z$ can again be directly read off from the asymptotic behavior of the solution of the initial value problem for $L_{b_0}\phidot_0=0$ with the same data. By Fredholm stability arguments as above, we can therefore solve
\begin{equation}
\label{EqIntroStrLinMod}
  L_b\wt\phi' = -z \in \cZ_b,
\end{equation}
for $b$ close to $b_0$, with $\wt\phi'=\cO(e^{-\alpha t_*})$ when one chooses $z$ suitably (again depending linearly on the initial data). Writing $z=\cD\theta+L_b(\chi\phi_b^{\prime\Ups}(b'))$, and assuming the initial data satisfy the linearized constraints, this shows that $\phidot:=\chi\phi_b^{\prime\Ups}(b')+\wt\phi'$ solves
\begin{equation}
\label{EqIntroStrLinEq}
  D_{\phi_b}P(\phidot)=0
\end{equation}
in the gauge $D_{\phi_b}\Ups(\phidot)-\theta=0$, proving the linear stability. Mode stability of $\phi_b$ is a simple consequence.

A crucial feature of this argument is that the exponential decay for the solution of \eqref{EqIntroStrLinMod} does not use any assumptions on the initial data; only the last step, relating the solution of \eqref{EqIntroStrLinMod} to a solution of the linearized equation \eqref{EqIntroStrLinEq}, uses the linearized constraints. This allows us to solve the non-linear equation as follows: writing
\begin{equation}
\label{EqIntroStrGrafted}
  \phi_{b_0,b} = (1-\chi)\phi_{b_0} + \chi\phi_b,
\end{equation}
which interpolates between $\phi_{b_0}$ initially and $\phi_b$ eventually, we solve
\begin{equation}
\label{EqIntroStrNLEq}
  \Phi(\wt\phi,b,\theta) := P(\phi_b + \wt\phi) - \cD(\Ups(\phi_b + \wt\phi) - \Ups(\phi_{b_0,b}) - \theta) = 0
\end{equation}
for $\wt\phi=\cO(e^{-\alpha t_*})$, $\theta\in\Theta$ and $b\in\R^N$. (The additional modification of the gauge term here addresses the incompatibility described in Remark~\ref{RmkIntroStrIncompatible}. We use $\phi_{b_0,b}$ rather than $\phi_b$ in order to leave the gauge condition unchanged near the Cauchy surface $\Sigma_0$.) One solves equation~\eqref{EqIntroStrNLEq} by solving a linearized equation globally at each step of a Newton (or Nash--Moser) iteration scheme.

We stress that we can solve equation~\eqref{EqIntroStrNLEq} globally for \emph{any} Cauchy data; \emph{only once one has found the solution} does one need to use the assumption that the initial data satisfy the constraints in order to conclude that $P(\phi_b+\wt\phi)=0$ in the gauge $\Ups(\phi_b+\wt\phi)-\Ups(\phi_{b_0,b})-\theta=0$.

\subsubsection{Illustration: Maxwell's equations}
\label{SubsubsecIntroStrM}

We consider the linear Maxwell equation in the 4-potential formulation, even though this does not fit exactly into the above framework.\footnote{In any case, this is not the right equation to consider for charged black holes, as it ignores the coupling of the electromagnetic field and the metric tensor.} Thus, we study
\begin{equation}
\label{EqIntroStrMEq}
  \delta_g d A = 0,
\end{equation}
where $g=g_{b_0}$ is a non-degenerate RNdS metric on the manifold $M^\circ$ defined in \eqref{EqIntroNLManifold}. In the above terminology, we thus have $P=\delta_g d$, acting on sections of the bundle $E=T^*M^\circ$. The gauge group is $G=\CI(M^\circ)$, acting by $a\cdot A:=A+d a$. We choose the Lorenz gauge $\Ups(A):=\delta_g A$, so the bundle $F$ is the trivial line bundle $\ul\R=M^\circ\times\R$. The gauge-fixed equation \eqref{EqIntroStrGaugeFix} reads $(\delta_g d + \td\delta_g)A = 0$, where $\td$ (called $-\cD$ above) is given by
\begin{equation}
\label{EqIntroStrMtd}
  \td u = d u + \gamma_3\,u\,dt_*
\end{equation}
with $\gamma_3\in\R$ specified momentarily. The generalized Bianchi identity is simply the identity $\delta_g^2=0$, so $\cB(A)=\delta_g$ in \ref{ItIntroStrBianchi} above.

The mode stability analysis for \eqref{EqIntroStrMEq} reveals\footnote{This can be deduced from \cite[Theorem~1]{HintzVasyKdsFormResonances} applied to the equation $(\delta_g d+d\delta_g)A=0$.} that for frequencies $\sigma$ with $\Im\sigma\geq 0$, $\sigma\neq 0$, (generalized) mode solutions are pure gauge, i.e.\ exact 1-forms, while at $\sigma=0$, there is a 1-dimensional space of solutions, namely $\mathspan\{A_{b_0}\}$. Arranging constraint damping requires that all solutions of $\delta_g\td y=0$ decay exponentially fast. Note that for $\gamma_3=0$, a solution is given by $y\equiv 1$, which does \emph{not} decay; one can however show that for $\gamma_3>0$ small (by perturbative arguments) or large (using semiclassical analysis), constraint damping does hold, see \S\ref{SecCD}. The high energy estimates for the operator $L_{b_0}=\delta_g d+\td\delta_g$ follow from a computation of the subprincipal operator of $L_{b_0}$ at the trapped set; this relies on \cite{HintzPsdoInner} and a continuity argument if $\gamma_3$ is small, and can be checked for general $\gamma_3>0$ by a direct calculation. We discuss this further below.

We conclude that for given initial data, we can choose a function $\theta$ within a suitable finite-dimensional space $\Theta\subset\CIc(M^\circ)$ such that
\[
  \delta_g d A + \td(\delta_g A-\theta) = 0
\]
has a solution $A=c A_{b_0} + \wt A$, with $c\in\R$ and $\wt A=\cO(e^{-\alpha t_*})$, for a suitable choice of $\theta$. If the initial data satisfy the constraint equation $\la\delta_g d A,dt_*\ra=0$ at $\Sigma_0$, and one picks Cauchy data for $A$ so that $\delta_g A=0$ at $\Sigma_0$, then we conclude that $\delta_g d A=0$ holds, and $A$ satisfies the gauge condition $\delta_g A-\theta=0$.\footnote{In fact, for $\gamma_3>0$ small, one can check that $\Theta=0$. However --- and this is crucial for Einstein--Maxwell --- even a finite-dimensional family $\Theta$ of gauge modifications would cause no problems.}

\subsubsection{Einstein--Maxwell system}
\label{SubsubsecIntroStrEM}

We now apply the above framework to the charged black hole stability problem. (In our proof of Theorem~\ref{ThmIntroNL}, we will need to modify the precise formulation only slightly, see \S\S\ref{SubsecBasicNL} and \ref{SubsecNLProof}).
\[
  P(g,A) = \bigl( 2(\Ric(g)+\Lambda g - 2 T(g,d A)), \delta_g d A \bigr),
\]
so the vector bundle is $E=S^2 T^*M^\circ\oplus T^*M^\circ$. (The factor $2$ is explained by equation~\eqref{EqBasicNLLinRic}.) The gauge group consists of diffeomorphisms $\phi$ of $M^\circ$ and smooth functions $a$, acting by $(g,A)\mapsto(\phi^*g,\phi^*A + d a)$. For the gauge, we take $\Ups=(\Ups^E,\Ups^M)$, with
\[
  \Ups^E(g)^\kappa = g^{\mu\nu}(\Gamma(g)_{\mu\nu}^\kappa-\Gamma(g_{b_0})_{\mu\nu}^\kappa),
\]
which is thus a generalized wave coordinate gauge, and
\[
  \Ups^M(g,A) = \delta_g(A-A_{b_0}),
\]
which is a generalized Lorenz gauge; see \S\ref{SubsecBasicNL} for further discussion. (The gauge source functions in both cases were chosen so as to make $\Ups(g_{b_0},A_{b_0})=0$.) The bundle $F$ is given by $F=T^*M^\circ\oplus\ul\R$. To write down the gauge-fixed equation \eqref{EqIntroStrGaugeFix}, we use the operator
\[
  \cD = (2\tdel^*, \td),
\]
with $\td$ as in \eqref{EqIntroStrMtd}, and $\tdel^*=\delta_{g_{b_0}}^*+\gamma_1\,dt_*\cdot u-\gamma_2 u(\nabla t_*)g_{b_0}$ for large $\gamma_1,\gamma_2\in\R$. The generalized Bianchi identity uses the operator
\[
  \cB(g,A) = ( \delta_g G_g, \delta_g ),
\]
where $G_g=\Id-\frac{1}{2}g\tr_g$ is the trace-reversal operator; thus $G_g\Ric(g)=\Ein(g)$ is the Einstein tensor, which is divergence-free by the second Bianchi identity.

We prove the mode stability of the RNdS solution $(g_{b_0},A_{b_0})$ in \S\ref{SecMS}, see Theorem~\ref{ThmMS}. Constraint damping at this solution is shown in \S\ref{SecCD}, see Theorem~\ref{ThmCD}, using and extending results from \cite[\S8]{HintzVasyKdSStability}.

The high energy estimates are proved in \S\ref{SecESG}, see Theorem~\ref{ThmESG}. The key step here is to analyze the behavior of the linearized gauge-fixed operator $L_{b_0}$ at the trapped set, which is located in phase space over the photon sphere and generated by null-geodesics orbiting the black hole indefinitely without escaping through the event or cosmological horizon. The relevant object is the \emph{subprincipal operator} \cite{HintzPsdoInner} at the trapped set; this is a version of the well-known subprincipal symbol, and we show in \S\ref{SubsecESGTrap} that in a suitable sense it has a favorable sign at the trapping.

In summary:

\begin{thm}[More precise version of Theorem~\ref{ThmIntroNL}]
\label{ThmIntroNLFull}
  Suppose $(h,k,\bfE,\bfB)$ are initial data on $\Sigma_0$, satisfying the constraint equations, free of magnetic charges, and close (in the topology of $H^{21}$) to the initial data of the non-degenerate RNdS solution $(g_{b_0},A_{b_0})$. Then there exist KNdS black hole parameters $b$ close to $b_0$ and exponentially decaying tails $(\wt g,\wt A)=\cO(e^{-\alpha t_*})$ such that
  \[
    (g,A) = (g_{b_0,b}+\wt g,A_{b_0,b}+\wt A)
  \]
  (see \eqref{EqIntroStrGrafted}) solves the Einstein--Maxwell system \eqref{EqIntroEMA}, attains the given initial data, and satisfies the gauge conditions
  \[
    \Ups^E(g) - \Ups^E(g_{b_0,b}) - \theta = 0, \quad \Ups^M(g,A)-\Ups^M(g,A_{b_0,b}) - \kappa = 0,
  \]
  where the gauge modification $(\theta,\kappa)$ lies in a suitable finite-dimensional space $\Xi\subset\CIc(\Omega^\circ;T^*\Omega^\circ\oplus\ul\R)$.
\end{thm}

See Theorem~\ref{ThmNLKNdS} for the precise statement, describing the function spaces as well as the space $\Xi$. We stress the importance of techniques from (non-smooth) global microlocal analysis (b-analysis, semiclassical analysis) and scattering theory underlying the proof; these techniques enable us to build a framework which is sufficiently robust under perturbations of the geometry and the dynamical structure of the spacetime to allow for the global analysis of tensorial quasilinear wave equations.

\subsection{Notation}
\label{SubsecIntroNotation}

\newlength{\oldparindent}
\setlength{\oldparindent}{\parindent}
\setlength{\parindent}{0pt}

We give a list of notation used frequently throughout the present paper, together with a reference to the first appearance:

\setlength{\parindent}{\oldparindent}

\nomenstart
  \nomenitem{$A_{b_0,b}$}{4-potential interpolating between $A_{b_0}$ and $A_b$, cf.\ $g_{b_0,b}$ below}
  \nomenitem{$b_0$}{parameters of a fixed RNdS black hole, see \eqref{EqKNdS0Params}}
  \nomenitem{$b$}{black hole parameters, $b\in B$, see \S\ref{SecKNdS}}
  \nomenitem{$B$}{parameter space for slowly rotating, non-degenerate KNdS black holes, see \S\ref{SecKNdS}}
  \nomenitem{$B_m$}{parameter space for KNdS black holes allowing for magnetic charges, see \S\ref{SubsecKNdSMag}}
  \nomenitem{$\tdel^*$}{modification of the symmetric gradient, see \eqref{EqBasicNLtdel}}
  \nomenitem{$\td$}{modification of the exterior derivative, see \eqref{EqBasicNLtd}}
  \nomenitem{$\delta_g$}{(negative) divergence, $(\delta_g T)_{\mu_1\ldots\mu_N}=-T_{\lambda\mu_1\ldots\mu_N;}{}^\lambda$}
  \nomenitem{$\delta_g^*$}{symmetric gradient, $(\delta_g^*u)_{\mu\nu}=\frac{1}{2}(u_{\mu;\nu}+u_{\nu;\mu})$}
  \nomenitem{$\gamma_0$}{map assigning to a function on $\Omega$ its Cauchy data at $\Sigma_0$, see \eqref{EqBasicNLCauchyData}}
  \nomenitem{$g_{b_0,b}$}{metric interpolating between $g_{b_0}$ and $g_b$, see \eqref{EqNLMetricInterpolate}}
  \nomenitem{$g_b$}{KNdS metric with parameters $b\in B$, see \eqref{EqKNdSaMetric}}
  \nomenitem{$\theta$}{gauge source function for the wave coordinate gauge, see \eqref{EqLinKNdSIVPSol2}}
  \nomenitem{$i_b$}{map constructing Cauchy data from geometric initial data, see Proposition~\ref{PropKNdSIni}}
  \nomenitem{$\kappa$}{gauge source function for the Lorenz gauge, see \eqref{EqLinKNdSIVPSol2}}
  \nomenitem{$\mu$}{metric coefficient of the RNdS metric, see \eqref{EqKNdS0Mu}}
  \nomenitem{$\cM$}{static exterior region of an RNdS black hole, see \eqref{EqKNdS0StaticRegion}}
  \nomenitem{$M^\circ$}{underlying manifold of an extended KNdS black hole, see \eqref{EqKNdS0Mf}}
  \nomenitem{$M$}{compactification of $M^\circ$ at future infinity, see \eqref{EqKNdS0MfComp}}
  \nomenitem{$\cL_V$}{Lie derivative along the vector field $V$}
  \nomenitem{$\tcL_T V$}{equal to $\cL_V T$, see Definition~\ref{DefBasicLinLie}}
  \nomenitem{$\Omega^\circ$}{domain in $M$ on which we solve initial value problems, see \eqref{EqKNdSGeoDomCirc}}
  \nomenitem{$\Omega$}{compactification of $\Omega^\circ$ in $M$, see \eqref{EqKNdSGeoDom}}
  \nomenitem{$r_\pm$}{short for $r_{b_0,\pm}$, see \S\ref{SecMS}}
  \nomenitem{$r_{b,\pm}$}{radius of the event ($-$) and cosmological ($+$) horizons of the KNdS spacetime with parameters $b$, see Lemma~\ref{LemmaKNdSaHorizons}}
  \nomenitem{$\Sigma_0$}{Cauchy surface in $\Omega$, see \eqref{EqKNdSGeoCauchySurf}}
  \nomenitem{$S_\sub$}{subprincipal operator, see \eqref{EqLinAnaSubpr}}
  \nomenitem{$\tau$}{global boundary defining function of $M$, see \eqref{EqKNdS0MfComp}}
  \nomenitem{$\tau_0$}{null boundary defining function of $M$ near the horizons, see \eqref{EqKNdS0Tau0}}
  \nomenitem{$t$}{static or Boyer--Lindquist time coordinate, see \eqref{EqKNdS0Metric}}
  \nomenitem{$t_*$}{timelike function on $M^\circ$ which is smooth across the horizons, see \eqref{EqKNdS0CoordDef}}
  \nomenitem{$\tr_g^{ij}$}{trace in the $i$-th and $j$-th indices, $(\tr_g^{ij}T)_{\mu_1\ldots\mu_N}=T_{\mu_1\ldots\lambda\ldots}{}^\lambda{}_{\ldots\mu_N}$, with the two $\lambda$'s at the $i$-th and $j$-th slots}
  \nomenitem{$X$}{boundary of $M$ at future infinity}
  \nomenitem{$\Ups^E$}{gauge condition for the metric tensor, see \eqref{EqBasicNLUpsE}}
  \nomenitem{$\Ups^M$}{gauge condition for the 4-potential, see \eqref{EqBasicNLUpsM}}
\nomenend

\subsection*{Acknowledgments}

I would like to thank Andr\'as Vasy, Maciej Zworski, Jim Isenberg, Sergiu Klainerman, and Yakov Shlapentokh-Rothman for valuable discussions and for their interest and support. I am very grateful to Mihalis Dafermos, and to an anonymous referee, for valuable comments on both content and exposition. I would also like to thank the Miller Institute at the University of California, Berkeley, for support.

\section{Basic properties of the Einstein--Maxwell equations}
\label{SecBasic}

We shall only describe the aspects relevant for present purposes; we refer the reader to \cite[\S6]{ChoquetBruhatGR} for a detailed exposition.

\subsection{Derivation of the equations; charges and symmetries}
\label{SubsecBasicDer}

On an oriented, time-oriented 4-di\-men\-sion\-al Lorentzian spacetime $(M,g)$ with signature $(1,3)$, the Einstein--Maxwell system, with cosmological constant $\Lambda$, for the metric $g$ and the \emph{electromagnetic 4-potential} $A\in\CI(M^\circ;T^*M^\circ)$, arises as the Euler--Lagrange system for the Lagrangian density
\[
  \cL(g,A)=(R_g + 2\Lambda + 2|dA|_g^2)\,|dg|,
\]
where we define the squared norm of a 2-form $F$ by
\begin{equation}
\label{EqBasicDerNormSq}
  |F|_g^2=\sum_{\mu<\nu}F_{\mu\nu}F^{\mu\nu},
\end{equation}
raising indices using $g$; this gives
\begin{equation}
\label{EqBasicDerEM}
  \Ein(g) - \Lambda g = 2 T(g,d A),\qquad \delta_g d A = 0,
\end{equation}
where $\Ein(g)=\Ric(g)-\frac{1}{2}R_g g$ is the Einstein tensor, $\delta_g A=-A_{\mu;}{}^\mu$ the divergence (adjoint of $d$), and
\begin{equation}
\label{EqBasicDerEMTensor}
  T(g,F) = -\tr_g^{24}(F\otimes F) + \frac{1}{2}g|F|_g^2,
\end{equation}
or in index notation $T(g,F)_{\mu\nu} = -F_{\mu\lambda}F_\nu{}^\lambda + \frac{1}{4}g_{\mu\nu}F_{\kappa\lambda}F^{\kappa\lambda}$, is the \emph{energy-momentum tensor} associated with the electromagnetic tensor $F$. It is convenient to rewrite \eqref{EqBasicDerEM} by applying the trace-reversal operator $G_g=\Id-\frac{1}{2}g\tr_g$ to the first equation; since $\tr_g T(g,F)\equiv 0$, this yields
\[
  \Ric(g) + \Lambda g = 2 T(g,dA).
\]
Since the system \eqref{EqBasicDerEM} only involves $F:=d A$, one often writes it in the form
\begin{gather}
\label{EqBasicDerEMCurv1}
  \Ein(g)-\Lambda g=2 T(g,F), \\
\label{EqBasicDerEMCurv2}  d F=0,\quad \delta_g F=0.
\end{gather}
While this is locally equivalent to \eqref{EqBasicDerEM}, it is not so globally when $H^2(M^\circ,\R)\neq 0$, with $H^2(M^\circ,\R)$ the second cohomology group with real coefficients; the latter is indeed the case for the black hole spacetimes of interest in the present paper. We return to this point below.

The second Bianchi identity $\delta_g\Ein(g)\equiv 0$ applied to \eqref{EqBasicDerEMCurv1} implies the conservation law $\delta_g T(g,F)=0$. This is well-known to be satisfied provided $F$ solves \eqref{EqBasicDerEMCurv2}; concretely, we have
\begin{equation}
\label{EqBasicDerEMDelT}
  \delta_g T(g,F) = -\tr_g^{13}(\delta_g F\otimes F) - \frac{1}{2}\tr_g^{13}\tr_g^{24}(F\otimes d F).
\end{equation}

Assuming now that $(g,F)$ is a smooth solution of \eqref{EqBasicDerEMCurv1}--\eqref{EqBasicDerEMCurv2}, let $i\colon\Sigma_0\hra M^\circ$ be a spacelike hypersurface with future timelike unit normal field $N$, and let $h=-g|_{\Sigma_0}$ be the restriction of the metric tensor to $\Sigma_0$, which is thus a (positive definite) Riemannian metric. Declaring a basis $\sB$ of $T_p\Sigma_0$ to be positively oriented if and only if the basis $\{N\}\cup\sB$ of $T_p M^\circ$ is, we also have a Hodge star operator $\star_h$ on $\Sigma_0$. The electric and magnetic fields, as measured by observers with 4-velocity $N$, are then defined by
\begin{equation}
\label{EqBasicDerEandBFields}
  \bfE := -i^* \iota_N F = \star_h i^*\star_g F, \quad \bfB := i^* \iota_N\star_g F = \star_h i^*F  \in \CI(\Sigma_0;T^*\Sigma_0);
\end{equation}
thus, if $(t,x,y,z)$ are local coordinates near a point $p\in\Sigma_0$ with $N=\pa_t$ and $\{\pa_t,\pa_x,\pa_y,\pa_z\}$ an oriented orthonormal basis of $T_p M^\circ$, and we write $E_x=\bfE(\pa_x)$, etc., then
\begin{align*}
  F &= (E_x\,dx+E_y\,dy+E_z\,dz)\wedge dt \\
    &\quad + (B_x\,dy\wedge dz + B_y\,dz\wedge dx + B_z\,dx\wedge dy).
\end{align*}

\begin{rmk}
\label{RmkBasicDerHodgeStar}
  For a metric $g$ with signature $(p,q)$, $n=p+q$, we will frequently use the identities
  \[
    \star_g\star_g = (-1)^{k(n-k)+q},\quad \delta_g = (-1)^{n(k-1)+1+q}\star_g d\star_g,
  \]
  for the action on differential $k$-forms.
\end{rmk}

We moreover compute for $T_{g,F}\equiv T(g,F)$
\begin{equation}
\label{EqBasicDerEMTensorN}
\begin{gathered}
  T_{g,F}(N,N) = \frac{1}{2}(|\bfE|_h^2 + |\bfB|_h^2), \\
  T_{g,F}(N,\cdot) = \star_h(\bfB\wedge\bfE)\quad\tn{on }T\Sigma_0;
\end{gathered}
\end{equation}
$T(N,N)$ is the energy density of the electromagnetic field as measured by an observer with velocity $N$.

Observe now that the expression for $T_{g,F}(N,N)$ is invariant under `rotations'
\begin{equation}
\label{EqBasicDerEMRotation}
  (\bfE,\bfB)\mapsto(\bfE_\theta,\bfB_\theta):=(\cos(\theta)\bfE-\sin(\theta)\bfB,\sin(\theta)\bfE+\cos(\theta)\bfB);
\end{equation}
the same is true for $T_{g,F}(N,\cdot)$. On the other hand, $\bfE_\theta$ and $\bfB_\theta$ are the electric and magnetic field, respectively, corresponding to
\begin{equation}
\label{EqBasicDerEMRotationForm}
  F_\theta = \cos(\theta)F + \sin(\theta)\star_g F.
\end{equation}
It follows that $T_{g,F}(N,N)=T_{g,F_\theta}(N,N)$ for \emph{all} timelike vectors $N$; since $T_{g,F}(X,X)$ is a quadratic form in $X$, this holds for all vectors $N$. Thus, by the standard polarization identity, $T_{g,F_\theta}=T_{g,F}$ for all $\theta$. Moreover $d F_\theta=0$ and $\delta_g F_\theta=0$ for all $\theta$ provided \eqref{EqBasicDerEMCurv2} holds. Thus:

\begin{lemma}
\label{LemmaBasicDerEMRotation}
  If $(g,F)$ solves the Einstein--Maxwell system \eqref{EqBasicDerEMCurv1}--\eqref{EqBasicDerEMCurv2}, then so does $(g,F_\theta)$ for all $\theta\in\R$, with $F_\theta$ defined in \eqref{EqBasicDerEMRotationForm}.
\end{lemma}

Let us now assume for simplicity
\begin{equation}
\label{EqBasicDerEMHole}
  M^\circ \cong \R_{t_*}\times I_r \times\Sph^2,\quad \Sigma_0 = \{t_*=0\}\subset M^\circ,
\end{equation}
with $I\subset\R$ a non-empty open interval, and assume that $\Sigma_0$ is spacelike. If $S=\{t_*=t_0,\ r=r_0\}\subset M$ (with $t_0\in\R$, $r_0\in I$) is a 2-sphere, we can define the \emph{electric} and \emph{magnetic charges}, associated with $F$ solving \eqref{EqBasicDerEMCurv2}, by
\begin{equation}
\label{EqBasicDerCharges}
  Q_e(g,F) := \frac{1}{4\pi}\int_S \star_g F,\qquad Q_m(F) := \frac{1}{4\pi}\int_S F;
\end{equation}
by Stokes' theorem, $Q_e$ and $Q_m$ are independent of the specific choice of $S$. In particular, we can take $S\subset\Sigma_0$, in which case we find, in terms of \eqref{EqBasicDerEandBFields}, that
\[
  Q_e(g,F) = Q_e(h,\bfE) := \frac{1}{4\pi}\int_S \star_h\bfE = \frac{1}{4\pi}\int_S \la\bfE,\nu\ra d\sigma,
\]
with $\nu$ the outward pointing unit normal to $S$ (with respect to $h$) and $d\sigma$ the volume element; and likewise for $Q_m(F)=Q_m(\bfB):=(4\pi)^{-1}\int_S\star_h\bfB$.

\begin{lemma}
\label{LemmaBasicDerEMRotationCharges}
  In the setting \eqref{EqBasicDerEMHole}, and given any 2-form $F$ solving \eqref{EqBasicDerEMCurv2}, there exists $\theta\in\R$ such that $Q_m(F_\theta)=0$, with $F_\theta$ defined in \eqref{EqBasicDerEMRotationForm}; for such $\theta$, we have
  \[
    Q_e(g,F_\theta)^2 = Q_m(F)^2 + Q_e(g,F)^2.
  \]
\end{lemma}
\begin{proof}
  We have $Q_m(F_\theta)=\cos(\theta)Q_m(F)+\sin(\theta)Q_e(g,F)$, so it suffices to take $\theta$ such that $(\cos\theta,\sin\theta)\perp(Q_m(F),Q_e(g,F))$, which can always be arranged.
\end{proof}

Without loss of generality, we may therefore assume $Q_m(F)=0$, equivalently $F=d A$ for some 1-form $A$, reducing the study of the (Einstein--)Maxwell equations on $M$ in the form \eqref{EqBasicDerEMCurv1}--\eqref{EqBasicDerEMCurv2} to the study of the system \eqref{EqBasicDerEM}.

\subsection{Initial value problems for the non-linear system}
\label{SubsecBasicNL}

Initial data for the Einstein--Maxwell system \eqref{EqBasicDerEMCurv1}--\eqref{EqBasicDerEMCurv2} are 5-tuples
\begin{equation}
\label{EqBasicNLData}
  (\Sigma_0,h,k,\bfE,\bfB),
\end{equation}
where $\Sigma_0$ is a smooth 3-manifold with Riemannian metric $h$, $k$ is a symmetric 2-tensor, and $\bfE,\bfB$ are 1-forms on $\Sigma_0$. The initial value problem asks for a (globally hyperbolic) 4-manifold $M^\circ$ equipped with a Lorentzian metric $g$ and a 2-form $F$ solving \eqref{EqBasicDerEMCurv1}--\eqref{EqBasicDerEMCurv2}, together with an embedding of $\Sigma_0\hra M^\circ$ as a Cauchy surface, such that $h=-g|_{\Sigma_0}$ is the induced metric, $k$ the second fundamental form of $\Sigma_0$, and such that $F$ induces the given fields $\bfE,\bfB$ on $\Sigma_0$ via \eqref{EqBasicDerEandBFields}.

Evaluating \eqref{EqBasicDerEMCurv1} on the pair of vectors $(N,N)$, with $N$ the future timelike unit normal to $\Sigma_0$, and on pairs $(N,X)$ with $X\in T\Sigma_0$, and pulling back $d F=0$ and $d\star_g F=0$ to $\Sigma_0$, we find as necessary conditions for the well-posedness of the initial value problem the \emph{constraint equations}
\begin{gather}
\begin{gathered}
\label{EqBasicNLConstraints1}
  R_h - |k|_h^2 + (\tr_h k)^2 - 2\Lambda = 2(|\bfE|_h^2+|\bfB|_h^2), \\
  \delta_h k + d\tr_h k = 2\star_h(\bfB\wedge\bfE);
\end{gathered} \\
\label{EqBasicNLConstraints2}
  \delta_h\bfE = 0,\quad \delta_h\bfB = 0.
\end{gather}

For future use, we point out that if we have a solution of the form $F=d A$, then the constraint on $\bfB$ is automatically satisfied because of $d^2=0$, while the constraint $\delta_h\bfE=0$ is equivalent to
\begin{equation}
\label{EqBasicNLConstraints22}
  (\delta d A)(N) = 0,
\end{equation}
in formal analogy to the derivation of the constraints arising from \eqref{EqBasicDerEMCurv1}.

We next recall the proof that these constraints are also sufficient for the local solvability of the Einstein--Maxwell system \eqref{EqBasicDerEMCurv1}--\eqref{EqBasicDerEMCurv2} \cite[\S6.10]{ChoquetBruhatGR}. (See \cite{ChoquetBruhatGerochMGHD,SbierskiMGHD} for a discussion of the maximal globally hyperbolic development.) Thus, suppose we are given initial data \eqref{EqBasicNLData} satisfying the constraint equations, and define $M^\circ=\R_{t_*}\times\Sigma_0$. We construct a Lorentzian metric $g$ and a 2-form $F$ in a neighborhood of $\Sigma_0$ in $M^\circ$ solving the Einstein--Maxwell equations and attaining the given data at $\Sigma_0$. We eliminate the diffeomorphism invariance by fixing a gauge: following DeTurck \cite{DeTurckPrescribedRicci}, see also \cite{GrahamLeeConformalEinstein,FriedrichHyperbolicityEinstein}, we fix a pseudo-Riemannian background metric $t$ and define the \emph{gauge 1-form}
\begin{equation}
\label{EqBasicNLUpsE}
  \Ups^E(g) := g t^{-1}\delta_g G_g t.
\end{equation}
As discussed in \cite[\S2.1]{HintzVasyKdSStability}, we have $\Ups^E(g)\equiv 0$ if and only if the identity map $\Id\colon(M^\circ,g)\to(M^\circ,t)$ is a wave map. Conversely, given any metric $g$ on $M^\circ$, one may solve the wave map equation $\Box_{g,t}\phi=0$, with $\phi=\Id$ to second order at $\Sigma_0$, and then $\Ups^E(\phi_*g)=0$; in particular, if $(g,F)$ solves the Einstein--Maxwell equations, then so does $(\phi_*g,\phi_*F)$, and the latter satisfies the wave map gauge.

Aiming now to construct a solution of the initial value problem in the gauge $\Ups^E(g)=0$, we consider the modified system
\begin{equation}
\label{EqBasicNLDeTurckF}
  P_{DT}(g,F) = 0, \qquad (d\delta_g+\delta_g d)F = 0,
\end{equation}
where
\begin{equation}
\label{EqBasicNLDeTurckFOp}
  P_{DT}(g,F) := \Ric(g)+\Lambda g - \delta_g^*\Ups^E(g) - 2 T(g,F),
\end{equation}
and $(\delta_g^*u)_{\mu\nu}=\frac{1}{2}(u_{\mu;\nu}+u_{\nu;\mu})$ is the symmetric gradient. Now, \eqref{EqBasicNLDeTurckF} is a quasilinear hyperbolic system for $(g,F)$ (which is principally scalar if one multiplies the first equation by $2$).

The second equation in \eqref{EqBasicNLDeTurckF} is linear in $F$, but the expression in a local coordinate system involves second derivatives of $g$ due to the fact that $F$ is not a scalar; on the other hand, $F$ itself appears in the first equation only in undifferentiated form. Thus, one can still prove the local well-posedness of initial value problems for the evolution equation \eqref{EqBasicNLDeTurckF} if one controls $g$ in $\cC^0 H^s\cap\cC^1 H^{s-1}$ and $F$ in $\cC^0 H^{s-1}\cap\cC^1 H^{s-2}$; see the discussions in \cite[\S18.8]{TaylorPDE} and \cite[\S10.4]{ChoquetBruhatGR}.

To formulate the Cauchy problem for \eqref{EqBasicNLDeTurckF}, denote for any section $u$ of an associated bundle $E\to M^\circ$ of the frame bundle of $M^\circ$ (such as $S^2 T^*M^\circ$, or $S^2 T^*M^\circ\oplus\Lambda^2 T^*M^\circ$) its Cauchy data by
\begin{equation}
\label{EqBasicNLCauchyData}
  \gamma_0(u) := (u|_{\Sigma_0}, \cL_{\pa_{t_*}}u|_{\Sigma_0}) \in \CI(\Sigma_0;E_{\Sigma_0})\oplus\CI(\Sigma_0;E_{\Sigma_0}).
\end{equation}
Now, given initial data $(h,k,\bfE,\bfB)$ on $\Sigma_0$, one first constructs the 1-jet of the Lorentzian metric $g$ at $\Sigma_0$, that is, $g_0,g_1\in\CI(\Sigma_0;S^2 T^*_{\Sigma_0}M^\circ)$, with the following property: if $g$ is any Lorentzian metric with $\gamma_0(g)=(g_0,g_1)$, then $g$ induces the data $(h,k)$ on $\Sigma_0$, and $\Ups^E(g)=0$ at $\Sigma_0$. (Since the second fundamental form and $\Ups^E$ only depend on first derivatives of the metric tensor, these requirements are independent of the choice of $g$.) Next, $g$ determines a unit normal vector field $N$ on $\Sigma_0$ (in fact, its 1-jet), and one can thus construct $F_0\in\CI(\Sigma_0;\Lambda^2 T^*_{\Sigma_0}M^\circ)$ inducing $(\bfE,\bfB)$ on $\Sigma_0$ according to \eqref{EqBasicDerEandBFields} (with $F_0$ in place of $F$). We obtain the second piece $F_1\in\CI(\Sigma_0;\Lambda^2 T^*_{\Sigma_0}M^\circ)$ of Cauchy data for $F$ by enforcing Maxwell's equations \eqref{EqBasicDerEMCurv2} on $\Sigma_0$: the requirements $\cL_N F=d \iota_N F$ and $\cL_N\star_g F=d \iota_N\star_g F$ lead to
\[
  i^*\cL_N F=d(i^* \iota_N F),\qquad i^*(\cL_N\star_g  F)=d(i^*\iota_N\star_g F),
\]
which at $\Sigma_0$ determines $F_1=(\cL_N F)|_{\Sigma_0}$. See \S\ref{SubsecKNdSIni} for further details about the construction of Cauchy data.

We can now solve the system \eqref{EqBasicNLDeTurckF} with initial data $\gamma_0(g,F)=(g_0,F_0;g_1,F_1)$ uniquely near $\Sigma_0$. By construction, we have $d F=0$ and $\delta_g F=0$ at $\Sigma_0$; using the second equation in \eqref{EqBasicNLDeTurckF} together with the constraints \eqref{EqBasicNLConstraints2}, one in fact finds $\gamma_0(d F)=0$, $\gamma_0(\delta_g F)=0$. But $d F$ and $\delta_g F$ satisfy the wave equations $(d\delta_g+\delta_g d)(d F)=0$, $(d\delta_g+\delta_g d)(\delta_g F)=0$, hence we conclude that $d F=0$ and $\delta_g F=0$ everywhere. Therefore $\delta_g T(g,F)=0$ by \eqref{EqBasicDerEMDelT}, and the second Bianchi identity gives
\[
  0 = \delta_g G_g P_{DT}(g,F) = -\delta_g G_g\delta_g^*\Ups^E(g),
\]
which is a wave equation for $\Ups^E(g)$. Now $\Ups^E(g)|_{\Sigma_0}=0$ by construction, and then $P_{DT}(g,F)=0$ together with the constraint equations imply $\gamma_0(\Ups^E(g))=0$, as follows from evaluating $G_g P_{DT}(g,F)=0$ on pairs of vectors $(N,N)$ and $(N,X)$, $X\in T\Sigma_0$. Thus, $\Ups^E(g)=0$ everywhere, so $(g,F)$ solves \eqref{EqBasicDerEM} in the gauge $\Ups^E(g)=0$, as desired.

The specific choice of the hyperbolic formulation has dramatic consequences for the long-time behavior of solutions \eqref{EqBasicNLDeTurckF}. The first useful modification concerns the choice of gauge: namely, we may replace the gauge condition $\Ups^E(g)=0$ by $\Ups^E(g)=\theta$ for a suitable 1-form $\theta$, supported away from $\Sigma_0$. The second useful modification of the operator \eqref{EqBasicNLDeTurckFOp} is the replacement of $\delta_g^*$ by an operator
\begin{equation}
\label{EqBasicNLtdel}
  \tdel^* \in \Diff^1(M^\circ;T^*M^\circ,S^2 T^*M^\circ),
\end{equation}
independent of $g$ for simplicity, which agrees with $\delta_g^*$ to leading order --- this condition is independent of the choice of $g$ indeed --- that is, $\tdel^*$ has principal symbol $\sigma_1(\tdel^*)(\zeta)=i\zeta\otimes(\cdot)$, $\zeta\in T_p M^\circ$, $p\in M^\circ$. The constraint propagation equation becomes
\[
  \delta_g G_g\tdel^*\Ups^E(g)=0
\]
in this case. Neither modification affects the principal (high energy) part of the operator $P_{DT}$, but the low energy behavior is affected in a crucial manner, as explained in \S\ref{SubsecIntroStr}.

\emph{Under the additional assumption} $[\star_h\bfB]=0\in H^2(\Sigma_0,\R)$, one can instead solve the system \eqref{EqBasicDerEM} for $(g,A)$, with electromagnetic tensor then given by $F=d A$. We present the proof in a coordinate-invariant form which will be useful later: the system \eqref{EqBasicDerEM} has an additional gauge freedom, namely the group
\[
  \Diff(M^\circ) \ltimes \CI(M^\circ)
\]
acts on the space of solutions, with $(\phi,a)\in\Diff(M^\circ)\times\CI(M^\circ)$ acting by
\begin{equation}
\label{EqBasicNLGaugeAction}
  (g,A) \mapsto (\phi_*g,\phi_*A+d a).
\end{equation}

A typical method of fixing a gauge for $A$ is by demanding $\delta_g A=0$, called \emph{Lorenz gauge}; it is in fact convenient to note $\delta_g=-\tr_g\delta_g^*$ and to introduce the \emph{gauge function}
\begin{equation}
\label{EqBasicNLUpsM}
  \Ups^M(g,A) := \tr_g\delta_t^* A,
\end{equation}
where $t$ is a fixed background metric as above;\footnote{The point is that this expression does not involve first derivatives of $g$, which has the advantage of making the gauge-fixed Einstein--Maxwell system principally scalar without having to assign different regularities to $g$ and $A$.} thus, $\Ups^M(g,A)$ equals $-\delta_g A$ up to terms of order $0$ in $A$. Note that given any $A$, one may solve the equation $\Ups^M(g,d a)\equiv \tr_g\delta_t^*d a=-\Ups^M(g,A)$, which is a linear scalar wave equation for $a$, and then $\Ups^M(g,A+d a)=0$. If one wants to put $(g,A)$ into the wave map and Lorenz gauge simultaneously, one first finds $\phi$ with $\Ups^E(\phi_*g)=0$ by solving a wave map equation, as explained above, and then $a$ such that $\Ups^M(\phi_*g,\phi_*A+d a)=0$, as just discussed.

We then consider the gauge-fixed system
\begin{equation}
\label{EqBasicNLDeTurckLorenzA}
  P_{DT}(g,d A) = 0,\qquad P_L(g,A) = 0,
\end{equation}
with $P_{DT}$ as in \eqref{EqBasicNLDeTurckFOp}, and
\[
  P_L(g,A) := \delta_g d A - d\Ups^M(g,A).
\]
This system is again quasilinear and principally scalar (after multiplying the first equation by $2$); in fact, due to our definition of $\Ups^M(g,A)$, the equation $P_L(g,A)$ does \emph{not} involve second derivatives of $g$ (while $P_{DT}(g,d A)$ now involves first derivatives of $A$), hence $g$ and $A$ live at the same level of regularity.

In order to solve the initial value problem for the Einstein--Maxwell system for $(g,A)$ by means of the hyperbolic formulation \eqref{EqBasicNLDeTurckLorenzA}, one first constructs Cauchy data $(g_0,g_1)$ for the metric $g$ as above. Next, one constructs Cauchy data for $A$; for example, one can write $\star_h\bfB=d A_0$ on $\Sigma_0$, with $A_0\in\CI(\Sigma_0;T^*\Sigma_0)\subset\CI(\Sigma_0;T^*_{\Sigma_0}M^\circ)$ (via the metric $g$), and one can then pick $A_1\in\CI(\Sigma_0;T^*_{\Sigma_0}M^\circ)$ so that any $A\in\CI(M^\circ;T^*M^\circ)$ with $\gamma_0(A)=(A_0,A_1)$ satisfies $\star_h i^*d A=\bfB$, $-i^* \iota_N d A=\bfE$, and $\Ups^M(g,A)=0$ at $\Sigma_0$; see \S\ref{SubsecKNdSIni} for details.

Having these Cauchy data at hand, one can now solve the system \eqref{EqBasicNLDeTurckLorenzA}. The first constraint equation in \eqref{EqBasicNLConstraints2}, written as \eqref{EqBasicNLConstraints22}, shows that $N\Ups^M(g,A)=0$ at $\Sigma_0$, so $\gamma_0(\Ups^M(g,A))=0$; as before, we also find $\gamma_0(\Ups^E(g))=0$. But then
\[
  0 = \delta_g P_L(g,A) = -\delta_g d \Ups^M(g,A)
\]
implies that $\Ups^M(g,A)\equiv 0$, hence we in fact have a solution of $\delta_g d A=0$ in this gauge. This in turn, using the second Bianchi identity and \eqref{EqBasicDerEMDelT} again, implies $\Ups^E(g)\equiv 0$, and we have therefore found a solution of the Einstein--Maxwell system \eqref{EqBasicDerEM} in the gauge $\Ups^E(g)=0$, $\Ups^M(g,A)=0$.

Again, there are two immediate ways in which one can modify the particular hyperbolic operator $P_L$: first, by using a different gauge condition $\Ups^M(g,A)=\kappa$ for a suitable function $\kappa\in\CI(M^\circ)$ supported away from $\Sigma_0$; and second, by replacing the second $d$ by a first order differential operator
\begin{equation}
\label{EqBasicNLtd}
  \td \colon \CI(M^\circ) \to \CI(M^\circ;T^*M^\circ),
\end{equation}
which agrees with $d$ to leading order, i.e.\ $\sigma_1(\td)(\zeta)=i\zeta$, $\zeta\in T_p M^\circ$, $p\in M^\circ$. The constraint propagation equation for $\Ups^M(g,A)$ then becomes
\[
  \delta_g\td\Ups^M(g,A)=0.
\]
As in the discussion of modifications of $P_{DT}$, such modifications are only relevant for the low energy, long time behavior of $(g,A)$, see \S\ref{SecCD}.

In particular, this discussion applies in the case when $\Sigma_0\cong I_r\times\Sph^2$ as in \eqref{EqBasicDerEMHole} and the magnetic charge $Q_m(\bfB)$ vanishes. By the discussion around \eqref{EqBasicDerEMRotation} and Lemma~\ref{LemmaBasicDerEMRotationCharges}, the constraint equations are invariant upon replacing $(\bfE,\bfB)$ by $(\bfE_\theta,\bfB_\theta)$, and for suitable $\theta$, we indeed have $[\star_h\bfB_\theta]=0$, thus we can solve the Einstein--Maxwell system for $(g,A)$ with initial data $(\bfE_\theta,\bfB_\theta)$ for $d A$, and then $F:=(d A)_{-\theta}$ yields a solution of \eqref{EqBasicDerEMCurv1}--\eqref{EqBasicDerEMCurv2} with initial data $(\bfE,\bfB)$ for $F$. Due to the additional flexibility of the formulation \eqref{EqBasicDerEM} afforded by the gauge freedom in $A$, the formulation \eqref{EqBasicDerEM} (or rather \eqref{EqBasicNLDeTurckLorenzA}), together with this argument, is thus what we will use to study the stability problem for charged black holes.

\subsection{Initial value problems for the linearized system}
\label{SubsecBasicLin}

For our stability arguments, the key is to understand certain properties of the linearization around special solutions of the Einstein--Maxwell system, as well as of the linearization of suitable gauge-fixed formulations.

Thus, suppose $(g_s,F_s)$, $s$ near $0$, is a smooth family of solutions of the Einstein--Maxwell system \eqref{EqBasicDerEMCurv1}--\eqref{EqBasicDerEMCurv2} on a spacetime $M^\circ=\R\times\Sigma_0$; let $(g,F):=(g_0,F_0)$. Then $(\gdot,\Fdot)=\frac{d}{ds}(g_s,F_s)|_{s=0}$ satisfies the \emph{linearized Einstein--Maxwell system}
\begin{equation}
\label{EqBasicLinEM}
\begin{gathered}
  D_g(\Ric+\Lambda)(\gdot) = 2 D_{g,F}T(\gdot,\Fdot), \\
  d\Fdot=0,\quad D_{g,F}(\delta_{(\cdot)}(\cdot))(\gdot,\Fdot)\equiv\frac{d}{ds}(\delta_{g_s}F_s)\Big|_{s=0} = 0.
\end{gathered}
\end{equation}
Since for each $s$ and any fixed $\theta\in\R$, Lemma~\ref{LemmaBasicDerEMRotation} applies to $(g_s,F_s)$, we find that a solution $(\gdot,\Fdot)$ of the linearized system gives rise to a family $(\gdot,\Fdot_\theta)$ of solutions, where
\begin{equation}
\label{EqBasicLinFdotRotation}
  \Fdot_\theta := \cos(\theta)\Fdot + \sin(\theta)\frac{d}{ds}(\star_{g_s}F_s)\Big|_{s=0}
\end{equation}
only depends on $(g,F)$ and $(\gdot,\Fdot)$.

Suppose $\Sigma_0$ is spacelike for $g$. The derivatives $(\hdot,\kdot,\bfEdot,\bfBdot)$ of the initial data $(h_s,k_s,\bfE_s,\bfB_s)$ at $\Sigma_0$ then satisfy the linearizations of the constraint equations \eqref{EqBasicNLConstraints1}--\eqref{EqBasicNLConstraints2} around the data $(h,k,\bfE,\bfB)=(h_0,k_0,\bfE_0,\bfB_0)$ induced by $(g,F)$. Alternatively, if $N_s$ denotes the smooth family of future timelike unit normals to $\Sigma_0$ with respect to $g_s$, one can differentiate the constraints $((\Ein-\Lambda)(g_s)-2 T_{g_s,F_s})(N_s,V)=0$, $V\in T_{\Sigma_0}M^\circ$, at $s=0$, obtaining
\[
  \bigl(D_g(\Ein-\Lambda)(\gdot) - 2 D_{g,F}T(\gdot,\Fdot)\bigr)(N,V) = 0,\quad V\in T_{\Sigma_0}M^\circ,
\]
which is equivalent to the linearization of \eqref{EqBasicNLConstraints1}. The constraints for the electromagnetic field on the other hand read
\[
  d(i^*\Fdot) = 0,\quad D_{g,F}(\delta_{(\cdot)}(\cdot))(\gdot,\Fdot)(N) = 0,
\]
which are equivalent to the linearization of \eqref{EqBasicNLConstraints2}.

The transformation \eqref{EqBasicLinFdotRotation} induces a transformation $(\bfEdot,\bfBdot)\mapsto(\bfEdot_\theta,\bfBdot_\theta)$ on the level of initial data. If $M^\circ\cong\R\times\Sigma_0$, we can define linearized charges by differentiating \eqref{EqBasicDerCharges} for $(g_s,F_s)$ in place of $(g,F)$ at $s=0$, and then the analogue of Lemma~\ref{LemmaBasicDerEMRotationCharges} holds for the linearized fields and charges.

Given initial data satisfying the linearized constraints, we now indicate briefly how to solve the linearized system \eqref{EqBasicLinEM}: one considers the gauge-fixed version
\begin{equation}
\label{EqBasicLinGaugedF}
\begin{gathered}
  D_g(\Ric+\Lambda)(\gdot) - \delta_g^*D_g\Ups^E(\gdot) = 2 D_{g,F}T(\gdot,\Fdot), \\
  d D_{g,F}(\delta_{(\cdot)}(\cdot))(\gdot,\Fdot) + \delta_g d\Fdot = 0.
\end{gathered}
\end{equation}
(If all $g_s$ satisfy the gauge condition $\Ups^E(g_s)=0$, then $(\gdot,\Fdot)$ indeed satisfies this equation, as follows from differentiating $\delta_{g_s}^*\Ups^E(g_s)=0$ at $s=0$.) This is now a linear, principally scalar hyperbolic system, up to multiplication of the first equation by $2$, and taking into account the relative regularity of $\gdot$ and $\Fdot$, as discussed after equation~\eqref{EqBasicNLDeTurckF}. Indeed, we recall from \cite[\S3]{GrahamLeeConformalEinstein} that
\begin{equation}
\label{EqBasicNLLinRic}
  (D_g\Ric)(\gdot) = \frac{1}{2}\Box_g\gdot - \delta_g^*\delta_g G_g\gdot + \sR_g(\gdot),
\end{equation}
where $\sR_g(r)_{\mu\nu}=\Riem(g)_{\kappa\mu\nu\lambda}r^{\kappa\lambda} + \frac{1}{2}(\Ric(g)_{\mu\lambda}r_\nu{}^\lambda+\Ric(g)_{\nu\lambda}r_\mu{}^\lambda)$, and
\begin{equation}
\label{EqBasicNLLinUpsE}
  D_g\Ups^E(\gdot) = -\delta_g G_g\gdot + \sE_g(\gdot),
\end{equation}
with $\sE_g(r)_\mu=C^\lambda_{\kappa\nu}(g_{\mu\lambda}r^{\kappa\nu}-r_{\mu\lambda}g^{\kappa\nu})$ and $C_{\mu\nu}^\kappa=\frac{1}{2}(t^{-1})^{\kappa\lambda}(t_{\mu\lambda;\nu}+t_{\nu\lambda;\mu}-t_{\mu\nu;\lambda})$, while for a 1-form $u$, we have
\[
  D_g(\delta_{(\cdot)}^*u)(\gdot)_{\mu\nu}:=\frac{d}{ds}(\delta_{g+s\gdot}^*u)_{\mu\nu}|_{s=0}=-\frac{1}{2}(\gdot_{\mu\kappa;\nu}+\gdot_{\nu\kappa;\mu}-\gdot_{\mu\nu;\kappa})u^\kappa,
\]
which is thus a 0-th order differential operator with coefficients depending only on first covariant derivatives of $g$. Lastly, $T(g,F)$ is of order $0$ in both $g$ and $F$, hence its linearization in $(g,F)$ is of the same type.

If one arranges for the map taking initial data satisfying the non-linear constraints into correctly gauged Cauchy data for $(g,F)$ to be differentiable, the derivative of this map maps linearized initial data into correctly gauged Cauchy data for this linearized system. After solving \eqref{EqBasicLinGaugedF}, one can then verify that
\begin{equation}
\label{EqBasicLinFdotSol}
  d\Fdot=0, \quad D_{g,F}(\delta_{(\cdot)}(\cdot))(\gdot,\Fdot)=0,
\end{equation}
In order to conclude the argument, one needs to verify that this implies
\begin{equation}
\label{EqBasicLin2ndBianchi}
  \delta_g G_g\bigl(D_g(\Ric+\Lambda)(\gdot) - 2 D_{g,F}T(\gdot,\Fdot)\bigr) = 0,
\end{equation}
since this gives the wave equation  $\delta_g G_g\delta_g^*(D_g\Ups^E(\gdot))=0$, and it is then easy to conclude $D_g\Ups^E(\gdot)\equiv 0$, finishing the argument. In order to prove \eqref{EqBasicLin2ndBianchi}, we observe that the second Bianchi identity and \eqref{EqBasicDerEMDelT} imply
\[
  \delta_g G_g(\Ric(g)+\Lambda g-2 T(g,F)) = 2\tr_g^{13}(\delta_g F\otimes F)+\tr_g^{13}\tr_g^{24}(F\otimes d F)
\]
for \emph{all} $(g,F)$; linearizing around the $(g,F)$ at hand, for which $\Ric(g)+\Lambda g-2 T(g,F)=0$, we deduce \eqref{EqBasicLin2ndBianchi} by using $\delta_g F=0$, $d F=0$ as well as \eqref{EqBasicLinFdotSol}.

\emph{Assuming} that $F_s=d A_s$ for all $s$, with $A_s$ depending smoothly on $s$, and restricting to those linearized fields $\Fdot$ for which $[\Fdot]=0\in H^2(M^\circ,\R)$, so $\Fdot=d\Adot$, we can consider the linearized Einstein--Maxwell system in the form
\begin{equation}
\label{EqBasicLinEinsteinMaxwell}
  \sL(\gdot,\Adot) = 0,
\end{equation}
where
\begin{equation}
\label{EqBasicLinEinsteinMaxwellExpl}
\begin{split}
  &\sL(\gdot,\Adot) = \bigl(\sL_1(\gdot,\Adot),\sL_2(\gdot,\Adot)\bigr), \\
  &\qquad \sL_1(\gdot,\Adot) = D_g(\Ric+\Lambda)(\gdot) - 2 D_{g,d A}T(\gdot,d\Adot), \\
  &\qquad \sL_2(\gdot,\Adot) = D_{g,A}(\delta_{(\cdot)}d(\cdot))(\gdot,\Adot),
\end{split}
\end{equation}
One can detect the condition $[\Fdot]=0$ already on the level of linearized initial data: indeed, with $i\colon\Sigma_0\hra M^\circ$ the inclusion (which induces an isomorphism in cohomology), we simply have $i^*\Fdot=D_{h,\bfB}(\star_{(\cdot)}(\cdot))(\hdot,\bfBdot)$. (By the linearization of the second constraint in \eqref{EqBasicNLConstraints2}, this form is indeed closed.) Moreover, in the special case considered in Lemma~\ref{LemmaBasicDerEMRotationCharges}, we have $[\Fdot]=0$ if and only if the linearized magnetic charge $\dot Q_m=\int_S\Fdot$ vanishes; by means of a `rotation' of the form \eqref{EqBasicLinFdotRotation}, which on the level of initial data corresponds to a rotation of $(\bfEdot,\bfBdot)$ akin to \eqref{EqBasicDerEMRotation}, this can always be arranged. In the context of the present paper, it thus suffices to study \eqref{EqBasicLinEinsteinMaxwell}.

One can then solve the initial value problem for the linearized system \eqref{EqBasicLinEinsteinMaxwell} by considering the gauge-fixed system
\begin{equation}
\label{EqBasicLinAEq}
\begin{gathered}
  D_g(\Ric+\Lambda)(\gdot) - \delta_g^*D_g\Ups^E(\gdot) = 2 D_{g,d A}T(\gdot,d\Adot), \\
  D_{g,A}(\delta_{(\cdot)} d(\cdot))(\gdot,\Adot) - d\bigl(D_{g,A}\Ups^M(\gdot,\Adot)\bigr) = 0.
\end{gathered}
\end{equation}
Note here that the second equation indeed has principal part $(\delta_g d+d\delta_g)\Adot$, and moreover only involves up to first derivatives of $\gdot$. Thus, we can solve the initial value problem in a manner that is analogous to our treatment of \eqref{EqBasicLinGaugedF}; see \S\ref{SubsecKNdSIni} for details. The only additional input is that we need to derive the propagation equation for the linearized gauge condition $D_{g,A}\Ups^M(\gdot,\Adot)=0$. This however is straightforward: since $\delta_g d A=0$ for the solution $(g,A)$ around which we linearize, and since one \emph{always} has $\delta_g\delta_g d A=0$, one finds
\[
  \delta_g\bigl(D_{g,A}(\delta_{(\cdot)} d(\cdot))(\gdot,\Adot)\bigr) = 0,
\]
and thus \eqref{EqBasicLinAEq} implies the wave equation $\delta_g d(D_{g,A}\Ups^M(\gdot,\Adot))=0$, as desired.

The action \eqref{EqBasicNLGaugeAction} of the gauge group for the non-linear Einstein--Maxwell system induces, by linearization around $(\phi,a)=(\Id,0)$, an action of $\CI(M^\circ;T M^\circ)\times\CI(M^\circ) \ni (V,a)$ on a solution $(\gdot,\Adot)$ of \eqref{EqBasicLinEinsteinMaxwell} via
\begin{equation}
\label{EqBasicLinGaugeAction}
  (\gdot,\Adot) \mapsto (\gdot + \cL_V g, \Adot + \cL_V A + d a);
\end{equation}
equivalently, $\CI(M^\circ;T^*M^\circ)\times\CI(M^\circ)\ni(\omega,a)$ acts on $(\gdot,\Adot)$ by the same formula via the identification $V=\omega^\sharp$. We note that $\cL_{\omega^\sharp}g=2\delta_g^*\omega$. For notational convenience, we make the following definition:

\begin{definition}
\label{DefBasicLinLie}
  For a vector field $V$ and a tensor $T$, define $\tcL_T V := \cL_V T$.
\end{definition}

Using the action \eqref{EqBasicLinGaugeAction}, we can put any given solution $(\gdot,\Adot)$ of \eqref{EqBasicLinEinsteinMaxwell} into the gauge $D_g\Ups^E(\gdot)=0$, $D_{g,A}\Ups^M(\gdot,\Adot)=0$ by solving the forced wave equation
\begin{equation}
\label{EqBasicLinGaugeG}
  D_g\Ups^E(\gdot+\cL_V g)=0
\end{equation}
for the vector field $V$ --- this is indeed a wave equation, since $D_g\Ups^E\circ\tcL_g$ is equal to $\Box_g$ (on vector fields) modulo lower order terms --- and then solving the forced wave equation
\begin{equation}
\label{EqBasicLinGaugeA}
  D_{g,A}\Ups^M(\gdot+\cL_V g,\Adot+\cL_V A + d a)=0
\end{equation}
for the scalar function $a$; here, we note that $\Ups^M\circ d$ is equal to $\Box_g$ (on functions) modulo lower order terms. If one chooses trivial Cauchy data for $\omega$ and $a$ on $\Sigma_0$, then $(\gdot+\cL_V g,\Adot+\cL_V A+d a)$ solves \eqref{EqBasicLinEinsteinMaxwell}, and induces the same initial data $(\hdot,\kdot,\bfEdot,\bfBdot)$ as $(\gdot,\Adot)$.

Furthermore, as in the non-linear setting, we have a considerable freedom in the precise formulation of the gauge-fixed equation \eqref{EqBasicLinAEq}, in that we can replace $\delta_g^*$ in the first equation and the second $d$ in the second equation, respectively, by any operator $\tdel^*$, resp.\ $\td$, with the same principal symbol, thus modifying the way in which the gauge term enters the equation. The propagation equations for the linearized gauge terms then become
\begin{equation}
\label{EqBasicLinGaugePropMod}
  \delta_g G_g\tdel^*(D_g\Ups^E(\gdot))=0,\quad \delta_g\td(D_{g,A}\Ups^M(\gdot,\Adot))=0.
\end{equation}
Moreover, we can use different gauge conditions, for example $D_g\Ups^E(\gdot)=\theta$ and $D_{g,A}\Ups^M(\gdot,\Adot)=\kappa$. In fact, the choice of the hyperbolic formulation of the \emph{non-linear} Einstein--Maxwell system will to a large extent be made so that the linearized system has desirable properties, specifically that is implements constraint damping and has good high frequency properties.

\section{Kerr--Newman--de~Sitter black holes}
\label{SecKNdS}

With the cosmological constant $\Lambda>0$ fixed, the KNdS family of solutions of the Einstein--Maxwell system is parameterized by the mass $\bhm>0$ of the black hole, the angular momentum per unit mass $a\in\R$, and the electric charge $Q_e\in\R$. It will be convenient to parameterize the angular momentum redundantly by a vector $\bfa\in\R^3$. We thus package these parameters into the parameter space
\[
  B = \{ (\bhm,\bfa,Q_e) \} \subset \R\times\R^3\times\R.
\]
In \S\ref{SubsecKNdSMag}, we will extend this parameter space to also allow for magnetic charges, in which case we need to use the more general formulation \eqref{EqBasicDerEMCurv1}--\eqref{EqBasicDerEMCurv2} of the Einstein--Maxwell equations.

We first introduce the Reissner--Nordstr\"om--de~Sitter (RNdS) subfamily of subextremal, non-rotating charged black holes in \S\ref{SubsecKNdS0}; these spacetimes have parameters $(\bhm,\bfzero,Q_e)$, and we shrink the parameter space $B$ to be a small neighborhood of a fixed member of the RNdS family. In \S\ref{SubsecKNdSa}, we then show how to realize the family of slowly rotating KNdS metrics with parameters in $B$ as a smooth family of stationary metrics on a fixed 4-manifold.

\subsection{Reissner--Nordstr\"om--de~Sitter black holes}
\label{SubsecKNdS0}

We fix a set of parameters
\begin{equation}
\label{EqKNdS0Params}
  b_0 = (\bhm[0],\bfzero,Q_{e,0});
\end{equation}
the RNdS solution of the Einstein--Maxwell equations is then described by the spherically symmetric, static metric
\begin{equation}
\label{EqKNdS0Metric}
  g_{b_0} = \mu(r)\,dt^2 - \mu(r)^{-1}\,dr^2 - r^2\slg,
\end{equation}
$\slg$ the round metric on $\Sph^2$, where
\begin{equation}
\label{EqKNdS0Mu}
  \mu(r) = 1-\frac{2\bhm[0]}{r}-\frac{\Lambda r^2}{3} + \frac{Q_{e,0}^2}{r^2},
\end{equation}
and the 4-potential and electromagnetic field
\begin{equation}
\label{EqKNdS0PotentialStatic}
  \breve A_{b_0} = -Q_{e,0} r^{-1}\,dt,\quad F_{b_0}=d\breve A_{b_0} = Q_{e,0}r^{-2}\,dr\wedge dt.
\end{equation}
(We reserve the name $A_{b_0}$ for the 4-potential that we eventually use, see \eqref{EqKNdS0Potential}.) The notation for the electric charge is consistent; indeed, one has $Q_e(g_{b_0},F_{b_0})=Q_{e,0}$, while $Q_m(F_{b_0})=0$, see \eqref{EqBasicDerCharges}. We assume that the parameter $b_0\in\R^5$ is \emph{non-degenerate}:

\begin{definition}
\label{DefKNdS0NonDeg}
  The parameters $(\Lambda,\bhm[0],Q_{e,0})$, with $\Lambda>0$, are \emph{non-degenerate} if and only if for $b_0=(\bhm[0],\bfzero,Q_{e,0})$ the two largest positive roots
  \begin{equation}
  \label{EqKNdS0Horizons}
    r_{b_0,-}<r_{b_0,+}
  \end{equation}
  of $\mu$ are simple (thus $\mu(r)>0$ for $r\in(r_{b_0,-},r_{b_0,+})$), and in the interval $(r_{b_0,-},r_{b_0,+})$, the function $r^{-2}\mu$ has a unique non-degenerate critical point $r_P$.
\end{definition}

The expressions \eqref{EqKNdS0Metric} and \eqref{EqKNdS0PotentialStatic} are then valid in the \emph{static region}
\begin{equation}
\label{EqKNdS0StaticRegion}
  \cM = \R_t \times \cX,\quad \cX = (r_{b_0,-},r_{b_0,+})_r \times \Sph^2.
\end{equation}
Let us discuss conditions for non-degeneracy: the sharp condition for uncharged black holes, i.e.\ $Q_{e,0}=0$ is
\[
  0<9\Lambda\bhm[0]^2<1,
\]
see \cite[\S2]{HintzVasyCauchyHorizon}. Since non-degeneracy is an open condition, black holes with $0<9\Lambda\bhm[0]^2<1$ and small charge $Q_{e,0}\neq 0$ are non-degenerate as well. Another class of (nearly extremal) non-degenerate spacetimes can be obtained by fixing $|Q_{e,0}|<\bhm[0]$ and then taking $\Lambda>0$ small. Following \cite[Appendix~A]{KodamaIshibashiCharged}, we determine the conditions on $\Lambda,\bhm[0]$ and $Q_{e,0}$ in general:

\begin{prop}
\label{PropRNdSNondeg}
  The parameters $(\Lambda,\bhm,Q)$ are the parameters of a non-de\-gen\-er\-ate RNdS spacetime if and only if $Q=0$ and $0<9\Lambda\bhm^2<1$, or $Q\neq 0$ and
  \begin{equation}
  \label{EqRNdSNondeg}
    D := 9\bhm^2-8 Q^2>0, \quad \max\Bigl(0,\frac{6(\bhm-\sqrt D)}{(3\bhm-\sqrt D)^3}\Bigr) < \Lambda < \frac{6(\bhm+\sqrt D)}{(3\bhm+\sqrt D)^3}.
  \end{equation}
\end{prop}

In particular, for $Q<\bhm$, the lower bound for $\Lambda$ simply reads $\Lambda>0$.

\begin{proof}[Proof of Proposition~\ref{PropRNdSNondeg}]
  Write
  \[
    \mu=1-\frac{2\bhm}{r}-\frac{\Lambda r^2}{3}+\frac{Q^2}{r^2},
  \]
  and
  \[
    \mu'=\frac{2 r}{3}(\wh\mu-\Lambda), \quad \wh\mu=\frac{3}{r^2}\Bigl(\frac{\bhm}{r}-\frac{Q^2}{r^2}\Bigr).
  \]
  Since $\mu(r)\to\infty$ as $r\to 0+$ and $\mu(r)\to-\infty$ as $r\to\infty$, non-degeneracy requires that there exists an $\rho\in(0,r_-)$ such that $\mu'(\rho)=0$ and
  \[
    \mu(\rho) = 1-\frac{3\bhm}{\rho}+\frac{2 Q^2}{\rho^2} < 0;
  \]
  this can only hold if $D=9\bhm^2-8 Q^2>0$. In this case, define
  \begin{equation}
  \label{EqRNdSCriticalPoints}
    r_{1-} := \frac{3\bhm - \sqrt D}{2}, \quad r_{2-} := \frac{3\bhm + \sqrt D}{2};
  \end{equation}
  thus,
  \[
    \mu'(r)=0 \quad\Longrightarrow\quad \mu(r)=r^{-2}(r-r_{1-})(r-r_{2-}).
  \]
  Next, we note that $\wh\mu'=3r^{-5}(4 Q^2-3\bhm r)=0$ iff $r=\wh r:=\frac{4 Q^2}{3\bhm}$, while $\pm\wh\mu'<0$ for $r>0$, $\pm(r-\wh r)>0$, so $\wh\mu$ is strictly monotonically increasing for $0<r<\wh r$ and decreasing for $r>\wh r$. We remark that $r_{1-}<\wh r<r_{2-}$. Now, the non-degeneracy implies the existence of roots $0<r_1<r_2$ of the equation $\mu'=0$, or equivalently of the equation $\wh\mu=\Lambda$, and they satisfy $\mu(r_1)<0$, $\mu(r_2)>0$. See Figure~\ref{FigRNdSNondeg}. The first inequality is equivalent to $r_1\in(r_{1-},r_{2-})$, and the second then to $r_2>r_{2-}$. The latter can be rephrased as $\wh\mu(r_{2-})>\Lambda$, which reads
  \[
    \Lambda < \frac{6(\bhm+\sqrt{D})}{(3\bhm+\sqrt{D})^3};
  \]
  the former on the other hand is equivalent to $\min(\wh\mu(r_{1-}),\wh\mu(r_{2-}))<\Lambda$; since $\wh\mu(r_{2-})>\Lambda$, this simply reads $\wh\mu(r_{1-})<\Lambda$, or
  \[
    \Lambda > \frac{6(\bhm-\sqrt D)}{(3\bhm-\sqrt D)^3},
  \]
  thus proving the necessity of \eqref{EqRNdSNondeg}.

  \begin{figure}[!ht]
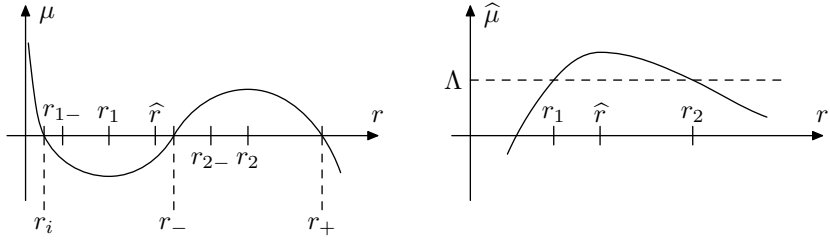

  \centering
  \inclfig{RNdSNondeg}
  \caption{\textit{Left:} Schematic graph of $\mu$, with its zeros $r_i$, $r_-$ and $r_+$, critical points $r_1$ and $r_2$, and the critical points $r_{1-}$ and $r_{2-}\equiv r_P$ of $r^{-2}\mu$. \textit{Right:} Schematic graph of $\wh\mu$, with its critical point $\wh r$.}
  \label{FigRNdSNondeg}
  \end{figure}
  
  Conversely, the above arguments show that \eqref{EqRNdSNondeg} implies the existence of $r_1,r_2$ with $r_{1-}<r_1<r_{2-}<r_2$ such that $\mu(r_1)<0$, $\mu(r_2)>0$; in particular, since $r^2\mu$ is a quartic polynomial with a negative root (due to $(r^2\mu)|_{r=0}>0$ and $\mu(r)\to-\infty$ as $r\to-\infty$), $\mu$ has exactly three simple positive roots $r_i,r_-,r_+$ with
  \begin{equation}
  \label{EqRNdSRoots}
    0<r_i<r_{1-}<r_-<r_{2-}<r_+.
  \end{equation}
  It remains to check that $r^{-2}\mu$ has a unique non-degenerate critical point $r_P$ in $(r_-,r_+)$; but
  \[
    -\frac{r^5}{2}(r^{-2}\mu)' = r^2-3\bhm r+2 Q^2=(r-r_{1-})(r-r_{2-}),
  \]
  hence we have $r_P=r_{2-}$, and $(r^{-2}\mu)''<0$ at $r=r_P$. The proof is complete.
\end{proof}

The proof also shows that for non-degenerate RNdS spacetimes, $\mu$ has a unique non-degenerate maximum at some radius $r_c\in(r_{b_0,-},r_{b_0,+})$.

The singularity of the expression \eqref{EqKNdS0Metric} at $r=r_{b_0,\pm}$ is merely a coordinate singularity. One way to extend the metric past $r=r_{b_0,\pm}$ is to introduce, near $r=r_{b_0,\pm}$, the function
\begin{equation}
\label{EqKNdS0NullCoord}
  t_0 = t - T_0(r),\quad T_0'=\pm\mu^{-1},
\end{equation}
in which case we find
\begin{equation}
\label{EqKNdS0MetricNull}
  g_{b_0} = \mu(r)\,dt_0^2 \pm 2\,dt_0\,dr - r^2\slg.
\end{equation}
This extends analytically across $r=r_{b_0,\pm}$; we can take the 4-potential to be $-Q_{e,0}r^{-1}dt_0$, which differs from $\breve A_{b_0}$ by an exact form. By analyticity, the thus extended metric and 4-potential continue to solve the Einstein--Maxwell system. As we will discuss in more detail in \S\ref{SubsecKNdSGeo}, there is an \emph{event horizon} $\cH^+$ at $r=r_{b_0,-}$, and a \emph{cosmological horizon} $\ol\cH{}^+$ at $r=r_{b_0,+}$. Note that level sets of $t_0$ are null hypersurfaces transversal to the event horizon (near $r_{b_0,-}$) and the cosmological horizon (near $r_{b_0,+}$). This form of the metric will be useful for calculations near the horizons.

Since the function $t_0$, defined only locally near $r_{b_0,\pm}$, does not give a single global extension of the metric, we instead define a function $T(r)\in\CI((r_{b_0,-},r_{b_0,+}))$ such that
\begin{equation}
\label{EqKNdS0CoordT}
  T'=\pm(\mu^{-1}+c_{b_0,\pm})\quad\tn{near }r_{b_0,\pm},
\end{equation}
with $c_{b_0,\pm}(r)$ smooth for $r$ up to and including $r_{b_0,\pm}$, and put
\begin{equation}
\label{EqKNdS0CoordDef}
  t_* = t - T(r).
\end{equation}
One can choose $c_{b_0,\pm}$ and thus $t_*$ in such a way that $dt_*$ is everywhere timelike; see Figure~\ref{FigKNdS0TimeFns}.
\begin{figure}[!ht]
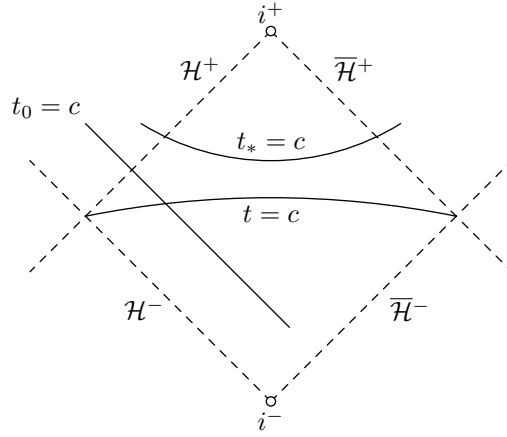

\centering
\inclfig{KNdS0TimeFns}
\caption{Penrose diagram of a RNdS spacetime, together with level sets of the static time coordinate $t$, the null coordinate $t_0$ from \eqref{EqKNdS0NullCoord} (for the bottom sign), and the timelike function $t_*$ from \eqref{EqKNdS0CoordDef}.}
\label{FigKNdS0TimeFns}
\end{figure}
In fact, the proof of \cite[Lemma~3.1]{HintzVasyKdSStability} carries over verbatim to give:
\begin{lemma}
\label{LemmaKNdS0Ext}
  Let $r_c$ denote the (unique and non-degenerate) critical point of $\mu$ in $(r_{b_0,-},r_{b_0,+})$, and put $c^2=\mu(r_c)^{-1/2}$. For a suitable choice of $t_*$, the metric $g_{b_0}$ and the dual metric $G_{b_0}$ take the form
  \begin{gather*}
    g_{b_0} = \mu\,dt_*^2 - 2\nu\,dt_*\,dr - c^2\,dr^2 - r^2\slg, \\
    G_{b_0} = c^2\,\pa_{t_*}^2 - 2\nu\,\pa_{t_*}\,\pa_r - \mu\,\pa_r^2 - r^{-2}\slG,
  \end{gather*}
  where the smooth function $\nu=\nu(r)$, defined near $[r_{b_0,-},r_{b_0,+}]$, is given by
  \[
    \nu(r) = \mp\sqrt{1-c^2\mu},\quad \pm(r-r_c)>0.
  \]
  In particular, $|dt_*|^2=c^2>0$ is constant.
\end{lemma}

This provides an extension of the metric to the 4-manifold
\begin{equation}
\label{EqKNdS0Mf}
  M^\circ := \R_{t_*} \times X,\qquad X=I_r\times\Sph^2,
\end{equation}
where
\begin{equation}
\label{EqKNdS0MfRad}
  I_r=(r_{I,-},r_{I,+}), \quad r_{I,\pm}=r_{b_0,\pm}\pm 2\eps
\end{equation}
for $\eps>0$ small and fixed. We can then define the electromagnetic field on $M^\circ$ via
\begin{equation}
\label{EqKNdS0Potential}
  A_{b_0} = -Q_{e,0}r^{-1}\,dt_*,\quad F_{b_0} = Q_{e,0}r^{-2}\,dr\wedge dt_*.
\end{equation}
The thus extended metric and potential $(g_{b_0},A_{b_0})$ furnish a solution of the Einstein--Maxwell system \eqref{EqBasicDerEM}.

The stationary structure of RNdS spacetimes is conveniently encoded by partially compactifying $M^\circ$ at future infinity. Thus, we define
\begin{equation}
\label{EqKNdS0MfComp}
\begin{split}
  M &= \bigl(M^\circ \sqcup ([0,\infty)_\tau \times X)\bigr) / \sim, \\
    &\qquad (t_*,x)\sim(\tau,x),\quad \tau:=e^{-t_*},
\end{split}
\end{equation}
with the algebra of smooth functions on $M$ generated by $\tau$ and smooth functions on $X$; thus, $\tau$ is a boundary defining function. We often identify
\[
  \pa M = \{\tau=0\} \cong X.
\]
By stationarity $g_{b_0}$ is a smooth non-degenerate Lorentzian b-metric on $M$,
\[
  g_{b_0} \in \CI(M;S^2\,\Tb^*M),
\]
and $A_{b_0}$, resp.\ $F_{b_0}$, is a smooth b-1-form, resp.\ b-2-form,
\[
  A_{b_0}\in\CI(M;\Tb^*M),\quad F_{b_0}\in\CI(M;\Lambda^2\,\Tb^*M).
\]
Indeed, these assertions follow from $dt_*=-\frac{d\tau}{\tau}$. See also Appendix~\ref{SecB}.

Near the horizons, the coordinates used in \eqref{EqKNdS0MetricNull} give rise to a local boundary defining function of $M$, namely
\begin{equation}
\label{EqKNdS0Tau0}
  \tau_0:=e^{-t_0},
\end{equation}
which is equivalent to (i.e.\ a smooth positive multiple of) $\tau$ on its domain of definition. In the black hole exterior region, but away from the horizons, i.e.\ for $r\in(r_{b_0,-},r_{b_0,+})$ bounded away from $r_{b_0,\pm}$, the static time coordinate $t$ also gives rise to a smooth boundary defining function, namely $\tau_s:=e^{-t}$, which is equivalent to $\tau$ there.

\subsection{Slowly rotating KNdS black holes}
\label{SubsecKNdSa}

We next recall the form of the KNdS solution of the Einstein--Maxwell system (see e.g.\ \cite[\S{A}]{PodolskyGriffithsKerrNewmanAcc}), given parameters
\[
  b = (\bhm,\bfa,Q_e),\quad a=|\bfa|\neq 0.
\]
In Boyer--Lindquist coordinates $(t,r,\theta,\phi)$ on $\R_t\times(r_{b,-},r_{b,+})\times\Sph^2_{\theta,\phi}$, with $r_{b,\pm}$ defined momentarily, and with the polar coordinates chosen such that the axis $\theta=0$ through the north poles of the spheres $r=const.$ is parallel to $\bfa$, we let
\begin{equation}
\label{EqKNdSaMetric}
\begin{split}
  g_b &= -\rho_b^2\Bigl(\frac{dr^2}{\wt\mu_b}+\frac{d\theta^2}{\kappa_b}\Bigr) + \frac{\wt\mu_b}{(1+\lambda_b)^2\rho_b^2}(dt-a\sin^2\theta\,d\phi)^2 \\
    &\qquad - \frac{\kappa_b\sin^2\theta}{(1+\lambda_b)^2\rho_b^2}(a\,dt-(r^2+a^2)\,d\phi)^2,
\end{split}
\end{equation}
where
\begin{gather*}
  \lambda_b=\frac{\Lambda a^2}{3},\quad \kappa_b=1+\lambda\cos^2\theta, \quad \rho_b^2=r^2+a^2\cos^2\theta, \\
  \wt\mu_b=(r^2+a^2)\Bigl(1-\frac{\Lambda r^2}{3}\Bigr)-2\bhm r+(1+\lambda_b)^2 Q_e^2;
\end{gather*}
then $r_{b,-}<r_{b,+}$ are the largest two roots of $\wt\mu_b$, see Lemma~\ref{LemmaKNdSaHorizons} below. In these coordinates, the 4-potential can be taken to be
\[
  \breve A_b = -\frac{Q_e r}{\rho_b^2}(dt-a\sin^2\theta\,d\phi),
\]
and the electromagnetic tensor is correspondingly given by
\begin{align*}
  F_b = d\breve A_b &= \frac{Q_e}{\rho_b^4}\bigl((r^2-a^2\cos^2\theta)\,dr\wedge(dt-a\sin^2\theta\,d\phi) \\
    &\qquad\qquad - 2 a r\cos\theta\sin\theta\,d\theta\wedge(a\,dt-(r^2+a^2)\,d\phi)\bigr).
\end{align*}
Thus, $Q_m(F_b)=0$. On the other hand, we compute
\begin{align*}
  \star_{g_b}F_b&=-\frac{Q_e}{\rho_b^4}\bigl(2 a r\cos\theta\,dr\wedge(dt-a \sin^2\theta\,d\phi) \\
  &\qquad\quad\qquad + (r^2-a^2\cos^2\theta)\sin\theta\,d\theta\wedge(a\,dt-(r^2+a^2)\,d\phi)\bigr),
\end{align*}
hence
\[
  Q_e(g_b,F_b)=\frac{1}{4\pi}\int_S \star_{g_b}F_b = \frac{Q_e}{2} \int_0^\pi \frac{(r^2-a^2\cos^2\theta)(r^2+a^2)}{(r^2+a^2\cos^2\theta)^2}\sin\theta\,d\theta = Q_e,
\]
as can be seen by first substituting $z=\cos\theta$, and then $y=\arctan(a z/r)$; this justifies the normalization of $Q_e$ in the form of $g_b$ and $\breve A_b$.

Since we focus on slowly rotating black holes in this paper, we assume that the parameter space $B$ is a small neighborhood of $b_0$. We remark however that in the present section, we could equally well work much more generally with any KNdS spacetime which is non-degenerate in a suitable sense, as discussed in the Kerr--de~Sitter setting in \cite[\S6.1]{VasyMicroKerrdS} and the discussion around \cite[equation~(6.13)]{VasyMicroKerrdS}.

\begin{lemma}
\label{LemmaKNdSaHorizons}
  There exist unique smooth functions $b\mapsto r_{b,\pm}$, defined for $b$ near $b_0$, which agree with $r_{b_0,\pm}$ in \eqref{EqKNdS0Horizons} at $b=b_0$, such that $\wt\mu_b(r_{b,\pm})=0$. Thus, $r_{b,\pm}\in I_r$ for $b$ near $b_0$, with $I_r$ defined in \eqref{EqKNdS0MfRad}.
\end{lemma}
\begin{proof}
  This follows from the simplicity of the roots $r_{b_0,\pm}$ of $\mu=r^{-2}\wt\mu_{b_0}$ and the implicit function theorem.
\end{proof}

We now describe the extension of $g_b$, $F_b$ and $\breve A_b$ (up to a gauge change) beyond the horizons $r=r_{b,\pm}$; this discussion is similar to \cite[\S3.2]{HintzVasyKdSStability}. Thus, we write
\[
  t_* = t - T_b(r), \quad \phi_* = \phi - \Phi_b(r),
\]
where
\[
  T_b' = \pm\Bigl(\frac{(1+\lambda_b)(r^2+a^2)}{\wt\mu_b} + c_{b,\pm}\Bigr), \quad
  \Phi_b' = \pm\Bigl(\frac{a(1+\lambda_b)}{\wt\mu_b} + \tilde c_{b,\pm}\Bigr);
\]
a simple argument, see \cite[Lemma~3.4]{HintzVasyKdSStability}, shows that one can choose $c_{b,\pm}$, resp.\ $a^{-1}\tilde c_{b,\pm}$, to agree with $c_{b_0,\pm}$ (defined implicitly in Lemma~\ref{LemmaKNdS0Ext}), resp.\ $0$, for $b=b_0$, and to depend smoothly on $b$ in the neighborhood $B$ of $b_0$. That is, fixing $r_{b_0,-}<r_1<r_2<r_{b_0,+}$, one can arrange $c_{b,-}(r),a^{-1}\tilde c_{b,-}(r)\in\CI(B_b\times(r_{I,-},r_2)_r)$ and $c_{b,+}(r),a^{-1}\tilde c_{b,+}(r)\in\CI(B_b\times(r_1,r_{I,+})_r)$. One then computes the dual metric $G_b$ of $g_b$ to be
\begin{equation}
\label{EqKNdSaDualMetric}
\begin{split}
  \rho_b^2 G_b &= -\wt\mu_b(\pa_r\mp c_{b,\pm}\pa_{t_*}\mp\tilde c_{b,\pm}\pa_{\phi_*})^2 \\
    &\qquad \pm 2 a(1+\lambda_b)(\pa_r \mp c_{b,\pm}\pa_{t_*} \mp \tilde c_{b,\pm}\pa_{\phi_*})\pa_{\phi_*} \\
    &\qquad \pm 2(1+\lambda_b)(r^2+a^2)(\pa_r\mp c_{b,\pm}\pa_{t_*}\mp\tilde c_{b,\pm}\pa_{\phi_*})\pa_{t_*} \\
    &\qquad - \frac{(1+\lambda_b)^2}{\kappa_b\sin^2\theta}(a\sin^2\theta\,\pa_{t_*} + \pa_{\phi_*})^2 - \kappa_b\pa_\theta^2,
\end{split}
\end{equation}
see \cite[equation~(3.17)]{HintzVasyKdSStability}. By construction, the two choices of sign give the same result in the overlap region $r_1<r<r_2$, and the arguments leading up to \cite[Proposition~3.5]{HintzVasyKdSStability} apply directly here as well, thus proving that
\[
  g_b \in \CI(B_b\times M^\circ;S^2 T^*M^\circ)
\]
is a smooth family of stationary metrics on $M^\circ$, and hence, in the compactified picture, $g_b \in \CI(B_b\times M;S^2\,\Tb^*M)$.

Next, we consider the vector potential
\[
  \breve A_b = -\frac{Q_e r}{\rho_b^2}\bigl(dt_*\pm c_{b,\pm}\,dr - a\sin^2\theta(d\phi_*\pm\tilde c_{b,\pm}\,dr)\bigr) \mp \frac{Q_e(1+\lambda_b)r}{\wt\mu_b}\,dr.
\]
The last term is an exact differential, hence we can define the 4-potential on $M^\circ$ by
\[
  A_b = -\frac{Q_e r}{\rho_b^2}\bigl(dt_*\pm c_{b,\pm}\,dr - a\sin^2\theta(d\phi_*\pm\tilde c_{b,\pm}\,dr)\bigr) \mp \chi\frac{Q_e(1+\lambda_b)r}{\wt\mu_b}\,dr,
\]
with $\chi=\chi(r)$ a smooth cutoff, identically $0$ near $r=r_{b,\pm}$, and identically $1$ on a large subinterval of $(r_{b,-},r_{b,+})$; where $\chi\equiv 1$, both choices of sign yield the same value of $A_b$, and near $r_{b,+}$, resp.\ $r_{b,-}$, we use the top, resp.\ bottom, sign. This expression is smooth near the poles $\theta=0,\pi$ of $\Sph^2$, as a simple coordinate change shows. The smoothness of $A_b$ as a smooth 1-form on $M^\circ$ depending on $b\in B$ near $b_0$ is easily established, similarly to the proof of \cite[Proposition~3.5]{HintzVasyKdSStability}: the smooth dependence of $a\tilde c_{b,\pm}\sin^2\theta\,dr$ on the parameter $b$ follows from that of $a^2\sin^2\theta=|\bfa|^2-\la\bfa,p/|p|\ra^2$ at a point $p\in X\subset\R^3$ (via the polar coordinate map), while the smooth dependence of $a\sin^2\theta\,d\phi_*=(a\pa_{\phi_*})^\flat$ (using the musical isomorphism on Euclidean $\R^3$) was established in the reference.

One can  alternatively use $c_{b,\pm}\equiv 0$, $\tilde c_{b,\pm}\equiv 0$ near $r_{b,\pm}$, giving rise to modifications $t_0$ and
\[
  \phi_0 = \phi - \Phi_b^0(r),\quad (\Phi_b^0)'=\pm\frac{a(1+\lambda_b)}{\wt\mu_b},
\]
of the functions $t_*$ and $\phi_*$; in the coordinates $(t_0,r,\theta,\phi_0)$ then, the expressions for $g_b$ and $A_b$ are analytic, and the fact that $(g_b,\breve A_b)$ and hence $(g_b,A_b)$ solves the Einstein--Maxwell system in the region of validity $r_{b,-}<r<r_{b,+}$ of the Boyer--Lindquist coordinates continues analytically to $M^\circ$. On the compactified manifold $M$, we can use $\tau_0$ as in \eqref{EqKNdS0Tau0} as a boundary defining function near the horizons.

The electromagnetic tensor
\[
  F_b = d A_b
\]
is then also a smooth (in $b$ near $b_0$) family of 2-forms on $M^\circ$. In the compactified picture, we in fact have $A_b\in\CI(B_b\times M;\Tb^*M)$ and $F_b\in\CI(B_b\times M;\Lambda^2\,\Tb^*M)$.

Given the smooth families $g_b$, $A_b$ and $F_b$, we can now define linearized KNdS solutions:
\begin{definition}
\label{DefKNdSaLin}
  For $b\in B$ and $b'\in\R^5$, define
  \[
    g'_b(b') = D_b(g_{(\cdot)})(b') = \frac{d}{ds}g_{b+s b'}|_{s=0},
  \]
  and similarly $A'_b(b')$ and $F'_b(b')$.
\end{definition}

Thus, $(g'_b(b'),A'_b(b'))$, resp.\ $(g'_b(b'),F'_b(b'))$, is a solution of the Einstein--Maxwell system linearized around $(g_b,A_b)$, resp.\ $(g_b,F_b)$, see \eqref{EqBasicLinEinsteinMaxwell}, resp.\ \eqref{EqBasicLinEM}.

\subsection{KNdS black holes with magnetic charge}
\label{SubsecKNdSMag}

Lemmas~\ref{LemmaBasicDerEMRotation} and \ref{LemmaBasicDerEMRotationCharges} immediately give the form of the metrics and electromagnetic tensors of magnetically charged KNdS black holes. Concretely, consider the enlarged parameter space
\begin{equation}
\label{EqKNdSMagParam}
  B_m = \{ (\bhm,\bfa,Q_e,Q_m) \} \subset \R\times\R^3\times\R\times\R;
\end{equation}
we take $B_m$ to be a small neighborhood of the parameters
\[
  b_{m,0} = (\bhm[0],\bfzero,Q_{e,0},Q_{m,0})
\]
of a non-degenerate RNdS spacetime with magnetic charge. We then define the map
\[
  \rho \colon B_m \ni (\bhm,\bfa,Q_e,Q_m) \mapsto (\bhm,\bfa,\sqrt{Q_e^2+Q_m^2}) \in B,
\]
and then, in the notation of \S\ref{SubsecKNdSa},
\begin{equation}
\label{EqKNdSMagSols}
\begin{split}
  b_m = (\bhm,&\bfa,Q_e,Q_m) \in B_m \\
    & \Longrightarrow \ g_{b_m} := g_{\rho(b_m)},\quad F_{b_m} := Q_e F_{(\bhm,\bfa,1)} + Q_m \star_{g_{b_m}} F_{(\bhm,\bfa,1)}.
\end{split}
\end{equation}
(For $Q_m=0$, this notation is consistent with the notation of the previous section for non-magnetically charged KNdS black holes.) Thus, the parameter range for general charged RNdS black holes is the same as for purely electrically charged RNdS black holes with charge $\sqrt{Q_e^2+Q_m^2}$. Since
\[
  T(g_{b_m},F_{b_m}) = T(g_{b_m}, F_{\rho(b_m)}) = T(g_{\rho(b_m)},F_{\rho(b_m)})
\]
by the discussion preceding Lemma~\ref{LemmaBasicDerEMRotation}, one concludes that $(g_{b_m},F_{b_m})$ is a solution of the Einstein--Maxwell system \eqref{EqBasicDerEMCurv1}--\eqref{EqBasicDerEMCurv2}; furthermore,
\[
  Q_e(g_{b_m},F_{b_m}) = Q_e,\quad Q_m(g_{b_m},F_{b_m}) = Q_m
\]
by construction.

\begin{rmk}
\label{RmkKNdSMagSmooth}
  In the expression \eqref{EqKNdSaMetric} for the electrically charged KNdS metric, the charge appears always in the second power, hence $g_{b_m}$ depends smoothly on $b_m\in B_m$. For $F_{b_m}$, the smooth dependence is then clear.
\end{rmk}

\subsection{Geometric and dynamical properties of KNdS spacetimes}
\label{SubsecKNdSGeo}

We refer the reader to \cite[\S6]{VasyMicroKerrdS} for a detailed discussion of the Kerr--de~Sitter geometry; the geometry of KNdS spacetimes is entirely analogous, thus we shall be brief. For $b$ near $b_0$, we define the dual metric function $G_b\in\CI(\Tb^*M)$ by $G_b(z,\zeta)=|\zeta|_{(G_b)_z}^2$, $z\in M$, $\zeta\in\Tb^*_z M$, and the characteristic set
\[
  \Sigma_b = G_b^{-1}(0) = \Sigma_b^+ \cup \Sigma_b^- \subset \rcTbdual M\setminus o,
\]
where $o$ is the 0-section of the vector bundle $\Tb^*M$. Since $\Sigma_b$ is conic in the fibers, it can be identified with its boundary at fiber infinity $\pa\Sigma_b\subset\Sb^*M$; here,
\[
  \Sigma_b^\pm = \Bigl\{ \zeta\in\Sigma_b \colon \pm\Big\la\zeta,-\frac{d\tau}{\tau}\Big\ra>0 \Bigr\}
\]
are the future ($+$) and past ($-$) light cones.

The null-geodesic flow near the b-conormal bundles
\[
  \cL_{b,\pm}=\Nb^*\{r=r_{b,\pm}\}\setminus o\subset\Sigma_b
\]
of the horizons has exactly the same structure (saddle point in the normal directions) as in the KdS case discussed in \cite[\S3.3]{HintzVasyKdSStability}. Indeed, the only difference between KNdS and KdS metrics is the precise form of the function $\wt\mu_b$, which however does not play any role in the relevant computations; only the non-degeneracy is important. Thus, let us use the coordinates
\[
  \sigma\,\frac{d\tau}{\tau}+\xi\,dr+\zeta\,d\phi+\eta\,d\theta
\]
on $\Tb^*M$ near $r=r_{b,\pm}$. Let then $\cL_{b,\pm}^\bullet=\cL_{b,\pm}\cap\Sigma_b^\bullet$, $\bullet=+,-$, denote the future (superscript `$+$'), resp.\ past (superscript `$-$'), component of the b-conormal bundles of the event (subscript `$+$'), resp.\ cosmological (subscript `$-$') horizon, and put $\cL_b^\pm=\cL_{b,+}^\pm\cap\cL_{b,-}^\pm$. The \emph{generalized radial sets}
\begin{equation}
\label{EqKNdSGeoRadialSet}
  \cR_{b,(\pm)}^{(\pm)} := \cL_{b,(\pm)}^{(\pm)}\cap \Tb^*_X M
\end{equation}
are then invariant under the $\ham_{g_b}$ flow; thus, their boundaries $\pa\cR_{b,(\pm)}^{(\pm)}\subset\Sb^*_X M$ at fiber infinity are invariant under the rescaled Hamilton flow $|\xi|^{-1}\ham_{G_b}$. As shown in \cite[\S3.3]{HintzVasyKdSStability}, $\pa\cR_b^+$ has a stable manifold $\pa\cL_b^+$ transversal to $\Sb^*_X M$, and an unstable manifold $\Sigma_b^+\cap\rcTbdual[X]M$ within $\rcTbdual[X]M$, with `stable' and `unstable' reversed when replacing `$+$' by `$-$'. Quantitatively, there exist smooth positive functions $\beta_{b,\pm,0},\beta_{b,\pm}\in\CI(\cR_{b,\pm})$ (see \cite[equation~(3.27)]{HintzVasyKdSStability}) such that with the defining function
\begin{equation}
\label{EqKNdSGeoFiberDefFn}
  \wh\rho=|\xi|^{-1}
\end{equation}
of fiber infinity $\Sb^*M\subset\rcTbdual M$, we have
\begin{equation}
\label{EqKNdSGeoRadPtQuant}
  \pm\wh\rho^{-1}\ham_{G_b}\wh\rho=\beta_{b,\bullet,0}, \quad \mp\tau^{-1}\ham_{G_b}\tau = \beta_{b,\bullet,0}\beta_{b,\bullet}\quad \text{at}\ \cR_{b,\bullet}^\pm,\ \bullet=+,-.
\end{equation}
See Figure~\ref{FigKNdSGeoRad}.

\begin{figure}[!ht]
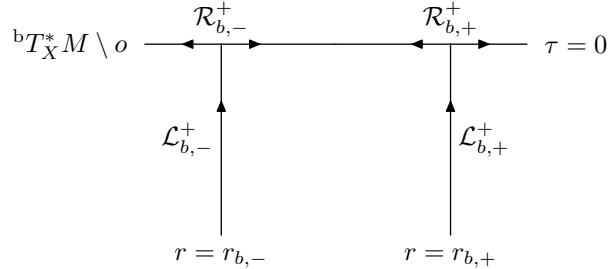

\centering
\inclfig{KNdSGeoRadial}
\caption{The future-directed null-geodesic flow near the b-conormal bundle $\cL^+_{b,\pm}$ of the event (`$-$'), resp.\ cosmological (`$+$') horizon. The saddle point structure of the rescaled Hamilton flow at $\pa\cR^+_{b,\pm}$, with unstable manifold $\Sigma_b^+\cap\rcTbdual[X]M$, is closely related to the classical red-shift effect.}
\label{FigKNdSGeoRad}
\end{figure}

We next compute the location of the trapped set in the exterior region $\cM$ of the RNdS spacetime with parameters $b_0$ as in \eqref{EqKNdS0Params}; we drop `$b_0$' from the notation. Writing covectors $\zeta\in\Tb^*M$ in $r_-<r<r_+$ using $\tau_s=e^{-t}$ as
\begin{equation}
\label{EqKNdSGeoTrapDual}
  \zeta = \sigma\,\frac{d\tau_s}{\tau_s} + \xi\,dr + \eta,\quad \eta\in T^*\Sph^2,
\end{equation}
the dual metric function equals $G_{b_0}=\mu^{-1}\sigma^2-\mu\xi^2-r^{-2}|\eta|^2$, with Hamilton vector field
\[
  \ham_G = 2\mu^{-1}\sigma\,\tau_s\pa_{\tau_s} - 2\mu\xi\pa_r - r^{-2}\ham_{|\eta|^2} + (\mu^{-2}\mu'\sigma^2+\mu'\xi^2-2r^{-3}|\eta|^2)\pa_\xi.
\]
In $\cM$, where $\mu>0$, we have $\ham_G r=-2\mu\xi=0$ if and only if $\xi=0$, in which case $\ham_G^2 r=0$ if and only if $\ham_G\xi=0$, which in the characteristic set $\Sigma$ is equivalent to $(\mu r^{-2})'=0$. By the non-degeneracy assumption on $b_0$, this equation has a unique solution $r_P\in(r_{b_0,-},r_{b_0,+})$, and the trapped set is
\begin{equation}
\label{EqKNdSGeoTrapSet}
  \Gamma = \{ (0,r_P,\omega;\sigma,0,\eta)\in \Tb^*_X M\setminus o \colon \mu^{-1}\sigma^2=r^{-2}|\eta|^2 \} \subset \Tb^*_X M;
\end{equation}
it has two connected components $\Gamma^\pm=\Gamma\cap\Sigma^\pm$. The Hamilton vector field at $\Gamma$ (as a b-vector field on $\Tb^*M$ restricted to $\Gamma$) is
\begin{equation}
\label{EqKNdSGeoTrapHam}
  \ham_G = 2\mu^{-1}\sigma\,\tau_s\pa_{\tau_s} - r^{-2}\ham_{|\eta|^2}.
\end{equation}

\begin{lemma}
\label{LemmaKNdSGeoNormHyp}
  For all non-degenerate RNdS spacetimes, the null-geodesic flow near the trapped set $\Gamma$ is $r$-normally hyperbolic for every $r$, see \cite{WunschZworskiNormHypResolvent}.
\end{lemma}
\begin{proof}
  By the homogeneity of $\ham_G$, it is sufficient to work at the frequency $\sigma=1$. The expansion rate of the $\ham_G$ flow within $\Gamma$ equals $0$ by spherical symmetry: null-geodesics within $\Gamma$ are simply (affinely reparameterized) geodesics on the round $\Sph^2$.
  
  We compute the linearization of $\ham_G$ in the normal directions to $\Gamma$ using the normal coordinates $r-r_P$ and $\xi$: the action of $\ham_G$ on $N^*(\Gamma\cap\{\sigma=1\})$ (within $\tau_s=0$, $\sigma=1$)\footnote{Here, $N^*(\Gamma\cap\{\sigma=1\})$ is the quotient $\cI/\cI^2$, where $\cI$ is the space of $\CI$ functions on $\{\tau_s=0,\ \sigma=1\}$ vanishing on $\Gamma$.} in the basis $\{d(r-r_P),d\xi\}$ is given by the matrix
  \[
    [\ham_G] = \begin{pmatrix}
      0          & \mu^{-2}r^2(r^{-2}\mu)''|_{r=r_P} \\
      -2\mu(r_P) & 0
    \end{pmatrix},
  \]
  with both off-diagonal entries strictly negative. Therefore, the minimal and maximal expansion rates of the $\ham_G$ flow are equal, and are given by the positive eigenvalue of $[\ham_G]$.
\end{proof}

Recalling the notation \eqref{EqKNdS0Mf}--\eqref{EqKNdS0MfRad}, we will study linear and non-linear waves on the domain with corners
\begin{equation}
\label{EqKNdSGeoDom}
  \Omega := [0,1]_\tau \times Y \subset M,
\end{equation}
where
\[
  Y= J\times\Sph^2 \subset X,\qquad J=[r_<,r_>],\ r_< :=r_{b_0,-}-\eps,\ r_> :=r_{b_0,+}+\eps.
\]
We also write
\begin{equation}
\label{EqKNdSGeoDomCirc}
  \Omega^\circ := [0,\infty)_{t_*} \times Y \subset M^\circ
\end{equation}
for the uncompactified domain. Apart from the boundary at future infinity, $\Omega$ has three boundary hypersurfaces: one is the Cauchy surface
\begin{equation}
\label{EqKNdSGeoCauchySurf}
  \Sigma_0 := \{\tau=1\} \cap \Omega = \{t_*=0\} \cap \Omega^\circ \cong Y.
\end{equation}
on which we pose initial/Cauchy data by means of the map $\gamma_0$ defined in \eqref{EqBasicNLCauchyData}; the other two are the lateral boundaries
\[
  \{\tau\leq 1,\ r=r_{b_0,\pm}\pm\eps\} \cap \Omega,
\]
both of which are spacelike: indeed, their outward pointing conormals $\pm dr$ are timelike as can be seen from \eqref{EqKNdSaDualMetric}, and they are future timelike if we take $\eps>0$ small. See Figure~\ref{FigKNdSDomain}.

\begin{figure}[!ht]
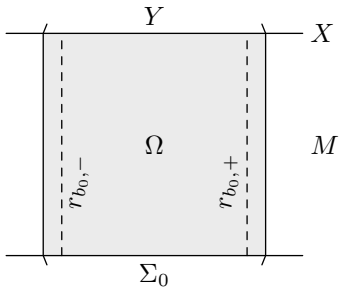

\centering
\inclfig{KNdSDomain}
\caption{The domain $\Omega\subset M$, with boundary at future infinity $Y\subset X=\pa M$; the event and cosmological horizon of the RNdS metric $g_{b_0}$ are indicated by dashed lines. The Cauchy hypersurface is $\Sigma_0$.}
\label{FigKNdSDomain}
\end{figure}

Therefore, linear wave equations with Cauchy data posed on $\Sigma_0$ and no data on the lateral boundary hypersurfaces are well-posed in $\Omega^\circ$.

\subsection{Constructing gauged Cauchy data from initial data sets}
\label{SubsecKNdSIni}

We return to the 4-potential formulation \eqref{EqBasicDerEM} of the Einstein--Maxwell system, hence to black holes with vanishing magnetic charge, and denote by $b_0$ the RNdS parameters \eqref{EqKNdS0Params}. The construction in \S\ref{SubsecKNdSa} gives rise to KNdS initial data
\[
  (h_b,k_b,\bfE_b,\bfB_b)
\]
on $\Sigma_0$, induced by the metric $g_b$ and the 4-potential $A_b$. We now show how to construct correctly gauged Cauchy data for the gauge-fixed system \eqref{EqBasicNLDeTurckLorenzA} from initial data given on $\Sigma_0$ which are close to the data $(h_{b_0},k_{b_0},\bfE_{b_0},\bfB_{b_0})$ induced by the RNdS metric with parameters $b_0$. We use the DeTurck (or wave map type) gauge 1-form for the Einstein part of the system,
\begin{equation}
\label{EqKNdSIniGaugeE}
  \Ups^E(g) := g g_{b_0}^{-1} \delta_g G_g g_{b_0},
\end{equation}
i.e.\ taking the background metric in \eqref{EqBasicNLUpsE} to be $g_{b_0}$, and the Lorenz type gauge function for the Maxwell part,
\begin{equation}
\label{EqKNdSIniGaugeM}
  \Ups^M(g,A) := \tr_g \delta_{g_{b_0}}^* A,
\end{equation}
see \eqref{EqBasicNLUpsM}. Since the realization of the KNdS family as a family of stationary metrics, described in \S\ref{SubsecKNdSa}, does not respect these gauges in general, we consider modifications by additional gauge source functions. We thus consider, for fixed $b\in B$, the gauge
\begin{equation}
\label{EqKNdSIniGaugeEM}
  \Ups^E(g) - \Ups^E(g_b) = 0, \quad \Ups^M(g,A-A_b) = 0;
\end{equation}
note that $(g,A)=(g_b,A_b)$ does satisfy these gauge conditions. In order to capture the vanishing magnetic charge, we consider for $s>3/2$ the subspace
\begin{equation}
\label{EqKNdSIniSpace}
\begin{split}
  Z^s &:= \Bigl\{ (h,k,\bfE,\bfB) \in H^{s+1}(\Sigma_0;S_{>0}^2 T^*\Sigma_0)\times H^s(\Sigma_0;S^2 T^*\Sigma_0) \\
    &\hspace{11em}\times H^s(\Sigma_0;T^*\Sigma_0)\times H^{s+1/2}(\Sigma_0;T^*\Sigma_0) \colon \\
    &\hspace{4em} d\star_h\bfB=0,\ \int_{\Sph^2}\star_h\bfB=0 \Bigr\},
\end{split}
\end{equation}
where $S^2_{>0}T^*\Sigma_0$ denotes the fiber bundle of positive definite inner products on $T\Sigma_0$. We now prove:

\begin{prop}
\label{PropKNdSIni}
  There exist a neighborhood
  \begin{align*}
    &(h_{b_0},k_{b_0},\bfE_{b_0},\bfB_{b_0}) \in \cU \subset Z^2
  \end{align*}
  of RNdS initial data and a smooth map
  \begin{align*}
    i_b \colon \cU \cap Z^s
        &\to H^{s+1}(\Sigma_0;S^2 T^*_{\Sigma_0}M^\circ) \times H^s(\Sigma_0;T^*_{\Sigma_0}M^\circ) \\
        &\quad \times H^s(\Sigma_0;S^2 T^*_{\Sigma_0}M^\circ) \times H^{s-1}(\Sigma_0;T^*_{\Sigma_0}M^\circ),
  \end{align*}
  for all $s\geq 2$, induced by continuity by a map on smooth sections (i.e.\ by a map $i_b$ for $s=\infty$), so that for $(h,k,\bfE,\bfB)\in\cU$, the sections
  \[
    (g_0,A_0;g_1,A_1) = i_b(h,k,\bfE,\bfB)
  \]
  induce the data $(h,k,\bfE,\bfB)$ on $\Sigma_0$, and they satisfy the gauge conditions \eqref{EqKNdSIniGaugeEM} in the sense that for any section $(g,A)$ of $S^2 T^*M^\circ\oplus T^*M^\circ$ near $\Sigma_0$ with $\gamma_0(g,A)=(g_0,A_0;g_1,A_1)$, the conditions~\eqref{EqKNdSIniGaugeEM} hold at $\Sigma_0$.

  Moreover, for exact KNdS data with parameter $b$, we have $i_b(h_b,k_b,\bfE_b,\bfB_b)=\gamma_0(g_b,A_b)$.
\end{prop}

Since the number of derivatives we use in the application of Proposition~\ref{PropKNdSIni} is rather large, see the statement of Theorem~\ref{ThmIntroNLFull}, we do not make any efforts to optimize the regularity assumptions.

We remark that in contrast to \cite[Proposition~3.10]{HintzVasyKdSStability}, the definition of the space $Z^s$ incorporates one of the constraint equations \eqref{EqBasicNLConstraints2}, which is necessary since we study the Einstein--Maxwell system using the potential $A$ rather than the electromagnetic 2-form $F$.

\begin{proof}[Proof of Proposition~\ref{PropKNdSIni}]
  The metric components of the map $i_b$, that is, the map $(h,k,\bfE,\bfB)\mapsto(g_0,g_1)$, was already constructed in \cite[Proposition~3.10]{HintzVasyKdSStability}; the construction there yields a smooth map $i_b$ on Sobolev spaces as well. We thus focus on the construction of the Cauchy data $(A_0,A_1)$ of the 4-potential $A$.

  The first step is to construct a bounded linear map
  \begin{equation}
  \label{EqKNdSIniInverse}
    \cA^\sharp \colon \Bigl\{ u\in H^{s+1/2}(\Sigma_0;\Lambda^2 T^*\Sigma_0) \colon d u=0,\ \int_{\Sph^2} u=0 \Bigr\} \to H^s(\Sigma_0;T^*\Sigma_0)
  \end{equation}
  such that $d\circ\cA^\sharp=\Id$. We present a direct argument, using the structure $\Sigma_0\cong J\times\Sph^2$, $J=[r_<,r_>]$, of $\Sigma_0$. Thus, let $\pi\colon\Sigma_0\to\Sph^2$ be the projection onto the second factor, and let $j\colon\Sph^2\hra\Sigma_0$ be the embedding $\omega\mapsto(r_<,\omega)$. Then the map $K\colon H^\sigma(\Sigma_0;\Lambda T^*\Sigma_0)\to H^\sigma(\Sigma_0;\Lambda T^*\Sigma_0)$, $\sigma\geq 0$, defined by linear extension from
  \[
    K(f(r,\omega) dr\wedge \pi^*u)(r,\omega) := (\pi^*u)(r,\omega)\int_{r_<}^r f(s,\omega)\,ds,\quad K(f(r,\omega)\pi^*u):=0,
  \]
  for $f\in H^\sigma(\Sigma_0)$ and $u\in\CI(\Sph^2;\Lambda T^*\Sph^2)$, satisfies $\Id-\pi^*j^*=d K+K d$. (See Bott--Tu \cite{BottTuAlgebraicTopology}.) Therefore, for closed $u$ as in \eqref{EqKNdSIniInverse}, we have $u=d K u + \pi^*(j^* u)$. Now $u':=j^*u\in H^s(\Sph^2;\Lambda^2 T^*\Sph^2)$ can be written uniquely as $u'=\slstar u_0+\sld u''$ where $u_0\in\R$ and $u''\in H^{s+1}(\Sph^2;T^*\Sph^2)$. Here, $\sld,\sldelta,\slstar$ denote the exterior differential, codifferential and Hodge star on $\Sph^2$. Since $0=\int u'=4\pi u_0$, we conclude that $u=d\wt A_0^\sharp$ with $\wt A_0^\sharp=\cA^\sharp u:=K u+\pi^* u''$, finishing the construction of the map \eqref{EqKNdSIniInverse}.

  Returning to the construction of the map $i_b$, we define
  \[
    \wt A_0^\sharp := \cA^\sharp(\star_h\bfB-\star_{h_b}\bfB_b),\quad A_0^\sharp:=i^*A_b + \wt A_0^\sharp \in H^s(\Sigma_0;T^*\Sigma_0);
  \]
  thus $d A_0^\sharp=i^*(\star_h\bfB)$, and $A_0^\sharp=i^*A_b$ if $(h,\bfB)=(h_b,\bfB_b)$. We make the ansatz
  \[
    A_j = a_j\,dt_* + A_j^\sharp,\quad j=0,1,
  \]
  with $a_j\in H^{s-j}(\Sigma_0)$ and $A_1^\sharp\in H^{s-1}(\Sigma_0;T^*\Sigma_0)$ to be determined. Note now that the knowledge of $g_0$ determines the future unit normal vector field $N$ to $\Sigma_0$, $N\in H^{s+1}(\Sigma_0;T_{\Sigma_0}M)$, which we can write as $N=c\pa_{t_*}+X$ with $0\neq c\in H^{s+1}(\Sigma_0)$ for some vector field $X\in H^{s+1}(\Sigma_0;T\Sigma_0)$ on $\Sigma_0$. The requirement that $-i^*\iota_N d A=\bfE$ at $\Sigma_0$ for $A=A_0+t_* A_1$ then reads $\bfE = c(d a_0-A_1^\sharp) - \iota_X d A_0^\sharp$. Observe that if $N$ and $\bfE$ are induced by $(g_b,A_b)$, then this holds for $a_0=\iota_{\pa_{t_*}}A_b$ and $A_1^\sharp=0$ by definition of $\bfE_b$ and due to the fact that $A_b$ is stationary. In general then, we define
  \[
    a_0 := \iota_{\pa_{t_*}}A_b, \quad A_1^\sharp:=d a_0 - c^{-1}(\bfE + \iota_X d A_0^\sharp),
  \]
  so in particular $A_1^\sharp=0$ if $(h,k,\bfE,\bfB)=(h_b,k_b,\bfE_b,\bfB_b)$.

  Lastly, we arrange the gauge condition $\Ups^M(g,A-A_b)=\tr_{g_0}\delta_{g_{b_0}}^*(A-A_b)=0$ at $\Sigma_0$, which is an equation for $a_1$ of the form
  \[
    G_0(\pa_{t_*},\pa_{t_*}) a_1 \in H^{s-1}
  \]
  here $G_0$ is the dual metric of $g_0$. This determines $a_1\in H^{s-1}(\Sigma_0)$ uniquely: indeed, $G_0(\pa_{t_*},\pa_{t_*})>0$ since this holds with $G_b$ replaced by $G_0$, and in general $G_0$ is close to $G_b$, so the positivity persists. In view of the uniqueness, we have $a_1=0$ if the data $(h,k,\bfE,\bfB)$ are induced by $(g_b,A_b)$, as claimed. The proof is complete.
\end{proof}

We can use the map $i_b$ to construct gauged Cauchy data for the initial value problem for the \emph{linearized} Einstein--Maxwell system \eqref{EqBasicLinEinsteinMaxwell}, linearized around $(g_b,A_b)$ and with initial data $(\hdot,\kdot,\bfEdot,\bfBdot)$; we need to assume that the linearized magnetic charge vanishes, i.e.\ the 2-form
\begin{equation}
\label{EqKNdSIniLinearizedFromPotential}
  D_{h_b,\bfB_b}(\star_{(\cdot)}(\cdot))(\hdot,\bfBdot)
\end{equation}
is closed and has vanishing integral over a non-trivial $\Sph^2\hra\Sigma_0$. For simplicity, we assume the linearized data are $\CI$. We then define
\[
  h(s) = h_b+s\hdot,\ k(s)=k_b+s\kdot,\ \bfE(s)=\bfE_b+s\bfEdot
\]
and
\[
  \star_{h+s\hdot}\bfB(s) = \star_h\bfB + s D_{h_b,\bfB_b}(\star_{(\cdot)}(\cdot))(\hdot,\bfBdot)
\]
for $s$ close to $0$; thus $h'(0)=\hdot$ etc., and $d(s):=(h(s),k(s),\bfE(s),\bfB(s))\in\cU$. It follows that $i_b(d(s))$ is well-defined and smooth in $s$, and then
\[
  (\gdot_0,\Adot_0;\gdot_1,\Adot_1) = \frac{d}{ds}i_b(d(s))\big|_{s=0} = D_{(h_b,k_b,\bfE_b,\bfB_b)}i_b(\hdot,\kdot,\bfEdot,\bfBdot)
\]
gives Cauchy data satisfying the gauge conditions
\begin{equation}
\label{EqKNdSIniLinGauge}
  D_{g_b}\Ups^E(\gdot)=0, \quad \Ups^M(g_b,\Adot)=0,
\end{equation}
for $\gdot$ and $\Adot$ with $\gamma_0(\gdot,\Adot)=(\gdot_0,\Adot_0;\gdot_1,\Adot_1)$.

\section{Linear analysis on slowly rotating KNdS spacetimes}
\label{SecLinAna}

We recall from \cite[\S5]{HintzVasyKdSStability} the main results on the global regularity and asymptotic behavior of solutions to initial value problems for linear wave equations on (asymptotically stationary perturbations of) slowly rotating KNdS\footnote{The results in \cite{HintzVasyKdSStability} are stated for slowly rotating KdS spacetimes, but extend immediately to slowly rotating KNdS spacetimes as well.} spacetimes, in particular in the presence of finite-dimensional spaces of non-decaying or exponentially growing resonant states. We only state those results which are directly needed for our present purposes; we refer the reader to \cite[\S5]{HintzVasyKdSStability} for a detailed discussion of the background, as well as for the proofs.

Let $E\to M$ be a stationary vector bundle, so $E=\pi_X^*E_X$, where $\pi_X\colon M=[0,\infty)_\tau\times X\to X$ is the projection, and $E_X\to X$ is a smooth complex finite rank vector bundle. Let then $L\in\Diffb^2(M;E)$ be an operator acting on sections of $E$ satisfying the following assumptions:
\begin{enumerate}
  \item $L$ is stationary, i.e.\ commutes with dilations in $\tau$.
  \item For fixed KNdS parameters $b$ close to $b_0$, the principal symbol of $L$ is $\sigma_{\bl,2}(L)=G_b\otimes\Id$, where $G_b\in\CI(\Tb^*M)$ is the dual metric function of $g_b$, and $\Id$ is the fiber-wise identity map on $\pi^*E\to\Tb^*M\setminus o$, with $\pi\colon\Tb^*M\to M$ the projection.
  \item With $\wh\rho$, defined in \eqref{EqKNdSGeoFiberDefFn}, denoting the defining function of fiber infinity $\Sb^*M\subset\rcTbdual M$ near the radial set $\cR_{b,\pm}$ at $r=r_{b,\pm}$, see \eqref{EqKNdSGeoRadialSet}, we define
  \[
    \beta \in \CI(\pa\cR_b),\quad \beta|_{\pa\cR_{b,\pm}}=\beta_{b,\pm} > 0,
  \]
  recalling the definition of $\beta_{b,\pm}$ from \eqref{EqKNdSGeoRadPtQuant}. Fixing a positive definite inner product on $E$ at $r=r_{b,\pm}$, write
  \begin{equation}
  \label{EqLinAnaWhBetaPM}
    \pm\wh\rho\sigma_{\bl,1}\Bigl(\frac{1}{2i}(L-L^*)\Bigr) = -\beta_{b,\bullet,0}\wh\beta_\bullet
  \end{equation}
  at $\pa\cR_{b,\bullet}^\pm$, with $\bullet=+,-$ specifying the cosmological (`$+$') and event (`$-$') horizon. Thus, $\wh\beta_\pm\in\CI(\pa\cR_b;\End(\pi^*E))$ is pointwise self-adjoint, and we define
  \begin{equation}
  \label{EqLinAnaWhBeta}
    \wh\beta := \inf_{\pa\cR_b}\wh\beta_\pm
  \end{equation}
  to be its smallest eigenvalue.
  \item Recall from \cite{HintzPsdoInner} the subprincipal operator, defined in a local trivialization of $E$ by
    \begin{equation}
    \label{EqLinAnaSubpr}
      S_\sub(L) = -i\ham_{G_b}\otimes\Id + \sigma_\sub(L) \in \Diffb^1(\Tb^*M\setminus o;\pi^*E),
    \end{equation}
    where we trivialize $\Omegab^\half(M)$ using $|d g_b|^\half$ to define the subprincipal symbol $\sigma_\sub(L)$. This operator is homogeneous of degree $1$ with respect to dilations in the fibers of $\Tb^*M\setminus o$. Denote $\sigma=\sigma_{\bop,1}(\tau D_\tau)$, see also \eqref{EqKNdSGeoTrapDual}. We then require that for all $\alpha_\Gamma>0$, there exists a smooth positive definite fiber inner product $h$ on $\pi^*E$ near $\Gamma$ such that
    \begin{equation}
    \label{EqLinAnaTrap}
      \pm\frac{1}{2 i|\sigma|}(S_\sub(L)-S_\sub(L)^*) < \alpha_\Gamma
    \end{equation}
    at $\Gamma^\pm=\Gamma\cap\{\pm\sigma<0\}$, where the adjoint on the left hand side is taken with respect to the symplectic volume b-density (defined by continuity from $M^\circ$ where it is equal to the standard symplectic volume density) and the fiber inner product $h$; note that the left hand side is a section of $\End(\pi^*E)$ which is pointwise self-adjoint with respect to $h$, and \eqref{EqLinAnaTrap} is a bound for its eigenvalues. (In fact it suffices to demand \eqref{EqLinAnaTrap} for some fixed constant $\alpha_\Gamma>0$ which depends on dynamical quantities associated with the trapping at $\Gamma$, see also \cite[Remark~5.1]{HintzVasyKdSStability}.)
\end{enumerate}

Thus, $\wh\beta$ controls the threshold regularity for microlocal propagation into $\pa\cR_b$, while the condition \eqref{EqLinAnaTrap} guarantees that solutions of $L u=0$ which have above-threshold regularity have an asymptotic expansion into resonances, up to an exponentially decaying remainder term. In order to state this precisely, we first recall the definition of the Mellin-transformed normal operator family $\wh L(\sigma)=\tau^{-i\sigma}L\tau^{i\sigma}$ acting on $\CIc(X;E_X)$; we then have:

\begin{thm}
\label{ThmLinAnaMero}
  (See \cite[Theorem~5.4]{HintzVasyKdSStability}.) There exist $\alpha>0$ and $C,C_1,C_2>0$ such that the following holds for fixed $s>1/2+\alpha\sup(\beta)-\wh\beta$: let\footnote{See Appendix~\ref{SecB} for the definition of Sobolev spaces $\Hext^s$ of extendible distributions.} $\cX^s=\{u\in\Hext^s(Y;E_Y)\colon\wh L(\sigma)\in\Hext^{s-1}(Y;E_Y)\}$, then
  \[
    \wh L(\sigma)\colon\cX^s \to \Hext^{s-1}(Y;E_Y),\quad \Im\sigma\geq-\alpha,
  \]
  is a holomorphic family of Fredholm operators. The inverse $\wh L(\sigma)^{-1}$ exists for $\Im\sigma>C_2$ (so $\wh L(\sigma)^{-1}$ is a meromorphic operator family) as well as for $\Im\sigma=-\alpha$, and the high energy estimate
  \begin{equation}
  \label{EqLinAnaMeroEst}
    \|u\|_{\Hext^s_{\la\sigma\ra^{-1}}(Y;E_Y)} \leq C\|\wh L(\sigma)u\|_{\Hext^{s-1}_{\la\sigma\ra^{-1}}}
  \end{equation}
  holds for $\Im\sigma>-\alpha$ and $|\Re\sigma|>C_1$, as well as for $\Im\sigma=-\alpha$.
\end{thm}

Indeed, due to \eqref{EqLinAnaTrap}, Dyatlov's results \cite{DyatlovSpectralGaps} apply at the trapped set, which in concert with radial point estimates \cite{VasyMicroKerrdS} yields the estimate \eqref{EqLinAnaMeroEst}. \emph{We fix $\alpha$ as in the statement of this theorem, and always assume that $s$ satisfies the lower bound.}

We define
\[
  \Res(L) \subset \{\sigma\in\C\colon \Im\sigma>-\alpha\}
\]
to be the set of poles of $\wh L(\sigma)^{-1}$ in $\Im\sigma>-\alpha$; for each $\sigma\in\Res(L)$, there are associated finite-dimensional spaces
\begin{align*}
  \Res(L,\sigma) &= \Bigl\{r=\sum_{j=0}^k e^{-i\sigma t_*}t_*^j r_j(x)\colon L r=0,\ r_j\in\CI(Y;E_Y)\Bigr\}, \\
  \Res^*(L,\sigma) &= \Bigl\{r=\sum_{j=0}^k e^{-i\bar\sigma t_*}t_*^j r_j(x)\colon L^*r=0,\ r_j\in\dot\sD(Y;E_Y)\Bigr\},
\end{align*}
of resonant, resp.\ dual resonant states. Thus, for $k=0$, these states are \emph{mode solutions} of $L r=0$, resp.\ $L^* r=0$, which are smooth on $Y$, resp.\ distributional with support in $\{r_{b,-}\leq r\leq r_{b,+}\}$; for $k>0$, they are \emph{generalized mode solutions}. For subsets $V\subset\C$, we also write
\[
  \Res(L,V) := \bigoplus_{\sigma\in V}\Res(L,\sigma),\quad \Res^*(L,V) := \bigoplus_{\sigma\in V}\Res^*(L,\sigma).
\]
Let us write
\begin{equation}
\label{EqLinAnaResTotal}
  R = \bigoplus_{\sigma\in\Res(L)}\Res(L,\sigma),\quad R^* = \bigoplus_{\sigma\in\Res(L)}\Res^*(L,\sigma),
\end{equation}
for the finite-dimensional total spaces of resonant and dual resonant states (recall the implicit assumption $\Im\sigma>-\alpha$). We can then define the continuous linear map
\[
  \lambda \colon \Hbsupp^{-\infty,\alpha}([0,1]_\tau;\Hext^{s-1}(Y;E_Y)) \ni f \mapsto \la f,\cdot\ra \in \cL(R^*,\bar\C).
\]
We recall from \cite[Proposition~5.7]{HintzVasyKdSStability} that $\lambda(f)=0$ is equivalent to the holomorphicity of $\wh L(\sigma)^{-1}\wh f(\sigma)$ in $\Im\sigma>-\alpha$; since $L$ is dilation-invariant, $\lambda(f)=0$ thus guarantees that the forward solution of $L u=f$ does not contain any resonant states from $R$, hence is exponentially decaying at rate $\alpha$. See Theorem~\ref{ThmLinAnaSolvFixed} for the precise statement.

Next, we recall from \cite[Definition~5.6]{HintzVasyKdSStability} the Banach space of initial data and forcing terms,
\[
  D^{s,\alpha}(\Omega;E) := \Hbext^{s,\alpha}(\Omega;E) \oplus \Hext^{s+1}(\Sigma_0;E_{\Sigma_0}) \oplus \Hext^s(\Sigma_0;E_{\Sigma_0}),
\]
equipped with the direct sum norm. Given $u\in\CI(\Omega;E)$, we denote by
\[
  \gamma_0(u) = (u|_{\Sigma_0}, \cL_{\pa_{t_*}}u|_{\Sigma_0})
\]
its Cauchy data, analogously to the definition \eqref{EqBasicNLCauchyData}. The map controlling the asymptotic behavior for the initial value problem $(L,\gamma_0)u=(f,u_0,u_1)\in D^{s,\alpha}(\Omega;E)$ is then
\[
  \lambda_\IVP \colon D^{s-1,\alpha}(\Omega;E) \ni (f,u_0,u_1) \mapsto \lambda(H f+[L,H](u_0+u_1 t_*)) \in \cL(R^*,\ol\C),
\]
where $H=H(t_*)$ is the Heaviside function; note that $[L,H]$, being a first order differential operator with coefficients which are $\delta$ distributions at $\Sigma_0$ (undifferentiated for the coefficients of the highest derivatives, and at most once differentiated otherwise), acting on $u_0+u_1 t_*$ produces an element of the space
\[
  \Hbsupp^{-3/2-0,\alpha}([0,1]_\tau;\Hext^{s-1}(Y;E_Y))
\]
which only depends on $u_0$ and $u_1$.

We can now describe the asymptotic behavior of solutions of the initial value problem
\[
  (L,\gamma_0)u = (f,u_0,u_1) \in D^{s,\alpha}(\Omega;E)
\]
under the assumptions of Theorem~\ref{ThmLinAnaMero}. Namely, $u$ has an asymptotic expansion into resonant states, up to an exponentially decaying remainder:
\begin{equation}
\label{EqLinAnaAsympExp}
  u = \sum_{\sigma\in\Res(L)} u_\sigma + \wt u,
\end{equation}
where $u_\sigma\in\Res(L,\sigma)$, and $\wt u\in\Hbext^{s,\alpha}(\Omega;E)$. More concretely, enumerating the resonances $\Res(L)=\{\sigma_1,\ldots,\sigma_N\}$, we have
\[
  u = \sum_{j=1}^N \sum_{k=1}^{d_j} \sum_{\ell=0}^{n_{j k}} u_{j k\ell} e^{-i\sigma_j t_*}t_*^\ell a_{j k\ell}(x) + \wt u,
\]
where $u_{j k\ell}\in\C$ are complex numbers (determined by the data), while the collection of resonant states
\[
  \sum_{\ell=0}^{n_{j k}} e^{-i\sigma_j t_*}t_*^\ell a_{j k\ell}(x), \quad k=1,\ldots,d_j,
\]
spans $\Res(L,\sigma_j)$.

The fundamental result guaranteeing the solvability of initial value problems for $L$ in decaying function spaces after suitably modifying the data by elements of a fixed finite-dimensional space is the following:

\begin{thm}
\label{ThmLinAnaSolvFixed}
  (See \cite[Corollaries~5.8 and 5.12]{HintzVasyKdSStability}.) With $\alpha>0$ and $s$ as in Theorem~\ref{ThmLinAnaMero}, suppose $\cZ\subset D^{s,\alpha}(\Omega;E)$ is a finite-dimensional linear subspace. If the map $\lambda_\IVP|_\cZ\colon\cZ\to\cL(R^*,\ol\C)$ is surjective, then for all $(f,u_0,u_1)\in D^{s,\alpha}(\Omega;E)$, there exists an element $z\in\cZ$, depending continuously on $(f,u_0,u_1)$, such that the solution of the initial value problem
  \[
    (L,\gamma_0)u=(f,u_0,u_1)+z
  \]
  satisfies $u\in\Hbext^{s,\alpha}(\Omega)$.
\end{thm}

This result suffices to prove the linear stability of non-degenerate RNdS spacetimes, see the discussion at the beginning of \S\ref{SecLin}.

We proceed to discuss the stability under perturbations. First, we consider finite-dimensional families of stationary perturbations. Thus, suppose $W\subset\R^{N_W}$ is an open neighborhood of some fixed $w_0\in\R^{N_W}$, and suppose that for each $w\in W$, we are given a stationary (dilation-invariant) operator
\[
  L_w\in\Diffb^2(M;E),
\]
depending continuously on $w$, such that $L_{w_0}$ satisfies the assumptions stated above for $L$. Suppose moreover that
\[
  \sigma_{\bl,2}(L_w)(\zeta) = |\zeta|_{G_{b(w)}}^2\otimes\Id,
\]
with $b(w)\in B$ depending continuously on $w\in W$. As shown in \cite[\S5.1.2]{HintzVasyKdSStability}, the conclusions of Theorem~\ref{ThmLinAnaMero} hold for the operator $L_w$, $w\in W$, as well, shrinking $W$ if necessary. One then has:

\begin{thm}
\label{ThmLinAnaSolvPertStat}
  (See \cite[Corollary~5.12]{HintzVasyKdSStability}.) Let $V$ be a small open neighborhood of $\Res(L_{w_0})$. Suppose $N_\cZ\in\N_0$, and suppose
  \[
    z \colon W\times\C^{N_\cZ} \to D^{s,\alpha}(\Omega;E)
  \]
  is continuous, and linear in the second argument. For $w\in W$, define the map
  \[
    \lambda_w \colon \C^{N_\cZ} \ni \bfc \mapsto \lambda_\IVP(z(w,\bfc)) \in \cL(\Res^*(L_w,V),\ol\C).
  \]
  Assume that $\lambda_{w_0}$ is surjective. Then $\lambda_w$ for $w\in W$ near $w_0$ is surjective as well, and there exists a continuous map
  \[
    W \times D^{s,\alpha}(\Omega;E) \ni (w,(f,u_0,u_1)) \mapsto \bfc \in \C^{N_\cZ},
  \]
  linear in the second argument, such that the solution of the initial value problem
  \[
    (L_w,\gamma_0)u=(f,u_0,u_1) + z(w,\bfc)
  \]
  satisfies $u\in\Hb^{s,\alpha}(\Omega;E)$, with continuous dependence on $(w,(f,u_0,u_1))$.
\end{thm}

This result will allow us to deduce the linear stability of slowly rotating KNdS black holes perturbatively from the linear stability of RNdS black holes, see Theorem~\ref{ThmLinKNdS}.

Finally, we discuss the extension to operators which are stationary up to an exponentially decaying, finite regularity remainder, which is the central analytic ingredient in our proof of non-linear stability in \S\ref{SecNL}. Thus, with $L_w$ as above, we consider operators
\[
  L_{w,\wt w} = L_w + \wt L_{w,\wt w},
\]
where
\[
  \wt w\in\wt W^s = \{\wt u\in\Hbext^{s,\alpha}(\Omega;E) \colon \|\wt u\|_{\Hbext^{14,\alpha}}<\eps\},
\]
with $\eps>0$ small and $s\geq 14$; further $\wt L_{w,\wt w}\in\Hbext^{s,\alpha}(\Omega)\Diffb^2(\Omega;E)$ is principally scalar, with coefficients decaying at the rate $\alpha$ given in Theorem~\ref{ThmLinAnaMero}. We assume $\wt L_{w_0,0}=0$, and we require
\[
  \|\wt L_{w_1,\wt w_1}-\wt L_{w_2,\wt w_2}\|_{\Hbext^{s,\alpha}\Diffb^2} \lesssim \bigl(|w_1-w_2|+\|\wt w_1-\wt w_2\|_{\Hbext^{s,\alpha}}\bigr),
\]
with the implicit constant depending on $s$. For $s=14$, this in particular states that $\wt L_{w,\wt w}$ is a bounded family of operators in $\Hbext^{14,\alpha}\Diffb^2(\Omega;E)$. Then:

\begin{thm}
\label{ThmLinAnaSolvPert}
  (See \cite[Theorem~5.14]{HintzVasyKdSStability}.) Suppose the eigenvalue $\wh\beta$, defined in \eqref{EqLinAnaWhBeta}, satisfies $\wh\beta\geq -1$. Let $N_\cZ\in\N_0$, and suppose
  \[
    z \colon W\times\wt W^s\times\C^{N_\cZ} \to D^{s,\alpha}(\Omega;E)
  \]
  is a continuous map, linear in the last argument. Suppose moreover that the map
  \[
    \C^{N_\cZ} \ni \bfc \mapsto \lambda_\IVP(z(w_0,0,\bfc)) \in \cL(R^*,\ol\C),
  \]
  with $R^*$ as in \eqref{EqLinAnaResTotal} for $L=L_{w_0,0}$, is surjective. Then there exists a continuous map
  \[
  \begin{split}
    S\colon W\times\wt W^\infty\times D^{\infty,\alpha}(\Omega;E) &\ni (w,\wt w,(f,u_0,u_1)) \\
      &\qquad \mapsto (\bfc, u)\in \C^{N_\cZ}\otimes\Hbext^{\infty,\alpha}(\Omega;E),
  \end{split}
  \]
  linear in the last argument, such that $u$ solves the initial value problem
  \[
    (L_{w,\wt w},\gamma_0)u=(f,u_0,u_1) + z(w,\wt w,\bfc).
  \]
  Furthermore, the map $S$ satisfies the estimates
  \begin{equation}
  \label{EqLinAnaSolvPertEst}
  \begin{gathered}
    |\bfc| \leq C\|(f,u_0,u_1)\|_{D^{13,\alpha}}, \\
    \|u\|_{\Hbext^{s,\alpha}} \leq C_s\bigl(\|(f,u_0,u_1)\|_{D^{s+3,\alpha}} + (1+\|\wt w\|_{\Hbext^{s+4,\alpha}})\|(f,u_0,u_1)\|_{D^{13,\alpha}}\bigr),
  \end{gathered}
  \end{equation}
  for $s\geq 10$. In fact, $S$ is defined for any $\wt w$ and $(f,u_0,u_1)$ for which the right hand sides of these estimates are finite; the thus extended $S$ still satisfies the same estimates.
\end{thm}

\section{Mode stability for the linearized Einstein--Maxwell system}
\label{SecMS}

We now study non-decaying (generalized) mode solutions of the Einstein--Max\-well system \eqref{EqBasicLinEinsteinMaxwell} linearized around the non-de\-gen\-er\-ate RNdS spacetime $(M,g_{b_0})$. We show that the only non-trivial such solutions (i.e.\ which are not pure gauge solutions) are stationary and correspond to infinitesimal changes in the parameters of the black hole. \emph{For the rest of this section, we drop the subscript `$b_0$', so $(g,A)\equiv(g_{b_0},A_{b_0})$, etc. We also write $Q\equiv Q_e$.}

\begin{thm}
\label{ThmMS}
  Let $\sigma\in\C$, $\Im\sigma\geq 0$, and $k\in\N_0$, and let
  \begin{equation}
  \label{EqMSGenMode}
    (\gdot,\Adot) = \sum_{j=0}^k e^{-i\sigma t_*}t_*^j (\gdot_j,\Adot_j), \quad \gdot_j\in\CI(Y;S^2 T^*_Y\Omega^\circ),\ \Adot_j\in\CI(Y;T^*_Y\Omega^\circ),
  \end{equation}
  be a generalized mode solution of the linearized Einstein--Maxwell system $\sL(\gdot,\Adot)=0$, linearized around the RNdS solution $(g,A)$, see \eqref{EqBasicLinEinsteinMaxwellExpl}. Then there exist $V\in\CI(\Omega^\circ;\TC\Omega^\circ)$ and $a\in\CI(\Omega^\circ;\C)$ as well as parameters $b'\in\R^5$ such that
  \begin{equation}
  \label{EqMSRes}
    (\gdot,\Adot) = (g'(b'),A'(b')) + (\cL_V g,\cL_V A + d a).
  \end{equation}
  In fact, $(V,a)$ is a generalized mode as well, containing at most one higher power of $t_*$ than $(\gdot,\Adot)$. Thus, if $\sigma\neq 0$, then $b'=0$, and $(\gdot,\Adot)$ is a pure gauge solution, while generalized mode solutions with frequency $\sigma=0$ are linearized KNdS solutions up to a pure gauge solution.
\end{thm}

This result relies in a crucial manner on the fact that the spacetime under consideration has low dimension (here $4$), as discussed in \S\ref{SubsecIntroWork}. We will follow the method of Kodama--Ishibashi for the study of perturbations of Schwarzschild--de~Sitter spacetimes \cite{KodamaIshibashiMaster}, building on \cite{KodamaIshibashiSetoBranes,KodamaSasakiPerturbation}, and extended to the charged case in \cite{KodamaIshibashiCharged}. The latter paper deals with a more general problem (working in general dimension, and including source terms beyond Maxwell),\footnote{In the notation of \cite{KodamaIshibashiCharged}, we discuss the case $n=2$, $\kappa^2=2$ (see \cite[equation~\KI{2}{12}]{KodamaIshibashiCharged}), $q=Q$, $\lambda=\Lambda/3$ (see \cite[equation~\KI{2}{14}]{KodamaIshibashiCharged}), so $E_0=Q/r^2$ (see \cite[equation~\KI{2}{9}]{KodamaIshibashiCharged}), and there are a number of sign changes due to the different sign convention adopted here.} but does not provide a full discussion of a number of special types of perturbations: spherically symmetric perturbations, which correspond to changes in the black hole parameters by the Birkhoff theorem, as pointed out in \cite[\S5]{KodamaIshibashiCharged}; scalar $l=1$ perturbations (see~\eqref{EqMSScalar} below), which are special due to their relationship with rotational Killing vector fields on $\Sph^2$ --- we extend the treatment in \cite[Appendix~D]{KodamaIshibashiCharged} by giving a detailed description of their pure gauge character as needed for the study of generalized modes; and lastly stationary, non spherically symmetric, perturbations, which were previously only treated in the uncharged case in \cite[\S4]{KodamaIshibashiMaster} --- we show how the arguments in the reference can be extended to the charged case. Since dealing with these special types necessitates working out the full perturbation equations, we give a detailed proof of Theorem~\ref{ThmMS} for \emph{all} types of perturbations in the setting of interest in the present paper; thus, we include a complete discussion also of those cases (non-stationary modes with $l\geq 2$ and vector $l=1$ modes) for which a sufficient treatment was already given in \cite{KodamaIshibashiCharged}.

The first step in the proof is a decomposition of $(\gdot,\Adot)$ into scalar and vector spherical harmonics; since the RNdS metric $g$ is spherically symmetric, each component by itself will be annihilated by $\sL$. To explain the decomposition, we write
\begin{equation}
\label{EqMSCalcMetric}
  g = \hg - r^2\slg
\end{equation}
on the 4-dimensional manifold $M^\circ=\hX_x\times\Sph_\omega^2$, with $\slg=\slg(\omega,d\omega)$ the round metric on $\Sph^2$, while $\hg=\hg(x,dx)$ is a Lorentzian metric on the oriented 2-dimensional manifold $\hX=\R_{t_*}\times I_r$, with $I_r$ defined in \eqref{EqKNdS0MfRad}; in static coordinates,
\begin{equation}
\label{EqMShg}
  \hg=\mu\,dt^2-\mu^{-1}\,dr^2,\quad \mu=1-\frac{2\bhm}{r}-\frac{\Lambda r^2}{3}+\frac{Q^2}{3},
\end{equation}
which extends past the horizons as discussed in \S\ref{SubsecKNdS0}. This product decomposition comes equipped with projections
\[
  \hpi\colon M^\circ\to\hX, \quad \slpi\colon M^\circ\to\Sph^2.
\]
We call functions \emph{aspherical} if they are constant on the fibers of $\wh\pi$, i.e.\ functions of $(t_*,r)$ only. We will henceforth identify
\[
  \hpi^*\CI(\hX)\cong\CI(\hX),\quad \hpi^*\CI(\hX;T^*\hX)\cong\CI(\hX;T^*\hX), \quad \tn{etc.}
\]
We split the cotangent bundle $T^*M^\circ$ into its \emph{aspherical} and \emph{spherical} parts,
\begin{equation}
\label{EqMSCalcCtgt}
  T^*M^\circ = \TAS^* \oplus \TS^*,\quad \TAS^*=\hpi^*T^*\hX,\ \TS^*=\slpi^*T^*\Sph^2;
\end{equation}
correspondingly, a 1-form $u\in\CI(M^\circ;T^*M^\circ)$ has a unique decomposition $u=u_1+u_2$, with $u_1\in\CI(M^\circ;\TAS^*)$ aspherical and $u_2\in\CI(M^\circ;\TS^*)$ spherical. The induced splitting of symmetric 2-tensors takes the form
\begin{equation}
\label{EqMSCalcS2}
  S^2T^*M^\circ = S^2\TAS^* \oplus (\TAS^*\otimes\TS^*) \oplus S^2\TS^*,
\end{equation}
where we identify $a\otimes s\in\TAS^*\otimes\TS^*$ with $a\otimes s+s\otimes a=2a\otimes_s s\in S^2T^*M^\circ$; thus, the splitting \eqref{EqMSCalcS2} induces a decomposition of symmetric 2-tensors on $M^\circ$ into an \emph{aspherical part} in $\CI(M^\circ;S^2\TAS^*)$, a \emph{mixed part} in $\CI(M^\circ;\TAS^*\otimes\TS^*)$, and a \emph{spherical part} in $\CI(M^\circ;S^2\TS^*)$.

Observe now that
\[
  L^2(M^\circ;\TAS^*)\cong L^2(\hX)\otimes L^2(\Sph^2;T^*\Sph^2),
\]
likewise for sections of the various other bundles; hence we can decompose the perturbation $(\gdot,\Adot)\in L^2(M^\circ;S^2T^*M^\circ\oplus T^*M^\circ)$ into an infinite sum of products of functions/1-forms/symmetric 2-tensors on $\hX$ and functions/1-forms/symmetric 2-tensors on $\Sph^2$. Let us describe this in more detail: consider the Helmholtz decomposition of a 1-form $u\in\CI(\Sph^2;T^*\Sph^2)$, namely $u=\sld v+w'$ with $v\in\CI(\Sph^2)$ and $w'\in\CI(\Sph^2;T^*\Sph^2)$, $\sldelta w'=0$. Rewriting the latter as $\sld\slstar w'=0$ and using that $H^1(\Sph^2)=0$, we have $w'=\slstar\sld w$, $w\in\CI(\Sph^2)$. From this, one can easily deduce that a complete orthogonal basis of $L^2(\Sph^2;T^*\Sph^2)$ of eigen-1-forms of the Hodge Laplacian $\slDelta_H=\sld\sldelta+\sldelta\sld$ is given by $\{\sld Y_l^m,\slstar\sld Y_l^m\colon l\in\N,|m|\leq l\}$, where the $Y_l^m$ are the scalar spherical harmonics; the eigenvalue for an eigen-1-form with index $l$ is $l(l+1)$. Thus, we have an orthogonal decomposition of the $(l(l+1)-1)$ eigenspace of the tensor Laplacian $\slDelta=\slDelta_H-\slRic=\slDelta_H-1$ into the scalar part
\[
  \sld\Sph, \quad \slDelta\Sph=l(l+1)\Sph, \ \Sph\in\CI(\Sph^2),\ l\geq 1,
\]
and the vector part $\vect\in\CI(\Sph^2;T^*\Sph^2)$,
\begin{equation}
\label{EqMSVectEigen}
  \slDelta\vect=(l(l+1)-1)\vect, \quad \sldelta\vect=0.
\end{equation}

On $\Sph^2$, the Helmholtz decomposition of symmetric 2-tensors $u\in\CI(\Sph^2;S^2T^*\Sph^2)$ gives $u=\sldelta^*_0 v + w\slg$, where $\sldelta^*_0=\sldelta^*+\frac{1}{2}\slg\sldelta$ is the trace-free symmetric gradient, the adjoint of the divergence $\sldelta$ acting on trace-free symmetric 2-tensors; this uses that $\ker\sltr\cap\ker\sldelta=0$ on $\Sph^2$, see \cite[\S{III}]{HiguchiSpherical}. Decomposing $v$ into 1-form spherical harmonics and further into scalar and vector parts, and likewise decomposing $w$ into scalar spherical harmonics, we conclude that it suffices to consider two classes of perturbations $(\gdot,\Adot)$: the first class consists of \emph{scalar perturbations} (also called \emph{even parity} modes \cite{ReggeWheelerSchwarzschild})
\begin{equation}
\label{EqMSScalar}
\begin{split}
  \gdot &= \wt f\Sph - \frac{2 r}{k}(f\otimes_s\sld\Sph) + 2r^2\bigl(H_L\Sph\slg+H_T\slH_k\Sph\bigr), \\
  \Adot &= \wt K\Sph - \frac{r}{k}K\sld\Sph,
\end{split}
\end{equation}
where we define a rescaling of the traceless Hessian,
\begin{equation}
\label{EqMSHessian}
  \slH_k := \frac{1}{k^2}\sldelta^*\sld + \frac{1}{2}\slg;
\end{equation}
$\Sph$ is a scalar eigenfunction, $\slDelta\Sph=k^2\Sph$, $k=(l(l+1))^{1/2}$, $l\geq 2$, and
\[
  H_L,H_T,K\in\CI(\hX), \quad f,\wt K\in\CI(\hX;T^*\hX),\quad \wt f\in\CI(\hX;\Lambda^2 T^*\hX)
\]
are aspherical. The second class consists of \emph{vector perturbations} (also called \emph{odd parity} modes)
\begin{equation}
\label{EqMSVector}
\begin{split}
  \gdot &= 2 r f\otimes_s\vect - \frac{2}{k}r^2 H_T\sldelta^*\vect, \\
  \Adot &= r K\vect,
\end{split}
\end{equation}
with $\vect$ as in \eqref{EqMSVectEigen}, and $k=(l(l+1)-1)^{1/2}$, $l\geq 2$. In \S\ref{SubsecMS2}, we show that all non-decaying generalized mode solutions with $l\geq 2$ are pure gauge.

There are three additional special cases:

\begin{enumerate}
\item Scalar perturbations with $l=0$, hence $\Sph$ is constant, take the form
  \[
    \gdot = \wt f + 2 r^2 H_L\slg, \quad \Adot=\wt K,
  \]
  and must be treated separately; they give rise to changes in the mass and charge of the black hole by a generalization of the Birkhoff theorem to the case of positive cosmological constant, as we discuss in \S\ref{SubsecMS0}.

\item For scalar perturbations with $l=1$ (so $k^2=2$), the quantity $H_T$ is not defined since $(\sldelta^*\sld+\slg)\Sph=0$, thus
  \begin{equation}
  \label{EqMSScalar1}
    \gdot = \wt f\Sph - \frac{2 r}{k}(f\otimes_s\sld\Sph) + 2 r^2 H_L\Sph\slg,\quad \Adot=\wt K\Sph+\frac{r}{k}K\sld\Sph;
  \end{equation}
  non-decaying generalized mode perturbations of this type will be shown to be pure gauge in \S\ref{SubsecMS1s}.

\item Lastly, if $\vect$ is a divergence-free $l=1$ 1-form spherical harmonic, as in \eqref{EqMSVectEigen}, then $\vect$ is a linear combination of $\slstar\sld Y_1^m$, $m=-1,0,1$ and hence easily checked to be a Killing 1-form, $\sldelta^*\vect=0$. Thus, such vector $l=1$ perturbations are described by
  \begin{equation}
  \label{EqMSVector1}
    \gdot = 2 r f\otimes\vect, \quad \Adot = r K\vect.
  \end{equation}
  In \S\ref{SubsecMS1v}, we show that non-decaying perturbations of this type are pure gauge for $\sigma\neq 0$, or give rise to infinitesimal changes in the angular momentum parameter for $\sigma=0$.
\end{enumerate}

Observe that the RNdS solution $(g,A)$ is spherically symmetric; in fact, it is invariant under the full orthogonal group $O(3)$ acting on $M^\circ$, with the action diagonal on $\hX\times\Sph^2$, namely trivial on $\hX$ and the standard action on $\Sph^2$. Therefore, the linearized operator $\sL$ commutes with $O(3)$. One can easily check that the only $O(3)$-invariant differential operators of order $\leq 2$ acting between sections over $\Sph^2$ of the trivial bundle, $T^*\Sph^2$ and $S^2T^*\Sph^2$ are $\slg$, $\sltr$, $\sld$, $\sldelta$ (on 1-forms and on symmetric 2-tensors), $\sldelta^*$, $\slDelta$ (on functions, 1-forms and symmetric 2-tensors) and their compositions of order $\leq 2$.\footnote{The $O(3)$-invariance, as opposed to merely $SO(3)$-invariance, rules out operators related to the orientation of $\Sph^2$ such as $\slstar$.} Thus, all spherical operators appearing in the aspherical--spherical decomposition of $\sL$ in the splittings \eqref{EqMSCalcCtgt} and \eqref{EqMSCalcS2} are given by these operators and their compositions. One can then immediately verify that $\sL$ preserves the class of scalar perturbations \eqref{EqMSScalar} associated with the \emph{fixed} spherical harmonic $\Sph$, and likewise $\sL$ preserves the class of vector perturbations \eqref{EqMSVector} associated with the \emph{fixed} divergence-free spherical harmonic 1-form $\vect$; in other words, $\sL$ does \emph{not} mix different eigenfunctions within eigenspaces. (One can also verify this fact directly by looking at the explicit form of all the terms in $\sL$ as given in \S\ref{SubsecMSCalc}.)

The arguments proving that all non-decaying generalized mode perturbations which do not belong to one of the exceptional classes are pure gauge solutions, proceed by first considering the transformations of the quantities $H_L,H_T,K,f,\wt K,\wt f$ upon adding a pure gauge solution preserving the class of scalar or vector perturbations under consideration, and thus finding all combinations of these quantities which are gauge-invariant; if one then shows that these gauge-invariant quantities must vanish, this will imply that the perturbation is trivial, i.e.\ pure gauge. We point out $\sL(\gdot,\Adot)$ itself \emph{is} gauge-invariant by virtue of $(g,A)$ solving the Einstein--Maxwell system: indeed, we clearly have $\sL(0,d a)=0$ for any function $a$, and moreover for any vector field $V$, we have
\[
  \sL(\gdot+\cL_V g,\Adot+\cL_V A) = \cL_V \bigl( (\Ric+\Lambda)(g)-2 T(g,d A), \delta_g d A \bigr) \equiv 0.
\]
Thus, one can write the equation $\sL(\gdot,\Adot)=0$ in terms of the gauge-invariant quantities. We will follow this strategy for each of the four cases considered above. 

\begin{rmk}
\label{RmkMSModeByMode}
  In our application of Theorem~\ref{ThmMS} to the study of non-decaying resonant states of the linearized gauge-fixed Einstein--Maxwell system in \S\ref{SecLin}, we will know \emph{a priori} that the total space of such resonant states $(\gdot,\Adot)$ is finite-dimensional. Thus, every such $(\gdot,\Adot)$ only has \emph{finitely} many non-zero components in the expansion into spherical harmonics; hence, using the mode stability result for each individual harmonic, we can write $(\gdot,\Adot)$ as in \eqref{EqMSRes}, with $(V,a)$ being a \emph{finite} sum of suitable (vector) spherical harmonics (with additional $(t_*,r)$ dependence). Therefore, smoothness of $V$ and $a$ is automatic.
  
  To prove Theorem~\ref{ThmMS} as stated, we also need to deal with the situation that $(\gdot,\Adot)$ has infinitely many non-zero components in the spherical harmonic expansion. However, an inspection of the argument producing the gauge transformations $(V,a)$ --- which will be (generalized) modes $(V,a)=\sum_{j=0}^k e^{-i\sigma t_*}t_*^j(V_j,a_j)$ themselves --- in the various cases shows that for a mode $(\gdot,\Adot)$ with multipole moment $l$, the function $(V_j,a_j)$, measured in $\cC^N$ in the radial variable $r$, has size bounded by $l^2$ times the radial $\cC^{N-2}$ size of $(\gdot,\Adot)$ for any $N$. Thus, for a general smooth mode $(\gdot,\Adot)$, one can sum up the gauge transformations obtained for each spherical harmonic part of $(\gdot,\Adot)$, and the resulting full gauge transformation is smooth as well.
\end{rmk}

\subsection{Detailed calculation of the linearized Einstein--Maxwell operator}
\label{SubsecMSCalc}

Recall the definition~\eqref{EqBasicLinEinsteinMaxwellExpl} of $\sL$, $\sL_1$, and $\sL_2$. For $\sL_1$, we use \eqref{EqBasicNLLinRic}, so
\[
  D_g(\Ric+\Lambda)(\gdot) = \frac{1}{2}\Box_g\gdot-\delta_g^*\delta_g G_g\gdot + \sR_g\gdot + \Lambda\gdot.
\]
In addition, writing $F=d A$ and $\Fdot=d\Adot$, we have
\begin{equation}
\label{EqMSDT}
\begin{split}
  D_g T(\gdot,F)_{\mu\kappa} &= \gdot^{\nu\lambda}\Bigl(F_{\mu\nu}F_{\kappa\lambda} - \frac{1}{2}g_{\mu\kappa}F_{\rho\nu}F^\rho{}_\lambda\Bigr) + \frac{1}{2}\gdot_{\mu\kappa}|F|_g^2, \\
  D_F T(g,\Fdot)_{\mu\kappa} &= -g^{\nu\lambda}(F_{\mu\nu}\Fdot_{\kappa\lambda}+\Fdot_{\mu\nu}F_{\kappa\lambda}) + G(F,\Fdot)\,g_{\mu\kappa},
\end{split}
\end{equation}
where $G$ denotes the inner product on 2-forms, with quadratic form given by \eqref{EqBasicDerNormSq}. For the second component $\sL_2$ of $\sL$, we note $D_A(\delta_g d(\cdot))(\Adot)=\delta_g d\Adot$; moreover:

\begin{lemma}
\label{LemmaMSLinDelta}
  For a metric $g$ and a 2-form $F$, we have
  \begin{equation}
  \label{EqMSLinDelta}
    D_g(\delta_{(\cdot)}F)(\gdot)_\nu = \gdot^{\mu\kappa}F_{\mu\nu;\kappa} - (\delta_g G_g\gdot)^\kappa F_{\kappa\nu} + \frac{1}{2}(\gdot_{\nu\kappa;\mu}-\gdot_{\nu\mu;\kappa})F^{\mu\kappa}.
  \end{equation}
\end{lemma}
\begin{proof}
  The derivative of
  \[
    (\delta_g F)_\nu=-g^{\mu\kappa}F_{\mu\nu;\kappa}=-g^{\mu\kappa}(\pa_\kappa F_{\mu\nu}-\Gamma(g)_{\mu\kappa}^\lambda F_{\lambda\nu}-\Gamma(g)_{\nu\kappa}^\lambda F_{\mu\lambda})
  \]
  in $g\in\CI(M^\circ,S^2 T^*M^\circ)$ in the direction $\gdot$ is equal to
  \begin{align*}
    \gdot^{\mu\kappa}F_{\mu\nu;\kappa} + \frac{1}{2}&\bigl((2\gdot_\mu{}^{\lambda;\mu} - \gdot_\mu{}^{\mu;\lambda})F_{\lambda\nu} + (\gdot_\nu{}^{\lambda;\mu}+\gdot^{\mu\lambda;}{}_\nu-\gdot_\nu{}^{\mu;\lambda})F_{\mu\lambda}\bigr) \\
      & = \gdot^{\mu\kappa}F_{\mu\nu;\kappa} - (\delta_g G_g\gdot)^\lambda F_{\lambda\nu} + \frac{1}{2}(\gdot_\nu{}^{\lambda;\mu}-\gdot_\nu{}^{\mu;\lambda})F_{\mu\lambda},
  \end{align*}
 since $\gdot^{\mu\lambda;}{}_\nu F_{\mu\lambda}=0$ by antisymmetry in $(\mu,\lambda)$.
\end{proof}

Using the decomposition \eqref{EqMSCalcMetric} of the RNdS metric and working in the splittings \eqref{EqMSCalcCtgt} and \eqref{EqMSCalcS2}, we proceed to calculate the form of the operators comprising $\sL$ in more detail. Let us denote the natural geometric operators associated with $\hg$ and $\slg$ by hats and slashes, respectively. Furthermore, we use abstract indices $a,b,c,d,e$ for coordinates on $\hX$, and abstract indices $i,j,k,\ell,m$ for coordinates on $\Sph^2$; we always work in product coordinates. All indices are raised and lowered using the metric $g$, with the exception that we write $\slg^{ij}=(\slg^{-1})_{ij}$ for the inverse metric on $\Sph^2$.

The Christoffel symbols of $g$ are then given by
\[
\setarraystretch
  \begin{array}{lll}
    \Gamma^c_{a b}=\hGamma^c_{a b},& \Gamma^c_{a j}=0, & \Gamma^c_{i j}=r r^{;c}\slg_{ij}, \\
    \Gamma^k_{a b}=0, & \Gamma^k_{a j}=r^{-1}r_{;a}\delta_j^k, & \Gamma^k_{i j}=\slGamma^k_{i j},
  \end{array}
\]
with $r_{;a}=\hnabla_a r$; we shall also write $\rho=\hd r$. The components of the Riemann curvature tensor are
\[
\setarraystretch
  \begin{array}{llll}
    R_{a b c d}=\hR_{a b c d}, & R_{i b c d}=0, & R_{i j c d}=0, & R_{i j k d}=0, \\
    R_{i b k d}=r r_{;b d}\slg_{i k}, & \multicolumn{3}{c}{R_{i j k \ell}=-r^2(1+|\rho|^2)(\slg_{i k}\slg_{j\ell}-\slg_{i\ell}\slg_{j k}).}
  \end{array}
\]
Here $|\rho|^2=r_{;a}r^{;a}$ (which will in fact be negative in our application). Correspondingly, the Ricci tensor takes the form
\begin{gather*}
  \Ric_{a b}=\hRic_{a b}-2r^{-1}r_{;a b}, \quad \Ric_{a j}=0, \\
  \Ric_{i j}=(-r\hBox r + (1+|\rho|^2))\slg_{i j},
\end{gather*}
where $\hBox r\equiv\Box_\hg r=-r_{;a}{}^a$. In the RNdS setting, we compute
\begin{equation}
\label{EqMSCalcHatQuantities}
\begin{gathered}
  r_{;a b} = -\frac{\mu'}{2}\hg_{a b}, \quad \hBox r=\mu', \quad \hR_{a b c d}=\frac{1}{2}\mu''(\hg_{a c}\hg_{b d}-\hg_{a d}\hg_{b c}), \\
  \hRic_{a b}=\frac{1}{2}\mu''\hg_{a b}, \quad \hR=\mu'', \quad |\rho|^2=-\mu,
\end{gathered}
\end{equation}
with $\hR$ the scalar curvature of the metric $\hg$. Next,
\[
  g = \begin{pmatrix} \hg \\ 0 \\ -r^2\slg \end{pmatrix},\quad \tr_g = \begin{pmatrix} \htr & 0 & -r^{-2}\sltr \end{pmatrix},
\]
and the trace reversal operator $G_g u=u-\frac{1}{2}g\tr_g u$ takes the form
\begin{equation}
\label{EqMSCalcTraceRev}
  G_g=
   \begin{pmatrix}
     G_\hg                    & 0 & \frac{1}{2}r^{-2}\hg\,\sltr \\
     0                        & 1 & 0 \\
     \frac{1}{2}r^2\slg\,\htr & 0 & G_\slg
   \end{pmatrix}.
\end{equation}
For the action of $(\sR_g u)_{\mu\nu}=R_{\kappa\mu\nu\lambda}u^{\kappa\lambda}+\frac{1}{2}(\Ric_\mu^\lambda u_{\lambda\nu}+\Ric_\nu^\lambda u_{\mu\lambda})$ on aspherical tensors, we compute
\begin{gather*}
  (\sR_g u)_{a b} = (\sR_\hg u)_{a b} - r^{-1}(r_{;a}{}^c u_{c b}+r_{;b}{}^c u_{a c}), \\
  (\sR_g u)_{a j} = 0, \quad (\sR_g u)_{i j}=-r r_{;a b}u^{a b}\slg_{i j};
\end{gather*}
on mixed tensors,
\begin{gather*}
  (\sR_g u)_{a b} = 0, \quad (\sR_g u)_{i j} = 0, \\
  (\sR_g u)_{a j}=\Bigl(\frac{1}{2}\hRic_a^b-2 r^{-1}r_{;a}{}^b\Bigr)u_{b j} + \frac{1}{2}\Bigl(\frac{\hBox r}{r}-\frac{1+|\rho|^2}{r^2}\Bigr)u_{a j};
\end{gather*}
on spherical tensors,
\begin{gather*}
  (\sR_g u)_{a b}=-r^{-3} r_{;a b}\slg^{i j}u_{i j}, \quad (\sR_g u)_{a j}=0, \\
  (\sR_g u)_{i j}=\frac{\hBox r}{r}u_{i j}-\frac{2(1+|\rho|^2)}{r^2}(G_\slg u)_{i j}.
\end{gather*}
These expressions can be simplified for the RNdS metric using \eqref{EqMSCalcHatQuantities}, giving
\[
  2\sR_g
  =\openbigpmatrix{4pt}
     2\mu'' G_\hg+2r^{-1}\mu' & 0                               & r^{-3}\mu'\hg\,\sltr \\
     0                        & \hRic+r^{-2}(\mu-1)+3r^{-1}\mu' & 0 \\
     r\mu'\slg\,\htr          & 0                               & 4r^{-2}(\mu-1)G_\slg+2r^{-1}\mu'
   \closebigpmatrix.
\]
Next, we compute first covariant derivatives: for aspherical tensors,
\begin{equation}
\label{EqMSCalcCovAsph}
\setarraystretch
  \begin{array}{ll}
    u_{a b;c}=\hnabla_c u_{a b}, & u_{a b;k}=\slnabla_k u_{a b}, \\
    u_{a j;c}=0, & u_{a j;k}=-r r^{;b}u_{a b}\slg_{j k}, \\
    u_{i j;c}=0, & u_{i j;k}=0;
  \end{array}
\end{equation}
for mixed tensors,
\[
\setarraystretch
  \begin{array}{ll}
    u_{a b;c}=0, & u_{a b;k}=-r^{-1}(r_{;a}u_{k b}+r_{;b}u_{a k}), \\
    u_{a j;c}=(\hnabla_c-r^{-1}r_{;c})u_{a j}, & u_{a j;k}=\slnabla_k u_{a j}, \\
    u_{i j;c}=0, & u_{i j;k}=-r r^{;a}(\slg_{i k}u_{a j}+\slg_{j k}u_{a i});
  \end{array}
\]
for spherical tensors,
\begin{equation}
\label{EqMSCalcCovSph}
\setarraystretch
  \begin{array}{ll}
    u_{a b;c}=0, & u_{a b;k}=0, \\
    u_{a j;c}=0, & u_{a j;k}=-r^{-1}r_{;a}u_{j k}, \\
    u_{i j;c}=(\hnabla_c-2 r^{-1}r_{;c})u_{i j}, & u_{i j;k}=\slnabla_k u_{i j}.
  \end{array}
\end{equation}
We then compute $(\delta_g u)_\nu=-u_{\mu\nu;}{}^\mu$, acting on general symmetric 2-tensors, to be equal to
\[
  \delta_g
  =\begin{pmatrix}
     r^{-2}\hdelta r^2 & -r^{-2}\sldelta   & -r^{-3}\rho\,\sltr \\
     0                 & r^{-2}\hdelta r^2 & -r^{-2}\sldelta
   \end{pmatrix},
\]
where in the second column $(\sldelta u)_a=-\slnabla^j u_{a j}=-\slg^{j k}\slnabla_k u_{a j}$; invariantly, one views $\TAS^*\otimes\TS^*\cong T^*\hX\boxtimes T^*\Sph^2$ and observes that the connection $\slnabla$ on $T^*\Sph^2$ induces a partial connection
\[
  \slnabla \colon \CI(M^\circ;T^*\hX\boxtimes T^*\Sph^2)\to \CI(M^\circ;T^*\Sph^2\otimes(T^*\hX\boxtimes T^*\Sph^2)).
\]
The operator $(\hdelta u)_i=-\hnabla^a u_{i a}$ has an analogous invariant interpretation. The adjoint $(\delta_g^* u)_{\mu\nu}=\frac{1}{2}(u_{\mu;\nu}+u_{\nu;\mu})$ is given by
\begin{equation}
\label{EqMSCalcDelGStar}
  \delta_g^*
  =\begin{pmatrix}
     \hdelta^*        & 0 \\
     \frac{1}{2}\sld  & \frac{1}{2}r^2\hd r^{-2} \\
     -r\slg\iota_\rho & \sldelta^*
   \end{pmatrix},
\end{equation}
where for an aspherical 1-form $u$ we write $\iota_u$ for the contraction with the vector $u^\sharp=\hG(u,\cdot)$. The composition $\delta_g G_g$ can be simplified using $\htr\,G_\hg=0=\sltr\,G_\slg$, and we obtain
\[
  \delta_g G_g
   = \begin{pmatrix}
       r^{-2}\hdelta r^2+\frac{1}{2}\hd\,\htr & -r^{-2}\sldelta & -\frac{1}{2}r^{-2}\hd\,\sltr \\
       \frac{1}{2}\sld\,\htr & r^{-2}\hdelta r^2 & -r^{-2}\sldelta - \frac{1}{2}r^{-2}\sld\,\sltr
     \end{pmatrix}
\]
Using \eqref{EqMSCalcCovAsph}--\eqref{EqMSCalcCovSph}, we can also compute the form of the tensor wave operator $(\Box_g u)_{\mu\nu}=-u_{\mu\nu;\kappa}{}^\kappa$: it equals
\begin{align*}
  \Box_g &= \hBox-r^{-2}\slDelta \\
   &\quad + \diag(-2 r^{-1}\hnabla_\rho + 4 r^{-2}\rho\otimes_s\iota_\rho, \\
   &\hspace{5em} 4 r^{-2}\rho\otimes\iota_\rho + (-r^{-1}\hBox r+r^{-2}|\rho|^2), \\
   &\hspace{5em} 2 r^{-1}\hnabla_\rho - 2r^{-1}\hBox r) \\
   &\quad
   + \begin{pmatrix}
       0                         & 4 r^{-3}\rho\otimes_s\sldelta & 2 r^{-4}(\rho\otimes\rho)\sltr \\
       -2 r^{-1}\sld \iota_\rho  & 0                             & 2 r^{-3}\rho\otimes\sldelta \\
       2\slg\iota_\rho\iota_\rho & -4 r^{-1}\sldelta^*\iota_\rho & 0
     \end{pmatrix};
\end{align*}
here, $\hBox$ operating on aspherical tensors is the tensor wave operator, operating on mixed tensors the 1-form wave operator acting on the aspherical part, and operating on a spherical tensor the scalar wave operator acting in the aspherical variables; the operator $\slDelta$ is defined in an analogous fashion. We remark that the Ricci term $\frac{1}{2}\hRic_a^b u_{b j}$ in the aspherical part of $\sR_g$ acting on aspherical tensors plus the tensor wave operator $\frac{1}{2}(\hBox u)_{a j}$ (acting as the 1-form wave operator in the aspherical variables) gives the Hodge--d'Alembertian $\frac{1}{2}(\hBox_H u)_{a j}$, $\hBox_H=\hd\hdelta+\hdelta\hd$ acting on aspherical 1-forms.

Next, let $A\in\CI(\hX;T^*\hX)$ be an aspherical 1-form, and let $F=d A=\hd A$, which is thus also aspherical. Since $\hX$ is 2-dimensional and orientable, with volume form
\[
  \hvol = dt_*\wedge dr,
\]
we automatically have
\[
  F=-Q r^{-2}\hvol
\]
for some aspherical function $Q$; in fact, for the RNdS electromagnetic field $F$, \emph{$Q$ is constant}, as we will assume from now on. For $T(g,F)_{\mu\nu}=-F_{\mu\kappa}F_\nu{}^\kappa+\frac{1}{2}g_{\mu\nu}|F|^2$ (with $|F|^2=\frac{1}{2}F_{\kappa\lambda}F^{\kappa\lambda}=-Q^2 r^{-4}$), we then recall \eqref{EqMSDT} and note that $u^{c d}\hvol_{a c}\hvol_{b d}=u_{a b}-\hg_{a b}\htr\,u$ since $\hg$ has Lorentzian signature, further $(\hstar u)_a=\hvol^b{}_a u_b$ for aspherical 1-forms $u$; one then finds
\begin{gather*}
  D_g T(\cdot,F)
  = -\frac{1}{2}Q^2 r^{-4}
    \begin{pmatrix}
      \hg\,\htr-1   & 0 & 0 \\
      0             & 1 & 0 \\
      r^2\slg\,\htr & 0 & 1
    \end{pmatrix}, \\
  D_A T(g,d(\cdot))
  =Q r^{-2}
     \begin{pmatrix}
       \hg \hstar\hd     & 0 \\
       -\hstar\sld       & \hstar\hd \\
       r^2\slg \hstar\hd & 0
     \end{pmatrix},
\end{gather*}
where we use that for an aspherical 2-form $u$, we have $\slG(u,\hvol)=\hstar u$ (indeed, $\hstar 1=\hvol$ and $\hstar\hvol=-1$). The Hodge dual $\hstar$ in the second row only acts on the $T^*\hX$ component. This finishes the computation of the form of all terms in the linearization of $\Ric(g)+\Lambda g-2 T(g,d A)$.

Next, we study the linearization of $(g,A)\mapsto \delta_g d A$. We use the splitting
\[
  \Lambda^2 T^*M^\circ = \Lambda^2\TAS^* \oplus (\TAS^*\wedge\TS^*) \oplus \Lambda^2\TS^*
\]
of the 2-form bundle over $M^\circ$, and in addition use the identification $\TAS^*\wedge\TS^*\cong\TAS^*\otimes\TS^*\cong T^*\hX\boxtimes T^*\Sph^2$ via $a\wedge s\mapsto a\otimes s$; then
\[
  d =
   \begin{pmatrix}
     \hd   & 0 \\
     -\sld & \hd \\
     0     & \sld
   \end{pmatrix}
\]
on 1-forms, and the divergence on 2-forms is given by
\[
  \delta_g =
    \begin{pmatrix}
      r^{-2}\hdelta r^2 & r^{-2}\sldelta & 0 \\
      0                 & \hdelta        & -r^{-2}\sldelta
    \end{pmatrix}.
\]
For $F$ as above, we calculate the first covariant derivatives
\[
  \begin{array}{lll}
    F_{a b;c}=-2 r^{-1}r_{;c}F_{a b}, & F_{a j;c}=0, & F_{i j;c}=0, \\
    F_{a b;k}=0, & F_{a j;k}=-Q r^{-1}(\hstar\hd r)_a\slg_{j k}, & F_{i j;k}=0;
  \end{array}
\]
thus, using Lemma~\ref{LemmaMSLinDelta}, we find
\begin{align*}
  D_g(\delta_{(\cdot)}F)=Q r^{-3}
  &\begin{pmatrix}
    2\hstar\iota_\rho & 0                  & r^{-2}(\hstar\rho)\sltr \\
    0                 & \iota_{\hstar\rho} & 0
  \end{pmatrix} \\
  & + Q r^{-2}
    \begin{pmatrix}
      r^{-2}\hstar\hdelta r^2+\frac{1}{2}\hstar\hd\,\htr & -r^{-2}\hstar\,\sldelta & -\frac{1}{2}r^{-2}\hstar\hd\,\sltr \\
      0 & 0 & 0
    \end{pmatrix} \\
  & - Q r^{-2}
    \begin{pmatrix}
      \hstar\hdelta+\hstar\hd\,\htr & 0 & 0 \\
      0 & \hstar\hd - r^{-1}\hstar(\rho\wedge) & 0
    \end{pmatrix}.
\end{align*}
This can be simplified, noting that $\hstar(\rho\wedge)=\iota_{\hstar\rho}=-\iota_\rho\hstar$, $2 r^{-1}\iota_{\hstar\rho}-\hstar\hd=-r^2\hdelta r^{-2}\hstar$, and $r^{-2}\hdelta r^2=\hdelta-2r^{-1}\iota_\rho$. Thus,
\[
  D_g(\delta_{(\cdot)}F)
  =Q r^{-2}
   \begin{pmatrix}
     -\frac{1}{2}\hstar\hd\,\htr & -r^{-2}\hstar\,\sldelta & -\frac{1}{2}\hstar\hd r^{-2}\sltr \\
     0 & -r^2\hdelta r^{-2}\hstar & 0
   \end{pmatrix}.
\]

For the calculation of gauge terms, we recall that $\cL_{u^\sharp}g=2\delta_g^* u$ for 1-forms $u$, with $\delta_g^*$ given in \eqref{EqMSCalcDelGStar}; furthermore, for $A\in\CI(\hX;T^*\hX)$ aspherical, we have
\[
  A_{a;b}=\hnabla_b A_a, \quad A_{a;j}=0, \quad A_{i;b}=0, \quad A_{i;j}=-r r^{;b}A_b\slg_{i j},
\]
hence $u\mapsto (\cL_{u^\sharp}A)_\mu=u_{b;\mu}A^b+u^\nu A_{\nu;\mu}$ is given by
\begin{equation}
\label{EqMSCalcLieA}
  \cL_{(\cdot)^\sharp}A
    = \begin{pmatrix}
        \hd\iota_{(\cdot)}A+\iota_{(\cdot)}\hd A & 0 \\
        \sld\,\iota_A & 0
      \end{pmatrix}.
\end{equation}

In order to compute the action of $\sL$ on the spherical harmonic decomposition (and at the same time confirming that $\sL$ preserves the scalar--vector decomposition in the strong sense described in the paragraphs preceding \S\ref{SubsecMSCalc}), we note that on functions, we have $\slDelta\sld=\sld(\slDelta-1)$ (with $\slDelta$ the \emph{tensor} Laplacian on 1-forms), while on 1-forms, we have $\slDelta\sldelta=\sldelta(\slDelta+1)$. Moreover, acting on 1-forms,
\[
  \slDelta\sldelta^* = \sldelta^*\slDelta - 3\sldelta^* - 2\slg\sldelta.
\]
In particular, for a 1-form $\vect $ with $\slDelta\vect=k^2\vect$, $\sldelta\vect=0$, then
\[
  \slDelta\sldelta^*\vect = (k^2-3)\sldelta^*\vect.
\]
If $\slDelta\Sph=k^2\Sph$ for a scalar function $\Sph$, this also gives
\[
  \slDelta\slH_k\Sph = (k^2-4)\slH_k\Sph,
\]
where we recall the definition of $\slH_k$ from \eqref{EqMSHessian}. This uses the identity $\slDelta\slg=\slg\slDelta$, thus $\slDelta\slg\Sph=k^2\slg\Sph$.\footnote{These three calculations provide us with a complete orthogonal basis of $L^2(\Sph^2;S^2T^*\Sph^2)$ and thus with the spectrum of $\slDelta$ acting on symmetric 2-tensors.} We also note that on 1-forms, we have
\[
  \sldelta\sldelta^* = \frac{1}{2}(\slDelta+\sld\sldelta-1),
\]
thus for $\vect$ and $\Sph$ as above,
\[
  \sldelta\sldelta^*\vect = \frac{k^2-1}{2}\vect,
\]
while
\[
  \sldelta\slH_k\Sph = \frac{k^2-2}{2 k^2}\sld\Sph.
\]

Lastly, we calculate the form of geometric operators on the Lorentzian manifold $(\hX,\hg)$; specifically, we only consider the exterior (static) region where $\mu>0$, so the metric takes the form $\hg=\mu\,dt^2-\mu^{-1}\,dr^2$, $\mu=\mu(r)$. We split
\begin{equation}
\label{EqMSCalchXBundles}
  T^*\hX = \la \hd t\ra\oplus\la \hd r\ra, \quad  S^2 T^*\hX=\la \hd t^2\ra\oplus\la 2 \hd t\,\hd r\ra\oplus \la \hd r^2\ra,
\end{equation}
so $\rho=\hd r=(0,1)$, further
\[
  \hd=\begin{pmatrix}\pa_t\\ \pa_r\end{pmatrix},
  \quad
  \hBox=-\mu^{-1}\pa_t^2 + \pa_r\mu\pa_r
\]
on scalar functions, while on 1-forms $\iota_\rho=(0,-\mu)$,
\begin{gather*}
  \hdelta = \begin{pmatrix} -\mu^{-1}\pa_t & \pa_r\mu \end{pmatrix},
  \quad
  \hdelta^*
    =\begin{pmatrix}
       \pa_t                       & -\frac{1}{2}\mu\mu' \\
       \frac{1}{2}\mu\pa_r\mu^{-1} & \frac{1}{2}\pa_t \\
       0                           & \mu^{-1/2}\pa_r\mu^{1/2}
     \end{pmatrix}, \\
  \hnabla_\rho = \begin{pmatrix} -\mu^{3/2}\pa_r\mu^{-1/2} & 0 \\ 0 & -\mu^{1/2}\pa_r\mu^{1/2} \end{pmatrix}, \\
  \hstar = \begin{pmatrix} 0 & \mu \\ \mu^{-1} & 0 \end{pmatrix},
  \quad
  2\rho\otimes_s(\cdot)
    =\begin{pmatrix} 
       0 & 0 \\
       1 & 0 \\
       0 & 2
     \end{pmatrix}
\end{gather*}
and on symmetric 2-tensors
\begin{equation}
\label{EqMSCalchXSym2}
\begin{gathered}
  \hdelta
   =\begin{pmatrix}
      -\mu^{-1}\pa_t & \pa_r\mu & 0 \\
      \frac{1}{2}\mu'\mu^{-2} & -\mu^{-1}\pa_t & \mu^{-1/2}\pa_r\mu^{3/2}
    \end{pmatrix},
  \quad
  \iota_\rho
   =\begin{pmatrix}
      0 & -\mu & 0 \\
      0 & 0    & -\mu
    \end{pmatrix}, \\
  \hnabla_\rho
   =\begin{pmatrix}
      -\mu^2\pa_r\mu^{-1} & 0 & 0 \\
      0 & -\mu\pa_r & 0 \\
      0 & 0 & -\pa_r\mu
    \end{pmatrix}
\end{gathered}
\end{equation}
and
\[
  \hBox
   =-\mu^{-1}\pa_t^2+\mu\pa_r^2
    +
    \begin{pmatrix}
      -\mu'\pa_r+\frac{\mu'^2}{2\mu}-\mu'' & 2\mu'\pa_t                     & -\frac{1}{2}\mu\mu'^2 \\
      \frac{\mu'}{\mu^2}\pa_t              & \mu'\pa_r - \frac{\mu'^2}{\mu} & \mu'\pa_t \\
      -\frac{\mu'^2}{2\mu^3}               & \frac{2\mu'}{\mu^2}\pa_t       & 3\mu'\pa_r+\frac{\mu'^2}{2\mu}+\mu''
    \end{pmatrix}.
\]

\subsection{Modes with \texorpdfstring{$l\geq 2$}{l at least 2}}
\label{SubsecMS2}

We show that both scalar and vector generalized mode solutions of the linearized Einstein--Maxwell system with $l\geq 2$ are pure gauge. Concretely, we prove this directly for mode solutions, i.e.\ we prove the $k=0$ case of Theorem~\ref{ThmMS} for $l\geq 2$, but with a stronger conclusion:

\begin{prop}
\label{PropMS2Nolog}
  Let $\sigma\in\C$, $\Im\sigma\geq 0$, and let
  \[
    (\gdot,\Adot)=e^{-i\sigma t_*}(\gdot_0,\Adot_0),\quad \gdot_0\in\CI(Y;S^2 T^*_Y\Omega^\circ),\ \Adot_0\in\CI(Y;T^*_Y\Omega^\circ),
  \]
  be a mode solution of $\sL(\gdot,\Adot)=0$ of the form \eqref{EqMSScalar} or \eqref{EqMSVector} with $l\geq 2$. Then there exist $V_0\in\CI(Y;T_Y\Omega^\circ)$ and $a\in\CI(Y)$ such that for $(V,a)=e^{-i\sigma t_*}(V_0,a_0)$, we have
  \[
    (\gdot,\Adot) = (\cL_V g,\cL_V A + d a).
  \]
\end{prop}

Thus, $(\gdot,\Adot)$ can be written as a pure gauge solution with gauge parameters $(V,a)$ \emph{which are modes as well}, that is, they do not contain powers of $t_*$. Before turning to the proof of this proposition, we note that it implies Theorem~\ref{ThmMS} for generalized $l\geq 2$ modes:

\begin{cor}
\label{CorMS2Log}
  Let $\sigma\in\C$, $\Im\sigma\geq 0$, $k\in\N_0$, and let $(\gdot,\Adot)$ as in \eqref{EqMSGenMode} be a generalized mode solution of $\sL(\gdot,\Adot)=0$ of the form \eqref{EqMSScalar} or \eqref{EqMSVector} with $l\geq 2$. Then there exist $V_j\in\CI(Y;T_Y\Omega^\circ)$ and $a_j\in\CI(Y)$, $j=0,\ldots,k$, such that for the gauge functions
  \[
    (V,a) = \sum_{j=0}^k e^{-i\sigma t_*}t_*^j (V_j,a_j),
  \]
  we have $(\gdot,\Adot)=(\cL_V g,\cL_V A+d a)$.
\end{cor}
\begin{proof}
  This is completely analogous to \cite[Lemma~10.1]{HintzVasyKdSStability}. Indeed, proceeding by induction, with the base case $k=0$ being Proposition~\ref{PropMS2Nolog}, we note that $\sL(\gdot,\Adot)=0$ implies in particular that $\sL(e^{-i\sigma t_*}(\gdot_k,\Adot_k))=0$, hence Proposition~\ref{PropMS2Nolog} gives $e^{-i\sigma t_*}(\gdot_k,\Adot_k)=(\cL_{V'}g,\cL_{V'}A+d a')$ for some $(V',a')=e^{-i\sigma t_*}(V_k,a_k)$, with $(V_k,a_k)$ stationary as in the statement of the corollary. But then
  \[
    (\gdot',\Adot') = (\gdot,\Adot) - (\cL_{t_*^k V'}g, \cL_{t_*^k V'}A + d(t_*^k a') )
  \]
  still solves $\sL(\gdot',\Adot')=0$, and is a generalized mode with the exponent of the highest power of $t_*$ at most $k-1$; applying the inductive hypothesis finishes the proof.
\end{proof}

The proof of Proposition~\ref{PropMS2Nolog} proceeds differently in the two cases \eqref{EqMSScalar} and \eqref{EqMSVector}; we discuss the scalar case in \S\S\ref{SubsubsecMS2Sc} and \ref{SubsubsecMS2Sc0}, and the vector case in \S\ref{SubsubsecMS2Vec}.

\subsubsection{Non-stationary scalar perturbations\texorpdfstring{ ($\sigma\neq 0$)}{}}
\label{SubsubsecMS2Sc}

Suppose $(\gdot,\Adot)$ is a scalar perturbation as in \eqref{EqMSScalar}, so
\[
  \gdot = \begin{pmatrix} \wt f \Sph \\ -\frac{r}{k}f\otimes\sld\Sph \\ 2 r^2(H_L\Sph\slg + H_T\slH_k\Sph) \end{pmatrix},
  \quad
  \Adot = \begin{pmatrix} \wt K \Sph \\ -\frac{r}{k}K\,\sld\Sph \end{pmatrix}
\]
in the splittings \eqref{EqMSCalcCtgt} and \eqref{EqMSCalcS2}. We consider its modifications by pure gauge solutions of the same type: these take the form
\begin{equation}
\label{EqMS2ScPureGauge}
  (\upd\gdot,\upd\Adot) = (\cL_{u^\sharp}g,\cL_{u^\sharp}A+d a),  \quad
  u = \begin{pmatrix}
        T\Sph \\
        -\frac{r}{k}L\sld\Sph \\
      \end{pmatrix},
  \ a=P\Sph,
\end{equation}
where $T\in\CI(\hX;T^*\hX)$ and $L,P\in\CI(\hX)$ are aspherical, while $\slDelta\Sph=k^2\Sph$, $k^2=l(l+1)$, $l\geq 2$. Using \eqref{EqMSCalcDelGStar} and \eqref{EqMSCalcLieA}, and writing $\upd\gdot$ and $\upd\Adot$ in terms of $\upd\wt f$, $\upd f$, $\upd H_L$, $\upd H_T$, $\upd\wt K$ and $\upd K$ analogously to \eqref{EqMSScalar}, we have
\begin{equation}
\label{EqMS2ScPureGaugeMods}
  \setarraystretch
  \begin{array}{ll}
    \upd\wt f=2\hdelta^*T,    & \upd f=-\frac{k}{r}T + r\hd(r^{-1}L), \\
    \upd H_T = -\frac{k}{r}L, & \upd H_L=-r^{-1}\iota_\rho T+\frac{k}{2 r}L, \\
    \upd\wt K = \hd\iota_A T+\iota_T\hd A + \hd P, & \upd K=-\frac{k}{r}(\iota_A T+P);
  \end{array}
\end{equation}
thus, defining
\begin{equation}
\label{EqMS2ScbfX}
  \bfX := \frac{r}{k}\Bigl(f+\frac{r}{k}\hd H_T\Bigr),
\end{equation}
which has $\upd\bfX=-T$, we conclude that the quantities\footnote{In these expressions, the operators $\hdelta^*$, $\iota_\rho\equiv\iota_{\rho^{\wh\sharp}}$ and $\iota_\bfX\equiv\iota_{\bfX^{\wh\sharp}}$, with $\wh\sharp$ indicating that one uses $\hg$ to compute the musical isomorphisms, the metric $\hg$ is fixed, i.e.\ these operators are \emph{not} subject to the modifications by $(\upd\gdot,\upd\Adot)$; thus, for instance, $\upd(\hdelta^*\bfX)=\hdelta^*(\upd\bfX)$.}
\begin{equation}
\label{EqMS2ScGaugeInv}
\begin{split}
  \wt F &:= \wt f+2\hdelta^*\bfX \in \CI(\hX;S^2T^*\hX), \\
  J &:= H_L + \frac{1}{2}H_T - r^{-1}\iota_\rho\bfX \in \CI(\hX), \\
  N &:= \wt K+\hd\Bigl(\frac{r}{k}K\Bigr)+\iota_\bfX\hd A \in \CI(\hX;T^*\hX)
\end{split}
\end{equation}
are gauge-invariant, i.e.\ $\upd\wt F=0$ etc. Moreover, they completely describe the perturbation $(\gdot,\Adot)$ up to pure gauge terms: indeed, if they all vanish, then one easily verifies
\begin{equation}
\label{EqMS2ScPureGaugeDescr}
 (\gdot,\Adot) = (\cL_{u^\sharp}g,\cL_{u^\sharp}A+d a),
 \quad
  u = \begin{pmatrix}
        -\bfX\Sph \\
        \frac{r^2}{k^2}H_T\hd\Sph
      \end{pmatrix},\ 
  a = -\frac{r}{k}K+\iota_A\bfX.
\end{equation}
Note that $u$ and $a$ are modes if $\gdot$ and $\Adot$ are modes.

Since the linearized Einstein--Maxwell equation is gauge-invariant, we can equivalently express it as a system for the gauge-invariant quantities \eqref{EqMS2ScGaugeInv}. This is most easily done using the following procedure: one first picks a gauge, i.e.\ one adds a pure gauge solution $(\upd\gdot,\upd\Adot)$ to $(\gdot,\Adot)$ for suitable $T$, $L$, $P$ as in \eqref{EqMS2ScPureGauge}, in which certain non gauge-invariant quantities take a simple form or vanish: concretely, let us take $T=\bfX$, $L=\frac{r}{k}H_T$ and $P=\frac{r}{k}K-\iota_A\bfX$, in which case $\bfX+\upd\bfX=0$, $H_T+\upd H_T=0$ (thus $f+\upd f=0$) and $K+\upd K=0$. Replacing $(\gdot,\Adot)$ by $(\gdot+\upd\gdot,\Adot+\upd\Adot)$, which may thus assume $\bfX=0$, $H_T=0$ (thus $f=0$) and $K=0$; the gauge-invariant quantities then are simply $\wt F=\wt f$, $J=H_L$ and $N=\wt K$. Next, using the calculations in \S\ref{SubsecMSCalc}, one can express the linearized Einstein--Maxwell system, acting on the new $(\gdot,\Adot)$, as a system of equations for quantities $\wt f,H_L,\wt K$. The resulting system of equations, which is expressed purely in terms of $\wt F$, $J$ and $N$, must be equal to the linearized Einstein--Maxwell system acting on our original $(\gdot,\Adot)$ by the gauge-invariance of this system.

\begin{rmk}
  This is analogous to local coordinate computations of invariantly defined operators in differential geometry: the calculations often simplify dramatically when one uses normal coordinates at a point; once one interprets the final result in invariant terms, the resulting expression is then the correct one in any coordinate system.
\end{rmk}

Writing the scalar type symmetric 2-tensor $2\sL_1(\gdot,\Adot)$ analogously to \eqref{EqMSScalar} in terms of the quantities $\wt f^E$, $f^E$, $H_T^E$ and $H_L^E$, and the scalar type 1-form $\sL_2(\gdot,\Adot)$ in terms of the quantities $\wt K^E$ and $K^E$, one finds $\sL(\gdot,\Adot)=0$ to be equivalent to the system
\begin{align}
\label{EqMS2ScwtfE}
    \begin{split}
      \wt f^E &= (\hBox-2\hdelta^*\hdelta-\hdelta^*\hd\,\htr)\wt F+2 r^{-1}(2\hdelta^*\iota_\rho\wt F-\hnabla_\rho\wt F) \\
        &\qquad +4\hdelta^*\hd J + 8 r^{-1}\rho\otimes_s\hd J \\
        & \qquad + (\mu''-k^2 r^{-2})\wt F + (2 Q^2 r^{-4}-\mu'')\hg\,\htr\wt F -4 Q r^{-2}\hg\,\hstar\hd N = 0,
    \end{split} \\
\label{EqMS2ScfE}
   -\frac{r}{k}f^E &= -\hdelta\wt F+2\hd J-r\hd r^{-1}\htr\wt F + 4 Q r^{-2}\hstar N = 0, \\
\label{EqMS2ScHLE}
    \begin{split}
      2 r^2 H_L^E &= \hBox(2 r^2 J)+2 r\iota_\rho\hdelta\wt F - 2\iota_\rho\iota_\rho\wt F + r\iota_\rho\hd\,\htr\wt F + \Bigl(r\mu'+2 Q^2 r^{-2}+\frac{k^2}{2}\Bigr)\htr\wt F \\
        &\qquad + (4\Lambda r^2+4 Q^2 r^{-2}-2 k^2)J - 4 Q\hstar\hd N = 0,
    \end{split} \\
\label{EqMS2ScHTE}
    2 r^2 H_T^E &= -k^2\htr\wt F = 0, \\
\label{EqMS2ScwtKE}
    \wt K^E &= r^{-2}\hdelta r^2\hd N-k^2 r^{-2}N - \frac{1}{2}Q r^{-2}\hstar\hd\,\htr\wt F-2 Q r^{-2}\hstar\hd J = 0, \\
\label{EqMS2ScKE}
    -\frac{r}{k}K^E &= -\hdelta N = 0.
\end{align}
By \eqref{EqMS2ScHTE}, all terms involving $\htr\wt F$ may be dropped. Plugging the expression
\[
  \hdelta\wt F = 2\hd J + 4 Q r^{-2}\hstar N
\]
from \eqref{EqMS2ScfE} into \eqref{EqMS2ScwtfE}, the term $4\hdelta^*\hd J$ drops out, and we obtain a wave equation for $\wt F$ (coupled to $J$ and $N$). Further, by \eqref{EqMS2ScKE}, we can add the vanishing term $\hd\hdelta N$ to \eqref{EqMS2ScwtKE}, the point being that then equations~\eqref{EqMS2ScwtfE}, \eqref{EqMS2ScHLE} (divided by $2 r^2$) and \eqref{EqMS2ScwtKE} can be written as a system
\begin{equation}
\label{EqMS2ScBoxOfPerturbation}
  \hBox P - \sD P = 0, \quad P=\begin{pmatrix}\wt F\\J\\N\end{pmatrix},
\end{equation}
with $\sD$ a stationary differential operator on $\hX$ of order $\leq 1$ acting on the bundle $S^2 T^*\hX\oplus\ulR\oplus T^*\hX$. When $P$ is a smooth mode, this equation is equivalent to an ODE on the 1-dimensional space $t_*=0$ with regular singular points at the horizons where $\mu=0$; thus, the vanishing of $P$ in the static black hole exterior region $\mu>0$ implies the vanishing of $P$ on all of $\hX$. (A more general argument, not relying on the 2-dimensional nature of $\hX$, uses unique continuation on the two components $\{\pm(r-r_\pm)>0\}$ of $\{\mu<0\}$, which are both asymptotically de Sitter-like spaces; see \cite[Lemma~1]{ZworskiRevisitVasy} or \cite[Proposition~5.3]{VasyWaveOndS}.) Our goal is therefore to prove $P=0$ in $\mu>0$.

By \eqref{EqMS2ScKE}, we may write
\begin{equation}
\label{EqMS2ScMaxwellMaster}
  N=\hstar\hd\cA, \quad \cA\in\CI(\hX),
\end{equation}
and then by equation~\eqref{EqMS2ScwtKE}, $0=r^2\hstar\wt K^E=\hd(r^2\hBox\cA-k^2\cA-2 Q J)$; since $\cA$ is only determined up to additive constants, we may thus normalize it such that
\begin{equation}
\label{EqMS2ScMaxwellEqn}
  \hBox\cA - k^2 r^{-2}\cA - 2 Q r^{-2} J = 0;
\end{equation}
this is \cite[equation~\KI{5}{20}]{KodamaIshibashiCharged} in the present context. We point out that $\sigma\neq 0$ allows us to choose $\cA$ to be a mode with frequency $\sigma$ as well; indeed, writing $N=e^{-i\sigma t_*}N_0$, $\hstar N_0=N_{00}\,dt_*+N_{01}\,dr$ and making the ansatz $\cA=e^{-i\sigma t_*}\cA_0$, with $N_{00}$, $N_{01}$ and $\cA_0$ functions of $r$ only, we conclude from $\hd\,\hstar N=0$ that $\pa_r N_{00}+i\sigma N_{01}=0$, while $\hd\cA=\hstar N$ is equivalent to $N_{00}=-i\sigma\cA_0$ and $N_{01}=\pa_r\cA_0$; the first equation is simply solved by $\cA_0=i\sigma^{-1}N_{00}$, and the second equation holds automatically then.

Next, we turn to the equations \eqref{EqMS2ScwtfE}--\eqref{EqMS2ScHTE}. They are not independent due to the (linearized) second Bianchi identity, see \eqref{EqBasicLin2ndBianchi}: indeed, suppose $\sL_2(\gdot,\Adot)=0$ (thus, the linearized stress-energy-momentum tensor satisfies the conservation law $\delta_g D_{g,F}T(\gdot,d\Adot)=0$), which is sufficient for \eqref{EqBasicLin2ndBianchi} to hold; then the identity $\delta_g G_g\sL_1(\gdot,\Adot)\equiv 0$ reads
\begin{align*}
  \delta_g G_g&
  \begin{pmatrix}
    \wt f^E\Sph \\
    -\frac{r}{k}f^E\otimes\sld\Sph \\
    2 r^2\bigl(H_L^E\Sph\slg + H_T^E\slH_k\Sph\bigr)
  \end{pmatrix} = 0 \nonumber\\
  &
  \quad\Longleftrightarrow\quad
  \left\{
  \begin{aligned}
    2 r^{-2}\hdelta(r^2\wt f^E)+\hd(\htr\wt f^E)+\frac{2 k}{r}f^E-4 r^{-2}\hd(r^2 H_L^E) &= 0, \\
    \htr\wt f^E - \frac{2}{k}r^{-2}\hdelta(r^3 f^E) - \frac{2(k^2-2)}{k^2}H_T^E &= 0.
  \end{aligned}
  \right.
\end{align*}
Thus, $f^E=0$ and $H_T^E=0$, i.e.\ equations \eqref{EqMS2ScfE} and \eqref{EqMS2ScHTE}, imply\footnote{This can of course also be checked directly from the equations \eqref{EqMS2ScwtfE}--\eqref{EqMS2ScHTE}.} $\htr\wt f^E=0$ and $\hdelta(r^2\wt E)=0$ for\footnote{By \eqref{EqMSCalcTraceRev}, this is the aspherical part of the tensor $2 G_g\sL_1(\gdot,\Adot)$, which is the sum of the linearized Einstein tensor and the contribution from the linearized stress-energy-momentum tensor.}
\[
  \wt E:=G_\hg\wt f^E+2 H_L^E\hg;
\]
thus, we can recover $\wt f^E$ and $H_L^E$ uniquely from $\wt E$ as the trace-free and pure trace part (up to a constant factor). By \eqref{EqMSCalchXSym2}, the $dr$-component of the latter equation reads
\[
  \frac{\mu'}{2\mu^2}\wt E_{t t} - r^2\mu^{-1}\pa_t \wt E_{t r} + \mu^{-1/2}\pa_r(r^2\mu^{3/2}\wt E_{r r}) = 0;
\]
since $\mu'/2\mu^2\neq 0$ almost everywhere, we conclude that $\wt E_{t t}=0$ is automatically satisfied if $\wt E_{t r}=0$ and $\wt E_{r r}=0$. Thus, the full system \eqref{EqMS2ScwtfE}--\eqref{EqMS2ScHTE} is equivalent to $f^E=0$, $H_T^E=0$, $\wt E_{t r}=0$ and $\wt E_{r r}=0$. \emph{We continue to tacitly use $H_T^E=0$, that is $\htr\wt F=0$.} Following \cite[equation~\KI{5}{27}]{KodamaIshibashiCharged}, we write
\begin{equation}
\label{EqMS2ScXYZDef}
  \wt F + 2 J \hg = \begin{pmatrix} \mu X \\ -\mu^{-1}Z \\ -\mu^{-1}Y \end{pmatrix}
\end{equation}
in the splitting \eqref{EqMSCalchXBundles}; we can recover $\wt F$ and $J$ via
\begin{equation}
\label{EqMS2ScXYZDef2}
  \wt F = \begin{pmatrix} \frac{\mu}{2}(X-Y) \\ -\mu^{-1}Z \\ \frac{1}{2\mu}(X-Y) \end{pmatrix}, \quad J = \frac{X+Y}{4}.
\end{equation}
Using the calculations following \eqref{EqMSCalchXBundles}, the equation $f^E=0$, so $\hdelta(\wt F+2 J\hg)=4 Q r^{-2}\hd\cA$, then reads
\begin{equation}
\label{EqMS2ScfEStatic}
\begin{gathered}
  \pa_t X + \pa_r Z = -4 Q r^{-2}\pa_t\cA, \\
  -\pa_r Y + \mu^{-2}\pa_t Z + \frac{\mu'}{2\mu}(X-Y) = 4 Q r^{-2}\pa_r\cA.
\end{gathered}
\end{equation}
The equation $\wt E_{t r}=\wt f^E_{t r}=0$ takes the form
\begin{equation}
\label{EqMS2ScwtEtr}
  \pa_t\pa_r X - \frac{\mu'}{2\mu}\pa_t X + \pa_t\pa_r Y + \Bigl(2 r^{-1}-\frac{\mu'}{2\mu}\Bigr)\pa_t Y + \frac{k^2}{r^2\mu}Z = 0,
\end{equation}
which corresponds to \cite[equation~\KI{5}{33c}]{KodamaIshibashiCharged}. Since $X$, $Y$, $Z$ and $\cA$ are modes, we have $\pa_t X=-i\sigma X$ etc., hence \eqref{EqMS2ScfEStatic} and \eqref{EqMS2ScwtEtr} imply for $X'=\pa_r X$ etc.\ the equations
\begin{equation}
\label{EqMS2ScXYZPrime}
\begin{split}
  \begin{pmatrix} X' \\ Y' \\ \frac{Z'}{i\sigma} \end{pmatrix} &= T \begin{pmatrix} X \\ Y \\ \frac{Z}{i\sigma} \end{pmatrix} + f, \\
  &\qquad
  T=
  \begin{pmatrix}
    0                 & \frac{\mu'}{\mu}-\frac{2}{r} & \frac{k^2}{r^2\mu}-\frac{\sigma^2}{\mu^2} \\
    \frac{\mu'}{2\mu} & -\frac{\mu'}{2\mu} & \frac{\sigma^2}{\mu^2} \\
    1                 & 0                  & 0
  \end{pmatrix},
  \ f=\frac{4 Q}{r^2}\begin{pmatrix}\cA' \\ -\cA' \\ \cA \end{pmatrix},
\end{split}
\end{equation}
which are \cite[equations~\KI{5}{34a}--\KI{5}{34c}]{KodamaIshibashiCharged} in the present context. Lastly, the equation $\mu\wt E_{r r}=\frac{1}{2}(\mu^{-1}\wt f^E_{t t}+\mu\wt f^E_{r r})-2 H_L^E=0$ takes the form
\begin{equation}
\label{EqMS2ScwtErr}
\begin{split}
  &\mu^{-1}\pa_t^2 X - \frac{\mu'}{2}\pa_r X + \Bigl(-\Lambda+\frac{Q^2}{r^4}\Bigr)X \\
  &\quad + \mu^{-1}\pa_t^2 Y - \Bigl(\frac{\mu'}{2}+\frac{2\mu}{r}\Bigr)\pa_r Y + \Bigl(\frac{k^2-2}{r^2}+\Lambda+\frac{3 Q^2}{r^4}\Bigr)Y \\
  &\quad + \frac{4}{r\mu}\pa_t Z + \frac{4 k^2 Q}{r^4}\cA = 0,
\end{split}
\end{equation}
cf.\ \cite[equations~\KI{5}{31a} and \KI{5}{35}]{KodamaIshibashiCharged}. Upon taking the Fourier transform in $t$ as above and using \eqref{EqMS2ScXYZPrime}, this reduces to a linear constraint on $X$, $Y$ and $Z$, namely
\begin{equation}
\label{EqMS2ScXYZConstraint}
  \gamma\begin{pmatrix}X \\ Y \\ \frac{Z}{i\sigma} \end{pmatrix} = h,
\end{equation}
with
\begin{align*}
  \gamma &= \Bigl(-\frac{\sigma^2}{\mu}-\Lambda+\frac{Q^2}{r^4}-\frac{(\mu')^2}{4\mu}-\frac{\mu'}{r}, \\
  &\qquad\quad -\frac{\sigma^2}{\mu}+\frac{k^2-2}{r^2}+\Lambda+\frac{3 Q^2}{r^4}-\frac{(\mu')^2}{4\mu}+\frac{2\mu'}{r}, \frac{2\sigma^2}{r\mu}-\frac{k^2\mu'}{2 r^2\mu}\Bigr), \\
  h &= -\frac{4 Q}{r^2}\Bigl(\frac{2\mu}{r}\cA' + \frac{k^2}{r^2}\cA\Bigr).
\end{align*}
(Multiplying \eqref{EqMS2ScXYZConstraint} by $-\mu r^2$ and using the expression \eqref{EqMShg} for $\mu$, this agrees with \cite[equation~\KI{5}{36}]{KodamaIshibashiCharged}.) Thus, a generic linear combination $\Phi$ of $X,Y,Z$ with $\cC^2$ coefficients (depending on $r$ only) satisfies a second order ODE with $\cC^0$ coefficients, and moreover $X$, $Y$ and $Z$ can be expressed as a linear combination of $\Phi$ and $\Phi'$ with continuous coefficients. Let us briefly explain how this works: writing $v=(X,Y,Z/i\sigma)$, and letting $\ell\colon\R_r\to(\R^3)^*$ be a $\cC^2$ function, we wish to derive a second order ODE for $\Phi:=\ell(v)$ provided $v$ solves the constrained system\footnote{There is a consistency condition for such a system to be well-posed. Indeed, differentiating the constraint, we find the necessary pointwise condition $(\gamma'+T^t\gamma)(v_0)-h'=0$ for all $v_0\in\R^3$ with $\gamma(v_0)=h$, which is equivalent to $\gamma'+T^t\gamma=\alpha\gamma$ for a scalar $\alpha$ satisfying $\alpha h=h'$.}
\begin{equation}
\label{EqMS2ScODERedux}
  v' = T v + f, \quad \gamma(v) = h.
\end{equation}
Denoting by $T^t$ the transpose of $T$, we have
\[
  \Phi' = \ell_1(v) + f_1, \quad \ell_1=\ell'+T^t\ell,\quad f_1=\ell(f),
\]
and $\Phi''=\ell_2(v)+f_2$ with
\[
  \ell_2=\ell''+2 T^t\ell'+(T')^t\ell+(T^t)^2\ell, \quad f_2=(2\ell'+T^t\ell)(f)+\ell(f').
\]
Thus, we need to find functions $a$, $b$ such that $\ell_2+a\ell_1+b\ell=0$ in $(\R^3)^*/\R\gamma$; this can be solved by solving
\begin{equation}
\label{EqMS2ScODEReduxLin}
  \ell_2 = -\begin{pmatrix} a & b & c \end{pmatrix} \begin{pmatrix} \ell_1 \\ \ell \\ \gamma \end{pmatrix}
\end{equation}
for $a$, $b$ and $c$, which can be done under the non-degeneracy condition that the $3\times 3$ matrix $(\ell_1,\ell,\gamma)$ is invertible. It then follows that $\Phi$ solves the equation
\[
  \Phi'' + a\Phi' + b\Phi  = F, \quad F := f_2 + a f_1 - c h.
\]
Conversely, for a non-degenerate choice of $\ell$, and for $\Phi=\ell(v)$ solving this ODE, one can recover the components of the solution $v$ of \eqref{EqMS2ScODERedux} by solving the linear system
\[
  \begin{pmatrix} \ell \\ \ell_1 \\ \gamma \end{pmatrix} v = \begin{pmatrix} \Phi \\ \Phi'-\ell(f) \\ h \end{pmatrix}
\]
for $v$.

In our application, we want to obtain a master equation for $\Phi$ of Schr\"odinger type, namely $(\mu\pa_r)^2\Phi - (V - \sigma^2)\Phi = F$, where $V$ is a potential and $F$ a forcing term with no dependence on $\sigma$. Expanding the derivative and dividing by $\mu^2$, we thus want to find $\ell$ such that $a=\mu^{-1}\mu'$ and $b=-\mu^{-2}(V(r)-\sigma^2)$. Moreover, for simplicity, we may ask for the forcing to not contain second or higher derivatives of $\cA$, thus we require $\ell(f')=0$, which holds if we take\footnote{Another motivation for choosing the coefficients of $X$ and $Y$ equal is the requirement that the master variables we will describe below be regular at the event and cosmological horizons; see the discussion around equation~\eqref{EqMS2ScRegAtHorizon}.} $\ell=(\alpha,\alpha,\beta)$ with $\alpha,\beta$ to be determined; the independence of $F$ on $\sigma$ then requires $c$ to be a function of $r$ only. Using these pieces of information as an ansatz for \eqref{EqMS2ScODEReduxLin},\footnote{The necessary symbolic calculations are quite lengthy and were performed using \texttt{math\-e\-mat\-i\-ca}.} one can first equate coefficients of the (finite) Taylor series in $\sigma$ beginning with the highest power $\sigma^2$; this gives an algebraic relation expressing $c$ in terms of $\alpha$. The terms with $\sigma$ in the first power trivially agree, being $0$. Next, the first and second components of \eqref{EqMS2ScODEReduxLin} take a very similar form; their difference yields $2\alpha'+(r\beta)'=0$, which suggests taking $\beta=-2 r^{-1}\alpha$. Then, one can solve for the potential $V$ (which only appears linearly) in terms of $\alpha$, leaving us with the task of determining $\alpha$; but $\alpha$ can then be seen to satisfy a simple first order ODE with solution
\begin{equation}
\label{EqMS2ScalphaH}
  \alpha = -\frac{r}{H}, \qquad H(r) := m+3 x-4 z,
\end{equation}
up to a constant prefactor, where we define
\begin{equation}
\label{EqMS2Scmxyz}
  m:=k^2-2,\quad x:=\frac{2\bhm}{r},\quad y:=\frac{\Lambda r^2}{3},\quad z:=\frac{Q^2}{r^2}.
\end{equation}
(The quantity $y$ will be used later.) Here, we note:

\begin{lemma}
\label{LemmaMS2ScHPos}
  In $[r_-,r_+]$, we have $H>0$; indeed, this holds for $l\geq 1$ ($k^2\geq 2$).
\end{lemma}
\begin{proof}
  Using merely that $l\geq 1$, hence $k^2-2\geq 0$, we have $r^2 H\geq 6\bhm r-4 Q^2$. For $Q=0$, this is clearly positive, so let us assume $Q\neq 0$. Then, by \eqref{EqRNdSCriticalPoints} and \eqref{EqRNdSRoots}, this can be estimated from below, using $D=9\bhm^2-8 Q^2>0$, by
  \[
    6\bhm r_{1-} - 4 Q^2 > 3\bhm(3\bhm-\sqrt{D}) - 4 Q^2,
  \]
  whose positivity in turn is equivalent to $9\bhm^2-4 Q^2>3\bhm\sqrt{D}$; since the left hand side is positive, squaring this inequality and cancelling terms gives the equivalent condition $16 Q^4>0$, which is evident.
\end{proof}

Thus, since $\sigma\neq 0$, we may take as our master variable for the system \eqref{EqMS2ScXYZPrime} and \eqref{EqMS2ScXYZConstraint} the function
\begin{equation}
\label{EqMS2ScPhiMaster}
  \Phi := \frac{\frac{2 Z}{i\sigma}-r(X+Y)}{H},
\end{equation}
which recovers \cite[equation~\KI{5}{37}]{KodamaIshibashiCharged}; we then have
\begin{equation}
\label{EqMS2ScMaster1}
  \mu(\mu\Phi')' - (V_\Phi-\sigma^2)\Phi = F_\Phi\cA,
\end{equation}
where $V_\Phi$ and $F_\Phi$ are given in equation~\eqref{EqFormulasVPhiFPhi}. (This recovers \cite[equations~\KI{5}{43}, \KI{5}{44} and \KI{C}{1}]{KodamaIshibashiCharged}.) The expressions for $X$, $Y$ and $Z/i\sigma$ as linear combinations of $\Phi$, $\Phi'$, $\cA$ and $\cA'$ are
\begin{equation}
\label{EqMS2ScRecoverXYZ}
\begin{split}
  X &= \Bigl(\frac{\sigma^2 r}{\mu}-\frac{P_{X 0}}{2 r H^2}\Bigr)\Phi + \frac{P_{X 1}}{2 H}\Phi' + \frac{2 Q P_{X\cA}}{r^2 H^2}\cA + \frac{8 Q\mu}{r H}\cA', \\
  Y &= \Bigl(-\frac{\sigma^2 r}{\mu}-\frac{P_{Y 0}}{2 r H^2}\Bigr)\Phi + \frac{P_{Y 1}}{2 H}\Phi' + \frac{2 Q P_{Y\cA}}{r^2 H^2}\cA - \frac{8 Q\mu}{r H}\cA', \\
  \frac{Z}{i\sigma} &= \frac{P_Z}{2 H}\Phi - r\mu\Phi' + \frac{8 Q\mu}{r H}\cA, \\
\end{split}
\end{equation}
where the functions $P_{X 0}$, $P_{X 1}$, $P_{X\cA}$, $P_{Y 0}$, $P_{Y 1}$, $P_{Y\cA}$ and $P_Z$ are given in equation~\eqref{EqFormulasPX0thruPZ}. (These equations are \cite[equations~\KI{5}{45}, \KI{5}{46} and \KI{C}{4}--\KI{C}{14}]{KodamaIshibashiCharged}, up to constant prefactors, in the present context.\footnote{In \cite[equation~\KI{C}{9b}]{KodamaIshibashiCharged}, the factor $r$ on the right hand side is extraneous.})

The equation \eqref{EqMS2ScMaster1} is coupled to the Schr\"odinger equation obtained from \eqref{EqMS2ScMaxwellEqn} by using the mode nature of $\Phi$, namely
\begin{equation}
\label{EqMS2ScMaster2prelim}
  \mu(\mu\cA')' - \Bigl(\frac{\mu k^2}{r^2}-\sigma^2\Bigr)\cA = \frac{Q\mu}{2 r^2}(X+Y).
\end{equation}
We can evaluate $X+Y$ using \eqref{EqMS2ScRecoverXYZ}, together with the identities
\begin{gather*}
  \frac{P_{X 0}+P_{Y 0}}{2 r H^2} = \frac{H}{r}-\frac{P_Z}{r H}, \quad \frac{P_{X 1}+P_{Y 1}}{2 H} = -2\mu, \quad \frac{2 Q(P_{X\cA}+P_{Y\cA})}{r^2 H^2} = \frac{16 Q\mu}{r^2 H},
\end{gather*}
to compute the form of $V_\cA$ and $F_\cA$. Thus, \eqref{EqMS2ScMaster2prelim} is equivalent to
\begin{equation}
\label{EqMS2ScMaster2}
  \mu(\mu\cA')' - (V_\cA-\sigma^2)\cA = F_{\cA 0}\Phi + F_{\cA 1}\mu\Phi'
\end{equation}
with
\[
  V_\cA = \mu\Bigl(\frac{k^2}{r^2} + \frac{8 Q^2\mu}{r^4 H}\Bigr), \quad F_{\cA 0} = -\frac{Q\mu}{2 r^2}\Bigl(\frac{H}{r}-\frac{P_Z}{H r}\Bigr), \quad F_{\cA 1} = -\frac{Q\mu}{r^2}.
\]
This is in agreement with what one obtains from \cite[equations~\KI{5}{45a}, \KI{5}{45b}, \KI{5}{46a}, \KI{5}{46b}, \KI{C}{9a}, \KI{C}{9b}, \KI{C}{16}]{KodamaIshibashiCharged}, see also \cite[equation~\KI{5}{49}]{KodamaIshibashiCharged}.

The procedure to find non-trivial linear combinations $\Psi:=a\Phi+b\cA$, with $a$, $b$ functions of $r$, which satisfy a single decoupled Schr\"odinger equation is straightforward: by \eqref{EqMS2ScMaster1} and \eqref{EqMS2ScMaster2}, the requirement that $\Psi$ satisfy an equation of the form $\mu(\mu\Psi')'-(V-\sigma^2)\Psi=0$ with $V$ only depending on $r$ translates into the conditions
\begin{gather*}
  a V_\Phi + b F_{\cA 0}+ \mu\mu' a'+\mu^2 a'' - a V = 0, \\
  b' = 0, \quad b F_{\cA 1}+2\mu a'=0, \quad a F_\Phi+b V_\cA - b V = 0.
\end{gather*}
The second equation implies that $b$ is constant. If $Q=0$, this system has two solutions
\begin{equation}
\label{EqMS2ScDecouplingConstantsQ0}
  (a_+,b_+)=(0,1),\quad (a_-,b_-)=(1,0).
\end{equation}
For $Q\neq 0$, we must have $b\neq 0$ (for otherwise the last equation gives $a=0$), so the third equation and the explicit form of $F_{\cA 1}$ yield $a=b(c-Q/2 r)$, with the constant $c$ is to be determined. The last equation gives
\begin{equation}
\label{EqMS2ScMasterPotential}
  V = V_\cA + \Bigl(c-\frac{Q}{2 r}\Bigr)F_\Phi;
\end{equation}
one then obtains possible values for $c$ by plugging this into the first equation, which upon division by $b$ reads
\[
  \Bigl(c-\frac{Q}{2 r}\Bigr)\Bigl(V_\Phi-V_\cA-\Bigl(c-\frac{Q}{2 r}\Bigr)F_\Phi\Bigr) + F_{\cA 0} + \frac{Q\mu}{r^2}\Bigl(\frac{\mu'}{2}-\frac{\mu}{r}\Bigr) = 0;
\]
this equation turns out to indeed have the two \emph{constant} solutions $c_\pm$,
\begin{equation}
\label{EqMS2ScCpmWtc}
  c_+ = -\frac{Q m}{2\tilde c}, \quad c_- = \frac{1}{8 Q}\tilde c,\qquad \tilde c:=3\bhm+\sqrt{9\bhm^2+4 Q^2 m}.
\end{equation}
Let us thus define
\begin{equation}
\label{EqMS2ScApmBpm}
  \begin{aligned}
    a_+ &= c_+ - \frac{Q}{2 r}, & b_+ &= 1, \\
    a_- &= \frac{\tilde c}{6\bhm} - \frac{2 Q^2}{3\bhm r}, & b_- &= \frac{4 Q}{3\bhm},
  \end{aligned}
\end{equation}
which are normalized so as to depend smoothly on $Q$ and to reduce to the solutions \eqref{EqMS2ScDecouplingConstantsQ0} for $Q=0$. We obtain the two master quantities
\[
  \Psi_\pm=a_\pm\Phi+b_\pm\cA,
\]
from which conversely $\Phi$ and $\cA$ are uniquely determined. ($\Psi_+$ and $\Psi_-$ are constant multiples of the quantities defined in \cite[equation~\KI{5}{56}]{KodamaIshibashiCharged} by comparison with \cite[equations~\KI{5}{57a} and \KI{5}{57b}]{KodamaIshibashiCharged}, see also \cite[equation~\KI{5}{59}]{KodamaIshibashiCharged}.) We then obtain the decoupled equations
\begin{equation}
\label{EqMS2ScFinalDecoupled}
  \mu(\mu\Psi_\pm')' - (V_\pm-\sigma^2)\Psi_\pm = 0,
\end{equation}
where the potentials $V_\pm$ can be computed by plugging $c=c_\pm$ into \eqref{EqMS2ScMasterPotential}; the explicit expressions are given in \cite[equation~\KI{5}{61}]{KodamaIshibashiCharged}, which, following \cite[equations~\KI{6}{21} and \KI{6}{23}]{KodamaIshibashiCharged}, we can also write as
\begin{equation}
\label{EqMS2ScSDef}
  V_\pm + \mu S_\pm' - S_\pm^2 = \wt V_\pm, \quad S_\pm := \mu S^0_\pm,\ S^0_\pm:=(\log H_\pm)',
\end{equation}
where
\[
  H_+:=1-\frac{4 Q^2}{\tilde c r},\ H_-:=k^2-2 + \frac{6\bhm}{r}\Bigl(1-\frac{4 Q c_+}{3\bhm}\Bigr),
\]
which are well-defined even for $l\geq 1$, and the potentials
\begin{equation}
\label{EqMS2ScSDefPot}
  \wt V_+ = \frac{k^2\mu}{r^2 H_+}, \quad \wt V_- = \frac{k^2(k^2-2)\mu}{r^2 H_-},
\end{equation}
are well-defined near $[r_-,r_+]$ and positive in $\mu>0$ due to the following lemma.

\begin{lemma}
\label{LemmaMS2ScHpmPos}
  In $[r_-,r_+]$, we have $H_\pm>0$; indeed, this holds for $l\geq 1$ ($k^2\geq 2$).
\end{lemma}
\begin{proof}
  Since $Q c_+\leq 0$, we certainly have $H_->0$. On the other hand, the inequality $H_+>0$ is equivalent to $(3\bhm+\sqrt{9\bhm^2+4(k^2-2)Q^2})r>4 Q^2$, which is clearly satisfied if $Q=0$, so let us assume $Q\neq 0$. We claim that this inequality holds under the weaker condition $r>r_{1-}$, see \eqref{EqRNdSCriticalPoints} and \eqref{EqRNdSRoots}: indeed, for $r=r_{1-}$, its left hand side (for $k^2\geq 2$) is minimized for $k^2=2$, in which case the inequality is equivalent to
  \[
    3\bhm(3\bhm-\sqrt{9\bhm^2-8 Q^2})-4 Q^2 > 0,
  \]
  which holds true, as we saw in the proof of Lemma~\ref{LemmaMS2ScHPos}.
\end{proof}

The point of the \emph{S-deformation} \eqref{EqMS2ScSDef}, see \cite[equations~\KI{6}{4}--\KI{6}{6}]{KodamaIshibashiCharged}, is that if we define the smooth function $\varphi_\pm$ in a neighborhood of $[r_-,r_+]$ by $\varphi_\pm'=S_\pm^0$ (thus we can take $\varphi_\pm=\log H_\pm$), then $\mu\pa_r+S_\pm=e^{-\varphi_\pm}\mu\pa_r e^{\varphi_\pm}$, and
\begin{equation}
\label{EqMS2ScConjMasterEqn}
  e^{-\varphi_\pm}\bigl((\mu\pa_r)^*(\mu\pa_r) + \wt V_\pm\bigr) e^{\varphi_\pm} = -(\mu\pa_r)^2 + V_\pm,
\end{equation}
with the adjoint on the left taken relative to the density $e^{-2\varphi_\pm}\mu^{-1}\,dr$ on $(r_-,r_+)$.

We are now ready to show that the mode $(\gdot,\Adot)$ with $\sigma\neq 0$ is a pure gauge mode; we merely have to show that $\Psi_\pm\equiv 0$ in their domain of definition $(r_-,r_+)$, where $\mu>0$. To do so, we first determine the precise behavior of $\Psi_\pm$ near $r=r_\pm$, which in turn is directly read off from the behavior of $\cA$ and $\Phi$. For the former, we simply observe that $\cA=e^{-i\sigma t_*}\cA_0$ with $\cA_0$ smooth in a neighborhood of $[r_-,r_+]$, as discussed after \eqref{EqMS2ScMaxwellEqn}. For the latter, we recall the definitions \eqref{EqMS2ScXYZDef} and \eqref{EqMS2ScPhiMaster}; now $\wt F=e^{-i\sigma t_*}\wt F_0$ and $J=e^{-i\sigma t_*}J_0$ with $\wt F_0\in\CI([r_-,r_+];S^2 T^*_{\{t_*=0\}}\hX)$ and $J_0\in\CI([r_-,r_+])$; now, writing
\[
  \wt F_0 + 2 J_0\hg = u\,dt_*^2 + 2 v\,dt_*\,dr + w\,dr^2
\]
with $u,v,w\in\CI([r_-,r_+])$, and using $t_*=t-T(r)$ with $dt_*=dt\mp(\mu^{-1}\pm c_\pm)$ from \eqref{EqKNdS0CoordDef}, we find that in the decomposition \eqref{EqMSCalchXBundles}, $\wt F_0$ takes the form
\begin{equation}
\label{EqMS2ScRegAtHorizon}
  \wt F_0 + 2 J_0\hg
  =\begin{pmatrix}
     u \\
     v\mp u(\mu^{-1}\pm c_\pm) \\
     w \mp 2 v(\mu^{-1}\pm c_\pm) + u(\mu^{-2}\pm 2\mu^{-1}c_\pm+c_\pm^2)
   \end{pmatrix};
\end{equation}
therefore, writing the modes $X$, $Y$ and $Z$ as $X=e^{-i\sigma t_*}X_0$ etc., we find that
\[
  Z_0 = -\mu v\pm u(1\pm\mu c_\pm)
\]
is smooth on $[r_-,r_+]$, as is
\[
  X_0+Y_0 = -\mu w \pm 2 v(1\pm\mu c_\pm) \mp 2 u c_\pm - \mu u c_\pm^2,
\]
and therefore so is $e^{i\sigma t_*}\Phi$, and thus so are the $t_*$-independent functions $\Psi_{\pm,0}:=e^{i\sigma t_*}\Psi_\pm$. Thus, the $t$-independent function $\Psi_{\pm,1}:=e^{i\sigma T}e^{\varphi_\pm}\Psi_{\pm0}$ satisfies
\begin{equation}
\label{EqMS2ScConjMasterEqn2}
  \bigl((\mu\pa_r)^*(\mu\pa_r)+(\wt V_{\pm}-\sigma^2)\bigl)\Psi_{\pm,1}=0,
\end{equation}
see \eqref{EqMS2ScConjMasterEqn}, and it lies in the space $|r-r_\pm|^{-\frac{i\sigma}{2\kappa_\pm}}\CI$ up to and including the horizon at $r=r_\pm$, where
\begin{equation}
\label{EqMS2ScSurfGrav}
  \kappa_\pm=\frac{1}{2}|\mu'(r_\pm)|
\end{equation}
are the surface gravities of the horizons. By standard arguments (pairing \eqref{EqMS2ScConjMasterEqn2} with $\Psi_{\pm,1}$ and integrating by parts for $\Im\sigma>0$; using a boundary pairing argument for $\sigma\in\R\setminus\{0\}$), equation~\eqref{EqMS2ScConjMasterEqn2} and the positivity of the potentials $\wt V_\pm$ imply $\Psi_{\pm,1}=0$, as desired.

This finishes the proof of Proposition~\ref{PropMS2Nolog} for scalar perturbations with $\sigma\neq 0$.

\subsubsection{Stationary scalar perturbations\texorpdfstring{ ($\sigma=0$)}{}}
\label{SubsubsecMS2Sc0}

Next, we discuss the modifications required to treat static perturbations, i.e.\ with frequency $\sigma=0$, still assuming $l\geq 2$. Thus, all gauge-invariant quantities $\wt F$, $J$, $N$, and therefore $X$, $Y$ and $Z$, are stationary.

First, we discuss the nature of $\cA$ in \eqref{EqMS2ScMaxwellMaster}: writing $\hstar N=N_0\,dt_*+N_1\,dr$ with $N_0$ and $N_1$ functions of $r$ only, the equation $\hd\,\hstar N=0$ gives $\pa_r N_0=0$; hence, writing $\cA_1(r)=\int_{r_-}^r N_1(s)\,ds$, we have the unique (up to additive constants) solution $\cA=N_0 t_*+\cA_1$, which is a generalized mode with frequency $0$ and grows at most linearly in $t_*$.

Now in the case $Q=0$, i.e.\ perturbing around an uncharged Schwarzschild--de~Sitter black hole, the electromagnetic and gravitational degrees of freedom decouple: equation~\eqref{EqMS2ScMaxwellEqn} and the positivity of $k^2 r^{-2}$ directly imply $\cA\equiv 0$. In general, for any value of $Q$, equation~\eqref{EqMS2ScwtEtr} implies
\begin{equation}
\label{EqMS2Sc0Zvanish}
  Z\equiv 0,
\end{equation}
hence if $Q\neq 0$, we obtain $\pa_t\cA=0$ from the first equation in \eqref{EqMS2ScfEStatic}. We conclude that for any value of $Q$, the quantity $\cA$ is necessarily stationary, and the rest of the discussion will apply to the cases $Q=0$ and $Q\neq 0$ simultaneously.

The derivation of the master variable $\Phi$, defined in \eqref{EqMS2ScPhiMaster}, and thus the derivation of the form of $\Psi_\pm$ given in the previous section does not directly work for $\sigma=0$. We remedy this as follows: for $\sigma\neq 0$, the linear constraint \eqref{EqMS2ScXYZConstraint} shows that $Z/i\sigma$ can be written as a linear combination of $X$, $Y$, $\cA$ and $\cA'$; thus, so can $\Phi$, and therefore, making the dependence of the construction on $\sigma\in\C$ explicit,
\[
  \Psi_\pm(\sigma) = C_{X\pm}(\sigma)X + C_{Y\pm}(\sigma)Y + C_{\cA\pm}(\sigma)\cA + C_{\cA'\pm}(\sigma)\cA',
\]
where the $C_{\bullet\pm}(\sigma)$ are rational functions of $r$ depending analytically on $\sigma\in\C$. Concretely, defining
\[
  \wt H(\sigma) := H (k^2\mu' - 4 r\sigma^2),
\]
and recalling the quantities $m$, $x$, $y$ and $z$ from \eqref{EqMS2Scmxyz}, we have
\begin{equation}
\label{EqMS2Sc0Coeffs}
  \begin{aligned}
    C_{X\pm} &= \frac{a_\pm P_X}{3\wt H}, & C_{Y\pm} &= \frac{a_\pm P_Y}{3\wt H}, \\
    C_{\cA+} &= 1 + \frac{\mu (m+2) P_\cA}{\tilde c r \wt H}, & C_{\cA-} &= b_-\Bigl(1 + \frac{6 a_-\mu (m+2)x}{r \wt H}\Bigr), \\
    C_{\cA'+} &= \frac{2\mu^2 P_\cA}{\tilde c\wt H}, & C_{\cA'-} &= \frac{4 b_-\mu^2(\tilde c-4 r z)}{r \wt H},
  \end{aligned}
\end{equation}
where the functions $P_X$, $P_Y$ and $P_\cA$, given in equation~\eqref{EqFormulasPXthruPA}, are smooth in $[r_-,r_+]$. For later use, we conversely compute the expressions for $X$, $Y$ and $\cA$ in terms of $\Psi_\pm(\sigma)$ and $\Psi_\pm'(\sigma)$: defining the positive function
\begin{equation}
\label{EqMS2Sc0HatC}
  \hat c:=r^{-1}(\tilde c-3\bhm)=r^{-1}\sqrt{9\bhm^2+4 m Q^2},
\end{equation}
we have
\begin{equation}
\label{EqMS2Sc0PsipmToXYA}
\begin{split}
  X &= -\Bigl(\frac{4 Q}{\hat c\mu}\sigma^2 + \frac{2 Q}{\hat c r^2 H^2}P_{X+}\Bigr)\Psi_+ + \Bigl(\frac{3\bhm}{\hat c\mu}\sigma^2+\frac{3}{8\tilde c\hat c H^2}P_{X-}\Bigr)\Psi_- \\
      &\quad\qquad - \frac{2 Q}{\hat c r H}Q_+\Psi'_+ + \frac{3 x}{4\hat c H} Q_-\Psi'_-, \\
  Y &= \Bigl(\frac{4 Q}{\hat c\mu}\sigma^2 - \frac{4 Q z}{r\hat c\tilde c H^2}P_{Y+}\Bigr)\Psi_+ - \Bigl(\frac{3\bhm}{\hat c\mu}\sigma^2 + \frac{x}{8\tilde c\hat c H^2}P_{Y-}\Bigr)\Psi_- \\
      &\quad\qquad + \frac{2 Q}{\hat c r}\Bigl(4\mu+\frac{Q_+}{H}\Bigr)\Psi'_+ - \frac{3 x}{4\hat c}\Bigl(4\mu+\frac{Q_-}{H}\Bigr)\Psi'_-, \\
  \cA &= \frac{\tilde c r^{-1}-4 z}{2\hat c}\Psi_+^0 + \frac{2 z(\tilde c+m r)}{b_-\hat c\tilde c}\Psi_-^0,
\end{split}
\end{equation}
where the functions $P_{X\pm}$, $P_{Y\pm}$ and $Q_\pm$, given in equation~\eqref{EqFormulasPXpthruPym}, are independent of $\sigma$, and smooth in $[r_-,r_+]$. From \eqref{EqMS2Sc0Coeffs}, we see that the functions $C_{\bullet\pm}(\sigma)$ are smooth in $\{r\colon\mu'(r)\neq 4 r\sigma^2/k^2\}$.

One can now check that if one defines
\begin{equation}
\label{EqMS2Sc0MasterVar}
  \Psi_\pm^0 := C_{X\pm}(0)X + C_{Y\pm}(0)Y + C_{\cA\pm}(0)\cA + C_{\cA'\pm}(0)\cA',
\end{equation}
then $\Psi_\pm^0$ solves the master equation \eqref{EqMS2ScFinalDecoupled}, for $\sigma=0$, away from the isolated critical point $r_2\in(r_-,r_+)$ of $\mu$ in the black hole exterior region. That is,\footnote{Without further calculations, this \emph{almost} follows directly from \eqref{EqMS2ScFinalDecoupled}: if $(X,Y,\cA)$ are arbitrary, i.e.\ not necessarily solutions of the linearized Einstein--Maxwell system, then the equation~\eqref{EqMS2ScFinalDecoupled} is equal to some linear combination (with coefficients meromorphic in $\sigma$) of the components of the linearized Einstein--Maxwell system and their derivatives; thus, \emph{if} the coefficients of this linear combination are regular at $\sigma=0$, then we can indeed conclude that for a stationary perturbation $(X,Y,\cA)$, the master variables $\Psi_\pm^0$ solve \eqref{EqMS2Sc0Master} in their domain of definition.}
\begin{equation}
\label{EqMS2Sc0Master}
  \mu(\mu(\Psi_\pm^0)')' - V_\pm\Psi_\pm^0 = 0,\quad r\neq r_2.
\end{equation}
Note that $\Psi_\pm^0$ is bounded by $1/\mu'(r)$ near $r_2$; indeed, $\mu'\Psi_\pm^0$ is smooth on $(r_-,r_+)$, as follows from the form \eqref{EqMS2Sc0Coeffs} of the coefficients. We contend that $\Psi_\pm^0$ is in fact smooth at $r_2$ as well. To see this, we note that by \eqref{EqMS2Sc0Zvanish}, equation~\eqref{EqMS2ScfEStatic}, evaluated at $r_2$, yields $Y'=-4Q r^{-2}\cA'$, which when plugged into \eqref{EqMS2ScwtErr} gives the linear relation
\begin{align*}
  C'_X X &+ C'_Y Y + C'_\cA \cA + C'_{\cA'} \cA' = 0, \\
  & C'_X = -\Lambda+\frac{Q^2}{r^4}, \ C'_Y=\frac{k^2-2}{r^2}+\Lambda+\frac{3 Q^2}{r^4},\ C'_\cA=\frac{4 k^2 Q}{r^4}, \ C'_{\cA'}=\frac{8 Q\mu}{r^3},
\end{align*}
at $r_2$; one can then check that at $r_2$, the vectors
\[
  (C_{X\pm}(0),C_{Y\pm}(0),C_{\cA\pm}(0),C_{\cA'\pm}(0))
\]
are both scalar multiples of $(C'_X,C'_Y,C'_\cA,C'_{\cA'})$, which implies that the smooth function $\mu'\Psi_\pm^0$ vanishes at the simple zero $r_2$ of $\mu'$. Therefore, $\Psi_\pm^0$ is smooth, as claimed, and consequently, $\Psi_\pm^0$ satisfies \eqref{EqMS2Sc0Master} globally on $(r_-,r_+)$.

Lastly, we observe that at either horizon $r=r_-$ or $r=r_+$, we have $C_{X\pm}(\sigma)=C_{Y\pm}(\sigma)$ for $\sigma$ near $0$, as follows from the construction in the previous section, or directly from the calculation $P_Y-P_X=12 H\mu$ and $\mu(r_\pm)=0$. This shows that $\Psi_\pm^0\in\CI([r_-,r_+])$. Using the S-deformation \eqref{EqMS2ScConjMasterEqn}, we thus find that $\Psi_{\pm,1}^0:=e^{\varphi_\pm}\Psi_\pm^0\in\CI([r_-,r_+])$ solves the equation \eqref{EqMS2ScConjMasterEqn2} for $\sigma=0$, that is
\begin{equation}
\label{EqMS2Sc0MasterTilde}
  \bigl((\mu\pa_r)^*(\mu\pa_r) + \wt V_\pm\bigr)\Psi_{\pm,1}^0 = 0,
\end{equation}
with the adjoint defined in \eqref{EqMS2ScConjMasterEqn}; recall here the form of the potentials from \eqref{EqMS2ScSDefPot}. Pairing this against $\chi(\eps^{-1}\mu)\Psi_{\pm,1}^0$, with $\chi(x)$ smooth in $x\geq 0$, vanishing for $x$ near $0$ and identically $1$ for large $x$, and letting $\eps\to 0+$ implies, using the positivity of $\wt V_\pm$, that $\Psi_{\pm,1}^0\equiv 0$, and thus $\Psi_\pm^0\equiv 0$. Now, $X$, $Y$ and $\cA$ can be expressed as linear combinations of\footnote{At frequencies $\sigma\neq 0$, this is again automatic, as $\Phi$ and $\cA$ are linear combinations of $\Psi_\pm$ by the definition of the latter, and then $X$ and $Y$ can be recovered from $\Phi$, $\cA$ and their derivatives by means of \eqref{EqMS2ScRecoverXYZ}.} $\Psi_\pm^0$ and $(\Psi_\pm^0)'$; these expressions are obtained from \eqref{EqMS2Sc0PsipmToXYA} by putting $\sigma=0$. Therefore $X=Y=\cA\equiv 0$, hence all gauge-invariant quantities vanish in $\mu>0$, and the unique continuation argument involving equation~\eqref{EqMS2ScBoxOfPerturbation} shows that they vanish on all of $\hX$.

Since the discussion around equation~\eqref{EqMS2ScPureGaugeDescr} applies without any restriction on $\sigma$, the proof of Proposition~\ref{PropMS2Nolog} is now complete.

\subsubsection{Vector perturbations}
\label{SubsubsecMS2Vec}

Given a vector perturbation \eqref{EqMSVector}, so
\[
  \gdot = \begin{pmatrix} 0 \\ r f\otimes\vect \\ -\frac{2}{k}r^2 H_T\sldelta^*\vect \end{pmatrix},
  \quad
  \Adot = \begin{pmatrix} 0 \\ r K\vect \end{pmatrix}
\]
in the splittings \eqref{EqMSCalcCtgt} and \eqref{EqMSCalcS2}, we consider pure gauge modifications of the same type,
\begin{equation}
\label{EqMS2VecGauge}
  (\upd\gdot,\upd\Adot)=(\cL_{u^\sharp}g,\cL_{u^\sharp}A), \quad u = r L\vect,
\end{equation}
with $L\in\CI(\hX)$. (There are no gauge modifications of $A$ which are of vector type.) Writing $\upd\gdot$ and $\upd\Adot$ in terms of $\upd f$, $\upd H_T$ and $\upd K$ analogously to \eqref{EqMSVector}, the expressions \eqref{EqMSCalcDelGStar} and \eqref{EqMSCalcLieA} give
\[
  \upd f = r\hd(r^{-1}L), \quad \upd H_T=-k r^{-1}L,\quad \upd K=0.
\]
Therefore, the perturbation $(\gdot,\Adot)$ is described by the gauge-invariant quantities
\[
  J := f + \frac{r}{k}\hd H_T \in \CI(\hX;T^*\hX), \quad K\in\CI(\hX).
\]
Indeed, suppose $J=0$ and $K=0$, then the choice $u=r L\vect$, $L=-k^{-1}r H_T$, gives $(\gdot,\Adot)=(\gdot,0)=(\cL_{u^\sharp}g,\cL_{u^\sharp}A)$, thus displaying $(\gdot,\Adot)$ as a pure gauge perturbation; note that $u^\sharp$ is a mode if $(\gdot,\Adot)$ is a mode.

Consider first the linearized second Bianchi identity \eqref{EqBasicLin2ndBianchi}, which holds provided $\sL_2(\gdot,\Adot)=0$: expressing the vector type symmetric 2-tensor $\sL_1(\gdot,\Adot)$ analogously to \eqref{EqMSVector} in terms of $f^E$ and $H_T^E$, this gives (without needing to assume $\sL_1(\gdot,\Adot)=0$)
\[
  \delta_g G_g \begin{pmatrix} 0 \\ r f^E\otimes\vect \\ -\frac{2}{k}r^2 H_T^E\sldelta^*\vect \end{pmatrix} \equiv 0
  \quad \Longleftrightarrow \quad
  r^{-2}\hdelta r^3 f^E + \frac{k^2-1}{k}H_T^E \equiv 0.
\]
In particular, $f^E=0$ implies $H_T^E=0$, since $k^2-1\neq 0$. Thus, instead of analyzing the system $\sL(\gdot,\Adot)=0$, one can equivalently study
\begin{equation}
\label{EqMS2VecBianchi}
  \hpi_2\sL_1(\gdot,\Adot)=0, \quad \sL_2(\gdot,\Adot)=0,
\end{equation}
where
\[
  \hpi_2 \colon \begin{pmatrix} 0 \\ r f^E\otimes\vect \\ -\frac{2}{k}r^2 H_T^E\sldelta^*\vect \end{pmatrix} \mapsto f^E.
\]
The linearized Einstein--Maxwell equation (and hence this system), being gauge-invariant, can be expressed in terms of the quantities $J$ and $K$: following the recipe explained in \S\ref{SubsubsecMS2Sc}, this calculation is most easily done by working in the gauge with $H_T=0$, in which case the gauge-invariant quantity $J$ is related to the perturbation variable $f$ simply by $J=f$. After some algebraic manipulations using the calculations in \S\ref{SubsecMSCalc}, the system \eqref{EqMS2VecBianchi} is then found to be equivalent to
\begin{equation}
\label{EqMS2VecSys}
\begin{gathered}
  r^{-2}\hdel r^4\hd(r^{-1}J) - (k^2-1)r^{-1}J - 4 Q r^{-2}\hstar\hd(r K)=0, \\
  \hBox(r K) - (k^2+1)r^{-1}K - Q\hstar\hd(r^{-1}J) = 0.
\end{gathered}
\end{equation}
The second Bianchi identity reads $\hdelta(r J)=0$, which follows from the first equation after applying $\hd r^2$. It implies that we can write
\begin{equation}
\label{EqMS2VecStarrF}
  \hstar r J = \hd(r\Phi)
\end{equation}
for some function $\Phi$ on $\hX$. After multiplication from the left by $r^2\hstar$, the first equation in \eqref{EqMS2VecSys} becomes 
\[
  \hd r^4\hdel r^{-2}\hd(r\Phi) - (k^2-1)\hd(r\Phi)-\hd(4 Q r K) = 0,
\]
therefore $r^4\hdel r^{-2}\hd(r\Phi) - (k^2-1)r\Phi - 4 Q r K=c$ is constant; since we can change $\Phi$ by adding $r^{-1}c'$ for any $c'\in\R$ without affecting \eqref{EqMS2VecStarrF}, we may assume $c=0$, so
\begin{equation}
\label{EqMS2VecSys1prelim}
  r\hdelta r^{-2}\hd r\Phi - (k^2-1)r^{-2}\Phi - 4 Q r^{-2} K = 0.
\end{equation}
This gives
\begin{equation}
\label{EqMS2VecSys1}
  \hBox\Phi - r^{-2}\Bigl(k^2+1-\frac{6\bhm}{r}+\frac{4 Q^2}{r^2}\Bigr)\Phi - 4 Q r^{-2} K=0.
\end{equation}
(This recovers \cite[equations~\KI{5}{14} and \KI{5}{15}]{KodamaIshibashiMaster} with $n=2$, $k_V=k$, $M=\bhm$ and $Q=0$, $K=0$.) On the other hand, the second equation in \eqref{EqMS2VecSys} can be rewritten, using \eqref{EqMS2VecStarrF} and \eqref{EqMS2VecSys1prelim}, as
\[
  \hBox(r K) - r^{-2}\Bigl(k^2+1+\frac{4 Q^2}{r^2}\Bigr)(r K) - (k^2-1)Q r^{-3}\Phi = 0.
\]
It is convenient to combine this with \eqref{EqMS2VecSys1} into a system for
\[
  \Psi:=\begin{pmatrix} (k^2-1)^{1/2}\Phi \\ 2 r K \end{pmatrix},
\]
namely
\begin{align*}
  \hBox\Psi &- (V + T)\Psi = 0, \\
    & \quad V = r^{-2}\Bigl(k^2+1-\frac{3\bhm}{r}+\frac{4 Q^2}{r^2}\Bigr), \\
    & \quad
      T = r^{-3}
          \begin{pmatrix}
            -3\bhm           & 2 Q(k^2-1)^{1/2} \\
            2 Q(k^2-1)^{1/2} & 3\bhm
          \end{pmatrix},
\end{align*}
with $\hBox$ acting component-wise; for $\Lambda=0$, this recovers \cite[equations~(18)--(20)]{MoncriefRNOdd}. Since $T$ is a constant coefficient symmetric matrix, this system can be diagonalized and yields scalar wave equations $\hBox\Psi_\pm-V_\pm\Psi_\pm=0$ for suitable linearly independent constant coefficient linear combinations $\Psi_\pm$ of the components of $\Psi$, where
\[
  V_\pm = V \pm 2Q(k^2-1)^{1/2}r^{-3} = r^{-2}\Bigl(\bigl((k^2-1)^{1/2}\pm Q r^{-1}\bigr)^2 + 2 - \frac{3\bhm}{r} + \frac{3 Q^2}{r^2}\Bigr);
\]
thus, $2r^2 V_\pm\geq 3\mu+1+\Lambda r^2+\frac{3 Q^2}{r^2}>0$ in the black hole exterior region $(r_-,r_+)$, where $\mu>0$. This implies that the operators $\hBox-V_\pm$ do not have any resonances in $\Im\sigma\geq 0$. Since the functions $\Psi_\pm$ are non-decaying modes, this implies $\Psi_\pm=0$, hence $J=0$ and $K=0$, and therefore $(\gdot,\Adot)$ is a pure gauge solution.

The proof of Proposition~\ref{PropMS2Nolog} is complete.

\subsection{Scalar \texorpdfstring{$l=1$}{l=1} modes}
\label{SubsecMS1s}

For scalar modes with $l=1$, so $k^2=2$, we now show how to obtain the analogue of Proposition~\ref{PropMS2Nolog}, i.e.\ only considering modes, rather than generalized modes, which as in Corollary~\ref{CorMS2Log} implies the $l=1$ scalar part of Theorem~\ref{ThmMS}.

Not assuming $(\gdot,\Adot)$ to be a mode at first, we recall from \eqref{EqMSScalar1} that for such perturbations, the component $H_T$ is no longer defined, so
\begin{equation}
\label{EqMS1sPerturbation}
  \gdot = \begin{pmatrix} \wt f\Sph \\ -\frac{r}{k}f\otimes\sld\Sph \\ 2 r^2 H_L\Sph\slg \end{pmatrix},
  \quad
  \Adot = \begin{pmatrix} \wt K\Sph \\ -\frac{r}{k}K\,\sld\Sph \end{pmatrix}
\end{equation}
in the splittings \eqref{EqMSCalcCtgt} and \eqref{EqMSCalcS2}. Likewise, the corresponding component $H^E_T$ of the tensor $2\sL_1(\gdot,\Adot)$ (see the discussion preceding \eqref{EqMS2ScwtfE}) is not defined. Let us \emph{define} $H_T=0$, and define $\bfX$, $\wt F$, $J$ and $N$ as in \eqref{EqMS2ScbfX} and \eqref{EqMS2ScGaugeInv}. Pure gauge modifications of the form \eqref{EqMS2ScPureGauge} change $\wt f$ etc.\ exactly as in \eqref{EqMS2ScPureGaugeMods}, with the equation for $\upd H_T$ absent, in particular, using the notation of \eqref{EqMS2ScPureGaugeMods},
\begin{equation}
\label{EqMS1sUpdXK}
  \upd\bfX = -T + \frac{r^2}{k}\hd(r^{-1}L), \quad \upd K=-\frac{k}{r}(\iota_A T+P).
\end{equation}
Thus, given any $L$, we can choose $T$ and $P$ so as to make $\bfX+\upd\bfX=0$ and $K+\upd K=0$; furthermore, if $L$ is a mode, then so are $T$ and $P$. Then, working in a gauge with
\begin{equation}
\label{EqMS1sGauge}
  \bfX=0, \quad K=0,
\end{equation}
we have
\begin{equation}
\label{EqMS1swtFJN}
  \wt F=\wt f, \quad J=H_L, \quad N=\wt K,
\end{equation}
and as in the discussion preceding \eqref{EqMS2ScwtfE}, we conclude that the linearized Einstein--Maxwell equation $\sL(\gdot,\Adot)=0$ is equivalent to the system of equations \eqref{EqMS2ScwtfE}, \eqref{EqMS2ScfE}, \eqref{EqMS2ScHLE}, \eqref{EqMS2ScwtKE}, and \eqref{EqMS2ScKE}. We stress that this formulation of $\sL(\gdot,\Adot)=0$ is \emph{only} valid in the gauge \eqref{EqMS1sGauge}. The quantities $\wt F$, $J$ and $N$ are no longer gauge-invariant, rather
\begin{equation}
\label{EqMS1sUpdPertVars}
  \upd\wt F = \frac{2}{k}\hdelta^*r^2\hd r^{-1}L,
  \quad
  \upd J=\frac{k}{2 r}L-\frac{r}{k}\iota_\rho\wh d(r^{-1}L),
  \quad
  \upd N=\frac{r^2}{k}\iota_{\hd(r^{-1}L)}\hd A,
\end{equation}
with the understanding that for any choice of $L$, we choose $T$ and $P$ as explained after \eqref{EqMS1sUpdXK} to enforce our choice of gauge. Working in the gauge \eqref{EqMS1sGauge} (thus $f=0$), we note that the vanishing of the quantities $\wt F$, $J$ and $N$ is equivalent to $(\gdot,\Adot)\equiv 0$.

\emph{Defining} $H^E_T:=-\frac{k^2}{2 r^2}\htr\wt F$, we claim that one can arrange for the absent equation~\eqref{EqMS2ScHTE}, $H^E_T=0$, to hold by a suitable choice of gauge, namely we need to solve
\begin{equation}
\label{EqMS1sUpdHET}
  \upd H^E_T = \frac{k}{r^2}\hdelta r^2\hd r^{-1}L = -k\Box_g(r^{-1}L) = -H^E_T.
\end{equation}
Thus, arranging $H^E_T+\upd H^E_T=0$ amounts to solving a scalar wave equation for $r^{-1}L$.

Let us now consider mode solutions with frequency $\sigma\in\C$, $\Im\sigma\geq 0$, $\sigma\neq 0$ in $t_*$. Define $\cA$ as in \eqref{EqMS2ScMaxwellMaster}, as well as $X$, $Y$ and $Z$ as in \eqref{EqMS2ScXYZDef}; these four functions can then be taken to be modes with frequency $\sigma$. Writing $H^E_T=e^{-i\sigma t_*}H^E_{T,0}$, equation~\eqref{EqMS1sUpdHET} for $L=e^{-i\sigma t_*}\hL(r)$ becomes the stationary equation
\[
  \wh{\Box_g}(\sigma)(r^{-1}\hL) = k^{-1}H^E_{T,0},
\]
which has a smooth solution $\hL$ since $\wh{\Box_g}(\sigma)^{-1}$ exists by our assumptions on $\sigma$. Thus, choosing $\hL$, thus $L$, and thus the gauge in this manner (i.e.\ adding a suitable pure gauge solution to $(\gdot,\Adot)$), the linearized Einstein--Maxwell system is again given by the system \eqref{EqMS2ScwtfE}--\eqref{EqMS2ScKE}. We can then follow the arguments in \S\ref{SubsubsecMS2Sc}, using the non-negativity of $\wt V_\pm$ in \eqref{EqMS2ScSDefPot}, to conclude that the (no longer gauge-invariant) quantities \eqref{EqMS1swtFJN} all vanish, in addition to $\bfX$ and $K$ by our choice of gauge \eqref{EqMS1sGauge}; but then $(\gdot,\Adot)=(0,0)$. Therefore, the original $(\gdot,\Adot)$ is a pure gauge solution, as desired.

For stationary perturbations, so $\sigma=0$, the function $H^E_T$ is stationary as well; since the scalar wave operator $\Box_g$ has a rank $1$ resonance at $0$ with resonant states given by constants, the equation \eqref{EqMS1sUpdHET} for $L$ can be solved by
\begin{equation}
\label{EqMS1sStationaryGaugeMod}
  r^{-1}L=\lambda t_*+r^{-1}L_0,
\end{equation}
with $\lambda\in\R$ and $L_0\in\CI(\hX)$ radial. Thus $L$ may not be stationary at this point, but rather grows linearly in general; however we will show that necessarily $\lambda=0$. For any $L$ as in \eqref{EqMS1sStationaryGaugeMod}, $\hd(r^{-1}L)$ is stationary still, hence adding to $(\gdot,\Adot)$ the pure gauge solution corresponding to this $L$ (and taking $T$ and $P$ as explained after \eqref{EqMS1sUpdXK}), the new perturbation variables $\wt F$ and $N$ are still stationary, see \eqref{EqMS1sUpdPertVars}; furthermore $\htr\wt F=0$, and $J-\frac{\lambda k t_*}{2}$ is stationary. If we define $X$, $Y$ and $Z$ as in \eqref{EqMS2ScXYZDef2}, we thus conclude that $Z=:Z_0$ is stationary, while
\[
  X = \lambda k t_* + X_{*,0}, \quad Y = \lambda k t_* + Y_{*,0}
\]
with $X_{*,0},Y_{*,0}\in\CI(\hX)$ stationary. Next, $\cA$ in \eqref{EqMS2ScMaxwellMaster} can be written as
\[
  \cA = \lambda' t_* + \cA_{*,0}
\]
with $\lambda'\in\R$ and $\cA_{*,0}\in\CI(\hX)$ stationary, see the first paragraph of \S\ref{SubsubsecMS2Sc0}. Using $t_*=t-T(r)$ as in \eqref{EqKNdS0CoordDef}, with $T\in\CI((r_-,r_+))$, $T(r)\to+\infty$ as $r\to r_\pm$, we write
\[
  X = \lambda k t + X_0, \quad Y = \lambda k t + Y_0, \quad \cA = \lambda' t + \cA_0,
\]
where
\begin{equation}
\label{EqMS1sPertVarSing}
  X_0=X_{*,0}-\lambda k T, \quad Y_0=Y_{*,0}-\lambda k T, \quad \cA_0=\cA_{*,0}-\lambda' T;
\end{equation}
note that $X_0$ etc.\ are only defined in $(r_-,r_+)$. Evaluating \eqref{EqMS2ScwtEtr}, we find (using $k=\sqrt{2}$) that $Z=Z_0$ is given by
\[
  Z_0=-\lambda k r\Bigl(1-\frac{3\bhm}{r}+\frac{2 Q^2}{r^2}\Bigr).
\]
(In contrast to the stationary $l\geq 2$ scalar-type case discussed in \S\ref{SubsubsecMS2Sc0}, we cannot conclude $Z_0\equiv 0$ in the present setting yet.) The first equation in \eqref{EqMS2ScfEStatic} then yields
\begin{equation}
\label{EqMS1sLambdaPrime}
  \lambda' = -\frac{Q}{k}\lambda;
\end{equation}
this is the same as what one obtains from comparing the $t$-coefficients in equation~\eqref{EqMS2ScwtErr}.

We now contend that the stationary functions $X_0$, $Y_0$ and $\cA_0$ satisfy \emph{the same equations} as $X$, $Y$ and $\cA$ in the stationary $l\geq 2$ case considered in \S\ref{SubsubsecMS2Sc0}; this is a consequence of the fact that, given our definition of $Z_0$, the equations \eqref{EqMS2ScfEStatic}, \eqref{EqMS2ScwtEtr} and \eqref{EqMS2ScwtErr} are either identities, or (equating coefficients of $t$ and constant coefficients, respectively) precisely the equations used in \S\ref{SubsubsecMS2Sc0}. In detail, since $\pa_t Z\equiv 0$, the second equation in \eqref{EqMS2ScfEStatic} is an equation for $X_0$, $Y_0$ and $\cA_0$ (formally obtained by deleting the term with $Z$, and replacing $X$ by $X_0$, etc.), likewise for the constant (in $t$) coefficient of \eqref{EqMS2ScwtErr} (formally, deleting all terms involving $\pa_t$, and replacing $X$ by $X_0$, etc.); in both of the resulting equations, $Z_0$ \emph{does not appear}, thus making the resulting first order ODE system take the same form as the system we implicitly used in \S\ref{SubsubsecMS2Sc0} (where we had $Z\equiv 0$). Lastly, equation~\eqref{EqMS2ScMaxwellEqn} for $\cA$ gives two equations, one corresponding to the coefficient of $t$, which one can verify is an identity given \eqref{EqMS1sLambdaPrime}, and one corresponding to the constant coefficient, which is formally obtained by replacing $\cA$ by $\cA_0$ and $J=(X+Y)/4$ by $(X_0+Y_0)/4$. This proves our contention.

Since all manipulations of equations in our entire mode stability discussion are purely algebraic, we can now follow the arguments in \S\ref{SubsubsecMS2Sc0}, replacing $X$ by $X_0$, etc., to deduce the master equations \eqref{EqMS2Sc0MasterTilde} for $\Psi_{\pm,1}^0=e^{\varphi_+}\Psi_\pm^0$, with $\Psi_\pm^0$ given in \eqref{EqMS2Sc0MasterVar}. (One may set $k^2=2$ in the discussion there without incurring any singular terms.)

Let us first consider the situation for $\Psi_{+,1}^0$: the coefficients $C_{X\pm}(0)$ etc.\ are smooth down to $\mu=0$, and we have
\[
  \bigl(C_{X+}(0),C_{Y+}(0),C_{\cA+}(0),C_{\cA'+}(0)\bigr) = (Q/4,Q/4,1,0)
\]
there. But then, using \eqref{EqMS1sLambdaPrime} in the expressions \eqref{EqMS1sPertVarSing}, the most singular contributions to $\Psi_{+,1}^0$ of size $T\sim|\log\mu|$ cancel due to $Q/4(-\lambda k)+Q/4(-\lambda k)+1\cdot(-\lambda')=0$, therefore in fact
\[
  \Psi_{+,1}^0 \in \CI([r_-,r_+]) + \mu T\cdot\CI([r_-,r_+]),
\]
which suffices to deduce $\Psi_{+,1}^0\equiv 0$ in $(r_-,r_+)$ from the master equation \eqref{EqMS2Sc0MasterTilde} by using the positivity of the potential $\wt V_+$, as explained in the paragraph following this equation.

On the other hand, the coefficients in the definition of $\Psi_{-,1}^0$ at $r=r_\pm$ (where $\mu=0$) are given by
\[
  \bigl(C_{X-}(0),C_{Y-}(0),C_{\cA-}(0),C_{\cA'-}(0)\bigr) = \frac{1}{6\bhm}(2 Q^2-3\bhm r_\pm,2 Q^2-3\bhm r_\pm,8 Q,0);
\]
in this case,
\[
  C_{X-}(0)(-\lambda k T)+C_{Y-}(0)(-\lambda k T) + C_{\cA-}(0)(-\lambda' T) = \lambda k r_\pm T;
\]
observe that for $\lambda\neq 0$, this gives
\begin{equation}
\label{EqMS1sPsi1Limits}
  \Psi_{-,1}^0\to(\sgn\lambda)\infty,\quad r\to r_\pm,
\end{equation}
i.e.\ the limit is either $+\infty$ at both horizons, or $-\infty$ at both horizons. Now, the equation for for $\Psi_{-,1}^0$ reads (after factoring out a non-vanishing prefactor, and using that $\wt V_-\equiv 0$ for $k^2=2$)
\[
  \pa_r e^{-2\varphi_+}\mu\pa_r \Psi_{-,1}^0 = 0,
\]
thus
\[
  \Psi_{-,1}^0 = c_1\int\frac{e^{2\varphi_+}}{\mu}\,dr + c_2
\]
for $c_1,c_2\in\R$. But then for $c_1\neq 0$, we have $\Psi_{-,1}^0\to\pm(\sgn c_1)\infty$ as $r\to r_\pm$, which contradicts \eqref{EqMS1sPsi1Limits}; therefore, we must have $c_1=0$, correspondingly $\lambda=0$, and therefore $\Psi_{-,1}^0=c_2$ is \emph{constant}, and $\Psi_\pm^0=c_2 H_+^{-1}=c_2 r/6\bhm$. Furthermore, we obtain $Z=Z_0\equiv 0$, and $X=X_0$, etc. We also point out that this shows that the gauge modification function $L$ in \eqref{EqMS1sStationaryGaugeMod} is in fact stationary.

Recall now that in the present, $l=1$, setting, $\Psi_{-,1}^0$ is not gauge-invariant; the next step is thus to show that $\Psi_{+,1}^0\equiv 0$ and $\Psi_{-,1}^0=c_2 r/6\bhm$ is a pure gauge solution. For this purpose, by linearity it suffices to consider the case $c_2=-6\bhm$, so $\Psi_\pm^0=-r$, which by virtue of \eqref{EqMS2Sc0PsipmToXYA} after a brief calculation corresponds to
\[
  X = 1,\quad Y = 1, \quad \cA = -\frac{Q}{2},
\]
i.e.\ to a perturbation with $\wt F\equiv 0$, $J=1/2$, and $N=0$, which according to \eqref{EqMS1sPerturbation}, \eqref{EqMS1sUpdPertVars} and the choice of gauge above is the perturbation given by
\[
  \gdot = r^2\Sph\slg, \quad \Adot=0;
\]
this however is equal to a pure gauge solution as in \eqref{EqMS2ScPureGauge} for $T=0$, $P=0$ and $L=2 r/k$. Therefore, upon adding to our original mode perturbation $(\gdot,\Adot)$ a pure gauge mode solution \emph{which is smooth on $\hX$} (not merely in $(r_-,r_+)$), the resulting smooth mode perturbation in the gauge \eqref{EqMS1sUpdPertVars} has $\wt F=0$, $J=0$ and $N=0$ in $(r_-,r_+)$; but then the unique continuation argument involving equation~\eqref{EqMS2ScBoxOfPerturbation} applies, and hence $\wt F=0$ etc.\ on all of $\hX$. In summary, we have shown that a stationary $l=1$ scalar perturbation is equal to a pure gauge solution, with the quantities $T$, $P$ and $L$ stationary as well.

This finishes the proof of Theorem~\ref{ThmMS} for scalar $l=1$ perturbations.

\subsection{Vector \texorpdfstring{$l=1$}{l=1} modes}
\label{SubsecMS1v}

We next consider a perturbation of the form \eqref{EqMSVector1}; thus, in the splittings \eqref{EqMSCalcCtgt} and \eqref{EqMSCalcS2}, we have
\[
  \gdot = \begin{pmatrix} 0 \\ r f\otimes\vect \\ 0 \end{pmatrix}, \quad \Adot=\begin{pmatrix} 0 \\ r K\vect \end{pmatrix}.
\] We assume that $(\gdot,\Adot)$ is a non-decaying generalized mode solution of $\sL(\gdot,\Adot)=0$, i.e.\ with frequency $\sigma$, $\Im\sigma\geq 0$. There is no analogue of Corollary~\ref{CorMS2Log} in the case that there are mode solutions of $\sL(\gdot,\Adot)=0$ which are not pure gauge; thus, in such cases, one has to deal with generalized modes directly.

The pure gauge terms which are of vector $l=1$ type are given in \eqref{EqMS2VecGauge}: they are determined by a 1-form $u=r L\vect$, and for $l=1$ vectors, thus $\sldelta^*\vect=0$, we have
\[
  \cL_{u^\sharp}g = 2\delta_g^*u = 2 r^2\hd r^{-1}L \otimes_s \vect,\quad \cL_{u^\sharp}A=0.
\]
Thus, adding $\cL_{u^\sharp}(g,A)$ to $(\gdot,\Adot)$ amounts to changing $f$ to $f+r\hd(r^{-1}L)$ and leaving $K$ unchanged. Therefore, the gauge-invariant quantities associated with the given perturbation $(\gdot,\Adot)$ are
\begin{equation}
\label{EqMS1vGaugeInv}
  \hd(r^{-1}f),\quad K.
\end{equation}
Indeed, if both vanish, then we can choose $L$ with $\hd(r^{-1}L)=r^{-1}f$ as the cohomology of $\hX$ is trivial: one can obtain $L$ by simply integrating $r^{-1}f$, starting from a fixed point in $M^\circ$. (In general, the expansion of $L$ in powers of $t_*$ will thus have one higher power of $t_*$ than $(\gdot,\Adot)$, thus $f$.) For the 1-form $u=r L\vect$, this gives
\[
  (\gdot,\Adot) = (2 r f\otimes_s\vect,0) = (\cL_{u^\sharp} g, \cL_{u^\sharp} A),
\]
as desired.

The equation $\sL(\gdot,\Adot)=0$, expressed in terms of the gauge-invariant quantities \eqref{EqMS1vGaugeInv}, is equivalent to the system
\begin{equation}
\label{EqMS1vSys}
\begin{gathered}
  r^{-2}\hdelta r^4\hd(r^{-1}f) - 4 Q r^{-2}\hstar\hd(r K) = 0, \\
  \hdelta\hd(r K)-2 r^{-1}K - Q\hstar\hd(r^{-1}f) = 0.
\end{gathered}
\end{equation}
Before analyzing this system, we point out that a linearized KNdS solution in which we change the angular momentum, rotating around the same axis as $\vect$, so
\[
  (\pa_a g_b|_{b_0},\pa_a\breve A_b|_{b_0}) = (2(1-\mu)\sin^2\theta\,dt\,d\phi, Q r^{-1}\sin^2\theta\,d\phi)
\]
in static coordinates, is of the form \eqref{EqMSVector1} with $f=(\frac{2\bhm}{r^2}+\frac{\Lambda r}{3}-\frac{Q^2}{r^3})dt$, $K=Q r^{-2}$ and $\vect=\sin^2\theta\,d\phi$; the gauge-invariant quantities associated with this particular perturbation are thus\footnote{A simple calculation verifies that these quantities indeed solve \eqref{EqMS1vSys}.}
\begin{equation}
\label{EqMS1vKNdS}
  \hstar\hd(r^{-1}f)=4 Q^2 r^{-5}-6\bhm r^{-4}, \qquad K = Q r^{-2}.
\end{equation}

Returning to the analysis of \eqref{EqMS1vSys} in general, we note that the first equation and the 2-dimensional nature of $\hX$ (thus $\hd(r^{-1}f)$ is a top degree form) imply
\[
  r^4\hstar\hd(r^{-1}f) - 4 Q r K = c
\]
for some constant $c\in\R$. Note that this does not use the fact that $(\gdot,\Adot)$ is a generalized mode solution. (For the perturbation \eqref{EqMS1vKNdS}, one has $c=-6\bhm$.) Plugging this into the second equation gives
\begin{equation}
\label{EqMS1vASWave}
  (\hBox - V)(r K) = Q r^{-4} c, \quad V=2r^{-2} + 4 Q^2 r^{-4}.
\end{equation}
Since $V>0$, the scalar operator $\hBox-V$ on $\hX$ has no resonances in the closed upper half plane; since $r K$ is a non-decaying generalized mode, we conclude that $r K$ must in fact be stationary. More precisely, since $K=Q r^{-2}$ solves this equation with $c=-6\bhm$ by the discussion of the linearized KNdS solution, we have $K = -\frac{c Q}{6\bhm r^2}$.

Thus, subtracting the linearized KNdS solution $-\frac{c}{6\bhm}(\pa_a g_b|_{b_0},\pa_a A_b|_{b_0})$ from $(\gdot,\Adot)$ --- which preserves the generalized mode nature of $(\gdot,\Adot)$ --- we can assume that $c=0$, hence $K=0$ (therefore $\Adot=0$), and thus $\hd(r^{-1}f)=0$, so all gauge-invariant quantities vanish. Our original $(\gdot,\Adot)$ is therefore the sum of a linearized KNdS solution and a pure gauge solution.

\begin{rmk}
\label{RmkMS1vDynamical}
  For purely gravitational perturbations, i.e.\ $\Adot=0$, the gauge-invariant quantity $K$ vanishes, and thus $\hd(r^{-1}f)=0$ and $(\gdot,0)$ is a pure gauge solution as above \emph{for general $\gdot$}, that is, without using the assumption that $\gdot$ is a non-decaying generalized mode; this is the case considered in \cite[\S7]{HintzVasyKdSStability}. Similarly, purely electromagnetic perturbations, i.e.\ $\gdot=0$, have $\hd(r^{-1}f)=0$, hence $K=-\frac{c}{4 Q r}$, which plugged into \eqref{EqMS1vASWave} yields $c=0$, and again $(0,\Adot)$ is a pure gauge solution for general $\Adot$. For general coupled gravitational and electromagnetic perturbations however, and for $c=0$, the wave equation \eqref{EqMS1vASWave} for $K$ is expected to have non-trivial exponentially decaying mode solutions; we do not study this here however. Thus, it is likely that there are dynamical degrees of freedom for vector $l=1$ perturbations, unlike in the purely gravitational context considered in \cite{HintzVasyKdSStability}; this was already pointed out by Moncrief \cite[\S{VII}]{MoncriefRNGaugeInv}.
\end{rmk}

\subsection{Spherically symmetric perturbations (\texorpdfstring{$l=0$}{l=0} modes)}
\label{SubsecMS0}

The simple proof of Birkhoff's theorem given by Schleich and Witt \cite{SchleichWittBirkhoff} can easily be extended to the spherically symmetric Einstein--Maxwell system; we proceed to describe the linearized version of such an extension. (An alternative approach uses the formulation of the Einstein--Maxwell system as a characteristic initial value problem in a double null gauge, see e.g.\ \cite[equations~(2.3)--(2.6)]{DafermosRendallSCC}.) A general spherically symmetric perturbation takes the form
\begin{equation}
\label{EqMS0Pert1}
  \gdot = \mudot\,dt_*^2 + 2\gdot_{01}\,dt_*\,dr + \gdot_{11}\,dr^2 + \gdot_S\,r^2\slg, \quad \Adot = \Adot_0\,dt_* + \Adot_1\,dr,
\end{equation}
where the coefficients are functions of $(t_*,r)$. Assuming that $\sL(\Adot,\gdot)=0$, we aim to show that this must be equal to a linearized RNdS solution $(g'(b'),A'(b'))$ for some $b'=(\dot\bhm,\bfzero,\Qdot)\in\R^5$, up to a pure gauge term. To do so, we first work locally near the horizons $r=r_\pm$ using the forms \eqref{EqKNdS0MetricNull} and \eqref{EqKNdS0Potential} of the metric and the 4-potential, so we write
\begin{equation}
\label{EqMS0Metric}
  g = \mu\,dt_0^2\pm 2\,dt_0\,dr - r^2\slg, \quad A = q\,dt_0,
\end{equation}
near $r=r_\pm$, where we set
\begin{equation}
\label{EqMS0MuQ}
  \mu = 1 - \frac{2\bhm}{r} - \frac{\Lambda r^2}{3} + \frac{Q^2}{r^2}, \quad q=-Q r^{-1}
\end{equation}
The function $t_0$, defined in \eqref{EqKNdS0NullCoord} near $r=r_\pm$, does not match up in the overlap region $(r_-,r_+)$, so should properly be denoted $t_{0,\pm}$, and correspondingly $A_\pm=-Q r^{-1}\,dt_{0,\pm}$, but we drop the subscript for brevity. (The $A$ here differs from \eqref{EqKNdS0Potential} by an exact 1-form; note here that adding an exact 1-form to $A$ not only does not change the form, but in fact not any of the \emph{coefficients} of the linearization of \eqref{EqBasicDerEM} around $(g,A)$, since \eqref{EqBasicDerEM} only involves $A$ through $d A$.) We note that in the form \eqref{EqMS0Metric}, the linearization of $(g,A)$ in $(\bhm,Q)$ in the direction $(\dot\bhm,\dot Q)$ takes the form
\begin{equation}
\label{EqMS0LinMetPot}
  g' = \Bigl(-\frac{2\dot\bhm}{r}+\frac{2 Q\dot Q}{r^2}\Bigr)\,dt_0^2,\quad A'=-\dot Q r^{-1}\,dt_0,
\end{equation}
i.e.\ this is the form of a linearized RNdS metric, which on its domain of definition differs from $(g'(b'),A'(b'))$ for $b'=(\dot\bhm,\bfzero,\dot Q)$ --- see Definition~\ref{DefKNdSaLin} and recall that we are dropping the subscript `$b_0$' --- by a pure gauge term.

We are free to add pure gauge terms to $(\gdot,\Adot)$, defined in \eqref{EqMS0Pert1}, corresponding to aspherical vector fields $V\in\CI(\hX;T\hX)$ and functions $a\in\CI(\hX)$; thus, in $(t_0,r)$ coordinates, $V=V_0(t_0,r)\pa_{t_0}+V_1(t_0,r)\pa_r$ and $a=a(t_0,r)$. We compute:
\begin{lemma}
\label{LemmaMS0Gauge}
  Let us consider the metric and 4-potential defined in \eqref{EqMS0Metric}. As a map between sections of
  \[
    \la\pa_{t_0}\ra\oplus\la\pa_r\ra \quad\tn{and}\quad \la dt_0^2\ra \oplus \la 2\,dt_0\,dr\ra \oplus \la dr^2\ra \oplus \la\slg\ra,
  \]
  we then have
  \[
    \tcL_g
      = \begin{pmatrix}
          2\mu\pa_{t_0}        & \pm 2\pa_{t_0}+\mu' \\
          \mu\pa_r\pm\pa_{t_0} & \pm\pa_r \\
          \pm 2\pa_r           & 0 \\
          0                    & -2 r
        \end{pmatrix},
  \]
  while as a map between sections of $\la\pa_{t_0}\ra\oplus\la\pa_r\ra$ and $\la dt_0\ra\oplus\la dr\ra$, we have
  \[
    \tcL_A
      = \begin{pmatrix}
          q\pa_{t_0} & q' \\
          q\pa_r     & 0
        \end{pmatrix}
  \]
  Mapping from sections of the trivial bundle $\ulR$ into $\la dt_0\ra\oplus\la dr\ra$, we have $d=(\pa_{t_0}, \pa_r)^t$.
  
  Lastly, for any smooth 1-form $\omega=u\,dt_0+v\,dr$, there exists a smooth function $a$ with $(\omega+d a)(\pa_r)=0$, i.e.\ $\omega+d a$ has no $dr$ component.
\end{lemma}
\begin{proof}
  The first three claims are simple calculations; the last claim follows by taking $a$ solving $\pa_r a=-v$.
\end{proof}

Thus, rewriting the perturbation \eqref{EqMS0Pert1} in the form
\[
  \gdot = \mudot\,dt_0^2 \pm 2\Xdot\,dt_0\,dr + \Zdot\,dr^2 - 2 r\Ydot\,\slg, \quad \Adot=\dot q\,dt_0 + \Adot_1\,dr
\]
(with a new function $\Adot_1$), we may add a pure gauge solution $(\cL_V g,\cL_V A)$ with $V=v\pa_{t_0}$, where $v$ solves $\pm 2\pa_r v=-\Zdot$, in order to eliminate the $dr^2$ component of $\gdot$. Further, by adding another such pure gauge solution, now with $V=-\Ydot\pa_r$, we can eliminate the $\slg$ component of $\gdot$. Afterwards, we can replace $\Adot$ by $\Adot+d a$ for suitable $a$ so as to eliminate the $dr$ component of $\Adot$. Up to pure gauge solutions, we can thus write the perturbation as
\begin{equation}
\label{EqMS0Pert2}
  \gdot = \mudot\,dt_0^2 \pm 2\Xdot\,dt_0\,dr,\quad \Adot=\dot q\,dt_0.
\end{equation}
To proceed, we use the linearized Einstein--Maxwell equations. For $(g,A)$ of the general form
\[
  g = \mu\,dt_0^2 \pm 2 X\,dt_0\,dr - r^2\slg,\quad A=q\,dt_0,
\]
where $\mu,X,q$ may be any functions of $(t_0,r)$, the vanishing of the $dr^2$ component of
\[
  \sE:=\Ric(g)+\Lambda g-2 T(g,dA)
\]
is equivalent to $\pa_r X=0$. Thus, the linearized Einstein--Maxwell equations for \eqref{EqMS0Pert2} imply $\pa_r\Xdot=0$, so $\Xdot$ is a function of $t_0$ only; writing $\Xdot=\pa_{t_0}f$ with $f=f(t_0)$, we may subtract $\tcL_g(2f\pa_{t_0})$ from $\gdot$ to eliminate the $dt_0\,dr$ component of $\gdot$. We may therefore assume
\[
  \Xdot \equiv 0, \quad\tn{so}\quad \gdot=\mudot\,dt_0^2,\quad \Adot=\qdot\,dt_0.
\]
with new functions $\mudot$ and $\qdot$. Considering $(g,A)$ of the form \eqref{EqMS0Metric}, but with $\mu$ and $q$ arbitrary functions of $(t_0,r)$, Maxwell's equation becomes
\begin{equation}
\label{EqMS0Max}
  \delta_g d A = \bigl(-\mu(2r^{-1}\pa_r q+\pa_{rr}q)\pm\pa_{t_0}\pa_r q\bigr)\,dt_0 \mp (2r^{-1}\pa_r q+\pa_{rr}q)\,dr = 0.
\end{equation}
When linearized around $\mu,q$ given in \eqref{EqMS0MuQ}, the $dr$ component implies $\pa_r r^2\pa_r\qdot=0$, so $\qdot=\qdot_0-\Qdot r^{-1}$ with $\qdot_0$ and $\Qdot$ functions of $t_0$ only. Since $\qdot_0(t_0)\,dt_0$ is exact, we may assume $\qdot_0\equiv 0$ by adding an exact 1-form to $\Adot$; thus $\qdot=-\Qdot r^{-1}$. The $dt_0$ component of \eqref{EqMS0Max} implies $\pa_{t_0}\Qdot=0$, therefore $\Qdot$ is constant.

Next, the vanishing of the spherical part of $\sE$ for arbitrary $(g,A)$ of the form \eqref{EqMS0Metric} is equivalent to
\[
  -1 + \Lambda r^2 + \pa_r(r\mu) + r^2(\pa_r q)^2 = 0.
\]
Linearized around \eqref{EqMS0MuQ} and using the already determined form of $\qdot$, this gives
\[
  \pa_r(r\mudot) + 2r^{-2}Q\Qdot = 0,
\]
hence
\[
  \mudot = -\frac{2\dot\bhm}{r} + \frac{2 Q\Qdot}{r^2},
\]
where $\dot\bhm$ is a constant of integration in $r$, thus it is a priori a function of $t_0$; this is almost the desired result \eqref{EqMS0LinMetPot}. But the $dt_0^2$ component of $\sE=0$ reads
\[
  \mu\bigl(2\pa_r\mu + r(2\Lambda-2(\pa_r q)^2+\pa_{rr}\mu)\bigr) \pm 2\pa_{t_0}\mu = 0.
\]
For $\mu,q$ given by \eqref{EqMS0MuQ}, the expression enclosed in large parentheses vanishes for all values of $\bhm$ and $Q$, and since it does not involve $t_0$ derivatives, it vanishes in fact for $\mu,q$ of the form \eqref{EqMS0MuQ} for which $\bhm$ depends on $t_0$. Thus, linearizing this equation around the $\mu,q$ given in \eqref{EqMS0MuQ} yields $\pa_{t_0}\mudot=0$, i.e.\ $\dot\bhm=0$. This proves $(\gdot,\Adot)=(g',A')$, given in \eqref{EqMS0LinMetPot}, up to pure gauge terms.

Since this argument was local near the horizons, we need to patch the pure gauge solutions together to get the desired global (in $r$) result. For a perturbation $(\gdot,\Adot)$ of the general form \eqref{EqMS0Pert1} solving the linearized Einstein--Maxwell system, we showed that
\[
  (\gdot,\Adot) = \bigl(g'(b'_\pm)+\cL_{V_\pm}g,A'(b'_\pm)+\cL_{V_\pm}A + d a_\pm \bigr)
\]
for $r<r_+$ (the `$-$' case), resp.\ $r>r_-$ (the `$+$' case). Now $(g'(b'_+),A'(b'_+))$ can be brought into the form used near $r_-$, i.e.\ \eqref{EqMS0Pert2} with the bottom sign; since this leaves $\mudot$ unaffected, we conclude by comparing coefficients that we must have $b'_+=b'_-\equiv b'$. Therefore $\cL_{V_+-V_-}g=0$, so $V_+-V_-$ is spherically symmetric and Killing, hence a constant multiple of the globally defined vector field $\pa_{t_*}$; thus $V_+=V_-+c_V\pa_{t_*}\equiv V$ for a suitable $c_V\in\R$, with $\cL_{V_\pm}g=\cL_V g$. We then find $d(a_+-a_-)=0$, hence $a_+=a_-+c_a\equiv a$ for a suitable $c_a\in\R$, and we conclude that
\[
  (\gdot,\Adot) = (g'(b')+\cL_V g,A'(b')+\cL_V A+d a),
\]
with $V$ and $a$ defined globally on $M^\circ$. This finishes the proof of Theorem~\ref{ThmMS} in spherical symmetry.

We remark that if $(\gdot,\Adot)$ was a generalized mode with highest power $t_*^k$, $k\in\N_0$, in the expansion, then $(V,a)$ constructed above is a generalized mode as well (with the same frequency $\sigma$), with highest power in $t_*$ at most $t_*^{k+1}$.

This completes the proof of Theorem~\ref{ThmMS}.

\section{Constraint damping}
\label{SecCD}

We return to the analysis of the \emph{gauge-fixed} Einstein--Maxwell system linearized around $(g,A)\equiv (g_{b_0},A_{b_0})$. Concretely, following the discussions in \S\ref{SubsecIntroStr} and \S\ref{SubsecBasicNL}, and preparing the non-linear analysis in \S\ref{SubsecNLProof}, we work with the formulation
\begin{gather}
\label{EqCDLinE}
  D_g(\Ric+\Lambda)(\gdot) - \tdel^* D_g\Ups^E(\gdot) = 2 D_{g,d A}T(\gdot,d\Adot), \\
\label{EqCDLinM}
  D_{g,A}(\delta_{(\cdot)} d(\cdot))(\gdot,\Adot) - \td\Ups^M(g,\Adot) = 0,
\end{gather}
with gauge terms defined by \eqref{EqKNdSIniGaugeE}--\eqref{EqKNdSIniGaugeM}, where $\tdel^*$ and $\td$ are modifications of $\delta_g^*$ and $d$ as discussed around \eqref{EqBasicNLtdel} and \eqref{EqBasicNLtd}; we will make concrete choices momentarily. (We remark that the gauge term in the second equation arises by linearizing $\Ups^M(g,A-A_{b_0})$ around $(g,A)=(g_{b_0},A_{b_0})$, see also \eqref{EqKNdSIniLinGauge}; this is the type of gauge term studied in \S\ref{SubsecKNdSIni}.) The constraint propagation equations for $\Ups^M$ and $\Ups^E$ are then
\begin{gather}
\label{EqCDPropM}
  \delta_g\td(\Ups^M(g,\Adot)) = 0, \\
\label{EqCDPropE}
  \delta_g G_g\tdel^*(D_g\Ups^E(\gdot)) = 0,
\end{gather}
cf.\ \eqref{EqBasicLinGaugePropMod}. We recall here that the first equation follows directly from \eqref{EqCDLinM}, hence it is \emph{decoupled} from \eqref{EqCDLinE}, while the second equation only follows from \eqref{EqCDLinE} once $\td\Ups^M(g,\Adot)=0$ is established.

As motivated in the introduction, we wish to show that non-decaying (generalized) mode solutions of the system \eqref{EqCDLinE}--\eqref{EqCDLinM} are in fact mode solutions of the \emph{ungauged} system \eqref{EqBasicLinEinsteinMaxwell}, that is, with the $\Ups^E$ and $\Ups^M$ terms absent; we argue now that this follows once we establish the constraint damping property, i.e.\ the absence of resonances in the closed upper half plane, for the operators $\delta_g\td$ and $\delta_g G_g\tdel^*$ \emph{separately}. Indeed, if $(\gdot,\Adot)$ is a generalized mode solution of the above gauge-fixed system, then mode stability for the operator $\delta_g\td$, applied to equation \eqref{EqCDPropM}, will imply $\Ups^M(g,\Adot)\equiv 0$, since this is a non-decaying solution to an equation whose solutions are exponentially decaying. But then $D_{g,A}(\delta_{(\cdot)} d(\cdot))(\gdot,\Adot)\equiv 0$ as well, and thus \eqref{EqCDPropE} holds; by mode stability for the operator $\delta_g G_g\tdel^*$, we find $D_g\Ups^E(\gdot)\equiv 0$, and thus $(\gdot,\Adot)$ solves the linearized \emph{ungauged} Einstein--Maxwell system, as claimed.

Thus, we merely need to establish mode stability separately for the operators appearing in \eqref{EqCDPropM} and \eqref{EqCDPropE}. In \cite[\S8]{HintzVasyKdSStability}, Vasy and the author proved:

\begin{thm}
\label{ThmCDE}
  See \cite[Theorem~8.1]{HintzVasyKdSStability}. Let $t_*$ denote the timelike function on RNdS constructed in Lemma~\ref{LemmaKNdS0Ext}, and define
  \[
    \tdel^*u = \delta_g^*u + \gamma_1\,dt_*\cdot u - \gamma_2 u(\nabla t_*)g,\quad \gamma_1,\gamma_2\in\R.
  \]
  Then there exist parameters $\gamma_1,\gamma_2>0$ and a constant $\alpha>0$ such that all resonances $\sigma$ of the constraint propagation operator $\delta_g G_g\tdel^*$ satisfy $\Im\sigma<-\alpha$.
\end{thm}
\begin{proof}
  The proof in the reference deals with the uncharged Schwarzschild--de~Sitter case, but works without any modifications in the RNdS setting as well: indeed, the RNdS metric, as given in Lemma~\ref{LemmaKNdS0Ext}, has the same form as the Schwarzschild--de~Sitter metric in the form \cite[equations~(8.5)--(8.6)]{HintzVasyKdSStability}.
\end{proof}

\emph{We continue to use the form of the metric given in Lemma~\ref{LemmaKNdS0Ext}.} We now establish mode stability for $\delta_g\td$ for a suitable modification of $\td$.

\begin{thm}
\label{ThmCDM}
  With $t_*$ as in Lemma~\ref{LemmaKNdS0Ext}, define
  \[
    \td u = d u + \gamma_3\,u\,dt_*,\quad \gamma_3\in\R
  \]
  for $u\in\CI(M^\circ)$. Then there exists a parameter $\gamma_3>0$ and a constant $\alpha>0$ such that all resonances $\sigma$ of the constraint propagation operator $\delta_g\td$ satisfy $\Im\sigma<-\alpha$.
\end{thm}

We give two proofs: the first is semiclassical --- thus, we show that $\gamma_3\gg 0$ works --- following the arguments in \cite[\S8]{HintzVasyKdSStability} in the present, much simplified setting; the second is perturbative, similar to the one used in \cite[Lemma~3.5]{HintzVasySemilinear}, in which we compute the behavior of the simple resonance at $0$ for $\delta_g d$ upon increasing $\gamma_3$ slightly above zero.

\begin{proof}[First proof]
  Write $\gamma_3=h^{-1}$, then
  \[
    \cP_h u := h^2\delta_g\td u = (h^2\Box_g - i L_h)u,\quad L_h u = G(-i h d u,dt_*) + i h(\Box_g t_*)u.
  \]
  The semiclassical principal symbol of $L_h$ is $\sigmabh(L_h)(\zeta)=G(dt_*,\zeta)$, which is positive, resp.\ negative, if $\zeta$ is future, resp.\ past, causal, and in particular elliptic in the causal double cone $\{\zeta\in\Tb^*M\setminus o\colon G(\zeta)\geq 0\}$. Since the semiclassical characteristic set of $h^2\Box_g\in\Diffbh^2(M)$ is $G^{-1}(\zeta)\subset\rcTbdual M$, this shows that $\cP_h$ is semiclassically elliptic at $\Tb^*M\setminus o$, i.e.\ away from the zero section and away from fiber infinity. The high frequency analysis in \cite[\S8.2]{HintzVasyKdSStability} applies directly to the operator $\cP_h$, giving semiclassical microlocal propagation estimates at real principal type points (with complex absorption $L_h$), radial points and at the trapped set in semiclassical weighted b-Sobolev spaces $\Hbh^{s,\rho}(M)$ for any fixed regularity $s$ and weight $\rho$ for sufficiently small $h>0$.

  The low frequency analysis is vastly simplified compared to \cite[\S8.3]{HintzVasyKdSStability} since $\cP_h$ is a scalar operator; it thus suffices to prove the analogues of \cite[Lemmas~8.17 and 8.18]{HintzVasyKdSStability}, since the arguments following these in \cite[\S8.3]{HintzVasyKdSStability} then apply directly to $\cP_h$. To do so, we write
  \[
    L_h = M_{t_*}\,h D_{t_*} + M_r\,h D_r + i h S
  \]
  similarly to \cite[Lemma~8.13]{HintzVasyKdSStability}, where now
  \[
    M_{t_*} = c^2, \quad M_r = -\nu, \quad S=\Box_g t_*=\nu'+2r^{-1}\nu,
  \]
  as follows from Lemma~\ref{LemmaKNdS0Ext} and \cite[Lemma~8.4]{HintzVasyKdSStability}. Thus indeed $M_{t_*}>0$, and $\pm M_r=\mp\nu>0$ is positive for $\pm(r-r_c)>0$, proving the analogue of \cite[Lemma~8.17]{HintzVasyKdSStability} in the present setting. Moreover, fixing the volume form $dt_*\,dr\,d\slg$ and defining
  \[
    \ell' = \frac{1}{2ih}(L_h-L_h^*) = \frac{1}{2}(\pa_r M_r + S + S^*) = \frac{1}{2}(3\nu'+8 r^{-1}\nu)
  \]
  analogously to \cite[equation~(8.18)]{HintzVasyKdSStability}, we have $\ell'<0$ near the critical point $r=r_c$ of $\nabla t_*$ (where $\nu=0$), establishing \cite[Lemma~8.18]{HintzVasyKdSStability} in the present context.

  Energy estimates near hypersurfaces, discussed in \cite[\S8.4]{HintzVasyKdSStability}, are standard in the present, scalar setting. Finally, the global estimates obtained in \cite[\S8.5]{HintzVasyKdSStability} hold for our scalar operator $\cP_h$ as well, and the absence of resonances in a half space including the closed upper half plane follows as in the reference.
\end{proof}

\begin{proof}[Second proof]
  Let us write $\td_\gamma := d + \gamma\,dt_*$ and $P_\gamma:=\delta_g\td_\gamma$. Then $P_0=\Box_g$ is the scalar wave operator on a non-degenerate RNdS spacetime; the only resonance of $P_0$ in a half space $\Im\sigma>-2\wt\alpha$, with $\wt\alpha>0$ small, is a simple rank $1$ resonance at $\sigma(0)=0$, see \cite[Lemma~2.14]{HintzVasyCauchyHorizon} and the references given there. By general perturbation results, see \cite[Proposition~5.11]{HintzVasyKdSStability}, the operator $P_\gamma$ for small $\gamma$ therefore has only one simple rank $1$ resonance in $\Im\sigma>-\wt\alpha$; denoting its location by $\sigma(\gamma)$, the function $\sigma(\cdot)$ is analytic near $0$, and one can moreover choose a resonant state $\phi(\gamma)\in\CI(Y)$, with $\phi(0)=1$, depending analytically on $\gamma$. With $\psi=1_{[r_-,r_+]}$ denoting a dual resonant state of $\Box_g$ at $0$, so $\wh{P_0}(0)^*\psi=0$, we then have
  \begin{equation}
  \label{EqCDPerturb}
  \begin{split}
    0 &= \Big\la\frac{d}{d\gamma}\wh{P_\gamma}(\sigma(\gamma))(\phi(\gamma))\big|_{\gamma=0}, \psi\Big\ra \\
      &= \sigma'(0)\Big\la\frac{d}{d\sigma}\wh{P_0}(\sigma)\big|_{\sigma=0}(1),\psi\Big\ra + \Big\la \frac{d}{d\gamma}\wh{P_\gamma}\big|_{\gamma=0}(0)(1),\psi\Big\ra,
  \end{split}
  \end{equation}
  where we integrate over $t_*=0$ with respect to the volume density $r^2\,dr\,|d\slg|$. Now $\wh{P_\gamma}(\sigma)=\wh{\Box_g}(\sigma)+\gamma(\Box t_*)-\gamma\la d(\cdot),dt_*\ra$ and
  \[
    \frac{d}{d\sigma}\wh{\Box_g}(\sigma)\big|_{\sigma=0}1=\frac{d}{d\sigma}(e^{i\sigma t_*}\Box_g e^{-i\sigma t_*})\big|_{\sigma=0} = -i\Box_g t_*.
  \]
  Since $\Box_g t_*=r^{-2}(r^2\nu)'$ and $\nu(r_\pm)=\mp 1$ in the notation of Lemma~\ref{LemmaKNdS0Ext} (see \cite[Lemma~8.4]{HintzVasyKdSStability}), so
  \[
    \int_{[r_-,r_+]\times\Sph^2}(\Box_g t_*)\,r^2\,dr\,|d\slg|=-4\pi(r_+^2+r_-^2) \neq 0,
  \]
  we conclude that \eqref{EqCDPerturb} implies $0=-i\sigma'(0)+1$, hence $\sigma'(0)=-i$. Therefore, for small $\gamma>0$, we have $\Im\sigma(\gamma)<0$, and Theorem~\ref{ThmCDM} follows.
\end{proof}

We summarize the above discussion as follows:
\begin{thm}
\label{ThmCD}
  Fix $\tdel^*$ and $\td$ as in Theorems~\ref{ThmCDE} and \ref{ThmCDM}. Then any non-decaying generalized mode solution $(\gdot,\Adot)$ (or in fact any generalized mode solution with temporal frequency $\sigma\in\C$ with $\Im\sigma\geq-\alpha$) of the linearized gauge-fixed Einstein--Maxwell system \eqref{EqCDLinE}--\eqref{EqCDLinM} on $\Omega$ satisfies $D_g\Ups^E(\gdot)=0$ and $\Ups^M(g,\Adot)=0$, and hence is a solution of the ungauged system \eqref{EqBasicLinEinsteinMaxwell}.
\end{thm}

\section{High frequency analysis of the linearized equation}
\label{SecESG}

We now analyze the high frequency properties of the linearized gauged Einstein--Maxwell system \eqref{EqCDLinE}--\eqref{EqCDLinM}, which is the linearization around
\[
  (g,A)=(g_{b_0},A_{b_0})
\]
of the full non-linear equation we will use in our proof of non-linear stability. Thus, let us define the linear, stationary second order differential operator
\begin{align*}
  L(\gdot,\Adot) &= \Bigl( 2\bigl(D_g(\Ric+\Lambda)(\gdot) - \tdel^*D_g\Ups^E(\gdot) - 2 D_{g,d A}T(\gdot,d\Adot)\bigr), \\
    &\qquad\qquad D_{g,A}(\delta_{(\cdot)}d(\cdot))(\gdot,\Adot) - \td \Ups^M(g,\Adot) \Bigr),
\end{align*}
acting on sections $(\gdot,\Adot)$ of the bundle $S^2\,\Tb^*_\Omega M\oplus\Tb^*_{\Omega}M$; here, as before, $\Ups^E$ and $\Ups^M$ are given by the expressions \eqref{EqKNdSIniGaugeE}--\eqref{EqKNdSIniGaugeM}, and $T$ is defined in \eqref{EqBasicDerEMTensor}. Furthermore, $\tdel^*$ and $\td$ are given by Theorems~\ref{ThmCDE} and \ref{ThmCDM}, so
\begin{equation}
\label{EqESGTildeOps}
  (\tdel^*-\delta_g^*)u = \gamma_1\,dt_*\cdot u - \gamma_2 u(\nabla t_*)g, \qquad (\td-d)v=\gamma_3\,v\,dt_*,
\end{equation}
for suitable $\gamma_1,\gamma_2,\gamma_3>0$, where $u$ a 1-form and $v$ a scalar function. We prove:

\begin{thm}
\label{ThmESG}
  (Cf.\ \cite[Theorem~4.4]{HintzVasyKdSStability}.) The operator $L$ satisfies the condition \eqref{EqLinAnaTrap} on the subprincipal operator at the trapped set. Furthermore, the quantity $\wh\beta$ in \eqref{EqLinAnaWhBeta}, related to the threshold regularity at the radial sets, satisfies $\wh\beta\geq 0$. Thus, there exists $\alpha>0$ such that for $s\geq 1$, the operator $\wh L(\sigma)$ satisfies the high energy estimates \eqref{EqLinAnaMeroEst} in $\Im\sigma\geq-\alpha$, and moreover $L$ has no resonances $\sigma$ with $-\alpha\leq\Im\sigma<0$.
\end{thm}
\begin{proof}
  The assertions at the trapped set and at the horizons will be verified in the remainder of this section. We recall that then $L$ has only finitely many resonances in $\Im\sigma\geq-\alpha$ for $\alpha$ as in Theorem~\ref{ThmLinAnaMero}; thus, reducing $\alpha$ if necessary guarantees that all resonances in $\Im\sigma\geq-\alpha$ in fact satisfy $\Im\sigma\geq 0$. The assumption $s\geq 1$ ensures that the above-threshold regularity condition on $s$ in Theorem~\ref{ThmLinAnaMero} is satisfied (again reducing $\alpha>0$ if necessary).
\end{proof}

The point of this theorem is that it implies that solutions of initial value problems for $L$ with data in the space $D^{s,\alpha}(\Omega;S^2\,\Tb^*_\Omega M\oplus\Tb^*_\Omega M)$ have an asymptotic expansion into finitely many (generalized) mode solutions up to an exponentially decaying remainder term in $\Hbext^{s,\alpha}(\Omega;S^2\,\Tb^*_\Omega M\oplus\Tb^*_\Omega M)$; see the discussion around \eqref{EqLinAnaAsympExp}.

The key calculation, which we will describe in \S\ref{SubsecESGTrap}, is that of the subprincipal operator $S_\sub(L)$ of $L$ at the trapped set; this is purely symbolic, and correspondingly the calculation of the main quantity, the spectrum of a certain endomorphism of the bundle $S^2\Tb^*_\Omega M\oplus\Tb^*_\Omega M$ (pulled back to $\Tb^*_\Omega M$ via the projection $\pi\colon\Tb^*_\Omega M\to\Omega$) associated with $S_\sub(L)$, is straightforward, even if algebraically somewhat lengthy as it involves high-dimensional linear algebra. The regularity assumption in Theorem~\ref{ThmESG} comes from a calculation of the threshold regularity at the radial sets at the event and cosmological horizons, see \S\ref{SubsecESGRad}.

We compute the form of $L$ in a bit more detail. For the first component, we use \eqref{EqBasicNLLinRic} as well as \eqref{EqBasicNLLinUpsE} together with $\sE_g\equiv 0$ (recalling that the background metric, called $t$ there, is equal to $g=g_{b_0}$, so the covariant derivatives defining $C_{\mu\nu}^\kappa$ there all vanish). In addition, writing $F=d A$ and $\Fdot=d\Adot$, we note that \eqref{EqMSDT} shows that $T(g,F)$ depends only on $g$ but no derivatives, hence $D_g T(\gdot,F)$ is of order $0$ in $\gdot$; on the other hand,
\[
  D_F T(g,\Fdot) = -2\tr_g^{24}(F\otimes_s\Fdot) + G(F,\Fdot)\,g.
\]
For the second component of $L$, we recall $D_A(\delta_g d(\cdot))(\Adot)=\delta_g d\Adot$; moreover, we write the last term in \eqref{EqMSLinDelta} as
\[
  \gdot_\nu{}^{\ell;\mu}-\gdot_\nu{}^{\mu;\ell} = (d^\nabla\gdot)_\nu{}^{\mu\ell}
\]
using the covariant exterior derivative of $\gdot$ viewed as a $T^*M^\circ$-valued 1-form, so
\[
  S^2 T^*M^\circ \hra T^*M^\circ \otimes T^*M^\circ \xra{d^\nabla} T^*M^\circ \otimes \Lambda^2 T^*M^\circ.
\]
Lastly, we simply have $\Ups^M(g,\Adot)=-\delta_g\Adot$ for our fixed metric $g=g_{b_0}$. Thus, defining the operators
\begin{gather*}
  L_{12}(\Adot) = 4\tr_g^{24}(F\otimes_s d\Adot) - 2 G(F,d\Adot)\,g, \\
  L_{21}(\gdot) = - \tr_g^{12}(\delta_g G_g\gdot\otimes F) + \frac{1}{2}\tr_g^{24}\tr_g^{35}(d^\nabla\gdot\otimes F),
\end{gather*}
we have
\begin{equation}
\label{EqESGL}
\begin{split}
  L(\gdot,\Adot) &\equiv \Bigl( \Box_g\gdot + 2(\tdel^*-\delta_g^*)\delta_g G_g\gdot + L_{12}(\Adot), \\
    &\qquad \Box_g\Adot + (\td-d)\delta_g\Adot + L_{21}(\gdot) \Bigr),
\end{split}
\end{equation}
modulo terms of order $0$, which are sub-subprincipal and thus irrelevant for the present high energy discussion, which only involves principal and subprincipal terms at the normally hyperbolic trapping and the radial set.

\subsection{Basic computations on the RNdS spacetime}
\label{SubsecESGComp}

We compute the explicit form of (parts of) operators appearing in the expression \eqref{EqESGL}.

\subsubsection{Warped products}

We first consider a warped product metric
\begin{equation}
\label{EqESGCompMet}
  g = \alpha^2\,dt^2 - h
\end{equation}
on $\cM=\R_t\times\cX_x$, with $\alpha\in\CI(\cX)$ positive, and $h=h(x,dx)$ a smooth (Riemannian) metric on the manifold $\cX$. We let
\[
  e_0:=\alpha^{-1}\pa_t,\quad e^0:=\alpha\,dt,\quad \Omega=\log\alpha,
\]
and denote the Levi-Civita connection of $h$ by $\nabla^h$. From \cite[\S6.1]{HintzVasyKdSStability}, we then recall for $v\in\CI(\cM,T\cX)$ and $w\in\CI(\cM,T^*\cX)$ the covariant derivatives
\begin{equation}
\label{EqESGCompConn1}
\begin{gathered}
  \nabla_{e_0}e^0=-d\Omega,\quad \nabla_{e_0}w=e_0 w-w(\nabla^h\Omega)e^0, \\
  \nabla_v e^0=0,\quad \nabla_v w=\nabla^h_v w.
\end{gathered}
\end{equation}
Furthermore, we define bundle splittings
\begin{equation}
\label{EqESGCompBundle1}
\begin{gathered}
  T^*\cM=W_N\oplus W_T,\quad \Lambda^2 T^*\cM=W_{NT}\oplus W_{TT}, \\
  S^2T^*\cM=V_{NN}\oplus V_{NT}\oplus V_T, \\
\end{gathered}
\end{equation}
where $W_N=\la e^0\ra$ is canonically trivial, so $W_N\cong\ulR$ is the trivial real rank $1$ bundle over $\cM$, and $W_T=T^*\cX$; the summands of the 2-form bundle are $W_{NT}=e^0\wedge W_T\cong W_T$ and $W_{TT}=\Lambda^2 T^*\cX$; and lastly
\[
  V_{NN}=\la e^0e^0\ra\cong\ulR,\quad V_{NT}=2e^0 W_T\cong W_T,\quad V_T=S^2T^*\cX.
\]
Here we write $uv=u\cdot v=u\otimes_s v=\frac{1}{2}(u\otimes v+v\otimes u)$ for the symmetrized product. For the second term in $L_{21}$, we also introduce the induced splitting
\begin{equation}
\label{EqESGCompBundle1T}
  T^*\cM\otimes\Lambda^2 T^*\cM = (W_N\otimes W_{NT}) \oplus (W_N\otimes W_{TT}) \oplus (W_T\otimes W_{NT}) \oplus (W_T\otimes W_{TT}),
\end{equation}
identifying $W_N\otimes W_\bullet\cong W_\bullet$ for $\bullet=NT,TT$, and $W_T\otimes W_{NT}\cong T^*\cX\otimes T^*\cX$. We decompose the covector $\zeta\in T^*\cM$, at which all symbols below are implicitly evaluated, according to \eqref{EqESGCompBundle1} as
\begin{equation}
\label{EqESGCompCovec}
  T^*\cM \ni \zeta = -\wt\sigma e^0 + \omega = \begin{pmatrix} -\wt\sigma \\ \omega \end{pmatrix}, \quad \wt\sigma\in\R,\ \omega\in T^*\cX;
\end{equation}
in particular, $\sigma_1(e_0)=-i\wt\sigma$, so the exterior derivative and the codifferential (divergence) acting on 1-forms have symbols
\begin{equation}
\label{EqESGCompDDel}
  \sigma_1(d) = i\zeta\wedge = i\begin{pmatrix} -\omega & -\wt\sigma \\ 0 & \omega\wedge \end{pmatrix},
  \quad
  \sigma_1(\delta_g) = -i \iota_\zeta = i\begin{pmatrix} \wt\sigma & \iota_\omega \end{pmatrix},
\end{equation}
where $\iota_\zeta:=\iota_{G(\zeta,\cdot)}$, while $\iota_\omega:=\iota_{H(\omega,\cdot)}$, i.e.\ for 1-forms on $\cX$, we use the musical isomorphism induced by $h$, so $\iota_\omega\omega':=H(\omega,\omega')$ for a 1-form $\omega'$ on $\cX$. For the symmetric gradient on 1-forms, $\sigma_1(\delta_g^*)=i\zeta\otimes_s(\cdot)$, and the divergence acting on symmetric 2-tensors, $\sigma_1(\delta_g)=-i\iota_\zeta$, we have
\[
  \sigma_1(\delta_g^*) = i\begin{pmatrix} -\wt\sigma & 0 \\ \frac{1}{2}\omega & -\frac{1}{2}\wt\sigma \\ 0 & \omega\otimes_s(\cdot) \end{pmatrix},
  \quad
  \sigma_1(\delta_g) = i\begin{pmatrix} \wt\sigma & \iota_\omega & 0 \\ 0 & \wt\sigma & \iota_\omega \end{pmatrix}.
\]
The metric and the trace reversal operator take the form
\[
  g = \begin{pmatrix} 1 \\ 0 \\ -h \end{pmatrix},
  \quad
  G_g = \begin{pmatrix}
          \frac{1}{2}  & 0 & \frac{1}{2}\tr_h \\
          0            & 1 & 0 \\
          \frac{1}{2}h & 0 & 1-\frac{1}{2}h\tr_h
        \end{pmatrix},
\]
thus
\begin{equation}
\label{EqESGCompDelgGg}
  2\sigma_1(\delta_g G_g)=
   i\begin{pmatrix}
      \wt\sigma & 2\iota_\omega & \wt\sigma\tr_h \\
      \omega & 2\wt\sigma       & 2\iota_\omega-\omega\tr_h
    \end{pmatrix}.
\end{equation}
Next, for 1-forms $a,b,c,d$, we have
\begin{align*}
  \tr_g^{24}((a\wedge b)\otimes_s(c\wedge d)) &= G(a,c)b\otimes_s d + G(b,d)a\otimes_s c \\
   &\qquad - G(a,d)b\otimes_s c - G(b,c)a\otimes_s d;
\end{align*}
therefore, for the 2-form
\begin{equation}
\label{EqESGCompF1}
  F=e^0\wedge\rho = \begin{pmatrix} \rho \\ 0 \end{pmatrix}, \quad \rho\in\CI(\cM,T^*\cX),
\end{equation}
on $\cM$, where the last expression is in terms of the decomposition \eqref{EqESGCompBundle1}, we have
\begin{equation}
\label{EqESGCompTrg24}
  \tr_g^{24}(F\otimes_s(\cdot))
    =\begin{pmatrix} 
       -\iota_\rho          & 0 \\
       0                    & \frac{1}{2}\iota_\rho \\
       \rho\otimes_s(\cdot) & 0
     \end{pmatrix},
  \quad
  G(F,\cdot) = \begin{pmatrix} -\iota_\rho & 0 \end{pmatrix}.
\end{equation}
on 2-forms. Acting on 1-forms, one computes
\begin{equation}
\label{EqESGCompTrg12}
  \tr_g^{12}((\cdot)\otimes F) = \begin{pmatrix} 0 & \iota_\rho \\ \rho & 0 \end{pmatrix}.
\end{equation}
Further, one computes the symbol of $d^\nabla\colon\CI(\cM;S^2T^*\cM)\to\CI(\cM;T^*\cM\otimes\Lambda^2 T^*\cM)$, so $\sigma_1(d^\nabla)(\zeta)(\psi\otimes\psi)=i\psi\otimes(\zeta\wedge\psi)$, in the decompositions \eqref{EqESGCompBundle1} and \eqref{EqESGCompBundle1T} to be
\[
  \sigma_1(d^\nabla)
   =i\begin{pmatrix}
       -\omega & -\wt\sigma            & 0 \\
       0       & \omega\wedge          & 0 \\
       0       & -(\cdot)\otimes\omega & -\wt\sigma \\
       0       & 0                     & i^{-1}\sigma_1(d^\nabla_\cX),
     \end{pmatrix},
\]
while
\[
  \tr_g^{24}\tr_g^{35}((\cdot)\otimes F)
    =\begin{pmatrix}
       -2\iota_\rho & 0 & 0                      & 0 \\
       0            & 0 & -2\Id\otimes\iota_\rho & 0
     \end{pmatrix}
\]
as a bundle map $T^*\cM\otimes\Lambda^2 T^*\cM\to T^*\cM$, therefore the composition is
\begin{equation}
\label{EqESGCompTrg24Trg35}
  \sigma_1(\tr_g^{24}\tr_g^{35}(d^\nabla(\cdot)\otimes F))
  = 2i
    \begin{pmatrix}
      \iota_\rho\omega & \wt\sigma\iota_\rho & 0 \\
      0                & \iota_\rho\omega    & \wt\sigma\iota_\rho
    \end{pmatrix}.
\end{equation}
Finally, in our application, we will have $t_*=t-T$ with $T\in\CI(\cX)$, see \eqref{EqKNdS0CoordDef}, so that
\[
  dt_* = \begin{pmatrix} \alpha^{-1} \\ -d_\cX T \end{pmatrix},
    \quad
  \nabla t_*=\begin{pmatrix} \alpha^{-1} & \iota_{d T} \end{pmatrix},
\]
so in view of \eqref{EqESGTildeOps}
\begin{equation}
\label{EqESGCompDiff1}
  \tdel^*-\delta_g^*
  =\gamma_1
     \begin{pmatrix}
       \alpha^{-1}     & 0                      \\
       -\frac{1}{2}d T & \frac{1}{2}\alpha^{-1} \\
       0               & -d T\otimes_s(\cdot)
     \end{pmatrix}
   -\gamma_2
     \begin{pmatrix}
       \alpha^{-1}   & \iota_{d T} \\
       0             & 0 \\
       -\alpha^{-1}h & -h\iota_{d T}
     \end{pmatrix}.
\end{equation}

\subsubsection{Specialization of the spatial metric}

Next, we assume that the spatial metric $h$ in \eqref{EqESGCompMet} takes the form
\[
  h = \alpha^{-2}\,dr^2 + r^2\slg,\quad \alpha=\alpha(r),
\]
with $\slg$ a metric (independent of $r$) on an $\sln$-dimensional manifold $S$, $\sln\in\N_0$. All slashed operators are those induced by (the Levi-Civita connection of) $\slg$. We let
\[
  e_1:=\alpha\pa_r,\quad e^1=\alpha^{-1}\,dr.
\]
From \cite[\S6.2]{HintzVasyKdSStability}, we recall that for $v\in\CI(\cX,T S)$, $w\in\CI(X,T^*S)$,
\begin{equation}
\label{EqESGCompConn2}
\begin{gathered}
  \nabla^h_{e_1}e^1=0, \quad \nabla^h_{e_1}w=e_1 w-\alpha r^{-1}w, \\
  \nabla^h_v e^1=r\alpha\iota_v\slg, \quad \nabla^h_v w=\slnabla_v w - \alpha r^{-1}w(v)e^1,
\end{gathered}
\end{equation}
extending $\slg$ to an element of $S^2 T^*\cX$ by means of the orthogonal projection $T\cX\to\la e_1\ra^\perp\cong T S$. We also split the natural bundles on $\cX$ by writing
\begin{equation}
\label{EqESGCompBundle2}
\begin{gathered}
  T^*\cX = W^T_N\oplus W^T_T, \quad \Lambda^2 T^*\cX=W^T_{NT}\oplus W^T_{TT}, \\
  S^2 T^*\cX = V^T_{NN}\oplus V^T_{NT}\oplus V^T_T,
\end{gathered}
\end{equation}
where $W^T_N=\la e^1\ra\cong\ulR$ and $W^T_T=T^*S$, while $W^T_{NT}=e^1\wedge W^T_T\cong W^T_T$ and $W^T_{TT}=\Lambda^2 T^*S$, and
\[
  V^T_{NN}=\la e^1e^1\ra\cong\ulR, \quad V^T_{NT}=2 e^1 W^T_T \cong W^T_T, \quad V^T_T=S^2 T^*S.
\]
Thus,
\[
  d\Omega = \begin{pmatrix} \alpha' \\ 0 \end{pmatrix},
    \quad
  \nabla^h\Omega=\begin{pmatrix} \alpha' & 0 \end{pmatrix};
    \quad
  h = \begin{pmatrix} 1 \\ 0 \\ r^2\slg \end{pmatrix},
    \quad
  \tr_h = \begin{pmatrix} 1 & 0 & r^{-2}\sltr \end{pmatrix}.
\]
We write the covector $\omega\in T^*\cX$ in \eqref{EqESGCompCovec} as
\begin{equation}
\label{EqESGCompCovec2}
  \omega = \wt\xi e^1 + \eta = \begin{pmatrix} \wt\xi \\ \eta \end{pmatrix}, \quad \wt\xi\in\R,\ \eta\in T^*S
\end{equation}
in the splitting \eqref{EqESGCompBundle2}. Consequently, acting on the bundle $T^*\cX$, we have
\begin{equation}
\label{EqESGCompOmega}
  \omega\wedge(\cdot)
   =\begin{pmatrix}
      -\eta & \wt\xi \\
      0     & \eta\wedge
    \end{pmatrix},
  \quad
  \iota_\omega
   =\begin{pmatrix}
      \wt\xi & r^{-2}\iota_\eta
    \end{pmatrix},
  \quad
  \omega\otimes_s(\cdot)
   =\begin{pmatrix}
      \wt\xi          & 0 \\
      \frac{1}{2}\eta & \frac{1}{2}\wt\xi \\
      0               & \eta\otimes_s(\cdot)
    \end{pmatrix},
\end{equation}
where $\iota_\eta\eta'=\slG(\eta,\eta')$ on 1-forms on $S$. On symmetric 2-tensors, we have
\[
  \iota_\omega=
    \begin{pmatrix} 
      \wt\xi & r^{-2}\iota_\eta & 0 \\
      0      & \wt\xi           & r^{-2}\iota_\eta.
    \end{pmatrix}
\]
We also specialize the expression \eqref{EqESGCompF1} for the 2-form $F=e^0\wedge\rho$ by demanding
\[
  \rho = q\,e^1 = \begin{pmatrix} q \\ 0 \end{pmatrix},\quad q\in\CI(\cX);
\]
hence, on 1-forms,
\begin{equation}
\label{EqESGCompIotaRho1}
  \iota_\rho = \begin{pmatrix} q & 0 \end{pmatrix}, \quad \rho\otimes_s(\cdot)=\begin{pmatrix} q & 0 \\ 0 & \frac{1}{2}q \\ 0 & 0 \end{pmatrix}
\end{equation}
while on 2-forms, resp.\ symmetric 2-tensors,
\begin{equation}
\label{EqESGCompIotaRho2}
  \iota_\rho = \begin{pmatrix} 0 & 0 \\ q & 0 \end{pmatrix},\quad \tn{resp.}\quad \iota_\rho = \begin{pmatrix} q & 0 & 0 \\ 0 & q & 0 \end{pmatrix},
\end{equation}
in the respective splittings given by \eqref{EqESGCompBundle2}. In our application on the RNdS spacetime, we have $q=-Q_e\alpha^{-1}r^{-2}$, see \eqref{EqKNdS0Potential}.

Finally, if the function $T$ in $t_*=t-T$ is a function of $r$ only, then
\begin{equation}
\label{EqESGCompDiff2}
  d T=\begin{pmatrix} \alpha T' \\ 0 \end{pmatrix}, \quad \iota_{d T}=\begin{pmatrix} \alpha T' & 0 \end{pmatrix}.
\end{equation}

\subsection{Subprincipal operator at the trapped set}
\label{SubsecESGTrap}

We continue dropping the subscript `$b_0$' from the notation; all computations use the objects (metric, 4-potential, electromagnetic field) defined for the RNdS spacetime with mass $\bhm$ and electric charge $Q_e$.

We recall the location of the trapped set $\Gamma$ from \eqref{EqKNdSGeoTrapSet}; the dual coordinates $\sigma$ and $\xi$ there are related to $\wt\sigma$ in \eqref{EqESGCompCovec} and $\wt\xi$ in \eqref{EqESGCompCovec2} by
\begin{equation}
\label{EqESGTrapCovecRel}
  \sigma=\alpha\wt\sigma, \ \xi=\alpha^{-1}\wt\xi,
\end{equation}
so
\begin{equation}
\label{EqESGTrapGamma}
  r=r_P,\quad \wt\sigma^2 = r^{-2}|\eta|^2,\quad \xi=0,
\end{equation}
at $\Gamma$, with $r_P$ the radius of the photon sphere; we have $\sigma\neq 0$ and $\eta\neq 0$ since $\Gamma$ by definition does not intersect $o\subset T^*\cM$. Let $\pi\colon T^*\cM\to\cM$ denote the projection. At $\Gamma$ then, using \eqref{EqKNdSGeoTrapHam}, we can compute the subprincipal operator $S_\sub(\Box_g)=-i\nabla^{\pi^*T^*\cM}_{\ham_G}$ of the wave operator on 1-forms at $\Gamma$, where $\alpha'=r^{-1}\alpha$, using \eqref{EqESGCompConn1} and \eqref{EqESGCompConn2} to be
\begin{equation}
\label{EqESGTrapSubpr}
\begin{split}
  S_\sub(\Box_g) &= -2\alpha^{-1}\wt\sigma D_t + i r^{-2}\nabla^{\pi^*T^*\cM}_{\ham_{|\eta|^2}} + S, \\
  &\hspace{5em} S = i\alpha r^{-1} \begin{pmatrix} 0 & -2\wt\sigma & 0 \\ -2\wt\sigma & 0 & -2r^{-2}\iota_\eta \\ 0 & 2\eta & 0 \end{pmatrix},
\end{split}
\end{equation}
acting on sections of $\pi^*T^*\cM\to T^*\cM$, in the combination of the splittings \eqref{EqESGCompBundle1} and \eqref{EqESGCompBundle2}, i.e.\ $T^*\cM=W_N\oplus(W^T_N\oplus W^T_T)$, pulled back to a vector bundle over $T^*\cM$ via $\pi$. (See also \cite[Proposition~4.7]{HintzPsdoInner}.) We note that we have
\[
  \nabla^{\pi^*T^*\cM}_{\ham_{|\eta|^2}} = \begin{pmatrix}\ham_{|\eta|^2} & 0 & 0 \\ 0 & \ham_{|\eta|^2} & 0 \\ 0 & 0 & \nabla^{\pi_{\Sph^2}^*T^*\Sph^2}_{\ham_{|\eta|^2}} \end{pmatrix}.
\]
Moreover, $D_t$ and $\nabla^{\pi^*T^*\cM}_{\ham_{|\eta|^2}}$ act in a canonical way on sections of $\pi^*T^*\cM$ due to the stationary and spherically symmetric nature of $\cM$; thus, the 0-th order part $S$ of $S_\sub(\Box_g)$ is well-defined.

As in \cite[\S9.1]{HintzVasyKdSStability}, it is convenient to perform a further `microlocal' decomposition of $\pi^*T^*\cM\to T^*\cM$ over the trapped set $\Gamma$. Thus, for $\eta\in T^*\Sph^2\setminus o$, we define $\eta^\perp:=\slstar\eta$, further
\[
  \wh\eta := \wt\sigma^{-1}\eta,\quad \wh\eta^\perp := \wt\sigma^{-1}\eta^\perp,
\]
and write
\begin{equation}
\label{EqESGTrapMlSplitSph}
\begin{gathered}
  (\pi_{\Sph^2}^*T^*\Sph^2)_\eta = \la \wh\eta \ra \oplus \la\wh\eta^\perp\ra, \quad (\pi_{\Sph^2}^*\Lambda^2 T^*\Sph^2)_\eta = \la\wh\eta\wedge\wh\eta^\perp\ra, \\
  (\pi_{\Sph^2}^*S^2T^*\Sph^2)_\eta = \la \wh\eta\wh\eta \ra \oplus \la 2\wh\eta\wh\eta^\perp\ra \oplus \la\wh\eta^\perp\wh\eta^\perp\ra,
\end{gathered}
\end{equation}
inducing the decompositions
\begin{equation}
\label{EqESGTrapMlSplit}
\begin{split}
  (\pi^* T^*\cM)_\zeta &= \la e^0\ra \oplus \Bigl(\la e^1\ra \oplus \bigl(\la \wh\eta \ra \oplus \la\wh\eta^\perp\ra\bigr)\Bigr), \\
  (\pi^*\Lambda^2 T^*\cM)_\zeta &= \Bigl(\la e^0\wedge e^1\ra \oplus\bigl(\la e^0\wedge\wh\eta \ra \oplus\la e^0\wedge\wh\eta^\perp\ra\bigr)\Bigr) \\
    &\hspace{6em} \oplus\bigl(\la e^1\wedge\wh\eta\ra \oplus \la e^1\wedge\wh\eta^\perp\ra \bigr) \oplus \la\wh\eta\wedge\wh\eta^\perp\ra, \\
  (\pi^*S^2 T^*\cM)_\zeta &= \Bigl(\la e^0e^0\ra \oplus \bigl(\la 2e^0e^1\ra \oplus \bigl( \la 2e^0\wh\eta\ra \oplus \la 2e^0\wh\eta^\perp\ra \bigr)\bigr)\Bigr) \\
    &\hspace{6em} \oplus \Bigl(\la e^1e^1\ra \oplus \bigl(\la 2e^1\wh\eta\ra \oplus \la 2e^1\wh\eta^\perp\ra \bigr)\Bigr) \\
    &\hspace{6em} \oplus \bigl(\la\wh\eta\wh\eta\ra \oplus \la 2\wh\eta\wh\eta^\perp\ra \oplus \la\wh\eta^\perp\wh\eta^\perp\ra \bigr),
\end{split}
\end{equation}
into trivial rank $1$ bundles over $T^*\cM$ near $\Gamma$, where the placing of the parentheses reflects the successive refinements \eqref{EqESGCompBundle1}, \eqref{EqESGCompBundle2} and \eqref{EqESGTrapMlSplitSph} of the various bundles. In terms of \eqref{EqESGTrapMlSplitSph}, writing $\slg=|\eta|^{-2}\eta\eta + |\eta^\perp|^{-2}\eta^\perp\eta^\perp$, and using \eqref{EqESGTrapGamma}, we have
\[
  r^2\slg=\begin{pmatrix}1\\0\\1\end{pmatrix},\quad r^{-2}\sltr=\begin{pmatrix}1&0&1\end{pmatrix},\quad
  \eta=\begin{pmatrix}\wt\sigma\\0\end{pmatrix},\quad r^{-2}\iota_\eta=\begin{pmatrix}\wt\sigma & 0 \end{pmatrix};
\]
further, acting on 1-forms,
\[
  \eta\otimes_s(\cdot)=\begin{pmatrix} \wt\sigma & 0 \\ 0 & \frac{1}{2}\wt\sigma \\ 0 & 0 \end{pmatrix},
  \quad
  \eta\wedge(\cdot)=\begin{pmatrix} 0 & \wt\sigma \end{pmatrix},
\]
while acting on symmetric 2-tensors,
\[
  r^{-2}\iota_\eta = \begin{pmatrix} \wt\sigma & 0 & 0 \\ 0 & \wt\sigma & 0 \end{pmatrix}.
\]
At $\Gamma$, we have $\zeta=-\wt\sigma e^0+\omega=(-\wt\sigma,0,\wt\sigma,0)^t$.

We note that $S$ in \eqref{EqESGTrapSubpr} in the splitting \eqref{EqESGTrapMlSplit} is equal to
\[
  S=i r^{-1}\sigma\begin{pmatrix}0&-2&0&0\\-2&0&-2&0\\0&2&0&0\\0&0&0&0\end{pmatrix}.
\]
Putting together the computations from \S\ref{SubsecESGComp}, we now compute the 0-th order part $S_L$ of $S_\sub(L)$ at $\Gamma$, with $L$ given in \eqref{EqESGL}. Acting on the bundle $S^2T^*\cM\oplus T^*\cM$, we can write $S_L$ as a $2\times 2$ block matrix,
\[
  (i r^{-1}\sigma)^{-1} S_L = \begin{pmatrix} S_L^{11} & S_L^{12} \\ S_L^{21} & S_L^{22} \end{pmatrix}.
\]
The $(1,1)$ block, corresponding to the first two terms in the first component of $L$, was already computed in \cite[\S9.1]{HintzVasyKdSStability};\footnote{There is an inconsequential typo in the matrix multiplying $\gamma_2''$ in \cite{HintzVasyKdSStability}; the matrix given here is the correct one.} it equals
\begin{align*}
  S_L^{11}&
  =\openbigpmatrix{2pt}
     0  & -4 & 0  & 0  & 0  & 0  & 0  & 0  & 0  & 0 \\
     -2 & 0  & -2 & 0  & -2 & 0  & 0  & 0  & 0  & 0 \\
     0  & 2  & 0  & 0  & 0  & -2 & 0  & 0  & 0  & 0 \\
     0  & 0  & 0  & 0  & 0  & 0  & -2 & 0  & 0  & 0 \\
     0  & -4 & 0  & 0  & 0  & -4 & 0  & 0  & 0  & 0 \\
     0  & 0  & -2 & 0  & 2  & 0  & 0  & -2 & 0  & 0 \\
     0  & 0  & 0  & -2 & 0  & 0  & 0  & 0  & -2 & 0 \\
     0  & 0  & 0  & 0  & 0  & 4  & 0  & 0  & 0  & 0 \\
     0  & 0  & 0  & 0  & 0  & 0  & 2  & 0  & 0  & 0 \\
     0  & 0  & 0  & 0  & 0  & 0  & 0  & 0  & 0  & 0
   \closebigpmatrix \\
  &+\gamma_1'
   \openbigpmatrix{2.3pt}
     2 & 0 & 4 & 0 & 2  & 0 & 0 & 2 & 0 & 2  \\
     0 & 2 & 0 & 0 & 0  & 2 & 0 & 0 & 0 & 0  \\
     1 & 0 & 2 & 0 & -1 & 0 & 0 & 1 & 0 & -1 \\
     0 & 0 & 0 & 2 & 0  & 0 & 0 & 0 & 2 & 0  \\
     0 & 0 & 0 & 0 & 0  & 0 & 0 & 0 & 0 & 0  \\
     0 & 0 & 0 & 0 & 0  & 0 & 0 & 0 & 0 & 0  \\
     0 & 0 & 0 & 0 & 0  & 0 & 0 & 0 & 0 & 0  \\
     0 & 0 & 0 & 0 & 0  & 0 & 0 & 0 & 0 & 0  \\
     0 & 0 & 0 & 0 & 0  & 0 & 0 & 0 & 0 & 0  \\
     0 & 0 & 0 & 0 & 0  & 0 & 0 & 0 & 0 & 0
   \closebigpmatrix
  -\gamma_1''
   \openbigpmatrix{2.3pt}
     0 & 0 & 0 & 0 & 0  & 0 & 0 & 0 & 0 & 0  \\
     1 & 0 & 2 & 0 & 1  & 0 & 0 & 1 & 0 & 1  \\
     0 & 0 & 0 & 0 & 0  & 0 & 0 & 0 & 0 & 0  \\
     0 & 0 & 0 & 0 & 0  & 0 & 0 & 0 & 0 & 0  \\
     0 & 4 & 0 & 0 & 0  & 4 & 0 & 0 & 0 & 0  \\
     1 & 0 & 2 & 0 & -1 & 0 & 0 & 1 & 0 & -1 \\
     0 & 0 & 0 & 2 & 0  & 0 & 0 & 0 & 2 & 0  \\
     0 & 0 & 0 & 0 & 0  & 0 & 0 & 0 & 0 & 0  \\
     0 & 0 & 0 & 0 & 0  & 0 & 0 & 0 & 0 & 0  \\
     0 & 0 & 0 & 0 & 0  & 0 & 0 & 0 & 0 & 0
   \closebigpmatrix \\
 &+\gamma_2'
   \openbigpmatrix{3pt}
     -1 & 0 & -2 & 0 & -1 & 0 & 0 & -1 & 0 & -1 \\
     0  & 0 & 0  & 0 & 0  & 0 & 0 & 0  & 0 & 0  \\
     0  & 0 & 0  & 0 & 0  & 0 & 0 & 0  & 0 & 0  \\
     0  & 0 & 0  & 0 & 0  & 0 & 0 & 0  & 0 & 0  \\
     1  & 0 & 2  & 0 & 1  & 0 & 0 & 1  & 0 & 1  \\
     0  & 0 & 0  & 0 & 0  & 0 & 0 & 0  & 0 & 0  \\
     0  & 0 & 0  & 0 & 0  & 0 & 0 & 0  & 0 & 0  \\
     1  & 0 & 2  & 0 & 1  & 0 & 0 & 1  & 0 & 1  \\
     0  & 0 & 0  & 0 & 0  & 0 & 0 & 0  & 0 & 0  \\
     1  & 0 & 2  & 0 & 1  & 0 & 0 & 1  & 0 & 1 
   \closebigpmatrix
 +\gamma_2''
   \openbigpmatrix{2.1pt}
     0 & -2 & 0 & 0 & 0 & -2 & 0 & 0 & 0 & 0 \\
     0 & 0  & 0 & 0 & 0 & 0  & 0 & 0 & 0 & 0 \\
     0 & 0  & 0 & 0 & 0 & 0  & 0 & 0 & 0 & 0 \\
     0 & 0  & 0 & 0 & 0 & 0  & 0 & 0 & 0 & 0 \\
     0 & 2  & 0 & 0 & 0 & 2  & 0 & 0 & 0 & 0 \\
     0 & 0  & 0 & 0 & 0 & 0  & 0 & 0 & 0 & 0 \\
     0 & 0  & 0 & 0 & 0 & 0  & 0 & 0 & 0 & 0 \\
     0 & 2  & 0 & 0 & 0 & 2  & 0 & 0 & 0 & 0 \\
     0 & 0  & 0 & 0 & 0 & 0  & 0 & 0 & 0 & 0 \\
     0 & 2  & 0 & 0 & 0 & 2  & 0 & 0 & 0 & 0
   \closebigpmatrix,
\end{align*}
where $\gamma_1'=\frac{1}{2}\alpha^{-2}r\gamma_1$, $\gamma_2'=\alpha^{-2}r\gamma_2$, and $\gamma_1''=\frac{1}{2}r T'\gamma_1$, $\gamma_2''=r T'\gamma_2$. The latter four terms here use \eqref{EqESGCompDelgGg}; in terms of \eqref{EqESGTrapMlSplit}, we have
\begin{equation}
\label{EqESGTrapDelgGgMl}
  \sigma_1(2\delta_g G_g)
    =i\wt\sigma
     \begin{pmatrix}
       1 & 0 & 2 & 0 & 1  & 0 & 0 & 1 & 0 & 1 \\
       0 & 2 & 0 & 0 & 0  & 2 & 0 & 0 & 0 & 0 \\
       1 & 0 & 2 & 0 & -1 & 0 & 0 & 1 & 0 & -1 \\
       0 & 0 & 0 & 2 & 0  & 0 & 0 & 0 & 2 & 0
     \end{pmatrix},
\end{equation}
and the above form of $S_L^{11}$ then follows easily from \eqref{EqESGCompDiff1} and \eqref{EqESGCompDiff2}.

Next, in the splitting \eqref{EqESGTrapMlSplit} of $\pi^*T^*\cM$, we find $\delta_g=i(\wt\sigma,0,\wt\sigma,0)$ and $(\td-d)=\gamma_3(\alpha^{-1},\alpha T',0,0)^t$, so we compute
\[
  S_L^{22} = \begin{pmatrix} 0 & -2 & 0 & 0 \\ -2 & 0 & -2 & 0 \\ 0 & 2 & 0 & 0 \\ 0 & 0 & 0 & 0 \end{pmatrix}
    + \gamma_3'
      \begin{pmatrix}
        1 & 0 & 1 & 0 \\
        0 & 0 & 0 & 0 \\
        0 & 0 & 0 & 0 \\
        0 & 0 & 0 & 0
      \end{pmatrix}
    + \gamma_3''
      \begin{pmatrix}
        0 & 0 & 0 & 0 \\
        1 & 0 & 1 & 0 \\
        0 & 0 & 0 & 0 \\
        0 & 0 & 0 & 0
      \end{pmatrix}
\]
for $\gamma_3'=\alpha^{-2}r\gamma_3$, $\gamma_3''=r T'\gamma_3$. For $S_L^{21}$, which is the principal symbol of $L_{21}$ at $\Gamma$ rescaled by $(i r^{-1}\sigma)^{-1}$, we find, using \eqref{EqESGCompTrg12}, \eqref{EqESGCompTrg24Trg35}, \eqref{EqESGCompIotaRho1}, \eqref{EqESGCompIotaRho2}, \eqref{EqESGTrapDelgGgMl}, and $\iota_\rho\omega=0$ at $\Gamma$:
\[
  S_L^{21}
  = q'
    \begin{pmatrix}
      0  & 0 & 0  & 0 & 0 & -2 & 0 & 0  & 0 & 0  \\
      -1 & 0 & -2 & 0 & 1 & 0  & 0 & -1 & 0 & -1 \\
      0  & 0 & 0  & 0 & 0 & 2  & 0 & 0  & 0 & 0  \\
      0  & 0 & 0  & 0 & 0 & 0  & 2 & 0  & 0 & 0
    \end{pmatrix},
\]
where $q'=\frac{1}{2}\alpha^{-1}r q$. Finally, $S_L^{12}$, the principal symbol of $L_{12}$ divided by $i r^{-1}\sigma$, can be computed at $\Gamma$ by using \eqref{EqESGCompDDel}, \eqref{EqESGCompTrg24}, \eqref{EqESGCompOmega}, \eqref{EqESGCompIotaRho1}, and \eqref{EqESGCompIotaRho2}; one finds
\[
  S_L^{12}
  = 4q'
    \begin{pmatrix}
      0  & 1  & 0  & 0 \\
      0  & 0  & 0  & 0 \\
      0  & -1 & 0  & 0 \\
      0  & 0  & 0  & 0 \\
      0  & -1 & 0  & 0 \\
      -1 & 0  & -1 & 0 \\
      0  & 0  & 0  & -1 \\
      0  & 1  & 0  & 0 \\
      0  & 0  & 0  & 0 \\
      0  & 1  & 0  & 0
    \end{pmatrix}.
\]
For $Q_{e,0}=0$, hence $q=0$, we have $S_L^{12}=0$, $S_L^{21}=0$, so $S_L$ is block-diagonal; furthermore, one easily checks that the eigenvalues of $S_L^{22}$ are $0$ and $\gamma_3'$. In \cite[\S9.1]{HintzVasyKdSStability}, it was checked that the eigenvalues of $S_L^{11}$ are either $0$ or positive multiples of $\gamma_1'$ and $\gamma_2'$. Thus, for $\gamma_j\geq 0$, the condition \eqref{EqLinAnaTrap} on the skew-adjoint part of the subprincipal operator at the trapped set is met. Since the condition is open, it immediately follows that it holds for small charges $Q_{e,0}$ as well. In fact, the smallness assumption on $Q_{e,0}$ is not necessary, as we proceed to show:

\begin{lemma}
\label{LemmaESGTrap}
  The eigenvalues of the $14\times 14$ matrix $(ir^{-1}\sigma)^{-1}S_L$ are
  \[
    0,\ 2\gamma_1',\ 4\gamma_1',\ 2\gamma_2',\ \gamma_3',\ -2 i q'\sqrt{2},\ 2 i q'\sqrt{2};
  \]
  in particular, for $\gamma_j\geq 0$, $j=1,2,3$, they all have real part $\geq 0$.
\end{lemma}
\begin{proof}
  We work in the fiber $V=(\pi^*(S^2T^*\cM\oplus T^*\cM))_\zeta$ over a single point $\zeta\in\Gamma$. The strategy is to find invariant subspaces of $V$ by which we quotient out to reduce the task to a simpler, less-dimensional one. Denote the canonical sections of $\pi^*S^2T^*\cM$ in the rank $1$ decomposition \eqref{EqESGTrapMlSplit} by $f_1,\ldots,f_{10}$, and the canonical sections of $\pi^*T^*\cM$ by $f_{11},\ldots,f_{14}$. Then $S_L$ preserves the subspace $V_1=\la f_4,f_7,f_9,f_{14}\ra$, with $K_1=\la f_4-f_9\ra\subset\ker(S_L|_{V_1})$; in the basis $[f_4],[f_7],[f_{14}]$ of $V_1/K_1$, the map induced by $S_L$ on $V_1/K_1$ is given by the matrix
  \[
    \begin{pmatrix}
      2\gamma_1'     & 0   & 0 \\
      -2-2\gamma_1'' & 0   & -4q' \\
      0              & 2q' & 0
    \end{pmatrix}
  \]
  with eigenvalues $2\gamma_1'$, $\pm 2 i q'\sqrt{2}$; thus the eigenvalues of $S_L|_{V_1}$ are these, together with $0$. We next note that $S_L$ preserves the subspace $V_2=\la f_1-f_8,f_2-f_6,f_1-f_3+f_8,f_{11}-f_{13}\ra$, and indeed $S_L|_{V_2}$ is lower triangular and nilpotent, hence its only eigenvalue is $0$. On the 6-dimensional quotient $V/(V_1\oplus V_2)$, for which we take the ordered basis $[f_1],[f_{11}],[f_2],[f_5],[f_{10}],[f_{12}]$, we have $[f_3]=2[f_1]$, $[f_6]=[f_2]$, $[f_8]=[f_1]$ and $[f_{13}]=[f_{11}]$, so the map induced by $L_V$ has the matrix
  \[
    \begin{pmatrix}
      4\gamma_1'     & 0             & 0                         & 0         & 0         & 0 \\
      0              & \gamma_3'     & 0                         & 0         & 0         & 0 \\
      -2+2\gamma_1'' & -4q'          & 2\gamma_1'                & 0         & 0         & 0 \\
      \gamma_2'      & 0             & -4+4\gamma_1''+\gamma_2'' & \gamma_2' & \gamma_2' & -4q'+\gamma_2' \\
      \gamma_2'      & 0             & \gamma_2''                & \gamma_2' & \gamma_2' & 4q'+\gamma_2' \\
      -q             & -2+\gamma_3'' & 0                         & q         & -q        & 0
    \end{pmatrix}
  \]
  Its eigenvalues are $4\gamma_1'$, $\gamma_3'$ and $2\gamma_1'$, together with the eigenvalues of the bottom right $3\times 3$ block, which are $2\gamma_2'$, $\pm 2 i q'\sqrt{2}$, finishing the proof.
\end{proof}

This establishes condition~\eqref{EqLinAnaTrap}, and thus gives the desired high energy estimates.

\subsection{Threshold regularity}
\label{SubsecESGRad}

As explained in \S\ref{SecLinAna}, the amount of regularity above which solutions of $L u=0$ have a resonance expansion up to an exponentially decaying remainder depends on the skew-adjoint part of $L$ at the radial set $\cR$ with respect to a positive definite inner product; the particular choice is informed by the structure of the subprincipal operator $S_\sub(L)$ at $\cR$.

Near the radial set $\cR_\pm=\Nb^*\{r=r_\pm\}\setminus o$ at the event ($-$) or cosmological ($+$) horizon, we use the coordinates \eqref{EqKNdS0NullCoord} and the resulting form \eqref{EqKNdS0MetricNull} of the metric; writing covectors as
\[
  -\sigma\,dt_0 + \xi\,dr + \eta,\quad \eta\in T^*\Sph^2,
\]
the dual metric function is then given by $G=\mp 2\sigma\xi-\mu\xi^2-r^{-2}\slG$, and at $\cR_\pm$, one finds
\begin{equation}
\label{EqESGRadHam}
  \ham_G=\pm 2\xi\pa_{t_0} \mp 2\kappa_\pm\xi^2\pa_\xi,
\end{equation}
with $\kappa_\pm$ the surface gravities, defined in equation~\eqref{EqMS2ScSurfGrav}.

The partial trivializations \eqref{EqESGCompBundle1} and \eqref{EqESGCompBundle2} are not defined at $r=r_\pm$; hence, following \cite[\S6.3]{HintzVasyKdSStability}, we compute the transition function to a smooth trivialization of $T^*M^\circ$ and $S^2 T^*M^\circ$ as follows: writing a 1-form in the static region $\cM$ as
\[
  u = u_N\,e^0 + u^T_N\,e^1 + u^T_T = \wt u_N\,dt_0 + \wt u^T_N\,dr+\wt u^T_T,
\]
with $u^T_T$ and $\wt u^T_T$ 1-forms on $\Sph^2$, we have
\begin{equation}
\label{EqESGRadTrans1}
  \begin{pmatrix} u_N \\ u^T_N \\ u^T_T \end{pmatrix} = \sC^{[1]}_\pm \begin{pmatrix} \wt u_N \\ \wt u^T_N \\ \wt u^T_T \end{pmatrix},
  \quad
  \sC^{[1]}_\pm = \begin{pmatrix} \alpha^{-1} & 0 & 0 \\ \mp\alpha^{-1} & \alpha & 0 \\ 0 & 0 & 1 \end{pmatrix}.
\end{equation}
The bundle splitting
\begin{equation}
\label{EqESGRadBundle1}
  T^*M^\circ = \la dt_0 \ra \oplus \la dr \ra \oplus T^*\Sph^2
\end{equation}
\emph{is} smooth at $r=r_\pm$. Consider the induced bundle splitting
\begin{equation}
\label{EqESGRadBundle2}
  S^2T^*M^\circ = \la dt_0^2 \ra \oplus \la 2dt_0\,dr\ra \oplus (2 dt_0\cdot T^*\Sph^2) \oplus \la dr^2 \ra \oplus (2 dr\cdot T^*\Sph^2) \oplus S^2T^*\Sph^2.
\end{equation}
Writing a section $u$ of $S^2T^*M^\circ$ as
\begin{align*}
  u &= u_{NN}\,e^0e^0 + 2 u_{NTN}\,e^0e^1 + 2e^0\cdot u_{NTT} + u^T_{NN}\,e^1e^1 + 2e^1\cdot u^T_{NT} + u^T_{TT} \\
    &= \wt u_{NN}\,dt_0^2 + 2 \wt u_{NTN}\,dt_0\,dr + 2dt_0\cdot \wt u_{NTT} + \wt u^T_{NN}\,dr^2 + 2dr\cdot \wt u^T_{NT} + \wt u^T_{TT},
\end{align*}
the transition matrix is defined by
\begin{equation}
\label{EqESGRadTrans2}
  \begin{pmatrix} u_{NN} \\ u_{NTN} \\ u_{NTT} \\ u^T_{NN} \\ u^T_{NT} \\ u^T_{TT} \end{pmatrix} = \sC_\pm^{(2)} \begin{pmatrix} \wt u_{NN} \\ \wt u_{NTN} \\ \wt u_{NTT} \\ \wt u^T_{NN} \\ \wt u^T_{NT} \\ \wt u^T_{TT} \end{pmatrix},
  \quad
  \sC_\pm^{(2)} = \openbigpmatrix{2pt} \alpha^{-2} & 0 & 0 & 0 & 0 & 0 \\ \mp\alpha^{-2} & 1 & 0 & 0 & 0 & 0 \\ 0 & 0 & \alpha^{-1} & 0 & 0 & 0 \\ \alpha^{-2} & \mp 2 & 0 & \alpha^2 & 0 & 0 \\ 0 & 0 & \mp\alpha^{-1} & 0 & \alpha & 0 \\ 0 & 0 & 0 & 0 & 0 & 1 \closebigpmatrix.
\end{equation}
We further recall from \cite[equations~(6.15) and (6.18)]{HintzVasyKdSStability} the forms of $\nabla^{\pi^*T^*M^\circ}_{\ham_G}$ and $\nabla^{\pi^*S^2T^*M^\circ}_{\ham_G}$ at $\cR_\pm$; thus
\begin{equation}
\label{EqESGRadPullback}
  -i\nabla^{\pi^*(S^2T^*M^\circ\oplus T^*M^\circ)}_{\ham_G} = \pm 2\xi D_{t_0} \mp 2\kappa_\pm\xi D_\xi \xi \mp 2 i\kappa_\pm \xi \diag(0,2,1,4,3,2,0,2,1),
\end{equation}
valid at $\cR_\pm$, with $\diag(a_1,\ldots,a_n)$ denoting the $n\times n$ diagonal matrix with entries $a_1,\ldots,a_n$. (The result \eqref{EqESGRadPullback} follows from \eqref{EqESGCompConn1}, \eqref{EqESGCompConn2}, \eqref{EqESGRadHam}, \eqref{EqESGRadTrans1} and \eqref{EqESGRadTrans2}.) The first two terms here are formally self-adjoint with respect to the symplectic volume form on $T^*M^\circ$.

In order to compute $S_\sub(L)$ at $\cR_\pm$, we in addition need to compute the principal symbol of the sum of the first order terms in \eqref{EqESGL} at $\cR_\pm$. This computation is easily accomplished by first calculating the symbols at $N^*\{r=r_0\}$, for $r_0\in(r_-,r_+)$ close to $r_\pm$, in the static coordinate system and bundle splittings \eqref{EqESGCompBundle1} and \eqref{EqESGCompBundle2}, conjugating by $\sC_\pm^{(2)}\oplus\sC_\pm^{[1]}$ and restricting to $\mu=0$, thereby computing the limit of the principal symbol as $r_0\to r_\pm$, which is no longer a singular limit in the smooth bundle splittings \eqref{EqESGRadBundle1} and \eqref{EqESGRadBundle2}, since $L$ has smooth coefficients across the horizons. (This is merely an invariant way of phrasing the strategy in \cite[\S9.2]{HintzVasyKdSStability}.)

We perform a decomposition
\begin{equation}
\label{EqESGRadBundle3}
  S^2 T^*\Sph^2 = \la r^2\slg \ra \oplus \slg^\perp,
\end{equation}
so
\[
  r^2\slg = \begin{pmatrix} 1 \\ 0 \end{pmatrix},\quad r^{-2}\sltr = \begin{pmatrix} 1 & 0 \end{pmatrix}.
\]
At $N^*\{r=r_0\}$, one then finds using \eqref{EqESGCompDelgGg} and $\wt\sigma=0$, so $\omega=\wt\xi\,e^1=\alpha^{-1}\xi\,dr$, see \eqref{EqESGTrapCovecRel}, that
\begin{equation}
\label{EqESGRadDelgGg}
  \sigma_1(2\delta_g G_g)
    = i\alpha\xi
      \begin{pmatrix}
        0 & 2 & 0 & 0 & 0 & 0  & 0 \\
        1 & 0 & 0 & 1 & 0 & -1 & 0 \\
        0 & 0 & 0 & 0 & 2 & 0  & 0
      \end{pmatrix}
\end{equation}
in the decomposition \eqref{EqESGCompBundle1} and \eqref{EqESGCompBundle2} of $T^*\cM$ and $S^2 T^*\cM$, with the summand $S^2 T^*\Sph^2$ decomposed further as in \eqref{EqESGRadBundle3}. Using this together with \eqref{EqESGCompDiff1} and \eqref{EqESGCompDiff2}, conjugating by $(\sC_\pm^{(2)})^{-1}$, writing $\mu T'=\pm(1+\mu c_\pm)$ as in \eqref{EqKNdS0CoordT} and evaluating at $\mu=0$ recovers the result of \cite[\S9.2]{HintzVasyKdSStability},
\begin{align}
\label{EqESGRadL11}
  \sigma_1\bigl(2&(\tdel^*-\delta_g^*)\delta_g G_g\bigr) \\
    & = 
  \mp i\gamma_1\xi
    \openbigpmatrix{2pt}
      2         & 0 & 0         & 0 & 0 & 0              & 0 \\
      \mp c_\pm & 0 & 0         & 0 & 0 & \pm\frac{1}{2} & 0 \\
      0         & 0 & 1         & 0 & 0 & 0              & 0 \\
      0         & 0 & 0         & 0 & 0 & -c_\pm         & 0 \\
      0         & 0 & \mp c_\pm & 0 & 0 & 0              & 0 \\
      0         & 0 & 0         & 0 & 0 & 0              & 0 \\
      0         & 0 & 0         & 0 & 0 & 0              & 0
    \closebigpmatrix
  \mp i\gamma_2\xi
    \openbigpmatrix{2pt}
      0           & 0 & 0 & 0 & 0 & 0     & 0 \\
      \pm 2 c_\pm & 0 & 0 & 0 & 0 & \mp 1 & 0 \\
      0           & 0 & 0 & 0 & 0 & 0     & 0 \\
      0           & 0 & 0 & 0 & 0 & 0     & 0 \\
      0           & 0 & 0 & 0 & 0 & 0     & 0 \\
      -2 c_\pm    & 0 & 0 & 0 & 0 & 1     & 0 \\
      0           & 0 & 0 & 0 & 0 & 0     & 0
    \closebigpmatrix \nonumber
\end{align}
at $\cR_\pm$. For the $(2,2)$ block of $L$, thus on 1-forms, we compute at $N^*\{r=r_0\}$
\[
  \sigma_1\bigl((\td-d)\delta_g\bigr) = i\gamma_3\xi\begin{pmatrix} 0 & 1 & 0 \\ 0 & -\mu T' & 0 \\ 0 & 0 & 0 \end{pmatrix}
\]
in the static bundle splitting; to translate this to the splitting \eqref{EqESGRadBundle1}, we conjugate by $(\sC_\pm^{[1]})^{-1}$ and evaluate at $\mu=0$, obtaining
\begin{equation}
\label{EqESGRadL22}
  \sigma_1\bigl((\td-d)\delta_g\bigr) = \mp i\gamma_3\xi \begin{pmatrix} 1 & 0 & 0 \\ \mp c_\pm & 0 & 0 \\ 0 & 0 & 0 \end{pmatrix}.
\end{equation}
For the $(2,1)$ block, we have at $N^*\{r=r_0\}$
\[
  \sigma_1(L_{21}) = \frac{1}{2}i\alpha q\xi
      \begin{pmatrix}
        1 & 0 & 0 & -1 & 0 & 1 & 0 \\
        0 & 0 & 0 & 0  & 0 & 0 & 0 \\
        0 & 0 & 2 & 0  & 0 & 0 & 0
      \end{pmatrix}
\]
in the static splitting, where we use \eqref{EqESGCompTrg12}, \eqref{EqESGCompTrg24Trg35} and \eqref{EqESGRadDelgGg}. Multiplying this from the left by $(\sC_\pm^{[1]})^{-1}$ and from the right by $\sC_\pm^{(2)}$, and evaluating at $\mu=0$ gives the expression at $\cR_\pm$,
\begin{equation}
\label{EqESGRadL21}
  \sigma_1(L_{21})
    =\mp i q \xi
     \begin{pmatrix}
       0 & 0     & 0     & 0 & 0 & 0            & 0 \\
       0 & \mp 1 & 0     & 0 & 0 & -\frac{1}{2} & 0 \\
       0 & 0     & \mp 1 & 0 & 0 & 0            & 0
     \end{pmatrix}.
\end{equation}
In the static splitting, the $(1,2)$ block finally has
\[
  \sigma_1(L_{12})
   = 2i\alpha q\xi
     \begin{pmatrix}
       1  & 0 & 0 \\
       0  & 0 & 0 \\
       0  & 0 & 1 \\
       -1 & 0 & 0 \\
       0  & 0 & 0 \\
       1  & 0 & 0 \\
       0  & 0 & 0
     \end{pmatrix},
\]
so in the smooth splitting (multiplying from the left by $(\sC_\pm^{(2)})^{-1}$ and from the right by $\sC_\pm^{[1]}$ and evaluating at $\mu=0$),
\begin{equation}
\label{EqESGRadL12}
  \sigma_1(L_{12})
    = \mp 2 i q \xi
      \begin{pmatrix}
        0     & 0 & 0 \\
        -1    & 0 & 0 \\
        0     & 0 & 0 \\
        0     & 0 & 0 \\
        0     & 0 & -1 \\
        \mp 1 & 0 & 0 \\
        0     & 0 & 0
      \end{pmatrix}
\end{equation}
at $\cR_\pm$.

We thus find that at $\cR_\pm$,
\[
  S_\sub(L) = \pm 2\xi D_{t_0} \mp 2\kappa_\pm \xi D_\xi \xi \mp i \xi S_L^\pm
\]
where in the decompositions \eqref{EqESGRadBundle1}, \eqref{EqESGRadBundle2} and \eqref{EqESGRadBundle3}, i.e.\ acting on the pullback of
\begin{align*}
  &\la dt_0^2 \ra \oplus \la 2 dt_0\,dr \ra \oplus (2 dt_0\cdot T^*\Sph^2) \oplus \la dr^2 \ra \oplus (2 dr\cdot T^*\Sph^2) \oplus \la r^2\slg \ra \oplus \slg^\perp \\
  & \qquad \oplus \la dt_0\ra \oplus \la dr \ra \oplus T^*\Sph^2
\end{align*}
under $\pi\colon T^*M^\circ\to M^\circ$, the bundle endomorphism $S_L^\pm$ is a block $10\times 10$ matrix, namely the sum of $2\kappa_\pm\diag(0,2,1,4,3,2,2,0,2,1)$, coming from \eqref{EqESGRadPullback} (with the additional splitting \eqref{EqESGRadBundle3}), and the $(7+3)\times(7+3)$ block matrix with $(1,1)$, $(2,2)$, $(2,1)$ and $(1,2)$ entries \eqref{EqESGRadL11}, \eqref{EqESGRadL22}, \eqref{EqESGRadL21}, \eqref{EqESGRadL12}, respectively.

\begin{lemma}
\label{LemmaESGRad}
  The distinct eigenvalues of $S_L^\pm$ are
  \[
    2\kappa_\pm,\ 4\kappa_\pm,\ 6\kappa_\pm,\ 8\kappa_\pm,\ 2\kappa_\pm+\gamma_1,\ 4\kappa_\pm+\gamma_2,\ 2\gamma_1,\ \gamma_3;
  \]
  in particular, for $\gamma_j\geq 0$, $j=1,2,3$, they all have real part $\geq 0$.
\end{lemma}
\begin{proof}
  $S_L^\pm$ preserves $\slg^\perp$, and equals scalar multiplication by $4\kappa_\pm$. Next, $S_L^\pm$ preserves $(2dt_0\cdot T^*\Sph^2)\oplus T^*\Sph^2\oplus(2dr\cdot T^*\Sph^2)$ and is in fact lower triangular (notice the different order of the bundles) with diagonal entries, hence eigenvalues, $2\kappa_\pm+\gamma_1$, $2\kappa_\pm$ and $6\kappa_\pm$. Lastly, $S_L^\pm$ preserves
  \[
    \la dt_0^2\ra \oplus \la dt_0\ra \oplus \la r^2\slg\ra \oplus \la 2dt_0\,dr\ra \oplus \la dr\ra \oplus \la dr^2\ra,
  \]
  and acting on this space (with bundles in this order), $S_L^\pm$ is lower triangular as well, with diagonal entries $2\gamma_1$, $\gamma_3$, $4\kappa_\pm+\gamma_2$, $4\kappa_\pm$, $4\kappa_\pm$, $8\kappa_\pm$.
\end{proof}

This plugs into \eqref{EqLinAnaWhBetaPM} and \eqref{EqLinAnaWhBeta}, giving $\wh\beta_\pm\geq 0$ there, and hence the threshold condition in Theorem~\ref{ThmLinAnaMero} is implied by $s>1/2+\alpha\sup(\beta)$. The proof of Theorem~\ref{ThmESG} is complete.

\section{Linear stability of KNdS black holes}
\label{SecLin}

With $\tdel^*$ and $\td$ as in Theorems~\ref{ThmCD}, and for KNdS parameters $b$ close to $b_0$, we define the linearized gauge-fixed Einstein--Maxwell operator
\begin{align*}
  L_b(\gdot,\Adot) &= \Bigl( 2\bigl(D_{g_b}(\Ric+\Lambda)(\gdot) - \tdel^*D_{g_b}\Ups^E(\gdot) - 2 D_{g_b,d A_b}T(\gdot,d\Adot)\bigr), \\
    &\quad\qquad D_{g_b,A_b}(\delta_{(\cdot)}d(\cdot))(\gdot,\Adot) - \td\Ups^M(g_b,\Adot)\Bigr),
\end{align*}
with $\Ups^E$, $\Ups^M$ and $T$ given in equations~\eqref{EqKNdSIniGaugeE}, \eqref{EqKNdSIniGaugeM}, and \eqref{EqBasicDerEMTensor}, respectively. We point out that $L_b$ is a principally scalar wave operator, and in view of Theorem~\ref{ThmESG}, $L:=L_{b_0}$ satisfies the assumptions of Theorem~\ref{ThmLinAnaSolvFixed}.

Using the mode stability result, Theorem~\ref{ThmMS}, which describes the structure of all non-decaying generalized resonant states of $L_b$, the high energy analysis, Theorem~\ref{ThmESG}, immediately implies the linear stability of non-degenerate RNdS spacetimes by means of the argument used in the proof of \cite[Theorem~10.2]{HintzVasyKdSStability}; see also Remark~\ref{RmkIntroStrNoCD}.

A robust proof, which prepares the non-linear stability argument, makes use of the fact that we have arranged for constraint damping (Theorem~\ref{ThmCD}); we proceed to explain this, following the arguments in \cite[\S\S10 and 11]{HintzVasyKdSStability} closely. First, we point out that while linearized KNdS solutions $(g'_b(b'),A'_b(b'))$ (see Definition~\ref{DefKNdSaLin}) solve the linearized Einstein--Maxwell equations $\sL(g'_b(b'),A'_b(b'))=0$ in the notation of Theorem~\ref{ThmMS}, they generally do not respect the linearized gauge conditions; as explained in \S\ref{SubsecIntroStr} and around equations~\eqref{EqBasicLinGaugeG} and \eqref{EqBasicLinGaugeA}, this can be remedied by solving
\begin{equation}
\label{EqLinGaugeG}
  D_{g_b}\Ups^E\bigl(g'_b(b')+\tcL_{g_b}V'_b(b')\bigr) = 0
\end{equation}
for the vector field $V'_b(b')\in\CI(\Omega^\circ;T\Omega^\circ)$, and then solving
\begin{equation}
\label{EqLinGaugeA}
  \Ups^M\bigl(g_b,\Adot+\tcL_{A_b}V'_b(b')+d a'_b(b')\bigr) = 0
\end{equation}
for the function $a'_b(b')\in\CI(\Omega^\circ)$; note that this differs slightly from equation~\eqref{EqBasicLinGaugeA} due to the different choices of linearized gauges (the present one coming from our choice \eqref{EqNLBabyNonlinearOp} of the non-linear gauge-fixed operator below). To ensure linear dependence of $V'_b(b')$ and $a'_b(b')$ on $b'$, we demand that their Cauchy data vanish:
\[
  \gamma_0(V'_b(b')) = 0,  \quad \gamma_0(a'_b(b')) = 0.
\]
Thus, the sections
\begin{equation}
\label{EqLinGaugedLinKNdS}
  \bigl(g_b^{\prime\Ups}(b'),A_b^{\prime\Ups}(b')\bigr) := \bigl(g'_b(b') + \tcL_{g_b}V'_b(b'), A'_b(b')+\tcL_{A_b}V'_b(b')+d a'_b(b')\bigr)
\end{equation}
solve the linearized gauge-fixed Einstein--Maxwell equation $L_b(g_b^{\prime\Ups}(b'),A_b^{\prime\Ups}(b'))=0$.

\begin{rmk}
\label{RmkLinGaugeVF}
  For $b=b_0$, equation~\eqref{EqLinGaugeG} reads
  \[
    (\Box_{g_{b_0}}+\Lambda)V'_{b_0}(b') = -2 D_{g_{b_0}}\Ups^E(g'_{b_0}(b')),
  \]
  hence $V'_{b_0}(b')$ has an asymptotic expansion with an $\cO(e^{-\alpha t_*})$ remainder term by the main result of \cite{HintzPsdoInner}.
\end{rmk}

Fixing a cutoff
\begin{equation}
\label{EqLinCutoff}
  \chi\in\CI(\R_{t_*}), \quad \chi(t_*)\equiv 0,\ t_*\leq 1,\ \ \chi(t_*)\equiv 1,\ t_*\geq 2,
\end{equation}
we could then use $L_b(\chi g^{\prime\Ups}_b(b'),\chi A^{\prime\Ups}_b(b'))\in\CIc(\Omega^\circ;S^2 T^*\Omega^\circ\oplus T^*\Omega^\circ)$ as a modification of the range of $L_b$ along the lines indicated in \S\ref{SubsecIntroStr} and Theorem~\ref{ThmLinAnaSolvFixed}. For future purposes, it is however better to use instead
\begin{equation}
\label{EqLinGeneratesLinKNdS}
\begin{split}
  K_b(b') &:= L_b\bigl(\chi g'_b(b')+\tcL_{g_b}(\chi V'_b(b')),\chi A'_b(b')+\tcL_{A_b}(\chi V'_b(b')) + d(\chi a'_b(b'))\bigr) \\
    &= L_b(\chi g'_b(b'),\chi A'_b(b')) + (2\tdel^*\theta_b(b'), \td\kappa_b(b'))
\end{split}
\end{equation}
with
\begin{align*}
  \theta_b(b') &= -\tdel^* D_{g_b}\Ups^E(\tcL_{g_b}(\chi V'_b(b'))), \\
  \kappa_b(b') &= -\td\Ups^M\bigl(g_b,\tcL_{A_b}(\chi A'_b(b')) + d(\chi a'_b(b'))\bigr).
\end{align*}
This shows directly how the forcing term $K_b(b')$ generates the linearized KNdS solution in the particular form $(g'_b(b'),A'_b(b'))$ after a suitable modification of the gauge.

\begin{thm}
\label{ThmLinKNdS}
  Fix $s\geq 2$, and let $\alpha>0$ be as in Theorem~\ref{ThmESG}. There exists a finite-dimensional space of gauge modifications
  \begin{equation}
  \label{EqLinKNdSGaugeMod}
    \Xi \subset \CIc(\Omega^\circ;T^*\Omega^\circ\oplus\ul\R),
  \end{equation}
  with $\ul\R=\Omega^\circ\times\R$ the trivial line bundle, such that the following holds: let $b\in B$ be close to the parameters $b_0$ of a non-degenerate RNdS black hole, and let
  \begin{align*}
    (\hdot,\kdot,\bfEdot,\bfBdot)&\in H^{s+1}(\Sigma_0;S^2 T^*\Sigma_0)\oplus H^s(\Sigma_0;S^2 T^*\Sigma_0) \\
      &\qquad\qquad \oplus H^s(\Sigma_0;T^*\Sigma_0) \oplus H^{s+1/2}(\Sigma_0;T^*\Sigma_0)
  \end{align*}
  be initial data for the linearized Einstein--Maxwell system, linearized around the KNdS solution $(g_b,A_b)$, which satisfy the linearized constraint equations, linearized around the KNdS initial data $(h_b,k_b,\bfE_b,\bfB_b)$. Assume that the linearized magnetic charge vanishes, see the discussion around \eqref{EqKNdSIniLinearizedFromPotential}.

  Then there exist linearized KNdS parameters $b'\in\R^5$, an exponentially decaying tail $\wt r\in\Hbext^{s,\alpha}(\Omega;S^2\Tb^*_\Omega M\oplus \Tb^*_\Omega M)$ and a gauge modification $(\theta,\kappa)\in\Xi$ such that the solution of the initial value problem
  \begin{equation}
  \label{EqLinKNdSIVP}
    \begin{cases}
      L_b\wt r = -K'_b(b')-(2\tdel^*\theta,\td\kappa) & \tn{in }\Omega^\circ, \\
      \gamma_0(\wt r) = D_{(h_b,k_b,\bfE_b,\bfB_b)}i_b(\hdot,\kdot,\bfEdot,\bfBdot) & \tn{on }\Sigma_0.
    \end{cases}
  \end{equation}
  is exponentially decaying, $\wt r\in\Hbext^{s,\alpha}(\Omega;S^2\,\Tb^*_\Omega M\oplus\Tb^*_\Omega M)$. In other words, with $\wt r=(\wt g,\wt A)$, the linearized metric and 4-potential
  \begin{equation}
  \label{EqLinKNdSIVPSol1}
  \begin{split}
    \gdot &:= \chi g'_b(b') + \wt g, \\
    \Adot &:= \chi A'_b(b') + \wt A
  \end{split}
  \end{equation}
  solve the linearized Einstein--Maxwell equation $\sL(\gdot,\Adot)=0$, attaining the given initial data at $\Sigma_0$, in the gauge
  \begin{equation}
  \label{EqLinKNdSIVPSol2}
  \begin{split}
    D_{g_b}\Ups^E(\gdot) - \theta'_b(b') - \theta = 0, \\
    \Ups^M(g_b,\Adot) - \kappa'_b(b') - \kappa = 0.
  \end{split}
  \end{equation}
\end{thm}

We shall not state the extension of this to the case of magnetically charged black holes; see Remark~\ref{RmkNLMagnetic} for the non-linear case.

\begin{proof}[Proof of Theorem~\ref{ThmLinKNdS}]
  With $\alpha$ as in Theorem~\ref{ThmESG}, denote the finite number $N$ of resonances of $L_{b_0}$, the linearization around the \emph{RNdS} spacetime, in $\Im\sigma\geq-\alpha$ by $\sigma_1=0$, $\sigma_2,\ldots,\sigma_N\neq 0$. We recall that Theorem~\ref{ThmCD} ensures that all resonant states $(\gdot,\Adot)$ of $L_{b_0}$ at frequency $\sigma_j$, $j=1,\ldots,N$, in fact satisfy the gauge conditions
  \begin{equation}
  \label{EqLinKNdSResStateGauged}
    D_{g_{b_0}}\Ups^E(\gdot) = 0, \quad \Ups^M(g_{b_0},\Adot) = 0,
  \end{equation}
  and the linearized ungauged Einstein--Maxwell system $\sL(\gdot,\Adot)=0$; thus, by Theorem~\ref{ThmMS}, $(\gdot,\Adot)$ is a linearized KNdS solution (only appearing at the resonance $\sigma_1=0$) plus a pure gauge solution.
  
  For $j\neq 1$ then, pick a basis $\{(\gdot_{j 1},\Adot_{j 1}),\ldots,(\gdot_{j N_j},\Adot_{j N_j})\}$, $N_j\in\N$, of the space $\Res(L_{b_0},\sigma_j)$ of resonant states; these are all pure gauge solutions. Hence, we can write
  \begin{equation}
  \label{EqLinKNdSPureGauge1}
    \gdot_{j k} = \cL_{V_{j k}}g_{b_0}, \quad \Adot_{j k} = \cL_{V_{j k}}A_{b_0} + d a_{j k}
  \end{equation}
  for suitable $V_{j k}\in\CI(\Omega^\circ;\TC\Omega^\circ)$ and $a_{j k}\in\CI(\Omega^\circ;\C)$. With $\chi$ a cutoff as in \eqref{EqLinCutoff}, the pure gauge modification giving rise to the asymptotic behavior $(\gdot_{j k},\Adot_{j k})$ is then
  \begin{equation}
  \label{EqLinKNdSPureGauge2}
    L_{b_0}\bigl(\tcL_{g_{b_0}}(\chi V_{j k}), \tcL_{A_{b_0}}(\chi V_{j k}) + d(\chi a_{j k})\bigr) = (2\tdel^*\theta_{j k},\td\kappa_{j k})
  \end{equation}
  with
  \begin{equation}
  \label{EqLinKNdSPureGauge3}
  \begin{gathered}
    \theta_{j k} = -D_{g_{b_0}}\Ups^E\bigl(\tcL_{g_{b_0}}(\chi V_{j k})\bigr) \in \CIc(\Omega^\circ;\TC^*\Omega^\circ), \\
    \kappa_{j k} = -\Ups^M\bigl(g_{b_0},\tcL_{A_{b_0}}(\chi V_{j k}) + d(\chi a_{j k})\bigr) \in \CIc(\Omega^\circ;\C),
  \end{gathered}
  \end{equation}
  the compact support in $\Omega^\circ$ being a consequence of \eqref{EqLinKNdSResStateGauged}.

  Next, for $j=1$, i.e.\ for the zero resonance, we note that the correctly gauged linearized KNdS solutions $(g^{\prime\Ups}_{b_0}(b'),A^{\prime\Ups}_{b_0}(b'))$ in \eqref{EqLinGaugedLinKNdS}, being smooth and lying in the kernel of $L_{b_0}$, have an asymptotic expansion up to an exponentially decaying remainder term by Theorem~\ref{ThmESG}. In fact, $V'_{b_0}(b')$ and $a'_{b_0}(b')$ themselves enjoy such expansions: for $V'_{b_0}(b')$, this was discussed in Remark~\ref{RmkLinGaugeVF}; for $a'_{b_0}(b')$, this then follows from the scalar wave equation \eqref{EqLinGaugeA} for $a'_{b_0}(b')$ and the fact that $V'_{b_0}(b')$ enjoys an asymptotic expansion. (We remark that the same is true if we replace $b_0$ by $b$ close to $b_0$.) Denote by $(g^{\prime\Ups}_{b_0}(b')_0,A^{\prime\Ups}_{b_0}(b')_0)$ the part of the asymptotic expansion of $(g^{\prime\Ups}_{b_0}(b'),A^{\prime\Ups}_{b_0}(b'))$ at frequency $0$; this is a generalized mode, and therefore
  \[
    K := \bigl\{ (g^{\prime\Ups}_{b_0}(b')_0, A^{\prime\Ups}_{b_0}(b')_0) \colon b'\in\R^5 \} \subset \Res(L_{b_0},0).
  \]
  By Theorem~\ref{ThmMS} (which applies at $0$ frequency as well due to Theorem~\ref{ThmCD}), \emph{all} generalized modes with frequency $0$ are equal to a linearized KNdS solution plus a pure gauge solution; therefore, we can pick a vector space complement
  \[
    \mathspan \bigl\{(\gdot_{1 1},\Adot_{1 1}),\ldots,(\gdot_{1 N_1},\Adot_{1 N_1}) \bigr\}
  \]
  of $K$ in the space $\Res(L_{b_0},0)\cap\CI(\Omega^\circ;T^*\Omega^\circ\oplus\ul\R)$ of real-valued zero resonant states consisting entirely of pure gauge solutions. (The space $\Res(L_{b_0},0)$ is spanned over $\C$ by its real-valued elements, since $L_{b_0}$ has real coefficients.) Thus, we have \eqref{EqLinKNdSPureGauge1}--\eqref{EqLinKNdSPureGauge3} for these zero resonant states as well.
  
  We can now define
  \begin{align*}
    \Xi^\C &:= \mathspan\{ (\theta_{j k},\kappa_{j k}) \colon j=1,\ldots,N,\ k=1,\ldots,N_j \} \\
      &\qquad\qquad \subset \CIc(\Omega^\circ;\TC^*\Omega^\circ\oplus\ul\C),
  \end{align*}
  and then the space of gauge modifications
  \[
    \Xi := \Xi^\C \cap \CIc(\Omega^\circ;T^*\Omega^\circ\oplus\ul\R)
  \]
  is the space of real-valued elements of $\Xi$. Since the operator $L_{b_0}$ has real coefficients, complex conjugation induces isomorphisms $\Res(L_{b_0},\sigma)\cong\Res(L_{b_0},-\bar\sigma)$ for all $\sigma\in\C$; therefore the vector space $\Res(L_{b_0},\sigma) + \Res(L_{b_0},-\bar\sigma)$ is spanned by its real-valued elements, indeed by the real parts of the elements of $\Res(L_{b_0},\sigma)$, and $\Xi^\C$ is correspondingly spanned by the real parts of the $(\theta_{j k},\kappa_{j k})$, so $\Xi^\C = \mathspan_\C \Xi$. Let $N_\Xi:=\dim_\R\Xi<\infty$, and fix a linear isomorphism
  \begin{equation}
  \label{EqLinKNdSXiIso}
    \R^{N_\Xi} \ni \bfc' \mapsto (\theta(\bfc'), \kappa(\bfc')) \in \Xi.
  \end{equation}

  Next, let $B\subset\R^5$ be a small neighborhood of $b_0$, let $N_\Xi:=\dim_\R\Xi$, and define the continuous map
  \begin{align*}
    z \colon B \times \R^5 \times \R^{N_\Xi} &\to \CIc(\Omega^\circ;S^2 T^*\Omega^\circ\oplus T^*\Omega^\circ), \\
    (b,b',\bfc') &\mapsto K_b(b') + (2\tdel^*\theta(\bfc'),\td\kappa(\bfc')),
  \end{align*}
  linear in $(b',\bfc')$, with $K_b(b')$ defined in \eqref{EqLinGeneratesLinKNdS}. Denote by $z^\C$ the map obtained from $z$ by $\C$-linear extension in $(b',\bfc')$, then $z^\C$ satisfies the assumptions of Theorem~\ref{ThmLinAnaSolvPertStat} (with $W=B$, $N_\cZ=5+N_\Xi$, and $w_0=b_0$) by construction. Therefore, we can indeed solve \eqref{EqLinKNdSIVP}, however a priori with $b'\in\C^5$, $(\theta,\kappa)\in\Xi^\C$, and $\wt r$ taking values in the complexified bundles. Since the initial data as well as the coefficients of $L_b$ are real, the real parts of $b'$ etc.\ yield a solution of \eqref{EqLinKNdSIVP} as well; but the solution is unique, hence we conclude that $b'\in\R^5$ and $(\theta,\kappa)\in\Xi$, with $\wt r$ real as well. Recalling the computation \eqref{EqLinGeneratesLinKNdS}, this shows that the solution of \eqref{EqLinKNdSIVP} indeed satisfies \eqref{EqLinKNdSIVPSol1}--\eqref{EqLinKNdSIVPSol2}.
\end{proof}

As a corollary (which by itself is much weaker than the above full linear stability result), we can prove the (generalized) mode stability of slowly rotating KNdS black holes: our robust setup, using constraint damping (Theorem~\ref{ThmCD}) and the high frequency analysis (Theorem~\ref{ThmESG}), allows us to infer this from the mode stability of \emph{non-rotating} black holes, which we proved using separation of variables in \S\ref{SecMS}.

\begin{thm}
\label{ThmLinMode}
  (Cf.\ \cite[Theorem~10.8]{HintzVasyKdSStability}.) Suppose $b$ is close to the parameters $b_0$ of a non-degenerate RNdS black hole. Let $\sigma\in\C$, $\Im\sigma\geq 0$, and $k\in\N_0$, and suppose
  \begin{equation}
  \label{EqLinMode}
    (\gdot,\Adot) = \sum_{j=0}^k e^{-i\sigma t_*}t_*^j(\gdot_j,\Adot_j),\quad \gdot_j\in\CI(Y;S^2 T^*_Y\Omega^\circ),\ \Adot_j\in\CI(Y;T^*_Y\Omega^\circ)
  \end{equation}
  is a generalized mode solution of the linearized Einstein--Maxwell system $\sL(\gdot,\Adot)=0$, linearized around the KNdS solution $(g_b,A_b)$, see \eqref{EqBasicLinEinsteinMaxwellExpl}. Then there exist generalized modes $V\in\CI(\Omega^\circ;\TC\Omega^\circ)$ and $a\in\CI(\Omega^\circ;\C)$ as well as parameters $b'\in\R^5$ such that
  \[
    (\gdot,\Adot) = (g'_b(b'),A'_b(b')) + (\cL_V g,\cL_V A + d a).
  \]
  If $\sigma\neq 0$, then $b'=0$, so $(\gdot,\Adot)$ is a pure gauge solution, while for $\sigma=0$, $(\gdot,\Adot)$ is equal to a linearized KNdS solution (linearized around $(g_b,A_b)$) plus a pure gauge solution.
\end{thm}
\begin{proof}
  We put $(\gdot,\Adot)$ into the gauge $D_{g_b}\Ups^E(\cdot)=0$, $\Ups^M(g_b,\cdot)=0$, as follows: solve the wave equations
  \[
    D_{g_b}\Ups^E(\gdot+\tcL_{g_b}V_0) = 0,\quad \Ups^M(g_b,\Adot+\tcL_{A_b}V_0 + d a_0) = 0,
  \]
  with vanishing initial data, for the vector field $V_0$ and the function $a_0$, both of which have an asymptotic expansion up to exponentially decaying remainders, as discussed in the proof of Theorem~\ref{ThmLinKNdS}. If $V_1$ and $a_1$ denote the parts of their respective asymptotic expansions which are generalized modes with frequency $\sigma$, we may replace $(\gdot,\Adot)$ by $(\gdot+\cL_{V_1}g_b,\Adot+\cL_{V_1}A_b+d a_1)$, which satisfies \eqref{EqLinMode} still, with a possibly larger value of $k$. Thus, we may assume $L_b(\gdot,\Adot)=0$; since $\sL(\gdot,\Adot)=0$, the initial data $(\gdot,\Adot)$ satisfy the linearized constraint equations at $\Sigma_0$, and satisfy the linearized gauge (encoded in the operator $L_b$) by construction. Thus, Theorem~\ref{ThmLinKNdS} implies that for suitable $(\theta,\kappa)\in\Xi$ and $b'\in\R^5$, the initial value problem
  \[
    L_b\wt r = -K'_b(b') - (2\tdel^*\theta,\td\kappa), \quad \gamma_0(\wt r) = \gamma_0(\gdot,\Adot)
  \]
  has an exponentially decaying solution, which by definition of $K'_b(b')$ and $\Xi$ means that, writing $\wt r=(\wt g,\wt A)$,
  \[
    L_b\bigl(\wt g + \chi g'_b(b') + \tcL_{g_b}(\chi V), \wt A + \chi A'_b(b') + \tcL_{A_b}(\chi V) + d(\chi a)\bigr) = 0
  \]
  for suitable sums of generalized modes $V$ and $a$. By uniqueness of solutions of the initial value problem, we conclude
  \[
    \gdot = \wt g + \chi g'_b(b') + \tcL_{g_b}(\chi V), \quad \Adot = \wt A + \chi A'_b(b') + \tcL_{A_b}(\chi V) + d(\chi a),
  \]
  which implies the equality of those parts of the asymptotic expansions of both sides of these two equations which are generalized modes with frequency $\sigma$. This proves the claim.
\end{proof}

\section{Non-linear stability of KNdS black holes}
\label{SecNL}

Following the strategy explained in \S\ref{SubsecIntroStr}, we now put together the mode stability, constraint damping, high energy estimates, and parts of the linear stability argument to prove Theorem~\ref{ThmIntroNL}.

\subsection{Proof of non-linear stability}
\label{SubsecNLProof}

In an iterative argument for proving non-linear stability, solving a linearized gauged Einstein--Maxwell equation globally at each step, we would like to use the change $b'$ of the final black hole parameters suggested by the solution \eqref{EqLinKNdSIVPSol1} of the linearized gauged equation to update the current guess at the final parameters from $b$ to $b+b'$, and we would like to update the gauge modification $(\theta,\kappa)$ in a similar fashion.

Now, the linearized KNdS solution $(g'_{b_0}(b'),A'_{b_0}(b'))$ does not solve the linearized gauged equation by itself; rather, as explained around equation~\eqref{EqLinGaugeG}, one has to add a pure gauge solution $(\cL_{V_{b_0}(b')}g_{b_0},\cL_{V_{b_0}(b')}A_{b_0}+d a_{b_0}(b'))$ in order to correct the failure of $(g'_{b_0}(b'),A'_{b_0}(b'))$ to satisfy the linearized gauge conditions $D_{g_{b_0}}\Ups^E(\cdot)=0$ and $\Ups^M(g_{b_0},\cdot)=0$; since the linearized ungauged Einstein--Maxwell operator annihilates this pure gauge solution, this is equivalent to $(g'_{b_0}(b'),A'_{b_0}(b'))$ solving the linearized \emph{gauged} Einstein--Maxwell system after a change of the gauge source functions. On the other hand, the solution $(g,A)=(g_b,A_b)$ of the ungauged Einstein--Maxwell system certainly satisfies the gauge conditions
\[
  \Ups^E(g)-\Ups^E(g_b)=0, \quad \Ups^M(g_{b_0},A-A_b)=0;
\]
the linearization of the gauge source functions $\Ups^E(g_b)$ and $\Ups^M(g_{b_0},A_{b_0}-A_b)$ in $b$ (around $b=b_0$) thus also yields a change of the gauge source functions after which $(g'_{b_0}(b'),A'_{b_0}(b'))$ solves the linearized gauged Einstein--Maxwell system. Indeed, the equality of the gauge source functions in these two arguments is equivalent to the equations \eqref{EqLinGaugeG}--\eqref{EqLinGaugeA} (for $b=b_0$) defining $V_{b_0}(b')$ and $a'_{b_0}(b')$.

Solving the non-linear Einstein--Maxwell system in the schematic form
\begin{gather*}
  (\Ric+\Lambda)(g_b+\wt g) - \tdel^*(\Ups^E(g_b+\wt g)-\Ups^E(g_b)) - 2 T(g_b+\wt g,d(A_b+\wt A)) = 0, \\
  \delta_{g_b+\wt g}d (A_b+\wt A) - \td\Ups^M(g_b+\wt g,\wt A) = 0,
\end{gather*}
valid for large $t_*$ (and thus ignoring finite-dimensional, compactly supported gauge modifications), one thus expects to be able to subsume \emph{both} the linearized KNdS solution \emph{and} the gauge change (given by $V_{b_0}(b')$ and $a_{b_0}(a')$), coming from the solution of the linearized gauged equation, into a change $b'$ of the asymptotic parameter $b$. (At a transition region, located at finite interval in $t_*$, one will also need to patch together gauge conditions to avoid the need to change the gauge of the Cauchy data at each step of the non-linear iteration).

We now make this precise; part of this computation was already performed in \cite[\S11.2]{HintzVasyKdSStability}. To keep the notation manageable, we introduce
\[
  \Ups^M_g(A) := \Ups^M(g,A)
\]
and recall $\tcL_T V := \cL_V T$ for tensors $T$ and vector fields $V$, so $[\tcL_T,\chi]V=\cL_{\chi V}T-\chi\cL_V T$; further
\begin{equation}
\label{EqNLProofNotation}
\begin{gathered}
  V := V_{b_0}(b'),\ a := a_{b_0}(b'), \\
  g' := g'_{b_0}(b'),\ A' := A'_{b_0}(b'), \\
  g^{\prime\Ups} := g'+\cL_V g_{b_0}, \ A^{\prime\Ups} := A'+\cL_V A_{b_0} + d a,
\end{gathered}
\end{equation}
hence dropping the dependence of these quantities on $b_0$ and $b'$ from the notation. Define the non-linear operator
\begin{equation}
\label{EqNLBabyNonlinearOp}
\begin{split}
  P_0(b,\wt g,\wt A) &= \Bigl( 2\bigl((\Ric+\Lambda)(g) - \tdel^*\bigl(\Ups^E(g)-\Ups^E(g_{b_0,b})\bigr) - 2 T(g, d A)\bigr), \\
  &\qquad\qquad \delta_g d A - \td \Ups^M(g,\wt A) \Bigr),
\end{split}
\end{equation}
where
\begin{equation}
\label{EqNLMetricInterpolate}
  g_{b_0,b} = (1-\chi)g_{b_0} + \chi g_b
\end{equation}
interpolates between the fixed metric $g_{b_0}$ near $\Sigma_0$ and the KNdS metric $g_b$ for late times, and where we have set $g:=g_{b_0,b}+\wt g$ and $A:=A_{b_0,b}+\wt A$. Then the linearization at the RNdS solution $b=b_0$, $\wt g=0$, $\wt A=0$,
\[
  L_{b_0} := D_{0,0}P_0(b_0,\cdot,\cdot),
\]
takes the form
\begin{equation}
\label{EqNLProofLb0}
\begin{split}
  L_{b_0}(\gdot,\Adot) &= \Bigl( 2\bigl(D_g(\Ric+\Lambda)(\gdot) - \tdel^* D_g\Ups^E(\gdot) - 2 D_{g,d A}T(\gdot,d\Adot)\bigr), \\
   & \qquad\qquad D_{g,A}(\delta_{(\cdot)}d(\cdot))(\gdot,\Adot) - \td\Ups^M(g,\Adot)\Bigr).
\end{split}
\end{equation}
where we put
\begin{equation}
\label{EqNLProofNotation2}
  g=g_{b_0},\ A=A_{b_0}.
\end{equation}

We then have:

\begin{lemma}
\label{LemmaNLProofParamChange}
  We use the notation \eqref{EqNLProofNotation}--\eqref{EqNLProofNotation2} and recall the definitions \eqref{EqLinGaugeG}--\eqref{EqLinGaugeA}, as well as \eqref{EqLinGeneratesLinKNdS}. Then
  \begin{equation}
  \label{EqNLProofParamChange}
  \begin{split}
    D_{b_0}P_0(\cdot,0,0)(b') &= L_{b_0}\bigl(\chi g' + \cL_{\chi V}g, \chi A' + \cL_{\chi V}A + d(\chi a)\bigr) \\
     &\qquad\qquad + \bigl(2\tdel^*\vartheta_\chi(b'), \td\varkappa_\chi(b')\bigr) \\
     & = K_{b_0}(b') + \bigl(2\tdel^*\vartheta_\chi(b'), \td\varkappa_\chi(b')\bigr),
  \end{split}
  \end{equation}
  where
  \begin{align*}
    \vartheta_\chi(b') &:= [D_g\Ups^E\circ\tcL_g,\chi]V + [D_g\Ups^E,\chi]g', \\
    \varkappa_\chi(b') &:= [\Ups^M_g\circ\tcL_{A},\chi]V + [\Ups^M_g\circ d,\chi]a + [\Ups^M_g,\chi]A';
  \end{align*}
  these are compactly supported and smooth in $\Omega^\circ$, with $\vartheta_\chi(b')$ a 1-form and $\varkappa_\chi(b')$ a scalar function.
\end{lemma}

Before proving the lemma, we explain how this achieves what we discussed schematically above: the forcing term effecting a (linearized) change in the final black hole parameter (i.e.\ the left hand side in \eqref{EqNLProofParamChange}) is the same as the forcing term generating the linearized gauged KNdS solution $(g^{\prime\Ups},A^{\prime\Ups})$ for large $t_*$ (the first term on the right hand side in \eqref{EqNLProofParamChange}), up to an additional compactly supported gauge modification $(\vartheta_\chi(b'),\varkappa_\chi(b'))$ used to patch up the gauge in the transition region $\supp d\chi$.

We point out that the left hand side in \eqref{EqNLProofParamChange}, which is equal to
\begin{equation}
\label{EqNLProofParamChangeFormula}
  \Bigl( 2\bigl(D_g(\Ric+\Lambda)(\chi g') - 2 D_{g,d A}T(\chi g',\chi A')\bigr), D_{g,A}(\delta_{(\cdot)}d(\cdot))(\chi g',\chi A') \Bigr),
\end{equation}
has compact support in $M^\circ$ not intersecting $\Sigma_0$, as does the first term on the right hand side; the latter property was already used in \S\ref{SecLin}.

\begin{proof}[Proof of Lemma~\ref{LemmaNLProofParamChange}]
  Comparing \eqref{EqNLProofParamChangeFormula} with \eqref{EqNLProofLb0}, we find
  \begin{equation}
  \label{EqNLProofParamChangeFormula2}
    D_{b_0}P_0(\cdot,0,0)(b') = L_{b_0}(\chi g',\chi A') + \bigl(2\tdel^*D_g\Ups^E(\chi g'),\td\Ups^M(g,\chi A')\bigr).
  \end{equation}
  On the other hand, using that $(\cL_{\chi V}g,\cL_{\chi V}A+d(\chi a))$ is a pure gauge term, one computes
  \begin{align*}
    &L_{b_0}\bigl(\chi g' + \cL_{\chi V}g, \chi A' + \cL_{\chi V}A + d(\chi a)\bigr) \\
    &\quad = L_{b_0}(\chi g',\chi A') + \bigl(-2\tdel^*D_g\Ups^E(\tcL_g(\chi V)), -\td\Ups^M_g(\tcL_{A}(\chi V)+d(\chi a))\bigr).
  \end{align*}
  We then manipulate the first component of the last term by using \eqref{EqLinGaugeG} and writing
  \begin{align*}
    -D_g\Ups^E(\tcL_g(\chi V)) &= -[D_g\Ups^E\circ\tcL_g,\chi]V + \chi D_g\Ups^E(g') \\
      & = -\vartheta_\chi(b') + D_g\Ups^E(\chi g'),
  \end{align*}
  and similarly the second component,
  \begin{align*}
    -\Ups^M_g(\cL_A(\chi V)+d(\chi a)) &= -[\Ups^M_g\circ\cL_A,\chi]V - [\Ups^M_g\circ d,\chi]a + \chi\Ups^M_g(A') \\
      & = -\varkappa_\chi(b') + \Ups^M(g,\chi A').
  \end{align*}
  Putting these expressions together, we obtain \eqref{EqNLProofParamChange}.
\end{proof}

More generally, we need to relate changes of the asymptotic black hole parameters to modifications of the gauge at \emph{all} steps in our non-linear Nash--Moser iteration scheme. To accomplish this, denote the linearization of $P_0$ around possibly non-zero $\wt g$ and $\wt A$ by
\begin{align*}
  L_{b,\wt g,\wt A}(\gdot,\Adot) &= D_{\wt g,\wt A}P_0(b,\cdot,\cdot)(\gdot,\Adot) \\
    &= \Bigl( 2\bigl(D_g(\Ric+\Lambda)(\gdot) - \tdel^*D_g\Ups^E(\gdot) - 2 D_{g,d A}T(\gdot,d\Adot)\bigr), \\
    &\qquad\qquad\qquad D_{g,A}(\delta_{(\cdot)}d(\cdot))(\gdot,\Adot) - \wt d D_{g,\wt A}\Ups^M(\gdot,\Adot)\Bigr),
\end{align*}
where we write $g=g_{b_0,b}+\wt g$ and $A=A_{b_0,b}+\wt A$. Then \eqref{EqNLProofParamChangeFormula2} generalizes to
\begin{align*}
  D_b P_0(\cdot,\wt g,\wt A)(b') &= L_{b,\wt g,\wt A}(\chi g'_b(b'),\chi A'_b(b')) \\
    &\quad + \bigl(\tdel^* 2 D_{g_{b_0,b}}\Ups^E(\chi g'_b(b')), \td \Ups^M(g,\chi A'_b(b'))\bigr).
\end{align*}

Combining this with the gauge modifications coming from the linear analysis in \S\ref{SecLin}, we can now describe the finite-dimensional modifications which we will use for the proof of non-linear stability below. As usual, let $B\subset\R^5$ be a small neighborhood of the RNdS parameters $b_0$; let further
\[
  \wt W^s := \bigl\{ (\wt g,\wt A) \in \Hbext^{s,\alpha}(\Omega;S^2\,\Tb^*_\Omega M) \oplus \Hbext^{s,\alpha}(\Omega;\Tb^*_\Omega M) \colon \|\wt g\|_{\Hbext^{14,\alpha}} + \|\wt A\|_{\Hbext^{14,\alpha}} < \eps \bigr\}
\]
for sufficiently small $\eps>0$. Let $\Xi$ denote the (fixed) space of gauge modifications from the linear theory, see \eqref{EqLinKNdSGaugeMod} and the proof of Theorem~\ref{ThmLinKNdS}, and recall the isomorphism \eqref{EqLinKNdSXiIso}. We then define the map
\begin{equation}
\label{EqNLProofModificationMap}
\begin{split}
  z \colon B\times\wt W^{s+2} \times \R^5 \times \R^{N_\Xi} &\to \Hbext^{s,\alpha}(\Omega;S^2\,\Tb^*_\Omega M)\oplus\Hbext^{s,\alpha}(\Omega;\Tb^*_\Omega M) \\
  &\qquad \hra D^{s,\alpha}(\Omega;S^2\,\Tb^*_\Omega M\oplus\Tb^*_\Omega M)
\end{split}
\end{equation}
by
\begin{equation}
\label{EqNLProofModificationMapDef}
\begin{split}
  z(b,(\wt g,\wt A),b',\bfc') &= D_b P_0(\cdot,\wt g,\wt A)(b') - (2\tdel^*\vartheta_\chi(b'),\td\varkappa_\chi(b')) \\
    &\qquad + (2\tdel^*\theta(\bfc'), \td\kappa(\bfc')).
\end{split}
\end{equation}
(The reason for putting two extra derivatives in the space $\wt W^{s+2}$ is due to the fact that $D_b P_0(\cdot,\wt g,\wt A)$ has coefficients with regularity $H^s$ if $(\wt g,\wt A)$ has regularity $H^{s+2}$.) This puts together
\begin{enumerate}
\item the asymptotic parameter change (first term), which for $b=b_0$, $\wt g=0$ and $\wt A=0$ is equal to the compactly supported modification of the range (giving rise to the solution $(g^{\prime\Ups}_b(b'),A^{\prime\Ups}_b(b'))$ for large $t_*$ of the linearized gauge-fixed Einstein--Maxwell equation), given by the first term in \eqref{EqNLProofParamChange}, upon subtracting $(2\tdel^*\vartheta_\chi(b'),\td\varkappa_\chi(b'))$, given by the second term in \eqref{EqNLProofParamChange},
\item the gauge modifications necessitated by the pure gauge resonances in the closed upper half plane, captured by the parameter $\bfc'$ and the map \eqref{EqLinKNdSXiIso}.
\end{enumerate}

We can now state and prove our main theorem, which we stated somewhat informally as Theorems~\ref{ThmIntroNL} and \ref{ThmIntroNLFull} in the introduction:

\begin{thm}
\label{ThmNLKNdS}
  Recall from \eqref{EqKNdSIniSpace} the space $Z^s$ of initial data for the Einstein--Maxwell system with vanishing magnetic charge. Let $(h,k,\bfB,\bfE)\in Z^\infty$ be a smooth initial data set, thus $h$ is a smooth Riemannian metric on $\Sigma_0$, $k$ a symmetric 2-tensor, $\bfE$ and $\bfB$ 1-forms, with $[\star_h\bfB]=0\in H^2(\Sigma_0,\R)$. Suppose $(h,k,\bfB,\bfE)$ is close to the initial data $(h_{b_0},k_{b_0},\bfB_{b_0},\bfE_{b_0})$ (induced by $(g_{b_0},A_{b_0})$ on $\Sigma_0$) of any fixed non-degenerate Reissner--Nordstr\"om--de~Sitter black hole in the topology of $H^{21}(\Sigma_0;S^2 T^*\Sigma_0\oplus S^2 T^*\Sigma_0\oplus T^*\Sigma_0\oplus T^*\Sigma_0)$.
  
  Then there exist KNdS black hole parameters $b\in B$ close to $b_0$ and exponentially decaying tails $(\wt g,\wt A)\in\wt W^\infty$ such that the smooth metric and 4-potential
  \[
    g = g_{b_0,b} + \wt g, \quad A = A_{b_0,b} + \wt A
  \]
  solve the coupled Einstein--Maxwell system
  \[
    \Ric(g) + \Lambda g = 2 T(g,d A), \quad \delta_g d A = 0,
  \]
  and attain the given initial data $(h,k,\bfB,\bfE)$ at $\Sigma_0$. More precisely, using the map $i_b$ from Proposition~\ref{PropKNdSIni}, we have $\gamma_0(g)=i_{b_0}(h,k,\bfB,\bfE)$; and, using the notation of \eqref{EqNLProofModificationMapDef}, there exists $\bfc\in\R^{N_\Xi}$ such that $g$ and $A$ satisfy the gauge conditions
  \begin{gather*}
    \Ups^E(g) - \Ups^E(g_{b_0,b}) - \theta(\bfc) + \vartheta_\chi(b-b_0) = 0, \\
    \Ups^M(g,A) - \Ups^M(g,A_{b_0,b}) - \kappa(\bfc) + \varkappa_\chi(b-b_0) = 0.
  \end{gather*}
\end{thm}

\begin{rmk}
\label{RmkNLMagnetic}
  The cohomological condition on $\bfB$, which is equivalent to the vanishing of the magnetic charge of the perturbed black hole, can be dropped in view of Lemmas~\ref{LemmaBasicDerEMRotation} and \ref{LemmaBasicDerEMRotationCharges}. The assumption on the initial data then is that they are close to the initial data of any fixed non-degenerate RNdS black hole with parameters $b_0=(\bhm,\bfzero,Q_e,Q_m)\in B_m$, see the definition~\eqref{EqKNdSMagParam}; the conclusion is that we have a solution of the initial value problem \eqref{EqBasicDerEMCurv1}--\eqref{EqBasicDerEMCurv2} with asymptotic behavior $g=g_{b_0,b}+\wt g$, $F=F_{b_0,b}+\wt F$, where now $b\in B_m$ is close to $b_0$, and $F_{b_0,b}=(1-\chi)F_{b_0}+\chi F_b$, with $g_b$ and $F_b$ defined in equation~\eqref{EqKNdSMagSols}.
\end{rmk}

\begin{proof}[Proof of Theorem~\ref{ThmNLKNdS}]
  We will use the Nash--Moser iteration scheme as presented by Saint-Raymond \cite{SaintRaymondNashMoser}, using the notation employed in \cite[Theorem~11.1]{HintzVasyKdSStability}. Thus, we define the real Banach spaces
  \begin{align*}
    B^s &= \R^5 \oplus \Hbext^{s,\alpha}(\Omega;S^2\,\Tb^*_\Omega M\oplus\Tb^*_\Omega M) \oplus \R^{N_\Xi}, \\
    \bfB^s &= D^{s,\alpha}(\Omega;S^2\,\Tb^*_\Omega M\oplus\Tb^*_\Omega M).
  \end{align*}
  Next, recalling the map $P_0$ from \eqref{EqNLBabyNonlinearOp}, we define the map $\Phi\colon B^\infty \to \bfB^\infty$,
  \begin{align*}
    \Phi&(b,(\wt g,\wt A),\bfc) = \Bigl(P_0(b,\wt g,\wt A) + \bigl(2\tdel^*(\theta(\bfc)+\vartheta_\chi(b-b_0)), \td(\kappa(\bfc)+\varkappa_\chi(b-b_0))\bigr), \\
      &\qquad\qquad\qquad\qquad \gamma_0(g_{b_0}+\wt g,A_{b_0}+\wt A)-i_{b_0}(h,k,\bfE,\bfB)\Bigr) \\
      & \equiv \Bigl( 2\bigl(\Ric(g)+\Lambda g-2 T(g,d A)-\tdel^*\bigl(\Ups^E(g)-\Ups^E(g_{b_0,b})-\theta(\bfc)-\vartheta_\chi(b-b_0)\bigr)\bigr), \\
      &\qquad\qquad \delta_g d A - \td\bigl(\Ups^M(g,A)-\Ups^M(g,A_{b_0,b})-\kappa(\bfc)-\varkappa_\chi(b-b_0)\bigr), \\
      &\qquad\qquad \gamma_0(\wt g,\wt A)-(i_{b_0}(h,k,\bfE,\bfB)-\gamma_0(g_{b_0},A_{b_0})) \Bigr),
  \end{align*}
  where $g=g_{b_0,b}+\wt g$, $A=A_{b_0,b}+\wt A$, defined in a neighborhood of $(b_0,(0,0),\bfzero)$ in the topology of $B^5$ (which implies that $\wt g$ and $\wt A$ are small in $\cC^2$). Note that our assumptions on the initial data imply that $i_{b_0}(h,k,\bfE,\bfB)-\gamma_0(g_{b_0},A_{b_0})$ is small in $H^{21}$. Now, in the Nash--Moser iteration scheme, we need to solve linearized equations of the form
  \[
    D_{(b,(\wt g,\wt A),\bfc)}\Phi(b',(\wt g',\wt A'),\bfc') = d \equiv (f,r_0,r_1) \in \bfB^\infty,
  \]
  which is equivalent to the initial value problem
  \begin{equation}
  \label{EqNLKNdSLin}
    \begin{cases}
      L_{b,\wt g,\wt A}(\wt g',\wt A') + z(b,(\wt g,\wt A),b',\bfc') = f, \\
      \gamma_0(\wt g',\wt A') = (r_0,r_1).
    \end{cases}
  \end{equation}
  Now for $b=b_0$ and $(\wt g,\wt A)=(0,0)$, and using Lemma~\ref{LemmaNLProofParamChange}, the first equation becomes
  \[
    L_{b_0}(\wt g',\wt A') + K'_{b_0}(b') + \bigl(2\tdel^*\theta(\bfc'), \td\kappa(\bfc')\bigr) = f,
  \]
  which by the construction of the map $z$ can always be solved (with the given initial data) for $(\wt g',\wt A')\in\Hbext^{\infty,\alpha}$, as follows from the proof of linear stability, Theorem~\ref{ThmLinKNdS}. By Theorem~\ref{ThmLinAnaSolvPert} then, the perturbed problem~\eqref{EqNLKNdSLin} can be solved as well if $b$ is close to $b_0$ and $(\wt g',\wt A')$ is small in $\wt W^{14}$. We remark that a priori one needs to allow $b'$, $\bfc'$ and $(\wt g',\wt A')$ to be complex-valued; however, arguing inductively (and noting that the smoothing operators in the Nash--Moser iteration can be chosen so as to preserve real sections), $b$, $\wt g$ and $\wt A$ in \eqref{EqNLKNdSLin} will always be real-valued, so by the same argument as at the end of the linear stability proof in \S\ref{SecLin}, $b'$ etc.\ will in fact be real-valued.
  
  In view of the estimate \eqref{EqLinAnaSolvPertEst}, the solution of \eqref{EqNLKNdSLin} satisfies the tame estimate
  \begin{align*}
    |b'|&+|\bfc'| + \|(\wt g',\wt A')\|_{\Hbext^{s,\alpha}} \\
       &\leq C\bigl( \|d\|_{D^{s+3,\alpha}} + (1+\|(\wt g,\wt A)\|_{\Hbext^{s+6,\alpha}})\|d\|_{D^{13,\alpha}}\bigr)
  \end{align*}
  for $s\geq 10$. (The additional two derivatives on $(\wt g,\wt A)$ are due to the same reason as in the definition of the map \eqref{EqNLProofModificationMap}.) We can therefore appeal to the Nash--Moser theorem \cite[Theorem~11.1]{HintzVasyKdSStability} to conclude the proof.
\end{proof}

\begin{rmk}
\label{RmkNLChargeConverges}
  By Stokes' theorem, the electric charge of the final black hole is equal to the electric charge $Q_e=\frac{1}{4\pi}\int_{\Sph^2}\star_h\bfE$ computed directly from the initial data. However, we point out that the value $Q_e^{(k)}$ of the electric charge at the $k$-th step of the Nash--Moser iteration may well be different from $Q_e$; however, $Q_e$ is of course (necessarily) recovered as the limit of $Q_e^{(k)}$ as $k\to\infty$.
\end{rmk}

\subsection{Initial data}
\label{SubsecNLIni}

We proceed to show how the conformal method can be used to construct initial data sets for the Einstein--Maxwell system in the context of the present paper. The formulation of the conformal method in this context goes back to Isenberg--Murchadha--York \cite{IsenbergMurchadhaYorkConstraintsFields}; see also \cite[\S7]{IsenbergCMCConstraints} for a study of the space of solutions in the constant mean curvature case, and the paper by Isenberg--Maxwell--Pollack \cite{IsenbergMaxwellPollackGluing} on gluing theorems for fairly general matter fields, including electromagnetic fields. We refer the reader to the survey by Bartnik and Isenberg \cite{BartnikIsenbergConstraints} for further references.

Let us denote the initial data induced on a spacelike hypersurface $\Sigma_i$ by a slowly rotating KNdS black hole solution $(g_b,A_b)$, so with $b$ close to $b_0$, by $(h_0,k_0,\bfE_0,\bfB_0)$. Below, we will prove a general theorem allowing us to construct initial data close to the KNdS data $(h_0,k_0,\bfE_0,\bfB_0)$, provided the mean curvature $H_0=\tr_{h_0}k_0$ is small. Observe now that $H_0$ is small indeed if we choose $\Sigma_i$ to be a small spacelike neighborhood of $\{t=0\}$ (in Boyer--Lindquist coordinates in the exterior region) within the maximal analytic extension of a slowly rotating KNdS solution: this follows directly for $\Sigma_i=\{t=0\}$ in the RNdS spacetime, since $\Sigma_i$ is totally geodesic in this case, being the set of fixed points of the isometry $t\mapsto -t$, and by continuity for slowly rotating KNdS black holes. See Figure~\ref{FigNLIniSurface}.

\begin{figure}[!ht]
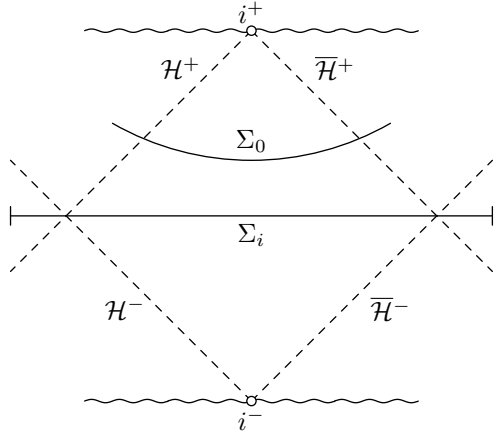

\centering
\inclfig{NLIniSurface}
\caption{The totally geodesic hypersurface $\Sigma_i$ (with boundary) within the maximal analytic extension of a RNdS spacetime. A neighborhood of the Cauchy surface $\Sigma_0$ lies in the closure of the domain of dependence of $\Sigma_i$.}
\label{FigNLIniSurface}
\end{figure}

Denote by
\[
  \cD_{h_0} = 2\delta_{h_0}^* + \frac{2}{3}h_0\delta_{h_0}
\]
the conformal Killing operator.

\begin{thm}
\label{ThmNLIni}
  Let $\Sigma_i$ be a smooth compact connected 3-dimensional manifold with boundary, and suppose $h_0,k_0\in\CI(\Sigma_i;S^2 T^*\Sigma_i)$, with $h_0$ a Riemannian metric, and $\bfE_0,\bfB_0\in\CI(\Sigma_i;T^*\Sigma_i)$ solve the Einstein--Maxwell constraint equations \eqref{EqBasicNLConstraints1}--\eqref{EqBasicNLConstraints2}. Write $k_0=H_0 h_0+Q_0$, where $H_0\in\CI(\Sigma_i)$, and $Q_0\in\CI(\Sigma_i;S^2 T^*\Sigma_i)$ is trace-free with respect to $h_0$.

  Let $s\geq s_0>3/2$. Then there exists an $\eps>0$ such that the following holds under the assumption $\|H_0\|_{H^{s_0}}<\eps$: let $\wt H\in H^s(\Sigma_i)$ and $\wt Q_1\in H^s(\Sigma_i;S^2 T^*\Sigma_i)$, with $\tr_{h_0}\wt Q_1=0$ and $\delta_{h_0}\wt Q_1=0$; let further $\wt\bfE,\wt\bfB\in H^s(\Sigma_i;T^*\Sigma_i)$, and assume $\delta_{h_0}\wt\bfE=0$ and $\delta_{h_0}\wt\bfB=0$. Let
  \[
    \Psi=(\wt H,\wt Q_1,\wt\bfE,\wt\bfB).
  \]
  Then if $\|\Psi\|_{H^{s_0}}<\eps$, there exist $\varphi=1+\psi$, $\psi\in H^{s+2}(\Sigma_i)$, and $V_2\in H^{s+1}(\Sigma_i;T\Sigma_i)$ such that
  \begin{equation}
  \label{EqNLIniData}
  \begin{gathered}
    h = \varphi^4 h_0,\quad k=(H_0+\wt H)h + \varphi^{-2}(Q_0+\wt Q_1+\cD_{h_0}V_2), \\
    \bfE = \varphi^{-2}(\bfE_0+\wt\bfE), \quad \bfB = \varphi^{-2}(\bfB_0+\wt\bfB)
  \end{gathered}
  \end{equation}
  solves the Einstein--Maxwell constraint equations. Furthermore, there exists a constant $C$ such that the estimate
  \begin{equation}
  \label{EqNLIniEst}
    \|\psi\|_{H^{s+2}} + \|V_2\|_{H^{s+1}} \leq C\|\Psi\|_{H^s}
  \end{equation}
  holds. In particular, we have (for another constant $C$)
  \begin{equation}
  \label{EqNLIniEst2}
    \|(h,k,\bfE,\bfB)-(h_0,k_0,\bfE_0,\bfB_0)\|_{H^s} \leq C\|\Psi\|_{H^s}.
  \end{equation}
\end{thm}

A neighborhood of the Cauchy surface $\Sigma_0$ used in Theorem~\ref{ThmNLKNdS} lies in the closure of the domain of dependence of $\Sigma_i$ (choosing $\Sigma_i$ large enough), hence one can evolve the initial data provided by Theorem~\ref{ThmNLIni} until $\Sigma_0$. Thus, Theorem~\ref{ThmNLIni} can be used to construct a sizeable set (or using the generalization discussed momentarily in fact a full neighborhood) of initial data close to slowly rotating KNdS data which, for $s=21$ and sufficiently small $\|\Psi\|_{H^{21}}$, can be used as initial data for Theorem~\ref{ThmNLKNdS}.

One can also replace $h_0$ by another metric $h_0'$ of class $H^s$, which is a perturbation of the given $h_0$ in the topology of $H^{s_0}(\Sigma_i;S^2 T^*\Sigma_i)$, and similarly perturb $Q_0$ to a trace-free symmetric 2-tensor $Q_0'$ with respect to $h_0'$, and perturb $\bfE_0$ and $\bfB_0$ to divergence-free (with respect to $h_0'$) 1-forms $\bfE_0'$ and $\bfB_0'$, and still obtain a solution of the constraint equations as in \eqref{EqNLIniData}, with $h_0$ etc.\ replaced by $h_0'$ etc. It is easy to see that \emph{every} initial data set $(h,k,\bfE,\bfB)$ satisfying \eqref{EqNLIniEst2} for $s=s_0$ and with $D_{s_0}$ sufficiently small arises via this generalization of the construction \eqref{EqNLIniData} (in fact with $\varphi\equiv 1$).

\begin{proof}[Proof of Theorem~\ref{ThmNLIni}]
  The divergence on 1-forms obeys the transformation rule $\delta_{\varphi^4 h_0}=\varphi^{-6}\delta_{h_0}\varphi^2$; the same formula is valid on symmetric trace-free 2-tensors. Furthermore, on 2-forms, the  Hodge star operator transforms as  $\star_{\varphi^4 h_0}=\varphi^{-2}\star_{h_0}$. Using the transformation rule for the scalar curvature
  \[
    R_{\varphi^4 h_0} = \varphi^{-4}(R_{h_0}+8\varphi^{-1}\Delta_{h_0}\varphi),
  \]
  with $\Delta_{h_0}\geq 0$, the constraint equations \eqref{EqBasicNLConstraints1} for the data \eqref{EqNLIniData} are equivalent to the system
  \begin{align}
    \label{EqNLIniEq1}
    \begin{split}
     &P_1(\varphi,V_2;\Psi):=\Delta_{h_0}\varphi + \frac{R_{h_0}}{8}\varphi - \frac{1}{8}|Q_0+\wt Q_1+\cD_{h_0}V_2|_{h_0}^2\varphi^{-7} \\
      &\qquad\qquad + \frac{1}{4}\bigl(3|H_0+\wt H|^2-\Lambda-|\bfE_0+\wt\bfE|_{h_0}^2 - |\bfB_0+\wt\bfB|_{h_0}^2\bigr)\varphi^{-5} = 0,
    \end{split} \\
    \label{EqNLIniEq2}
    \begin{split}
      &\delta_{h_0}\cD_{h_0}V_2 = P_2(\varphi;\Psi) := -\delta_{h_0}Q_0 - 2\varphi^6 d(H_0+\wt H) \\
      &\hspace{12em}+2\star_{h_0}\bigl((\bfB_0+\wt\bfB)\wedge(\bfE_0+\wt\bfE)\bigr);
    \end{split}
  \end{align}
  the constraints \eqref{EqBasicNLConstraints2} for $\bfE$ and $\bfB$ are automatically satisfied. \emph{We will henceforth drop the subscript `$h_0$' from the notation.} By assumption on the data set $(h_0,k_0,\bfE_0,\bfB_0)$, if the perturbation data are trivial, then $\varphi\equiv 1$ and $V_2\equiv 0$ solve this system, so $P_1(1,0;0)=0$ and $P_2(1;0)=0$.
  
  Similarly to the proof of \cite[Proposition~11.4]{HintzVasyKdSStability}, we embed $\Sigma_i$ into a closed, connected 3-manifold $\Sigma$, and denote by $\Sigma_e=\ol{\Sigma\setminus\Sigma_i}$ the closure of the region appended to $\Sigma_i$. Extending $h_0$ to a smooth Riemannian metric on $\Sigma$, likewise extending the symmetric 2-tensor $k_0$ and the 1-forms $\bfE_0$ and $\bfB_0$ \emph{in any smooth fashion}, and denoting the thus extended quantities by $h_0$ etc.\ still, we have
  \[
    P_1(1,0;0)=:p_1\in\CIdot(\Sigma_e),\quad P_2(1;0)=:p_2\in\CIdot(\Sigma_e;T^*\Sigma_e).
  \]
  Extend $\wt H$ to a function in $H^s(\Sigma_e)$, denoting the extension by $\wt H$ still; one can arrange for the extension to have an $H^{s_0}$ and $H^s$ norm which is bounded by a constant times the corresponding norms on $\Sigma_i$. Similarly, extend the other components of $\Psi$.

  We aim to find finite-dimensional subspaces $\cZ_1\subset\CIdot(\Sigma_e)$, $\cZ_2\subset\CIdot(\Sigma_e;T\Sigma_e)$, for which we can solve
  \begin{equation}
  \label{EqNLIniEqExt}
    P_1(\varphi,V_2;\Psi) = p_1 + z_1,\quad \delta \cD V_2=(P_2(\varphi;\Psi)-p_2)+z_2,
  \end{equation}
  for $z_j\in\cZ_j$, $j=1,2$, and $\varphi=1+\psi$, $\psi\in H^{s_0+2}(\Sigma_e)$, $V_2\in H^{s_0+1}(\Sigma_e;T\Sigma_e)$, with control of higher Sobolev norms in terms of higher norms of the data as in \eqref{EqNLIniEst}, provided $\wt H$ etc.\ are small in $H^{s_0}$. Since $p_j$ and $z_j$ vanish identically in $\Sigma_i$, the restriction of $\varphi$ and $V_2$ to $\Sigma_i$ yields a solution of \eqref{EqNLIniEq1}--\eqref{EqNLIniEq2} satisfying the estimate \eqref{EqNLIniEst} on $\Sigma_i$. The estimate \eqref{EqNLIniEst2} is then an immediate consequence of the estimate \eqref{EqNLIniEst} and the definition \eqref{EqNLIniData}.

  We proceed to solve the system \eqref{EqNLIniEqExt} for $\varphi=1+\psi$. We first discuss the second equation: the operator $\delta \cD$ is elliptic and self-adjoint, with $(\ran\delta \cD)^\perp=\ker\delta \cD$ given by conformal Killing vector fields. The arguments in \cite[\S2]{BeigChruscielSchoenKIDs} imply that any non-trivial conformal Killing field has support in all of $\Sigma$, in particular its support intersects $\Sigma_e$ in an open set. Thus, there exists a subspace $\cZ_2$ as above, of dimension equal to $\dim\ker\delta \cD$, with $H^s(\Sigma_e;T\Sigma_e)=\ran_{H^{s+2}}\delta \cD \oplus \cZ_2$ for all $s\in\R$; this gives a bounded linear solution operator
  \[
    (S_2,S_\cZ) \colon H^s(\Sigma_e;T\Sigma_e) \to H^{s+2}(\Sigma_e;T\Sigma_e)\oplus\cZ_2, \quad f \mapsto (S_2 f,S_\cZ f)=(V_2,z_2),
  \]
  for the equation $\delta \cD V_2=f+z_2$, with $s\in\R$ arbitrary. It therefore remains to solve the first equation in \eqref{EqNLIniEqExt} in the form
  \begin{equation}
  \label{EqNLIniEqRedux}
    \wt P_1(\psi;\Psi) := P_1\bigl(1+\psi,S_2(P_2(1+\psi;\Psi)-p_2);\Psi\bigr) - p_1 = z_1,
  \end{equation}
  with $z_1$ in the space $\cZ_1$ which we will specify momentarily. Let us write
  \[
    \wt P_1(\psi;0) = \wt L_1\psi + \wt P_1'(\psi),
  \]
  with
  \[
    \wt L_1=\Delta+\frac{R}{8}+\frac{7}{8}|Q_0|^2-\frac{5}{4}(3|H_0|^2-\Lambda-|\bfE_0|^2-|\bfB_0|^2)
  \]
  self-adjoint and differential, while
  \begin{align*}
    \wt P_1'(\psi) &= \frac{1}{8}\bigl(|Q_0|^2-\bigl|Q_0+\cD S_2(-2((1+\psi)^6-1)d H_0)\bigr|^2\bigr)(1+\psi)^{-7} \\
      &\qquad + \frac{1}{8}|Q_0|^2(1-7\psi-(1+\psi)^{-7}) \\
      &\qquad + \frac{1}{4}(3|H_0|^2-\Lambda-|\bfE_0|^2-|\bfB_0|^2)((1+\psi)^{-5}-1+5\psi)
  \end{align*}
  captures the non-local (pseudodifferential) part of $\wt P_1$ coming from $S_2$ --- which is \emph{small} due to the smallness of $H_0$ --- and the remaining non-linear terms in $\wt P_1$, which vanish quadratically at $\psi=0$; concretely, we have the estimate
  \begin{equation}
  \label{EqNLIniEqOperatorEst1}
    \|\wt P_1'(\psi)\|_{H^{s_0}} \leq C\bigl(\|H_0\|_{H^{s_0}}\|\psi\|_{H^{s_0}} + \|\psi\|_{H^{s_0}}^2\bigr)
  \end{equation}
  since $H^{s_0}$ is an algebra for $s_0>3/2$. Now, since $\wt P_1(0;0)=0$, equation \eqref{EqNLIniEqRedux} is equivalent to
  \begin{equation}
  \label{EqNLIniEqRedux2}
    \wt L_1\psi = -\wt P_1'(\psi) - (\wt P_1(\psi;\Psi)-\wt P_1(\psi;0)) + z_1,
  \end{equation}
  with the second term satisfying the estimate
  \begin{equation}
  \label{EqNLIniEqOperatorEst2}
    \|\wt P_1(\psi;\Psi)-\wt P_1(\psi;0)\|_{H^{s_0}} \leq C\|\psi\|_{H^{s_0}}\|\Psi\|_{H^{s_0}}.
  \end{equation}
  Now, since the unique continuation principle applies to elements of $(\ran\wt L_1)^\perp$, we can find $\cZ_1\subset\CIdot(\Sigma_e)$ (with dimension equal to $\dim\ker\wt L_1$) and a solution operator
  \[
    S_1 \colon H^s(\Sigma_e) \to H^{s+2}(\Sigma_e) \oplus \cZ_1,\quad f\mapsto(\psi,z_1),
  \]
  for the equation $\wt L_1\psi=f+z_1$. Identify $\cZ_1\cong\C^N$ for $N=\dim\cZ_1$. We can then solve the equation~\eqref{EqNLIniEqRedux2} for $\psi\in H^{s_0+2}(\Sigma)$ and $z_1\in\cZ_1$ using the Banach fixed point theorem applied to the map
  \begin{equation}
  \label{EqNLIniFixedPoint}
    (\psi,z_1) \mapsto S_1\bigl(-\wt P_1'(\psi)-(\wt P_1(\psi;\Psi)-\wt P_1(\psi;0))+z_1\bigr),
  \end{equation}
  which for sufficiently small $\eps>0$ and under the assumption $\|H_0\|_{H^{s_0}}+\|\Psi\|_{H^{s_0}}<\eps$ maps the ball $\{(\psi,z_1)\colon\|\psi\|_{H^{s_0+2}}+|z_1|\leq\eps\}$ into itself; this follows from the estimates \eqref{EqNLIniEqOperatorEst1} and \eqref{EqNLIniEqOperatorEst2}. Using similar estimates, one can show that the map \eqref{EqNLIniFixedPoint} is a contraction (reducing $\eps>0$ if necessary). This proves the existence of a solution of the system \eqref{EqNLIniEq1}--\eqref{EqNLIniEq2}. The higher regularity estimate~\eqref{EqNLIniEst} follows by elliptic regularity.
\end{proof}

\appendix
\section{Review of b-geometry and b-analysis}
\label{SecB}

We only give a very brief account of the aspects of b-geometry and b-analysis which are used in the present paper; for a  more detailed overview, we refer the reader to \cite[Appendix~A]{HintzVasyKdSStability} as well as to Melrose's book \cite{MelroseAPS} on the subject.

Fix a smooth connected $(n+1)$-dimensional manifold $M$ with non-empty boundary $\pa M$. We denote by $\Vb(M)\subset\cV(M)$ the space of \emph{b-vector fields}, smooth vector fields on $M$ which are tangent to $\pa M$. Away from $\pa M$, these are simply ordinary smooth vector fields. Near the boundary, with $(\tau,x^1,\ldots,x^n)$ denoting \emph{adapted} local coordinates near a point in $\pa M$, namely with $\pa M$ given by the vanishing of $\tau$, a b-vector field $V$ takes the form
\[
  V = a \tau\pa_\tau + \sum_{j=1}^n b_j \pa_{x^j},\quad a,b_j\in\CI(M).
\]
Correspondingly, b-vector fields are the space of sections of a natural vector bundle $\Tb M\to M$, called \emph{b-tangent bundle}, which over the interior $M^\circ$ is naturally isomorphic to the standard tangent bundle, and which near the boundary in the above coordinates has the basis $\{\tau\pa_\tau,\pa_{x^1},\ldots,\pa_{x^n}\}$; in particular, $\tau\pa_\tau$ is non-vanishing at $\tau=0$ as a b-vector field. One can check that $\tau\pa_\tau$ is in fact well-defined, i.e.\ independent of the choice of adapted local coordinates. The space $\Diffb^*(M)$ of \emph{b-differential operators} is the universal enveloping algebra of $\Vb(M)$, thus elements of $\Diffb^m(M)$ are finite linear combinations (with $\CI(M)$ coefficients) of products of up to $m$ b-vector fields. If $E,F\to M$ are two smooth vector bundles, one can more generally define $m$-th order b-differential operators $\Diffb^m(M;E,F)$ mapping $\CI(M;E)$ into $\CI(M;F)$, e.g.\ using local trivializations of $E$ and $F$.

The dual bundle $\Tb^*M$ of $\Tb M$, called the \emph{b-cotangent bundle}, is correspondingly spanned by $\frac{d\tau}{\tau},dx^1,\ldots,dx^n$; here $\frac{d\tau}{\tau}$ is smooth (and non-degenerate) as a b-1-form up to $\tau=0$. A smooth \emph{b-metric} $g$ on $M$ is then a smooth section of the second symmetric tensor power $S^2\,\Tb^*M$; in local coordinates as above, this means that
\[
  g = g_{00}\,\frac{d\tau^2}{\tau^2} + 2 g_{0j}\,\frac{d\tau}{\tau}\otimes_s dx^j + g_{ij}\,dx^i\otimes_s dx^j,\quad g_{\mu\nu}\in\CI(M).
\]
If $M$ arises as the compactification of a manifold $M^\circ$ without boundary as in equation~\eqref{EqKNdS0MfComp}, then a smooth b-metric on $M$ is asymptotically stationary on $M^\circ$ in the following sense: letting $t:=-\log\tau$, we have $\frac{d\tau}{\tau}=-dt$, and a smooth function $a\in\CI(M)$, having a Taylor expansion in powers of $\tau$, has a Taylor expansion on $M^\circ$ in powers of $e^{-t}$; thus, $g=g_0+\wt g$ with
\[
  g_0=g_{00}(0,x)\,dt^2 - 2 g_{0j}(0,x)\,dt\otimes_s dx^j + g_{ij}(0,x)\,dy^i\otimes_s dy^j,\quad \wt g=\cO(e^{-t}),
\]
approaches the stationary metric $g_0$ exponentially fast as $t\to\infty$. Conversely, if $M^\circ$ is equipped with a metric $g$ approaching a stationary metric exponentially fast at some rate $\alpha>0$, then $g$ extends to be a smooth b-metric on $M$ plus an error term (in general non-smooth) of size $\tau^\alpha$. In the case of interest in the present paper, this remainder term will be \emph{conormal}, or more generally lie in a weighted b-Sobolev space which we discuss further below.

We further have the \emph{b-differential} $\bdiff$, acting between sections of the exterior powers $\Lambda^k\,\Tb^*M$; they are defined by extension of the usual exterior differential $d$ from $M^\circ$; thus, acting on functions, one has
\[
  \bdiff a = (\tau\pa_\tau a)\frac{d\tau}{\tau} + (\pa_{x^j}a)dx^j,
\]
and in general $\bdiff\in\Diffb^1(M;\Lambda^k\,\Tb^*M,\Lambda^{k+1}\,\Tb^*M)$.

On $M$, we naturally have the b-density bundle $\Omegab^1(M)$, with local trivialization induced by $|\frac{d\tau}{\tau}dx^1\ldots dx^n|$; fixing a nowhere vanishing b-density $\nu$ on $M$, this allows us to define the $L^2$ space $L^2_\bl(M;\nu)\equiv L^2(M;\nu)$. We drop the density $\nu$ from the notation from now on. (For compact $M$, different choices of $\nu$ lead to equivalent norms.) For integer $k\geq 0$ and real $\alpha\in\R$, we then define the \emph{weighted b-Sobolev space}
\begin{align*}
  \Hb^{k,\alpha}(M) &= \{ u\in \tau^\alpha L^2_\bl(M) \colon V_1\ldots V_j u\in\tau^\alpha L^2_\bl(M), \\
  &\qquad\qquad 0\leq j\leq k,\ V_\ell\in\Vb(M),\ 1\leq\ell\leq j\}.
\end{align*}
For compact $M$, $\Hb^{k,\alpha}(M)$ can be endowed with a Hilbert space structure by means of a finite collection of b-vector fields which span $\Tb_p M$ over every $p\in M$; the norms for any two such collections are equivalent. For non-integer $s\in\R$, the space $\Hb^{s,\alpha}(M)$ is defined using duality, that is $\Hb^{s,\alpha}(M)^*=\Hb^{-s,-\alpha}(M)$, and interpolation. We point out that the definition of the space $\Hb^{s,\alpha}(M)$ as a Hilbert space for $M$ compact does \emph{not} require the choice of a metric. Elements of the space $\Hb^{\infty,\alpha}(M)=\bigcap_{s\in\R}\Hb^{s,\alpha}(M)$ are called \emph{conormal} (with respect to $L^2_\bl$); for $M$ compact, this space carries a natural Fr\'echet space structure. Near a point on $\pa M$, using coordinates $(\tau,x^1,\ldots,x^n)$ as above, and letting $t=-\log\tau$, the space $\Hb^{s,\alpha}(M)$ is locally the same (as a Hilbert space, up to equivalence of norms) as the space $e^{-\alpha t}H^s(M^\circ)$, where the Sobolev space on $M^\circ$ is defined by testing with products of the vector fields $\pa_t,\pa_{x^1},\ldots,\pa_{x^n}$.

Suppose next that $\Omega\subset M$ is a non-empty open subset of $M$. One can then define the space of \emph{supported distributions} $\dot\sD(\Omega)$ as the space of distributions $u\in\sD(M)=\CIdot(M;\Omega^1 M)^*$ with $\supp u\subset\Omega$. (The same definition applies for $M$ without boundary.) We then define $\Hbsupp^{s,\alpha}(\Omega)=\Hb^{s,\alpha}(M)\cap\dot\sD(\Omega)$; this thus consists of elements of $\Hb^{s,\alpha}(M)$ which are supported in $\bar\Omega$. The space of \emph{extendible distributions}, $\bar\sD(\Omega)$, is equal to the space of restrictions $u|_\Omega$ for $u\in\sD(M)$; we likewise define $\Hbext^{s,\alpha}(\Omega)=\Hb^{s,\alpha}(M)|_{\Omega}$, with the natural (quotient) norm. Thus, elements of $\Hbext^{s,\alpha}(\Omega)$ automatically have extensions to $\Hb^{s,\alpha}(M)$ (with the same norm).

If $E\to M$ is a smooth vector bundle, weighted b-Sobolev spaces $\Hb^{s,\alpha}(M;E)$ are defined using local trivializations of $E$; for $\Omega\subset M$ as above, one can likewise define spaces $\Hbsupp^{s,\alpha}(M;E)$ and $\Hbext^{s,\alpha}(M;E)$ of supported and extendible sections of $E$ over $\Omega$.

\section{Explicit expressions for the mode stability analysis}
\label{SecFormulas}

In this appendix, we list the explicit formulas for a number of functions arising in \S\ref{SecMS}; we recall that the quantities $x$, $y$, $z$, $m$ were defined in \eqref{EqMS2Scmxyz}, $H$ in \eqref{EqMS2ScalphaH}, $\tilde c$ in \eqref{EqMS2ScCpmWtc}, $a_+$ in \eqref{EqMS2ScApmBpm}, and $\hat c$ in \eqref{EqMS2Sc0HatC}.

The expressions for the functions used in equation~\eqref{EqMS2ScMaster1} are then:

\begin{align}
\label{EqFormulasVPhiFPhi}
  V_\Phi &= \frac{\mu}{r^2 H^2}\bigl(9 x^3 - 9(2 y+6 z-m)x^2 + (72 z^2-8(4 m-3)z + 3 m^2)x \\
      &\qquad\qquad + 8(9 x z-12 z^2-m z)y - 32 z^3 + 24 m z(z+1) + m^2(m+2)\bigr), \nonumber\\
  F_\Phi &= -\frac{8 Q\mu}{r^3 H^2}\bigl(2(3 x-8 z)y+2 x z-3 x^2+6 x+m(m+4)\bigr). \nonumber
\end{align}
The functions appearing in equation~\eqref{EqMS2ScRecoverXYZ} are given as follows:
\begin{align}
\label{EqFormulasPX0thruPZ}
  P_{X 0} &= \bigl(6(4 z+m)x-64 z^2-16 m z\bigr)y + 27 x^3 - 24(5 z-m)x^2 \\
    &\qquad + \bigl(152 z^2-2(35 m-12)z+3 m(3 m+2)\bigr)x - 64 z^3 + 48 m z^2 \nonumber\\
    &\qquad - 8 m(m-2) z + 2 m^2 (m+2), \nonumber\\
  P_{X 1} &= 2(4 z+m)y + 9 x^2 - (16 z - 5 m + 6)x + 8 z^2 - 6 m z - 4 m, \nonumber\\
  P_{X\cA} &= -4(4 z+m)y - 18 x^2 + 4(8 z-m+6)x \nonumber\\
    &\qquad - 16 z^2 + 4(m-4)z + 2 m(m+6), \nonumber\\
  P_{Y 0} &= 2\bigl(18 x^2-3(28 z-m)x+96 z^2-8 m z\bigr)y + 9 x^3 - 6(10 z-m)x^2 \nonumber\\
    &\qquad + \bigl(120 z^2-2(11 m-12)z + 3 m(m+2)\bigr)x \nonumber\\
    &\qquad - 64 z^3 + 16(m-4)z^2 - 8 m(m+2)z, \nonumber\\
  P_{Y 1} &= 2(6 x-12 z+m)y+3 x^2-(12 z+m+6)x + 8 z^2 + 2(m+8)z, \nonumber\\
  P_{Y\cA} &= -4(6 x-12 z+m)y - 6 x^2 + 4(6 z-m)x \nonumber\\
    &\qquad - 16 z^2 + 4(m-4)z - 2 m(m+2), \nonumber\\
  P_Z &= (-6 x+16 z)y+3 x^2+(-2 z+3 m)x-(4 m+8)z-2 m. \nonumber
\end{align}
The functions used in equation~\eqref{EqMS2Sc0Coeffs} are:
\begin{align}
\label{EqFormulasPXthruPA}
  P_X &= \bigl(9 x-36 z+3(12 y-6-m)\bigr)x - 6(12 z-m)y + 6(4 z+m+8)z, \\
  P_Y &= -3(9 x-16 z+5 m-6)x-6(4 z+m)y - 6(4 z-3 m)z + 12 m, \nonumber\\
  P_\cA &= -8(\tilde c+m r)z. \nonumber
\end{align}
The functions appearing in equation~\eqref{EqMS2Sc0PsipmToXYA} take the following form:
\begin{align}
\label{EqFormulasPXpthruPym}
  P_{X+} &= -8 z H \mu - 3\bigl(9 x-8(5 z-m)\bigr)x^2 - 3 m(3 m+2+2 y)x \\
   &\qquad + 2(35 m-12 y-12)x z+64 z^3 - 8 z^2(19 x-8 y+6 m) \nonumber\\
   &\qquad + 8 m(2 y+m-2)z - 2 m^2(m+2) \nonumber\\
   &\quad - \Bigl(4 z-\frac{\tilde c}{r}\Bigr)\bigl((9 x-16 z+2 m-12)x+2(m+4 z)y \nonumber\\
   &\qquad\qquad\qquad\qquad + 2(4 z-m+4)z - m(m+6)\bigr), \nonumber\\
  P_{X-} &= 81(4 z-x)x^4 - 18 x^4(4 m+3\hat c) \nonumber\\
    &\qquad - 3 x^3\bigl(144 z^2-2 z(36\hat c+23 m)+m(16\hat c+3(3 m+2)+6 y)\bigr) \nonumber\\
    &\qquad - 2 x^2\bigl(-96 z^3+8 z^2(18\hat c+m)-2 m z(35\hat c+5 m+24) \nonumber\\
    &\qquad\qquad\qquad\qquad + 3 m(-8 y z+(m+2)(m+3\hat c)+2\hat c(y-2))\bigr) \nonumber\\
    &\qquad -4 m^2(m+2)\hat c\,x-64 x z^3(m-2\hat c)-16 m(6\hat c+4 y-m+4)x z^2 \nonumber\\
    &\qquad + 8 m x z\bigl(m(3\hat c+m+6)+(4\hat c-2 m)y\bigr), \nonumber\\
  Q_+ &= 12 x^2+\bigl(3(y-9 z-3)+5 m+2\hat c\bigr)x + 16 z^2 \nonumber\\
    &\qquad - 2 z(3 m+\hat c-4) + 2(m+\hat c)y-2(2 m+\hat c), \nonumber\\
  Q_- &= -\frac{16 r a_+ \mu Q}{r^2}+(9 x-16 z+5 m-6)x+8 z^2+2(m+4 z)y-6 m z-4 m, \nonumber\\
  P_{Y+} &= -9 x^3+6 x^2(3 y+z+m-\hat c+3) + 16 m z^2 + 4 x z(-12 y-6 m+\hat c) \nonumber\\
    &\qquad +x\bigl(12(2 m+\hat c)y+m(7 m+12)+12\hat c\bigr) - 16 z\bigl((3 m+2\hat c)y-m-1\bigr) \nonumber\\
    &\qquad + 4 m^2 y+2(m+2)^2(-2 z+m+\hat c-2) + 16 m z + 8(m-\hat c+2), \nonumber\\
  P_{Y-} &= 81 x^4+54 x^3(6 y+4 z+m+\hat c+1) + 9 x^2\bigl(16 z^2-2 z(24 x-m+8\hat c) \nonumber\\
    &\qquad -6 x+24(\hat c-4 z)y+m(6 y+3 m+4\hat c+6)\bigr) + 6 x\bigl(16 z^2(6 y-3 m+\hat c) \nonumber\\
    &\qquad +2 z(24(m-2\hat c)y+m(m-3\hat c-12)) + 3 m(2 y+m+2)\hat c\bigr) \nonumber\\
    &\qquad + 8z\bigl(24 m z^2-6 z(4(3 m-2\hat c)y+m(m-4)) \nonumber\\
    &\qquad\qquad\qquad\qquad +3 m^2(2 y+m-\hat c+2) - 12 m \hat c\bigr). \nonumber
\end{align}

\bibliographystyle{alpha}
\newcommand{\etalchar}[1]{$^{#1}$}

\end{document}